\documentclass{biometrika-noname}

\usepackage[hang,flushmargin]{footmisc}
\usepackage{amssymb,amsmath,amsfonts,mathtools,bm,bbm,mathrsfs} 	% math fonts/symbols
\usepackage[dvipsnames]{xcolor} 								% for \textcolor
\usepackage[hidelinks]{hyperref} 								% for hyperlinks
\usepackage{graphicx} 											% for \includegraphics
\usepackage{caption} 											% for figure and table captions
\usepackage{verbatim}
\usepackage[misc]{ifsym} 		
\usepackage{mathabx}
\usepackage{cancel}
\usepackage{float}
\usepackage{lineno} % check number of lines per page
\usepackage{graphicx}
\usepackage{subcaption}
\usepackage{enumitem}
\usepackage{microtype}
\usepackage{booktabs}
\usepackage{wrapfig}

\usepackage[plain,noend,linesnumbered]{algorithm2e}
\usepackage{etoolbox}

\makeatletter
\renewcommand{\algocf@captiontext}[2]{#1\algocf@typo. \AlCapFnt{}#2} % text of caption
\def\@algocf@capt@plain{top}
\renewcommand{\algocf@makecaption}[2]{%
  \addtolength{\hsize}{\algomargin}%
  \sbox\@tempboxa{\algocf@captiontext{#1}{#2}}%
  \ifdim\wd\@tempboxa >\hsize%     % if caption is longer than a line
    \hskip .5\algomargin%
    \parbox[t]{\hsize}{\algocf@captiontext{#1}{#2}}% then caption is not centered
  \else%
    \global\@minipagefalse%
    \hbox to\hsize{\box\@tempboxa}% else caption is centered
  \fi%
  \addtolength{\hsize}{-\algomargin}%
}
\patchcmd{\@algocf@start}{-1.5em}{0pt}{}{}
\makeatother

\SetAlCapHSkip{0pt}

\newcommand{\AlgoStrut}{\rule[-0.5ex]{0pt}{0.5ex}}
\theoremstyle{plain}
\theoremsymbol{}
\theoremheaderfont{\hskip\parindent\normalfont\scshape}
\theorembodyfont{\itshape}
\newtheorem{conjecture}{Conjecture}
\makeatletter
\newtoks\remarkwrap@saveeverypar

\newenvironment{remarkwrap}[1][]{%
  \par\addvspace{\theorempreskipamount}%
  \ifcsname c@remark\endcsname
    \refstepcounter{remark}%
    \def\remarkwrap@num{\theremark}%
  \else
    \refstepcounter{theorem}%
    \def\remarkwrap@num{\thetheorem}%
  \fi
  \def\remarkwrap@note{#1}%
  \remarkwrap@saveeverypar=\everypar
  \def\remarkwrap@printhead{%
    \noindent\hspace*{\theoremindent}%
    {\itshape Remark \remarkwrap@num%
      \if\relax\detokenize{\remarkwrap@note}\relax.\else\ (\remarkwrap@note).\fi}%
    \normalfont\space
    \let\remarkwrap@printhead\relax
  }%
  \everypar{\remarkwrap@printhead\the\remarkwrap@saveeverypar}%
  \ignorespaces
}{%
  \par
  \everypar=\remarkwrap@saveeverypar
  \addvspace{\theorempostskipamount}%
}
\makeatother

\makeatletter
\newcommand\newtarget[2]{\Hy@raisedlink{\hypertarget{#1}{}}#2}
\makeatother

\newcommand{\E}{\mathbb E}								% expectation
\newcommand{\V}{\mathrm{Var}}							% variance
\renewcommand{\P}{\mathbb{P}}							% probability
\newcommand{\R}{\mathbb{R}}								% reals
\newcommand{\N}{\mathbb{N}}								% naturals
\newcommand{\indicator}{\mathbbm 1}						% indicator
\newcommand{\independent}{{\ \perp \! \! \! \perp\ }}	% independent
\newcommand{\iidsim}{\stackrel{\mathrm{i.i.d.}}{\sim}} 	% i.i.d. distributed
\newcommand{\indsim}{\stackrel{\mathrm{ind}}{\sim}}		% independently distributed
\newcommand{\convp}{\overset p \rightarrow}             % convergence in probability
\newcommand{\convd}{\overset d \rightarrow}             % convergence in distribution
\newcommand{\cF}{\mathcal F}						% sigma field
\newcommand{\TPR}{\operatorname{TPR}}
\newcommand{\TPP}{\operatorname{TPP}}
\newcommand{\spl}{{\operatorname{split}}}
\newcommand{\mean}{{\operatorname{mean}}}
\newcommand{\opt}{{\operatorname{opt}}}

\newcommand{\aug}{{\textnormal{aug}}}
\newcommand{\bonus}{{\operatorname{bonus}}}
\newcommand{\insamp}{{\operatorname{in-samp}}}

\newcommand{\learn}{{\operatorname{learn}}}

\newcommand{\score}{{\operatorname{score}}}

\newcommand{\cH}{\mathcal H}
\newcommand{\cR}{\mathcal R}

\newcommand{\Var}{\operatorname{Var}}

\DeclareMathOperator{\Cov}{Cov}

\begin{document}

%% Here are the title, author names and addresses
\title{Power of masking methods for adaptive testing \\ in a multivariate normal means problem}

\author{A. Chakraborty*}
\affil{Department of Statistics, Columbia University,\\ New York, NY 10027, U.S.A.
\email{ac4662@columbia.edu}}

\author{J. Lee*}
\affil{Department of Statistics and Data Science, University of Pennsylvania,\\ Philadelphia, PA 19104, U.S.A. \email{junulee@wharton.upenn.edu}}

\author{\and E. Katsevich}
\affil{Department of Statistics and Data Science, University of Pennsylvania,\\ Philadelphia, PA 19104, U.S.A.
\email{ekatsevi@wharton.upenn.edu}}

\maketitle

\renewcommand{\thefootnote}{}\footnotetext{*Equal contribution}\renewcommand{\thefootnote}{\arabic{footnote}}

\begin{abstract}

Many large-scale testing procedures learn signal structure from the data to boost power. Direct data reuse can inflate Type-I error (“double dipping”), so a common remedy is masking: withholding some information during learning and using it for testing. Sample splitting masks by withholding observations for testing, while null augmentation (e.g., knockoffs or full-conformal outlier detection) masks by appending null samples or variables and withholding their identities until testing. In many settings, little is known about how the power of masking methods compares across mechanisms, across tuning choices, or against more data-efficient non-masking alternatives. We study these questions in a stylized two-groups multivariate normal means model with an unknown signal direction learned from the data. Within this testbed, we develop a transparent, unified set of asymptotic power expressions for three parallel methods differing in masking choices: a sample splitting method, a full-conformal-style null augmentation method, and an oracle in-sample benchmark. Our main findings are: (1) the augmentation method is more powerful than the splitting method with matched tuning; (2) the power-optimal number of null samples for the augmentation method is a vanishing fraction of the number of tests, in which case its power approaches that of the in-sample benchmark; and (3) for a tractable approximation to the augmentation method, the optimal number of null samples scales as the square root of the number of tests, with empirical evidence suggesting a similar scaling for the method itself. These results characterize masking-induced power trade-offs in a tractable model and suggest qualitative lessons for other settings.

\end{abstract}


\section{Introduction}

\subsection{Masking: a versatile remedy for double dipping} \label{sec:masking}

Consider testing $m$ hypotheses, with the goal of controlling some notion of Type-I error. Often, the alternatives of interest are highly composite, e.g. multivariate or nonparametric. Examples include testing $m$ genetic variants for association with a multivariate vector of phenotypes \citep{Urbut2019, Urbut2023}, testing $m$ images for being outliers with respect to an inlier distribution accessed through samples \citep{Bates2021,Marandon2024,Lee2025}, and testing $m$ predictor variables for conditional association with a response variable, given the other predictors \citep{CetL16}. In all of these cases, omnibus tests have low power due to the high complexity of the alternative space. To address this issue, a common strategy is the following two-step procedure: 1) \textit{learn} some aspect of the alternative distributions, often by leveraging structure and/or borrowing strength across hypotheses, and 2) \textit{test} the hypotheses using this learned information. Except in special cases like nuisance-only learning \citep{Shah2018}, however, direct reuse of the data for learning and testing can inflate Type-I error due to double dipping.

One of the most prevalent approaches to avoiding double dipping in this context is \textit{masking}: withholding part of the information in the data at the learning step and then leveraging this information for valid testing. For example, consider the holdout randomization test (HRT; \cite{Tansey2018}) of conditional association between variables $X_1, \dots, X_m$ and a response $Y$. HRT \textit{masks} the data by withholding a subset of observations, \textit{learns} a predictive model $\hat f$ on the masked data, and then \textit{tests} each variable $j$ based on the contribution of $X_j$ to the predictive accuracy of $\hat f$ on the withheld observations. Another example is full-conformal outlier detection \citep{Lee2025}. This approach \textit{masks} the data by appending $\tilde m$ null (inlier) samples and withholding their identities, \textit{learns} a scoring function to distinguish inliers from outliers, and \textit{tests} based on the ranks of the original samples' scores relative to those of the null samples.

HRT and full-conformal outlier detection exemplify two common masking mechanisms: \textit{sample splitting} and \textit{null augmentation}, respectively. Sample splitting partitions the observations into two groups, one for learning and one for testing. It has been employed for variable selection \citep{Wasserman2009,Tansey2018}, adaptive multiple testing \citep{Wasserman2006a,Rubin2006b} and causal inference \citep{Bekerman2025}. Null augmentation appends $\tilde m$ real or synthetic null observations or variables to the data prior to learning, withholding which observations or variables are null until testing. Nulls can be added in a common pool shared among hypotheses (\textit{one-to-many} augmentation), or individually for each hypothesis (\textit{one-to-one} augmentation), relying on different exchangeability relationships for calibration. Conformal outlier detection \citep{Marandon2024,Lee2025} and variable selection with knockoffs \citep{BC15,CetL16} exemplify one-to-many and one-to-one null augmentation, respectively. Other masking mechanisms exist, such as hypothesis-wise splitting \citep{Ignatiadis2016,Liu2020}. Likewise, there are other data modification schemes \citep{Xing2019, Dai2023a} not falling within our definition of masking.

\subsection{Questions about masking power and existing progress} \label{sec:questions}

While the Type-I error control of masking methods is well understood, fundamental theoretical questions about their power properties have received less attention:
\begin{itemize}[labelwidth=0pt,leftmargin=*,align=left]
\item[\textbf{Q1: Which masking mechanism is more powerful?}] How do sample splitting and null augmentation (one-to-many and one-to-one) compare?
\item[\textbf{Q2: How much masking gives optimal power?}] Sample splitting and one-to-many null augmentation both have parameters controlling the effective amount of information withheld during learning: the splitting proportion and the number of null data points appended, respectively. What are power-optimal settings for these parameters?
\item[\textbf{Q3: How much power does masking sacrifice?}] Masking obviates the need to characterize the statistical properties of the learning procedure, making it compatible with arbitrary black-box procedures. However, it also uses only partial information for learning and/or testing, which can impact power. Compared to more tailored solutions using the data more efficiently, what is the price in power of the versatility of masking? 
\end{itemize}
These theoretical questions are motivated by practical considerations: masking mechanism choice, tuning, and whether to mask at all. These choices can substantially affect power, yet are often made heuristically. Answering Q1--Q3 does not by itself yield implementable decision rules, e.g., due to dependence on unknown parameters. Rather, theory can clarify the trade-offs governing these choices and identify the quantities that would need to be estimated for practical application. There has been some progress on addressing Q1--Q3 in specific contexts, which we review next.

For Q1, we are aware of no work comparing sample splitting to null augmentation. \citet{Weinstein2023} empirically compared knockoffs (one-to-one augmentation) to counting knockoffs \citep{Weinstein2017} (one-to-many augmentation), finding the latter to be more powerful. Limiting power expressions for both were obtained in \citet{Weinstein2017, Weinstein2023}, but not compared, potentially due to their complicated form.

For Q2, we are aware of one work shedding light on the optimal sample splitting proportion. In a causal inference context, \citet{Bekerman2025} derive an explicit asymptotic power formula for their sample splitting method, from which the optimal splitting proportion can be derived. For one-to-many null augmentation, \citet{Yang2021,Marandon2024} discuss heuristics for selecting the number of null samples $\tilde m$. % Both works recommend $\tilde m \asymp m$ based on arguments involving the granularity of empirical $p$-values. 
Finally, \citet{Weinstein2017} derive the asymptotic power of counting knockoffs. The dependence on the limiting ratio $\tilde m / m$ is complex, and the optimal choice was not analyzed. None of these works explore the trade-off between the quality of the learning step (decreasing with $\tilde m$) and the granularity of the $p$-values (increasing with $\tilde m$).

% Our results suggest that this recommendation can be suboptimal in certain settings. 

For Q3, the cost of masking has been studied for knockoffs. Several works have compared the power of knockoffs in the linear model, mostly with lasso-based test statistics, to that of non-masking benchmarks \citep{Weinstein2017, Wang2020b, Weinstein2023, Ke2021a}. The cost of masking was found to range from negligible to substantial, depending on factors including the knockoff statistic and the data-generating model. By contrast, to our knowledge, the cost of masking has not been quantified in other contexts, including one-to-many null augmentation or sample splitting.

% we are aware of one work that indirectly addresses the cost of masking, without framing the question in these terms \citep{Wang2020b}. In this work, a power comparison is conducted between knockoffs and the conditional randomization test (CRT; \cite{CetL16}) in a variable selection context. Knockoffs is a masking method based on null augmentation, while the CRT operates on the full data without masking. \citet{Wang2020b} find that the CRT is more powerful than knockoffs in their setting, due in part to the doubling of the aspect ratio incurred by appending null variables. We provide complementary insights on a knockoffs-style method in our setup.

Some of the above questions have been studied in the context of single testing, which is not the focus of this work, including Q2 \citep{Moran1973a,Cox1975a,Kim2020, Kim2024} and Q3 \citep{Wasserman2020b}. Masking can be viewed as a form of conditional inference (masking methods are usually valid conditionally on the learned information), and is therefore related to conditional post-selection inference \citep{Kuchibhotla2022}. However, the latter class of problems is beyond the scope of this work.

\subsection{Our contributions}

Despite the above progress, the theoretical understanding of masking power remains fragmented and incomplete. To address this gap, we study Q1--Q3 in a stylized testbed that isolates the power effects of masking while remaining tractable enough to yield explicit and interpretable power expressions.

We work in a two-groups multivariate normal means problem where alternative means are drawn from a one-dimensional subspace whose direction $\bm v$ is unknown \citep{Yang2021}, a simplified variant of the multi-outcome genetic association testing setting \citep{Urbut2019}. This setting captures the core problem structure from Section~\ref{sec:masking} in a canonical multiple-testing testbed \citep{Genovese2002} with a simple structured alternative. In this setting, we analyze three methods that \textit{learn} an alternative direction $\hat{\bm v}$ on a portion of the data (based on either a scaled sample mean or top eigenvector), \textit{score} each hypothesis by projecting a potentially different portion of the data in the direction of $\hat{\bm v}$, \textit{calibrate} these scores against a null distribution to obtain $p$-values, and then \textit{adjust} these $p$-values for multiplicity by applying the Benjamini-Hochberg (BH) procedure. In particular, we consider \textbf{Split BH} (sample splitting), \textbf{BONuS} (a non-iterative variant of the one-to-many null augmentation method proposed by \citet{Yang2021}), and \textbf{In-sample BH} (an oracle non-masking benchmark that learns and tests on the full data). The only differences among the methods are those dictated by the choice of whether and how to mask, setting up direct comparisons that facilitate addressing Q1 and Q3. 

We \textbf{derive the asymptotic powers of all three methods in a unified framework} that mirrors their common structure (Section~3). We then \textbf{analyze a tractable approximation to BONuS power} in the regime $\tilde m \ll m$, identified in Section 3 as most powerful, yielding finer-grained insight into the optimal number of null samples $\tilde m$ (Section~4). Finally, we \textbf{empirically compare a one-to-one null augmentation method to Split BH and BONuS} (Appendix~\ref{sec:knockoffs}). Our findings are as follows.

\paragraph{Q1 (The choice of masking mechanism).} We find that asymptotic power is driven by the effective proportions of the data allocated for learning and scoring, but not calibration. While Split BH uses partial information for learning and scoring, BONuS uses partial information for learning but full information for scoring. As a result, BONuS is asymptotically more powerful than Split BH when both are tuned to equalize the quality of the learned direction $\hat{\bm v}$. In an empirical comparison with a one-to-one null-augmentation analog, we find that BONuS can be more powerful, consistent with \citet{Weinstein2023}, and that the observed gap depends on the choice of knockoff statistic.

\paragraph{Q2 (The amount of masking).} For Split BH, we find that the optimal splitting proportion depends on unknown parameters only through the ``learning curve'' mapping the splitting proportion to the limiting alignment between $\hat{\bm v}$ and $\bm v$. This identifies a quantity whose estimation may facilitate a future data-driven method for choosing the splitting proportion. For BONuS, we find that the optimal number of null samples satisfies $\tilde m_{\text{opt}} \ll m$. To quantify the rate, we find for a tractable approximation to BONuS power that $\tilde m_{\text{opt}} \asymp \sqrt{m}$, and conjecture with empirical support that the same is true for BONuS. We derive the leading constant in this approximation as well, but leave the estimation of unknown parameters in this expression to future work. These results complement and refine the heuristic guidance of \citet{Yang2021,Marandon2024}.

\paragraph{Q3 (The cost of masking).} We quantify the cost of masking at the learning and scoring stages, and how these contribute to the overall power loss relative to In-sample BH. Since Split BH incurs losses at both stages, it remains less powerful than In-sample BH even when optimally tuned. By contrast, BONuS incurs a loss only at the learning stage, and this loss vanishes asymptotically when $\tilde m \ll m$. Thus, within our testbed, one-to-many null augmentation can asymptotically approach the power of the oracle non-masking benchmark while retaining the convenience of masking. These results complement existing work on the cost of masking for one-to-one null augmentation \citep{Weinstein2017, Wang2020b, Weinstein2023, Ke2021a}.

\paragraph{} The conclusions above are derived within our testbed, but this setting bears similarities to many of those that motivate our work (Section~\ref{sec:masking}). We therefore hypothesize that several of the above conclusions extend---at least qualitatively---beyond this paper's setting; we formulate three such conjectures in the Discussion (Section~\ref{sec:discussion}). Furthermore, our proofs suggest a route to establishing results for the methods we analyze in other settings. We prove ``master theorems'' (Appendix~\ref{sec:master-theorems}) for the power of Split BH, BONuS, and In-sample BH under conditions on their limiting score distributions, which are induced by the data-generating model, learning procedure, and test statistic. Extending the present analysis to other choices of these three elements can therefore be achieved by verifying the conditions of our master theorems. Taken together, these results provide a framework for broader analyses beyond our testbed, and we hope that the resulting insights will ultimately inform the application of masking methods in other settings.

\section{Problem setup and methods compared}
\label{sec:overview}

\subsection{Problem setup} \label{sec:problem-setup} 

% We consider testing multivariate normal means $\bm \theta_j \in \R^d$ for $j \in [m]$. The goal is to control the FDR at level $q$ for any fixed setting of $(\bm \theta_j)_{j=1}^m$, but we assess power under an empirical Bayes two-groups model where $\bm \theta_j$ are drawn from a mixture distribution \citep{Efron2008,Yang2021}. 

For each hypothesis $j \in [m]$, we observe a $d$-dimensional summary statistic $\bm X_j$, equal to the scaled average of $n$ underlying observations (Appendix~\ref{sec:multi-observation-model}), and assume
\begin{equation}
\bm X_{j} \indsim N(\bm \theta_j, \bm I_d), \quad \bm \theta_j \in \R^d, \quad j=1,\ldots,m.
\label{eq:data-generating}
\end{equation}
Our goal is to test $H_{0j}: \bm \theta_j = \bm 0$ for each $j \in [m]$. We compare procedures returning rejection sets $\widehat{\mathcal R} \subseteq [m]$ that target FDR control at level $q$ for each fixed setting of $(\bm \theta_j)_{j=1}^m$. We evaluate the power of these procedures with respect to a two-groups model \citep{Efron2008,Yang2021} where $\bm \theta_j$ are drawn from a mixture distribution
\begin{equation}
\bm \theta_j \iidsim (1-\gamma) \cdot \delta_{\bm 0} + \gamma \cdot \Lambda, \quad j = 1, \ldots, m,
\label{eq:two-groups}
\end{equation}
where $\gamma \in (0,1]$ is the non-null proportion and $\Lambda$ is the alternative prior. Denoting by $\mathcal H_0 \equiv \{j \in [m]: \bm \theta_j = \bm 0\}$ and $\mathcal H_1 \equiv \{j \in [m]: \bm \theta_j \neq \bm 0\}$ the sets of null and alternative hypotheses, respectively, we define the true positive proportion as $\TPP(\widehat{\mathcal R}) \equiv |\widehat{\mathcal R} \cap \mathcal H_1|/|\mathcal H_1|$. We quantify the asymptotic power of a method as the in-probability limit of its TPP in a suitable scaling regime with respect to the distributions~\eqref{eq:data-generating} and~\eqref{eq:two-groups}.

The alternative prior $\Lambda$ represents a structure that can be learned from the data to boost power (recall Section~\ref{sec:masking}). We consider two choices for $\Lambda$, corresponding to cases where $\Lambda$ is a point mass or a Gaussian distribution supported on a one-dimensional subspace \citep{Yang2021}:
\begin{equation}
\Lambda \equiv \delta_{h\bm v} \quad \text{or} \quad \Lambda \equiv N(\bm 0, h^2 \bm v \bm v^T), \quad \text{where} \quad h \geq 0, \ \bm v \in \R^d, \ \|\bm v\|_2 = 1.
\label{eq:priors}
\end{equation}
In both cases, the quantities $\bm v$ and $h$ capture the direction and magnitude of the alternative departure from the null, respectively. Given $\Lambda$, the optimal test of $H_{0j}$ is based on the log likelihood ratio statistic
\begin{equation}
T_{\Lambda}(\bm X_{j}) \equiv \log \frac{\int_{\R^d}\phi_d(\bm X_{j} - \bm \theta)d\Lambda(\bm \theta)}{\phi_d(\bm X_{j})},
\label{eq:likelihood-ratio-statistic}
\end{equation}
where $\phi_d$ is the PDF of $N(\bm 0, \bm I_d)$ \citep{Yang2021}. An empirical Bayes strategy for testing in the presence of the unknown distribution $\Lambda$ is to pool the data across $j$ and learn an approximation $\hat \Lambda$. If $d$ grows with $m$, however, $\hat \Lambda$ may be inconsistent, and reusing the data to learn $\hat \Lambda$ and test with $T_{\hat \Lambda}$ may create a double dipping issue. This makes our setting a natural testbed for masking methods, which we introduce next.

\subsection{Methods compared}
\label{sec:methods}

We compare Split BH (Algorithm~\ref{alg:split-bh}), BONuS (Algorithm~\ref{alg:bonus}), and In-sample BH (Algorithm~\ref{alg:in-sample-bh}). All three methods fit into a common framework (Figure~\ref{fig:common-paradigm}), where the data $\bm X \in \R^{m \times d}$ are manipulated into datasets $\bm X^\learn$ and $\bm X^\score$ (Steps 1-2) to learn $\hat \Lambda$ via an arbitrary learning procedure $L$ (Step 3) and to construct a score based on $T_{\hat \Lambda}$~\eqref{eq:likelihood-ratio-statistic} for each hypothesis (Step 4), respectively. A null distribution is then constructed for the resulting scores (Step 5), based on which finite-sample-valid $p$-values are computed (Step 6), which are passed through BH correction to obtain the rejection set $\widehat{\mathcal R}$ (Step 7). The methods differ only in how they define $\bm X^\learn$ and $\bm X^\score$ (Steps 1-2), and how they construct the null distribution for calibration (Step 5); Split BH and BONuS employ masking while In-sample BH does not. We elaborate on these method-specific details below. 

\begin{figure}[h!]
\centering
\includegraphics[width = 0.7\textwidth]{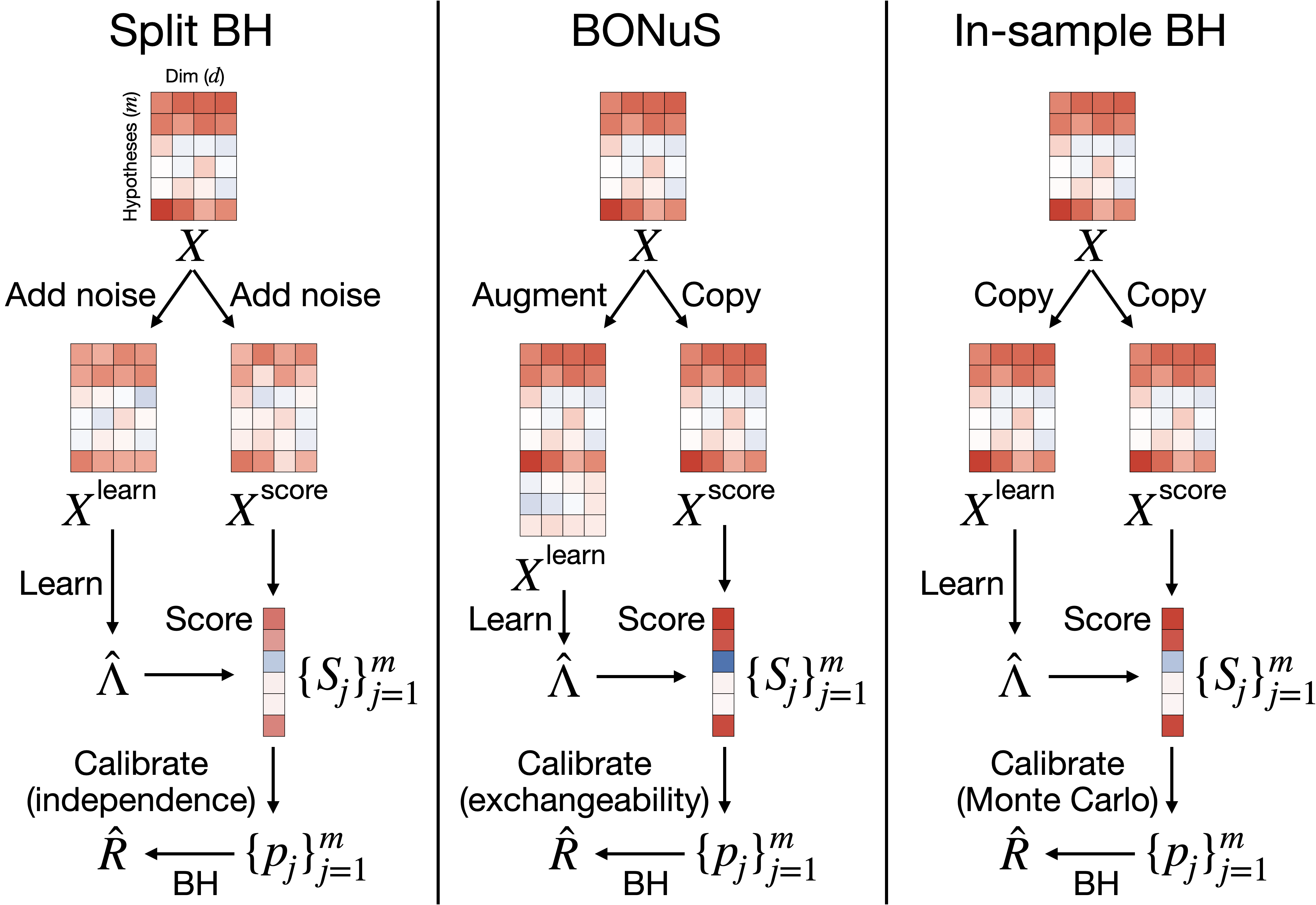}
\caption{Three methods compared in a two-groups multivariate normal means setting: Split BH (sample splitting), BONuS (null augmentation), and In-sample BH (a non-masking oracle benchmark). The colored squares represent example numeric values.}
\label{fig:common-paradigm}
\end{figure}

\begin{figure}[h]
\begin{minipage}[t]{0.49\textwidth}
\begin{algorithm}[H]
\LinesNumbered
\setlength{\algomargin}{0em}
\setlength{\algoheightrule}{0pt}
\setlength{\algotitleheightrule}{0pt}
\setlength{\interspacetitleruled}{0pt}
\setlength{\interspacealgoruled}{0pt}
\SetAlgoSkip{}
\SetInd{0em}{0em}
\caption{Split BH}
\label{alg:split-bh}
\KwIn{Splitting proportion $\smash{\pi_\spl} \in (0,1)$}
\setcounter{AlgoLine}{-1}
\AlgoStrut$\tilde{\bm X}_{k} \sim N(\bm 0, \bm I_d)$, $k = 1, \ldots, m$\;
\AlgoStrut$\bm X^\learn \gets \sqrt{\smash[b]{\pi_\spl}}\bm X + \sqrt{\smash[b]{1-\pi_\spl}}\bm {\tilde X}$\;
\AlgoStrut$\bm X^\score \gets \sqrt{\smash[b]{1-\pi_\spl}}\bm X - \sqrt{\smash[b]{\pi_\spl}}\bm {\tilde X}$\;
\AlgoStrut\(\hat \Lambda \gets L(\bm X^\learn)\)\;
\AlgoStrut$S_j \gets T_{\hat \Lambda}(\bm X_j^\score)$\;
\AlgoStrut$\bar G_{0m}(t) \gets \mathbb P[T_{\hat \Lambda}(\bm X_j^\score) \geq t \mid \hat \Lambda, \bm \theta_j = \bm 0]$\;
\AlgoStrut$p_j \gets \bar G_{0m}(S_j)$\;
\AlgoStrut\KwRet{\(\widehat{\mathcal R} = \textnormal{BH}(\{p_j\}_{j = 1}^m)\)}
\end{algorithm}
\end{minipage}
\hfill
\begin{minipage}[t]{0.51\textwidth}
\begin{algorithm}[H]
\LinesNumbered
\setlength{\algomargin}{0em}
\setlength{\algoheightrule}{0pt}
\setlength{\algotitleheightrule}{0pt}
\setlength{\interspacetitleruled}{0pt}
\setlength{\interspacealgoruled}{0pt}
\SetAlgoSkip{}
\SetInd{0em}{0em}
\caption{BONuS}
\label{alg:bonus}
\KwIn{Number of null samples $\tilde m \in \N$}
\setcounter{AlgoLine}{-1}
\AlgoStrut$\tilde{\bm X}_{k} \sim N(\bm 0, \bm I_d)$, $k = 1, \ldots, \tilde m$\;
\AlgoStrut$\bm X^\learn \gets [\bm X; \tilde{\bm X}]_\Pi$, \ \ $\Pi \sim \text{Sym}([m + \tilde m])$\;
\AlgoStrut$\bm X^\score \gets \bm X$\;
\AlgoStrut\(\hat \Lambda \gets L(\bm X^\learn)\)\;
\AlgoStrut$S_j \gets T_{\hat \Lambda}(\bm X_j^\score)$\;
\AlgoStrut$\bar G_{0m}(t) \gets \tfrac{1}{1 + \tilde m}(1 + \sum_{k = 1}^{\tilde m}\indicator(T_{\hat \Lambda}(\tilde{\bm X}_k) \geq t))$\;
\AlgoStrut$p_j \gets \bar G_{0m}(S_j)$\;
\AlgoStrut\KwRet{\(\widehat{\mathcal R} = \textnormal{BH}(\{p_j\}_{j = 1}^m)\)}
\end{algorithm}
\end{minipage}
\end{figure}

\paragraph{Split BH.} Split BH uses a fraction $\pi_\spl \in (0,1)$ of the observations for learning and the remainder for scoring. Recall that the observations are left implicit in the data-generating model~\eqref{eq:data-generating}. At the level of the scaled averages $\bm X_j$, sample splitting can be equivalently represented via a \textit{data fission} operation \citep{Neufeld2024b, Leiner2025}, which adds Gaussian noise to the data to create two datasets $\bm X^\learn, \bm X^\score \in \R^{m \times d}$ (Steps 0-2), each with weaker signal (Appendix~\ref{sec:multi-observation-model}). Conditionally on $(\bm \theta_j)_{j \in [m]}$, $\bm X^\learn$ and $\bm X^\score$ are independent, each drawn from the distribution~\eqref{eq:data-generating}, but with means scaled by $\sqrt{\pi_\spl}$ and $\sqrt{1-\pi_\spl}$, respectively. Given the (conditional) independence between $\bm X^\learn$ and $\bm X^\score$, the null distribution for $T_{\hat \Lambda}(\bm X_j^\score)$ is constructed by conditioning on $\hat \Lambda$ (Step 5).

\paragraph{BONuS.}
BONuS augments the original data $\bm X$ with $\tilde m \in \N$ null samples $\tilde{\bm X}_k$ and then scrambles the rows with a random permutation $\Pi$ (Steps 0-1) to obtain $\bm X^\learn = [\bm X; \tilde{\bm X}]_\Pi \in \R^{(m + \tilde m) \times d}$ ($[\bm A; \bm B]$ denotes vertically stacking matrices $\bm A$ and $\bm B$ and $\bm C_\Pi$ denotes permuting the rows of $\bm C$ by $\Pi$). On the other hand, the scoring data $\bm X^\score$ are simply the original data $\bm X$ (Step 2). The null distribution for $T_{\hat \Lambda}(\bm X_j^\score)$ is constructed based on the empirical distribution of the scores of the null samples $T_{\hat \Lambda}(\tilde{\bm X}_k)$, which are exchangeable with the null scores of the original data (Step 5). The resulting empirical $p$-values are closely related to conformal $p$-values \citep{Bates2021,Mary2021a,Marandon2024}. We note that, even though this method differs from that of \citet{Yang2021} in that it is non-iterative, we still call it ``BONuS'' for simplicity. Furthermore, \citet{Yang2021} did not represent their method in terms of BH applied to empirical $p$-values, but this representation is equivalent \citep[Proposition 8]{Lee2024}.

\paragraph{In-sample BH.} 

\begin{wrapfigure}[10]{r}{0.44\textwidth}
\vspace{-0.16in}                    % small pull-up (tune or drop)
\begin{minipage}[t]{0.44\textwidth}
\begin{algorithm}[H]
\LinesNumbered
\caption{In-sample BH}
\label{alg:in-sample-bh}
\KwIn{($\bm \theta_j)_{j = 1}^m$, $j_0$ s.t. $\bm \theta_{j_0} = \bm 0$}
\AlgoStrut$\bm X^\learn \gets \bm X$\;
\AlgoStrut$\bm X^\score \gets \bm X$\;
\AlgoStrut\(\hat \Lambda \gets L(\bm X^\learn)\)\;
\AlgoStrut$S_j \gets T_{\hat \Lambda}(\bm X_j^\score)$\;
\AlgoStrut$\bar G_{0m}(t) \gets \mathbb P[T_{\hat \Lambda}(\bm X_{j_0}) \geq t \mid (\bm \theta_j)_{j=1}^m]$\;
\AlgoStrut$p_j \gets \bar G_{0m}(S_j)$\;
\AlgoStrut\KwRet{\(\widehat{\mathcal R} = \textnormal{BH}(\{p_j\}_{j = 1}^m)\)}
\end{algorithm}
\end{minipage}
\end{wrapfigure}
In-sample BH uses the entire data $\bm X$ for learning and scoring (Steps 1-2). The null distribution $\smash{\mathbb P[T_{\hat \Lambda}(\bm X_{j}) \geq t \mid (\bm \theta_j)_{j=1}^m]}$ (Step 5) is complex due to the dependence between $\smash{\hat \Lambda}$ and $\bm X_j$ induced by direct data reuse. We approximate it via Monte Carlo: for each $b \in [B]$, we draw a sample from the data-generating model~\eqref{eq:data-generating}, compute $\hat \Lambda$, and compute $T_{\hat \Lambda}(\bm X_{j_0})$ for a null index $j_0$. Finally, we extract an empirical survival function of these $B$ sampled test statistics analogously to Step 5 of Algorithm~\ref{alg:bonus}.

\section{Comparing the asymptotic powers of the three methods} \label{sec:asymptotic-power-comparison}

We compare the asymptotic powers of the three methods under the following high-dimensional asymptotic regime:
\begin{equation}
h, \gamma \text{ fixed}, \quad m, d, \tilde m \rightarrow \infty \quad \text{with} \quad \frac{d}{m} \rightarrow c \geq 0, \quad \frac{m}{m + \tilde m} \rightarrow \pi_\aug \in [0,1].
\label{eq:asymptotic-regime}
\end{equation}
The ratio $\tfrac{m}{m + \tilde m}$ is the fraction of original samples in the augmented $\bm X^\learn = [\bm X; \tilde{\bm X}]_\Pi$, so we view it as the ``effective fraction of the data used for learning,'' analogous to $\pi_\spl$. This section begins with an overview (Section~\ref{sec:overview-of-asymptotic-power-results}) followed by results for the two cases of the prior $\Lambda$~\eqref{eq:priors}: point mass prior (Section~\ref{sec:point-mass-prior}) and subspace prior (Section~\ref{sec:subspace-prior}).

\subsection{Overview of asymptotic power results} \label{sec:overview-of-asymptotic-power-results}

We derive the power of three methods under two choices of prior in a unified framework (Figure~\ref{fig:theory-schematic}), which pinpoints how methodological differences translate to power differences. Our framework involves understanding the properties of five key objects computed by each method (Figure~\ref{fig:common-paradigm}): $\bm X^\learn$, $\bm X^\score$,  $\hat{\bm v}$, $\bm S$, and $\widehat{\mathcal R}$. Note that $\hat{\bm v}$ stands in place of $\hat \Lambda$ because it is the relevant portion of the prior to learn under the two priors considered~\eqref{eq:priors}.

\begin{figure}[h!]
\centering
\includegraphics[width = \textwidth]{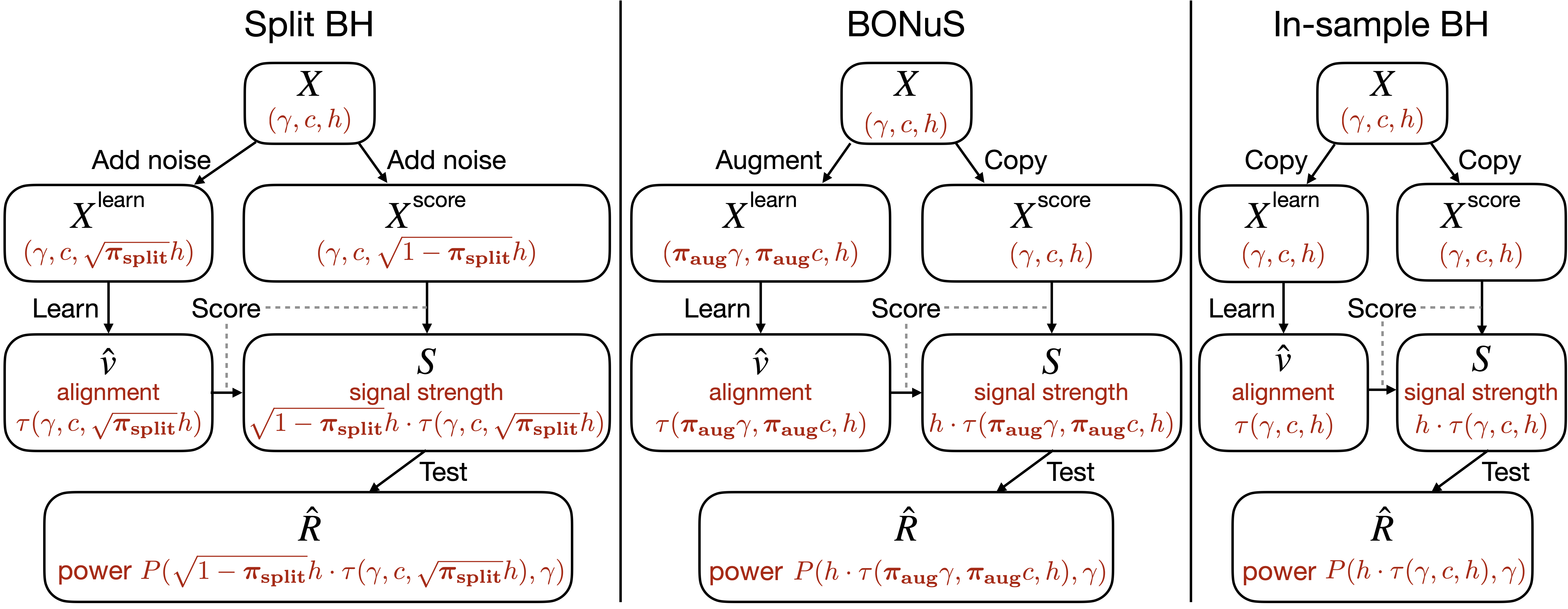}
\caption{A framework for dissecting the asymptotic power of the three methods under two prior choices, mirroring the structures of the methods themselves (Figure~\ref{fig:common-paradigm}).}
\label{fig:theory-schematic}
\end{figure}

\paragraph{$\bm X^\learn$.} Split BH and BONuS mask the input data $\bm X$ prior to learning. As \citet{Weinstein2017} did in a different setting, we express these masking operations as transformations of the three parameters governing the asymptotic behavior of the data~\eqref{eq:asymptotic-regime}: the non-null proportion $\gamma$, aspect ratio $c$, and signal strength $h$. It is easy to verify that splitting weakens the signal strength by a factor of $\sqrt{\pi_\spl}$ while null augmentation decreases the non-null proportion and aspect ratio by a factor of $\pi_\aug$:
\begin{equation}
\text{splitting} : (\gamma, c, h) \mapsto (\gamma, c, \sqrt{\pi_\spl}h), \quad \text{augmentation} : (\gamma, c, h) \mapsto (\pi_\aug \gamma, \pi_\aug c, h).
\label{eq:masking-parameter-transformations}
\end{equation}
By contrast, In-sample BH leaves the input data unchanged prior to learning.

\paragraph{$\bm X^\score$.} Only Split BH uses masked data $\bm X^\score$ for scoring; the other two methods use the original data $\bm X$. The masking operation for Split BH weakens the signal strength in the scoring data by a factor of $\sqrt{1 - \pi_\spl}$.

\paragraph{$\hat{\bm v}$.} We measure the quality of the learned direction $\hat{\bm v}$ via its limiting alignment with the true alternative direction $\bm v$. This alignment is a function $\tau$ of the three parameters governing $\bm X^\learn$. This function depends on the specific choices of prior and learner. In both cases, the quality of the learned directions $\hat{\bm v}$ degrades as $\pi_\spl$ and $\pi_\aug$ decrease, with In-sample BH achieving the highest limiting alignment among the three methods.

\paragraph{$\bm S$.} The scores $S_j$ are constructed based on $\hat{\bm v}$ and $\bm X^\score$ (Figure~\ref{fig:common-paradigm}). For both priors, the score for the $j$th hypothesis is a function of $\hat{\bm v}^\top \bm X^\score_j$, i.e., the projection of the $j$th vector onto the learned direction. Asymptotically, this quantity's distribution under the null and the alternative differs by a \textit{signal strength} $\mu$ (a normal mean shift or a $\chi^2$ scaling factor), the product of the alignment of $\hat{\bm v}$ with $\bm v$ and the signal strength in $\bm X^\score$:
\begin{equation}
\begin{cases}
\mu_\spl(\gamma, c, h; \pi_\spl) \!\!\! &\equiv h \sqrt{1 - \pi_\spl} \cdot \tau(\gamma, c, \sqrt{\pi_\spl}h), \\ 
\mu_\bonus(\gamma, c, h; \pi_\aug) \!\!\! &\equiv h \cdot \tau(\pi_\aug \gamma, \pi_\aug c, h), \\
\mu_\insamp(\gamma, c, h) \!\!\! &\equiv h \cdot \tau(\gamma, c, h).
\end{cases}
\label{eq:signal-strengths}
\end{equation}

\paragraph{$\widehat{\mathcal R}$.} The rejection set results from a calibration procedure on the scores $S_j$, or equivalently, an application of BH to score-derived $p$-values. The overall power can be expressed as a prior-dependent but method-independent function $P$ of the signal strength in the alternative scores as well as the non-null proportion $\gamma$. The function $P$ is defined in terms of the limiting power of BH applied to testing simple hypotheses of the form $H_0: Y \sim G_0 \ \text{versus} \ H_1: Y \sim G_1$, where a fraction $\gamma$ are non-nulls \citep{Genovese2002}:
\begin{definition}[BH TPR function] \label{def:bhtpr}
For a non-null proportion $\gamma \in (0,1]$ and null and alternative distributions $G_0$ and $G_1$ with survival functions $\bar G_0$ and $\bar G_1$, we define
\begin{equation}
\TPR_{\textnormal{BH}}(G_0, G_1, \gamma) \equiv \bar G_1(t_*), \ \text{where} \ t_* \equiv \inf\left\{t \in \R: \frac{\bar G_0(t)}{(1-\gamma) \bar G_0(t) + \gamma \bar G_1(t)} \leq q\right\}.
\end{equation}
\end{definition}

% With this overview, we proceed to stating our asymptotic power results.

\subsection{Case 1: Point mass prior} \label{sec:point-mass-prior}

Throughout this section, consider data $\bm X$ generated according to equations~\eqref{eq:data-generating} and~\eqref{eq:two-groups}, with $\Lambda = \delta_{h \bm v}$, and all asymptotic statements are with respect to the regime~\eqref{eq:asymptotic-regime}.

\paragraph{Likelihood ratio and learner.} We can derive that the likelihood ratio statistic~\eqref{eq:likelihood-ratio-statistic} is
\begin{equation}
T_\Lambda(\bm X_j) = h\bm v^\top \bm X_j - \tfrac12 h^2, \quad \text{or} \quad \bm v^\top \bm X_j \ \text{after monotonic transformation.}
\label{eq:likelihood-ratio-statistic-point-mass}
\end{equation}
Given the invariance of all three methods to monotonic transformations of the statistic, we use the functional form $T_\Lambda(\bm X_j) \equiv \bm v^\top \bm X_j$ for the remainder of Section~\ref{sec:point-mass-prior}. The important portion of the prior to learn is the direction $\bm v \in \R^d$, and since $\E[\bm X_j] =  \gamma h \bm v$, we consider a learner based on the average of $\bm X_j$, similar to that employed by \citet{Kim2024}:
\begin{equation}
\hat{\bm v}_{\text{mean}} \equiv L_{\text{mean}}(\bm X) \equiv \frac{\frac{1}{m}\sum_{j = 1}^m \bm X_j}{\|\frac{1}{m}\sum_{j = 1}^m \bm X_j\|}.
\label{eq:mean-learner}
\end{equation}
This simple estimator is not necessarily optimal, but facilitates transparent analysis.

\paragraph{Quality of learned directions $\hat{\bm v}$.} For In-sample BH, the above learner is applied to the original data. The learned direction has the following asymptotic alignment with $\bm v$:
\begin{lemma}[Quality of mean learner on original data] \label{lem:mean-learner-untransformed}
The direction $\bm{\hat v}$ $= L_{\textnormal{mean}}(\bm X)$ has the following asymptotic alignment with the true direction $\bm v$:
\begin{equation}
\hat{\bm v}^\top \bm v \convp \frac{h\gamma}{\sqrt{h^2\gamma^2 + c}} \equiv \tau_{\textnormal{mean}}(\gamma, c, h).
\end{equation}
\end{lemma}
\noindent For Split BH and BONuS, the learner is applied to the masked data $\bm X^\learn$, which is governed by modified sets of parameters~\eqref{eq:masking-parameter-transformations}. We therefore deduce the asymptotic alignment of the learned direction $\hat{\bm v}$ with the true direction $\bm v$ for both methods:
\begin{proposition}
[Quality of mean learner on masked data] \label{prop:mean-learner-masked}
The directions $\hat{\bm v}_\spl$ and $\hat{\bm v}_\bonus$ learned by $L_{\text{mean}}$~\eqref{eq:mean-learner} applied to the learning data $\bm X^\learn$ of Split BH and BONuS have the following limiting alignments with the true direction $\bm v$:
\begin{align}
\hat{\bm v}_\spl^\top \bm v \convp \tau_{\textnormal{mean}}(\gamma, c, \sqrt{\pi_\spl}h) = \frac{h \gamma}{\sqrt{h^2 \gamma^2 + c/\pi_\spl}}; \\
\hat{\bm v}_\bonus^\top \bm v \convp \tau_{\textnormal{mean}}(\pi_\aug \gamma, \pi_\aug c, h) = \frac{h \gamma}{\sqrt{h^2 \gamma^2 + c/\pi_\aug}}.
\end{align}
\end{proposition}
\noindent Note that the quality of the learned direction is the same across the two methods when $\pi_\spl = \pi_\aug$. In Figure~\ref{fig:point-mass-prior-alignment}, we illustrate how the alignment varies with $\pi_\spl$ or $\pi_\aug$ for one choice of $(\gamma, c, h)$ and provide numerical simulation results,\footnote{All replication code is available at \textcolor{blue}{\url{https://github.com/Katsevich-Lab/masking-power}}.} which support our theory. Unless stated otherwise, all figures use the default setting $(\gamma, c, h, m) = (0.2, 0.2, 4, 1000)$, and Monte Carlo results are averaged over 1000 independent replications.

\begin{figure}[h]
\centering
\includegraphics{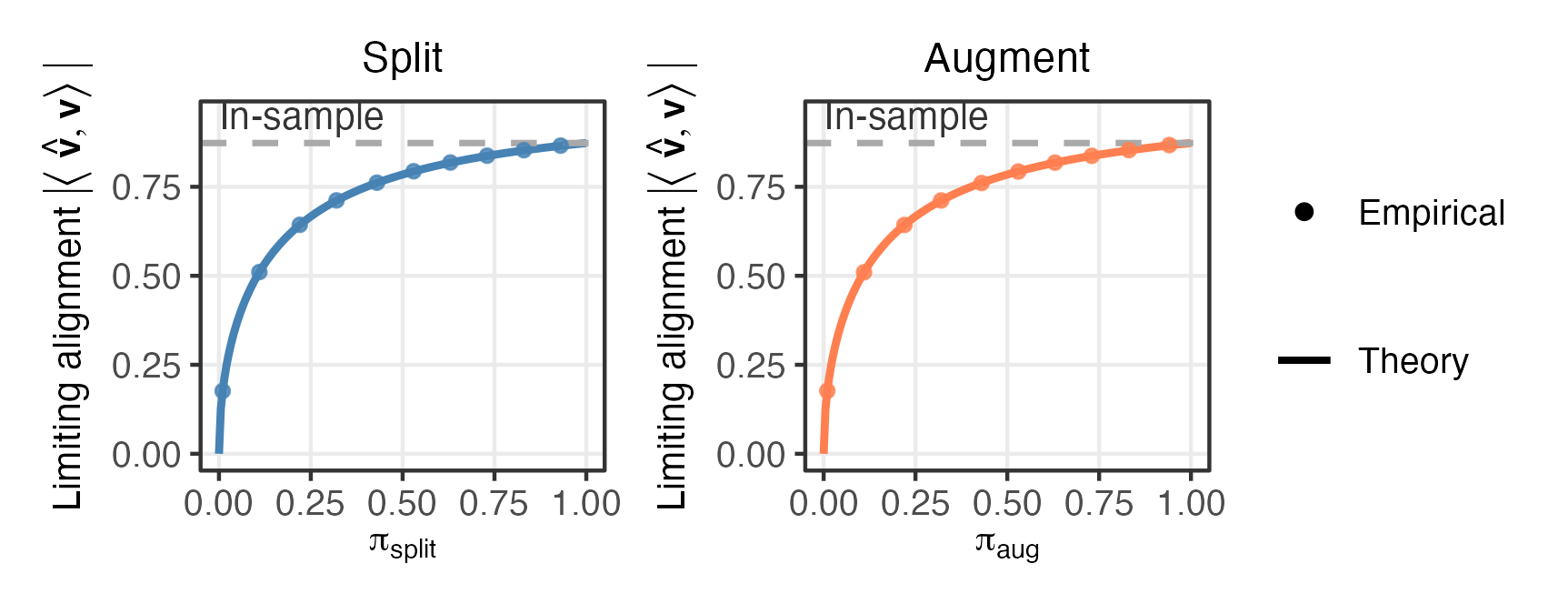}
\vspace{-0.1in}
\caption{Asymptotic alignment of learned direction $\hat{\bm v}$ with the true direction $\bm v$.}
\label{fig:point-mass-prior-alignment}
\end{figure}

\paragraph{Asymptotic distributions of scores.} We can derive the asymptotic signal strengths of the alternative scores based on equation~\eqref{eq:signal-strengths}:
\begin{equation}
\mu_\spl \equiv \frac{\sqrt{1-\pi_\spl} h^2 \gamma}{\sqrt{h^2 \gamma^2 + c/\pi_\spl}};\quad \mu_\bonus \equiv \frac{h^2 \gamma}{\sqrt{h^2 \gamma^2 + c/\pi_\aug}};\quad \mu_{\insamp} \equiv \frac{h^2 \gamma}{\sqrt{h^2 \gamma^2 + c}}.
\label{eq:point-mass-prior-signal-strengths}
\end{equation}
These signal strengths translate to mean shifts in the asymptotic score distributions.
\begin{proposition}[Asymptotic distributions of scores] \label{prop:point-mass-prior-score-distributions}
The scores produced by the three methods have the following limiting distributions:
\begin{align*}
&S^\spl_j \mid j \in \mathcal H_0 \convd N(0,1),\!\!\! &&S^\spl_j \mid j \in \mathcal H_1 \convd N(\mu_\spl, 1); \\
&S^\bonus_j \mid j \in \mathcal H_0 \convd N(K_\bonus,1),\!\!\! &&S^\bonus_j \mid j \in \mathcal H_1 \convd N(K_\bonus + \mu_\bonus, 1); \\
&S^{\insamp}_j \mid j \in \mathcal H_0 \convd N(K_{\insamp},1),\!\!\! && S^{\insamp}_j \mid j \in \mathcal H_1 \convd N(K_{\insamp} + \mu_{\insamp}, 1),
\end{align*}
where constants $K_\bonus$ and $K_{\insamp}$ depend on $(\gamma, c, h, \pi_\aug)$ and $(\gamma, c, h)$, respectively.
% \begin{equation*}
% K_\bonus \equiv \frac{c}{\sqrt{h^2\gamma^2 + c/\pi_\aug}} \quad \text{and} \quad K_{\insamp} \equiv \frac{c}{\sqrt{h^2\gamma^2 + c}}.
% \end{equation*}
\end{proposition}
\noindent The overlap between $\bm X^\learn$ and $\bm X^\score$ for BONuS and In-sample BH shifts null and alternative distributions by constants $K_\bonus$ and $K_{\insamp}$, respectively. These constants do not affect the power of the methods, as they simply shift the rejection thresholds.

\paragraph{Overall power.} Given the score distributions from Proposition~\ref{prop:point-mass-prior-score-distributions}, we can view all three methods as being asymptotically equivalent to testing for mean shifts in univariate normal models. This leads to our first main result.
\begin{theorem}[Asymptotic power under point mass prior] \label{thm:power-point-mass}
Let $\widehat{\mathcal R}$ represent the rejection set of Split BH, BONuS, or In-sample BH, and $\mu(\gamma, c, h)$ represent the corresponding signal strength~\eqref{eq:point-mass-prior-signal-strengths}. The TPP of $\widehat{\mathcal R}$ for each method has the following limit:
\begin{equation}
\TPP(\widehat{\mathcal R}) \convp \TPR_{\textnormal{BH}}(N(0,1), N(\mu(\gamma, c, h), 1), \gamma).
\end{equation}
\end{theorem}
\noindent This verifies the bottom row of Figure~\ref{fig:theory-schematic}, with $P(\mu, \gamma) \equiv \TPR_{\textnormal{BH}}(N(0,1), N(\mu, 1), \gamma)$. Figure~\ref{fig:point-mass-prior-theorem} illustrates the comparison for one choice of $(\gamma, c, h)$; others are explored in Appendix~\ref{sec:additional-experiments}. Our theory provides an excellent fit to the empirical power of each method, except for BONuS with $\pi_\aug \approx 1$. We address this and other observations next.
\begin{figure}[h!]
\centering
\includegraphics{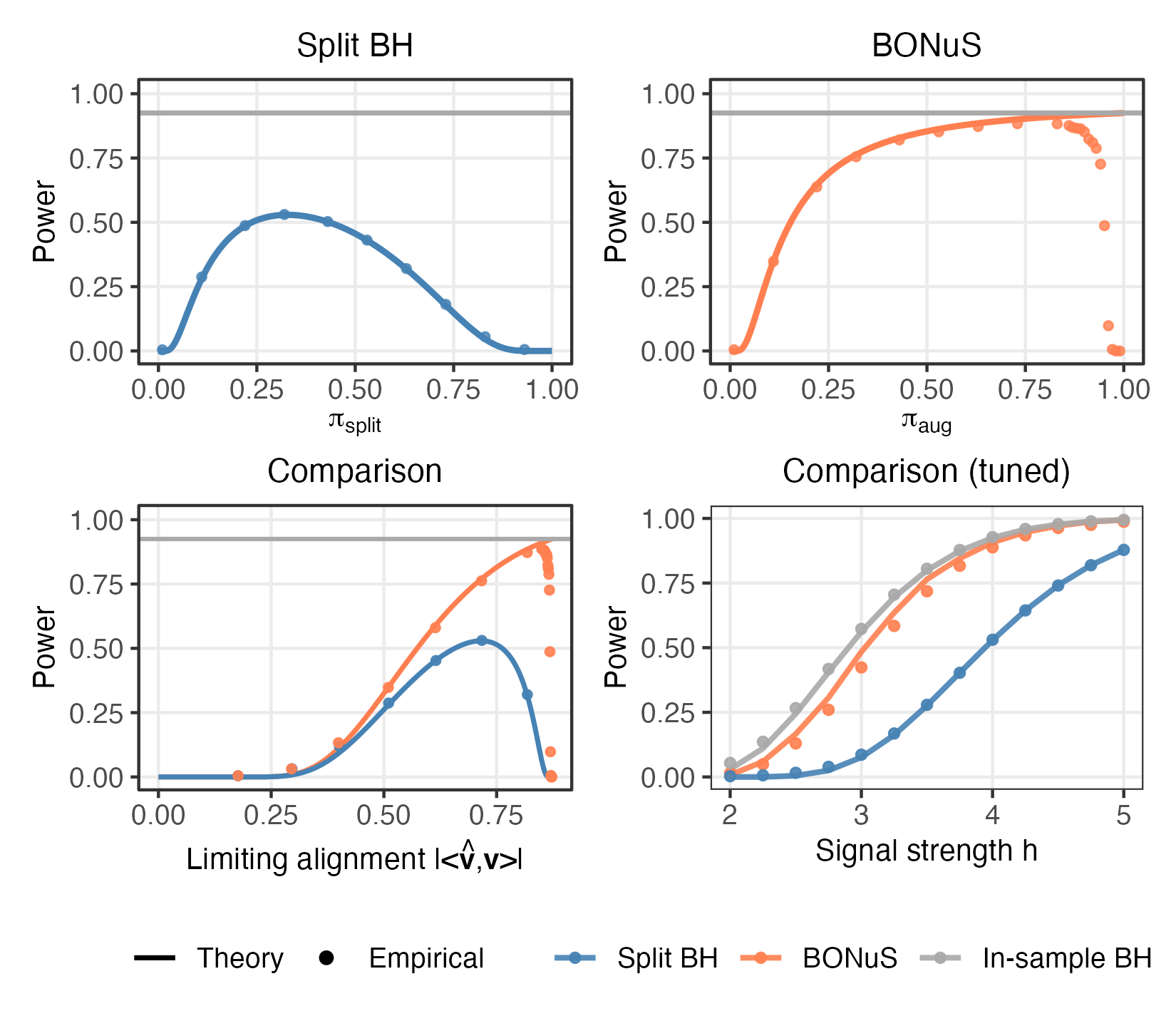}
\caption{The powers of Split BH (top left) and BONuS (top right) versus their tuning parameters, and compared to that of In-sample BH, for the point prior. The powers of the procedures are compared as a function of the limiting alignment of the learned direction with the true direction (bottom left). We also compare the empirical maximum powers attained by Split BH and BONuS, with the power of In-sample BH, which has no tuning parameters, as a function of the signal strength $h$ (bottom right).} 
\label{fig:point-mass-prior-theorem}
\end{figure}

\begin{wrapfigure}[12]{r}{0.4\textwidth}
\vspace{-0.1in}
\includegraphics{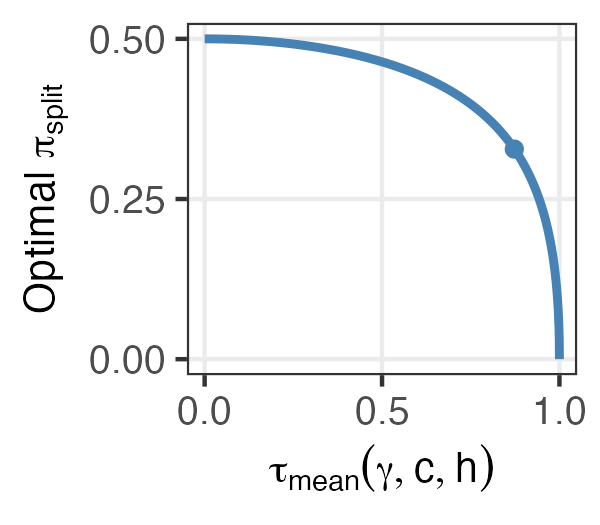}
\vspace{-0.15in}
\caption{Optimal Split BH learning proportion versus ease of learning, $\tau_{\text{mean}}$. Point marks default setting $(\gamma, c, h) = (0.2, 0.2, 4)$.}
\label{fig:optimal-pi-split}
\end{wrapfigure}
\begin{remarkwrap}[Power of Split BH]
The power of Split BH increases and then decreases as $\pi_\spl$ increases from 0 to 1 (Figure~\ref{fig:point-mass-prior-theorem}), due to the zero-sum trade-off between signal strength in $\bm X^\learn$ and $\bm X^\score$ (Figure~\ref{fig:theory-schematic}). Maximizing $\mu_\spl$ (\ref{eq:signal-strengths}, \ref{eq:point-mass-prior-signal-strengths}) over $\pi_\spl$ (Appendix~\ref{sec:pi-opt}) yields the optimal splitting proportion
\begin{equation}
\begin{split}
\pi_\spl^{\textnormal{opt}} &= \underset{\pi_\spl \in [0,1]}{\arg \max} \sqrt{1-\pi_\spl} \cdot \tau(\gamma, c, \sqrt{\pi_\spl}h) \\
 &= \frac{\sqrt{1 - \tau^2_{\text{mean}}(\gamma, c, h)}}{1 + \sqrt{1 - \tau^2_{\text{mean}}(\gamma, c, h)}}.
\label{eq:optimal-pi-split}
\end{split}
\end{equation}
The first equality applies to Cases 1 and 2, while the second equation is specific to Case 1. The first expression shows that the optimal splitting proportion depends on unknown parameters only through the curve $\pi_\spl \mapsto \tau(\gamma, c, \sqrt{\pi_\spl}h)$, which may be estimable from data (at least for $\pi_\spl \leq 1/2$) by examining alignments between $\hat{\bm v}$ learned on disjoint splits of the data. We leave this to future work. The second expression demonstrates that, in Case 1, the optimal splitting proportion depends on how easy it is to learn the direction $\bm v$ on the full data, as measured by $\tau_{\text{mean}}(\gamma, c, h)$ (Figure~\ref{fig:optimal-pi-split}). The easier it is to learn $\bm v$, the less data should be allocated to learning, with $\pi_\spl^{\text{opt}} \rightarrow 0$ as $\tau_{\text{mean}} \rightarrow 1$ and $\pi_\spl^{\text{opt}} \rightarrow 1/2$ as $\tau_{\text{mean}} \rightarrow 0$. A similar qualitative conclusion was reached by \citet{Bekerman2025}. 
\end{remarkwrap}

\begin{remark}[Power of BONuS]
In contrast to Split BH, the asymptotic power of BONuS increases monotonically as $\pi_\aug$ increases from 0 to 1 (Figure~\ref{fig:point-mass-prior-theorem}), since $\mu_\bonus$~\eqref{eq:point-mass-prior-signal-strengths} is an increasing function of $\pi_\aug$. Indeed, increasing $\pi_\aug$ improves the quality of $\hat{\bm v}$ without changing $\bm X^\score$, which is always the original data $\bm X$. The empirical power of BONuS matches this trend until $\pi_\aug$ approaches 1 ($\tilde m$ approaches zero), which causes empirical power to drop. In this regime, the empirical $p$-values from lines 5-6 of Algorithm~\ref{alg:bonus} become highly discrete, leading to conservative behavior; we investigate further in Section~\ref{sec:small-tilde-m}.
\end{remark}

\begin{remark}[Power comparison] \label{rem:power-comparison}
Comparing the powers of BONuS and Split BH (Figure~\ref{fig:point-mass-prior-theorem}, bottom row), we find that BONuS dominates across matched tuning parameters and, upon optimal tuning of each method, across signal strengths. This improvement stems from the stronger signal in $\bm X^\score$ (Figure~\ref{fig:theory-schematic}), manifesting as $\mu_\bonus \geq \mu_\spl$ when $\pi_\aug = \pi_\spl$~\eqref{eq:point-mass-prior-signal-strengths}. On the other hand, In-sample BH dominates BONuS by using the entire data for learning, leading to $\mu_\insamp \geq \mu_\bonus$ for all $\pi_\aug$. However, the asymptotic power of BONuS approaches that of In-sample BH as $\pi_\aug \rightarrow 1$, improving the quality of the learned direction for BONuS. Even with the finite-sample penalty of choosing large $\pi_\aug$, the power of optimally tuned BONuS is quite close to that of In-sample BH. We empirically compare these methods to an adaptation of knockoffs (Figure~\ref{fig:theorem-over_h}), which employs one-to-one null augmentation rather than BONuS's one-to-many augmentation. We observe that knockoff scores can be viewed as noisy BONuS scores, which can lead to lower power (though still at least that of Split BH). However, the masked likelihood ratio (MLR) knockoff statistic \citep{Spector2022e} has the effect of collapsing this noise and restoring power similar to that of BONuS.
\end{remark}

\subsection{Case 2: One-dimensional subspace prior} \label{sec:subspace-prior}
We retrace the steps of Section~\ref{sec:point-mass-prior} for the prior $\Lambda = N(\bm 0, h^2 \bm v \bm v^\top)$.

\paragraph{Likelihood ratio and learner.} Under this prior, the likelihood ratio statistic~\eqref{eq:likelihood-ratio-statistic} is
\begin{equation}
T_\Lambda(\bm X_j) = -\frac{\log(1 + h^2)}{2} + \frac{h^2(\bm v^\top \bm X_j)^2}{2(1+h^2)}, \ \text{or} \ (\bm v^\top \bm X_j)^2 \text{ after monotonic transformation.}
\label{eq:likelihood-ratio-statistic-subspace}
\end{equation}
We therefore use $T_\Lambda(\bm X_j) \equiv (\bm v^\top \bm X_j)^2$ in this section, and note that the signal manifests in
\begin{equation}
\V[\bm X_j] =  \gamma h^2 \bm v \bm v^\top + \bm I_d.
\label{eq:covariance-subspace-prior}
\end{equation}
Since $\bm v$ is the top eigenvector of this covariance matrix, we consider the PCA learner:
\begin{equation}
\hat{\bm v}_{\text{PCA}} = L_{\text{PCA}}(\bm X) \equiv \text{leading eigenvector of } \bm X^\top \bm X.
\label{eq:pca-learner}
\end{equation}

\paragraph{Quality of the learned directions $\hat{\bm v}$.} Given the rank-one spiked covariance structure~\eqref{eq:covariance-subspace-prior}, we use standard random matrix theory results \citep{Benaych-Georges2011} to compute the asymptotic alignment of the learned and true directions $\hat{\bm v}$ and $\bm v$.
\begin{lemma}[Quality of PCA learner on original data] \label{lem:pca-learner-untransformed}
The direction $\bm{\hat v}$ $= L_{\textnormal{PCA}}(\bm X)$ has the following squared asymptotic alignment with the true direction $\bm v$:
\begin{equation}
(\hat{\bm v}^\top \bm v)^2 \convp \frac{\left(1 - \frac{c}{\gamma^2 h^4}\right)_+}{1 + \frac{c}{\gamma h^2}} \equiv \tau^2_{\textnormal{PCA}}(\gamma, c, h), \quad \text{where} \quad (x)_+ \equiv \max(x,0).
\end{equation}
\end{lemma}
\noindent As in Section~\ref{sec:point-mass-prior}, applying the parameter transformations~\eqref{eq:masking-parameter-transformations} yields the asymptotic alignment of the learned direction $\hat{\bm v}$ with the true direction $\bm v$ for Split BH and BONuS:
\begin{proposition}
[Quality of PCA learner on masked data] \label{prop:pca-learner-masked}
The directions $\hat{\bm v}_\spl$ and $\hat{\bm v}_\bonus$ learned by $L_{\text{PCA}}$~\eqref{eq:pca-learner} applied to the learning data $\bm X^\learn$ of Split BH and BONuS have the following limiting squared alignments with the true direction $\bm v$:
\begin{align}
(\hat{\bm v}_\spl^\top \bm v)^2 \convp \tau^2_{\textnormal{PCA}}(\gamma, c, \sqrt{\pi_\spl}h) = \frac{\left(1 - \frac{c}{\gamma^2 \pi_\spl^2 h^4}\right)_+}{1 + \frac{c}{\gamma \pi_\spl h^2}}; \\
(\hat{\bm v}_\bonus^\top \bm v)^2 \convp \tau^2_{\textnormal{PCA}}(\pi_\aug \gamma, \pi_\aug c, h) = \frac{\left(1 - \frac{c}{\pi_\aug \gamma^2 h^4}\right)_+}{1 + \frac{c}{\gamma h^2}}.
\end{align}
\end{proposition}
%Unlike for the mean learner, the learned direction based on split data has lower quality than that based on the augmented data, when $\pi_\spl = \pi_\aug$. However, this observation does not in itself signal a deficiency of Split BH; it suggests that $\pi_\spl$ and $\pi_\aug$ are not on the same scale. 
Figure~\ref{fig:subspace-prior-alignment} illustrates how the theoretical alignment varies as a function of $\pi_\spl$ or $\pi_\aug$ for one choice of $(\gamma, c, h)$, as well as agreement with numerical simulations.

\paragraph{Asymptotic distributions of scores.} As in Section~\ref{sec:point-mass-prior}, we can derive the limiting distributions of the scores based on Lemma~\ref{lem:pca-learner-untransformed}, Proposition~\ref{prop:pca-learner-masked}, and equation~\eqref{eq:signal-strengths}:
\small
\begin{equation}
\mu^2_\spl = (1-\pi_\spl)h^2 \cdot \tfrac{\left(1 - \frac{c}{\gamma^2 \pi_\spl^2 h^4}\right)_+}{1 + \frac{c}{\gamma \pi_\spl h^2}};\ \mu^2_\bonus = h^2 \cdot \tfrac{\left(1 - \frac{c}{\pi_\aug \gamma^2 h^4}\right)_+}{1 + \frac{c}{\gamma h^2}};\ \mu^2_{\insamp} = h^2 \cdot \tfrac{\left(1 - \frac{c}{\gamma^2 h^4}\right)_+}{1 + \frac{c}{\gamma h^2}}.
\label{eq:subspace-prior-signal-strengths}
\end{equation}
\normalsize
These signal strengths translate to scaling factors in the asymptotic score distributions.
\begin{proposition}[Asymptotic distributions of scores] \label{prop:subspace-prior-score-distributions}
The scores produced by the three methods have the following limiting distributions:
\begin{align*}
&S^\spl_j \mid j \in \mathcal H_0 \convd \chi^2_1, \quad &&S^\spl_j \mid j \in \mathcal H_1 \convd (1 + \mu^2_\spl)\chi^2_1; \\
&S^\bonus_j \mid j \in \mathcal H_0 \convd  K_\bonus \chi^2_1, \quad &&S^\bonus_j \mid j \in \mathcal H_1 \convd K_\bonus (1 + \mu^2_\bonus)\chi^2_1; \\
&S^{\insamp}_j \mid j \in \mathcal H_0 \convd K_{\insamp}\chi^2_1, \quad && S^{\insamp}_j \mid j \in \mathcal H_1 \convd K_{\insamp}(1 + \mu^2_{\insamp})\chi^2_1,
\end{align*}
where constants $K_\bonus$ and $K_{\insamp}$ depend on $(\gamma, c, h, \pi_\aug)$ and $(\gamma, c, h)$, respectively.
\end{proposition}
\noindent Comparing to Proposition~\ref{prop:point-mass-prior-score-distributions} for the point mass prior, the score distributions are now scaled $\chi^2_1$ distributions rather than shifted normal distributions. The null and alternative distributions of BONuS and In-sample BH again are impacted by selection bias, this time in the form of multiplicative constants $K_\bonus$ and $K_{\insamp}$ that do not affect power. 

\paragraph{Overall power.} Given the score distributions from Proposition~\ref{prop:subspace-prior-score-distributions}, we can view all three methods as being asymptotically equivalent to testing for scale changes in univariate $\chi^2_1$ models. This leads to our second main result.
\begin{theorem}[Asymptotic power under one-dimensional subspace prior] \label{thm:power-one-dimensional}
Let $\widehat{\mathcal R}$ represent the rejection set of Split BH, BONuS, or In-sample BH, and let $\mu(\gamma, c, h)$ represent the corresponding signal strength~\eqref{eq:subspace-prior-signal-strengths}. Then, for each method, we have
\begin{equation}
\TPP(\widehat{\mathcal R}) \convp \TPR_{\textnormal{BH}}(\chi^2_1, (1 + \mu^2(\gamma, c, h))\chi^2_1, \gamma).
\label{eq:subspace-prior-asymptotic-power}
\end{equation}
\end{theorem}
\noindent This verifies the bottom row of Figure~\ref{fig:theory-schematic}, with $P(\mu, \gamma) \equiv \TPR_{\textnormal{BH}}(\chi^2_1, (1 + \mu^2)\chi^2_1, \gamma)$. Figure~\ref{fig:subspace-mass-prior-theorem} illustrates the comparison; the qualitative behavior matches that of Figure~\ref{fig:point-mass-prior-theorem} for the point prior, so the associated discussion mostly carries over. Nevertheless, we highlight a few departures from Case 1. First, setting $\pi_\spl = \pi_\aug$ does not yield the same quality of learned direction for Split BH and BONuS. However, it can be seen both from the specific signal strengths~\eqref{eq:subspace-prior-signal-strengths} for Case 2 and from the general expressions~\eqref{eq:signal-strengths} that $\mu_\bonus \geq \mu_\spl$ when $\pi_\spl$ and $\pi_\aug$ are chosen to equalize the quality of the learned direction, $\tau$. Second, there is no closed-form analog to the optimal splitting proportion for Split BH as in Case 1~\eqref{eq:optimal-pi-split}. Nevertheless, the optimal splitting proportion's dependency on the learning curve $\pi_\spl \mapsto \tau(\gamma, c, \sqrt{\pi_\spl}h)$ is the same. Finally, the learning curve for the PCA learner (Figure~\ref{fig:subspace-prior-alignment}) rises and plateaus more sharply than that of the mean learner (Figure~\ref{fig:point-mass-prior-alignment}), especially for augmentation. This leads to a broader range of $\pi_\aug$ for which the empirical and theoretical power of BONuS is near its maximum (Figure~\ref{fig:subspace-mass-prior-theorem}).

\section{Power of BONuS for small $\tilde m$} \label{sec:small-tilde-m}

In Section~\ref{sec:asymptotic-power-comparison}, we found that the asymptotic power of BONuS is maximized at $\pi_\aug = 1$, which corresponds to $\tilde m \ll m$. However, this result does not provide the actual growth rate of $\tilde m_{\text{opt}}$, the power-optimal value of $\tilde m$. Setting $\tilde m$ too small leads to a highly discrete empirical $p$-values, which can cause conservative behavior and power degradation. We observed this empirically in Figure~\ref{fig:point-mass-prior-theorem} for $\pi_\aug$ close to 1. Therefore, the optimal choice of $\tilde m$ must balance the trade-off between quality of the learned direction (benefiting from smaller $\tilde m$) and the granularity of $p$-values (benefiting from larger $\tilde m$). To dissect this trade-off, we carry out a higher-order asymptotic analysis to approximate the power of BONuS in the regime $\tilde m \ll m$ and use this to characterize the growth rate of $\tilde m_{\text{opt}}$. We do so for point mass prior in Sections~\ref{sec:tractable-approximation} and~\ref{sec:small-tilde-m-point-mass-prior} and for subspace prior in Appendix~\ref{sec:small-tilde-m-subspace-prior}. We also provide empirical evidence supporting our theory (Section~\ref{sec:numerical-evidence-small-tilde-m}).

\subsection{Quantifying the trade-off in choosing $\tilde m$} \label{sec:small-tilde-m-point-mass-prior}

It is challenging to directly analyze the impact of $\tilde m$ on the power of BONuS within the regime $\tilde m \ll m$. We instead analyze an approximation to BONuS power, derived in Appendix~\ref{sec:tractable-approximation} and defined as follows:
\begin{equation*}
\TPR'_\bonus(\tilde m) \equiv \bar \Phi(t'_*(\tilde m) - \mu'(\tilde m)),
\end{equation*}
where
\begin{equation}
\mu'(\tilde m) =  \frac{h^2 \gamma}{\sqrt{h^2 \gamma^2 + c(1 + \tilde m / m)}}
\label{eq:approximate-mean-shift-main}
\end{equation}
and
\begin{equation}
t'_*(\tilde m) = \inf\left\{t : \widehat{\text{FDR}}'(t; \tilde m)  \leq q\right\}; \quad \widehat{\text{FDR}}'(t; \tilde m) \equiv \frac{\frac{1}{1 + \tilde m}\left(1 + \tilde m \cdot \bar \Phi(t)\right)}{(1-\gamma)\bar \Phi(t) + \gamma \bar \Phi(t - \mu'(\tilde m))}.
\label{eq:approximate-fdr-main}
\end{equation}
While this expression is obtained in part by substituting quantities involving empirical $p$-values with their expectations~\eqref{eq:approximate-fdr}, it still captures the trade-off in $\tilde m$. Increasing $\tilde m$ decreases the quality of the learned direction $\hat{\bm v}$, which manifests as a decrease in the mean shift $\mu'(\tilde m)$ of the scores for non-null hypotheses~\eqref{eq:approximate-mean-shift-main}. Decreasing $\mu'(\tilde m)$ has the impact of decreasing $\TPR'_\bonus(\tilde m) = \bar \Phi(t'_*(\tilde m) - \mu'(\tilde m))$ both directly and indirectly (by increasing the approximate threshold $t'_*(\tilde m)$). On the other hand, decreasing $\tilde m$ increases \textit{pseudocount bias} in the numerator of $\smash{\widehat{\text{FDR}}'(t; \tilde m)}$~\eqref{eq:approximate-fdr-main}, which increases the approximate threshold $t'_*(\tilde m)$. However, it remains unclear how these two competing forces compare to one another, and how they interact to determine the optimal $\tilde m$. To understand this trade-off, we carry out an asymptotic analysis of $\TPR'_\bonus(\tilde m)$.

We observe that, in line with Remark~\ref{rem:power-comparison}, the limit of the approximate BONuS power as $m, \tilde m \rightarrow \infty, \tilde m / m \rightarrow 0$ tends to the asymptotic power of In-sample BH:
\begin{equation*}
% \lim_{\substack{
%   m \to \infty \\
%   \tilde m / m \to 0
% }}
\lim_{m, \tilde m \to \infty,\tilde m / m \to 0} \TPR'_\bonus(\tilde m) = \TPR_{\textnormal{BH}}(N(0,1), N(\mu_{\insamp}(\gamma, c, h), 1), \gamma) \equiv \TPR_{\insamp}.
\end{equation*}
This motivates us to expand $\TPR'_\bonus(\tilde m)$ about $\TPR_{\insamp}$ in this asymptotic regime.
\begin{theorem} \label{thm:point-mass-small-m}
We have
\begin{equation}\label{eq:point-mass-small-m}
\TPR'_\bonus(\tilde m) = \TPR_{\insamp} - \eta_1 \frac{1}{\tilde m} - \eta_2 \frac{\tilde m}{m} + o\left(\frac{1}{\tilde m} + \frac{\tilde m}{m}\right),
\end{equation}
where $\eta_1, \eta_2 > 0$ are constants depending on $(\gamma, c, h, q)$.
\end{theorem}
This result quantifies the trade-off in $\tilde m$: the ``cost'' of the pseudocount bias resulting from decreasing $\tilde m$ is $\asymp 1/\tilde m$ and the ``cost'' of the degraded learning performance resulting from increasing $\tilde m$ is $\asymp \tilde m / m$. Balancing these two terms implies that the optimal $\tilde m$ is on the order of $\sqrt{m}$. We conjecture that the same statement holds for BONuS itself.
\begin{conjecture} \label{conj:optimal-tilde-m-point-mass}
Let $\tilde m_{\textnormal{opt}}$ be the choice of $\tilde m$ that maximizes the power of BONuS in the case of point mass prior for a fixed value of $m$. Then, $\tilde m_{\textnormal{opt}} \asymp \sqrt{m}$.
\end{conjecture}
In Appendix~\ref{sec:small-tilde-m-subspace-prior}, we prove an analogous result for the subspace prior (Theorem~\ref{thm:subspace-small-m}) and make an analogous conjecture (Conjecture~\ref{conj:optimal-tilde-m-subspace}).
% \begin{remark}[Conjectured wider applicability of the $\sqrt{m}$ scaling]	
Having observed the $\sqrt{m}$ scaling for $\tilde m_{\textnormal{opt}}$ in two settings suggests that such a result might apply in a wider range of settings. % Inspecting our proofs of Theorems~\ref{thm:point-mass-small-m} and~\ref{thm:subspace-small-m} suggests that this relationship occurs any time the FDP estimator is asymptotically Lipschitz in $\pi_{\aug}$. 
% \end{remark}

%\begin{remark}[Optimal $\tilde{m}$ in previous literature.] \label{rem:previous-tilde-m}

%In the BONuS paper \citep{Yang2021} and the conformal literature \citep{Mary2021a, Marandon2024}, the recommended choice of $\tilde m$ is $\tilde m \approx {m} /  ({q |\widehat{\mathcal R}|})$. This recommendation stems from an equivalent interpretation of BONuS as BH applied to conformal $p$-values \cite[Proposition 8]{Lee2024}. Then, ${m} /  ({q|\widehat{\mathcal R}|})$ can be seen as the smallest value of $\tilde m$ yielding conformal $p$-values with sufficient granularity to pass the BH threshold. This heuristic is difficult to apply in practice, however, as $|\widehat{\mathcal R}|$ is usually unknown in advance. Furthermore, it does not account for the fact that $|\widehat{\mathcal R}|$ is itself a function of $\tilde m$ (see Theorems \ref{thm:power-point-mass}, \ref{thm:power-one-dimensional}, \ref{thm:point-mass-small-m}, and \ref{thm:subspace-small-m}). This creates a complex relationship between $|\widehat{\mathcal R}|$ and $\tilde m$ that our theory helps to clarify, resulting in a novel and perhaps surprising recommendation to choose $\tilde m = \Theta(\sqrt{m})$. 

%\end{remark}

\begin{remark}[Comparison to existing work on choosing $\tilde m$] \label{rem:choosing-m-tilde}

The question of choosing $\tilde m$ has been considered for the original iterative variant of BONuS \citep{Yang2021} and for the related AdaDetect conformal outlier detection procedure \citep{Marandon2024}. Both works recognized the trade-off inherent in choosing $\tilde m$ but neither formally derived an optimal choice. Nevertheless, they offered parallel heuristic discussions of this issue. Both procedures involve BH applied to empirical $p$-values, so to make any rejections, the smallest possible empirical $p$-value $1/(\tilde m + 1)$ must be less than or equal to the BH threshold $q |\widehat{\mathcal R}| / m$. This leads to the criterion $\tilde m \geq m/q|\widehat{\mathcal R}| - 1$, which \citet{Yang2021,Marandon2024} take as their starting point for choosing $\tilde m$. This criterion restricts $\tilde m$ only implicitly, as the rejection set size $\widehat{\mathcal R}$ is itself a function of $\tilde m$. Nevertheless, these works reasoned $\tilde m$ must be at least $\asymp m$ for weak- or sparse-signal scenarios where $\widehat{\mathcal R} = O(1)$ is possible. On the other hand, for $\widehat{\mathcal R} \asymp m$, \citet{Marandon2024} observed that $\tilde m = O(1)$ would suffice for sufficient granularity. Presumably to be conservative in the face of unknown signal, both prior works recommend a default setting of $\tilde m \asymp m$. Other works have studied the power of conformal procedures, but where score function is learned on a separate dataset \citep{Mary2021a,Gazin2024}.

Our simpler problem setup allows for a more detailed analysis, explicitly quantifying the trade-off between $p$-value granularity and quality of learned direction (Theorem~\ref{thm:point-mass-small-m}), leading to a potentially surprising conjecture for the optimal scaling: $\tilde m \asymp \sqrt{m}$. This is larger than the $O(1)$ scaling that \citet{Marandon2024} observed is required for sufficient granularity in our setting where $\widehat{\mathcal R} \asymp m$. While $\tilde m = O(1)$ is sufficient to make \textit{any rejections}, we believe it is not the best choice to make the \textit{largest number of rejections}. On the other hand, $\tilde m \asymp \sqrt{m}$ is smaller than the default $\tilde m \asymp m$ recommended by prior works. This suggests that accommodating sparse signals can lead to suboptimal choices in proportional scaling regimes, highlighting the need for signal-adaptive choices of $\tilde m$. We note that the $\sqrt m$ scaling is still a conjecture for BONuS, and even if this conjecture is correct, we do not suggest that scaling is optimal for proportional scaling regimes in all null augmentation problems. Nevertheless, we expect at least that $1 \ll \tilde m_{\text{opt}} \ll m$ in a broad range of such problems, as $\tilde m \gg 1$ eliminates the $p$-value discreteness problem in cases where the threshold converges to a constant while we expect $\tilde m \ll m$ eliminates the learning degradation problem in many settings. 
\end{remark}

\subsection{Numerical evidence for theoretical approximations} \label{sec:numerical-evidence-small-tilde-m}

We first compare the $\tilde m / m \rightarrow 0$ power given by Theorems~\ref{thm:point-mass-small-m} and~\ref{thm:subspace-small-m} to the general results in Theorems~\ref{thm:power-point-mass} and~\ref{thm:power-one-dimensional} (Section~\ref{sec:asymptotic-power-comparison}) as well as the empirical power of BONuS (Figure~\ref{fig:theorem-with-approx}). The refined approximation is not a perfect match for the empirical power, particularly for $\pi_\aug$ near zero (Theorems~\ref{thm:point-mass-small-m} and~\ref{thm:subspace-small-m} assume $\pi_\aug \rightarrow 1$) or too near one (these theorems still require $\tilde m \rightarrow \infty$). Nevertheless, it captures the qualitative behavior of the empirical power better than the results in Theorems~\ref{thm:power-point-mass} and~\ref{thm:power-one-dimensional}, particularly the increasing and then decreasing shape of the curves as $\pi_\aug$ increases from zero to one.
\begin{figure}[h]
\centering
\includegraphics{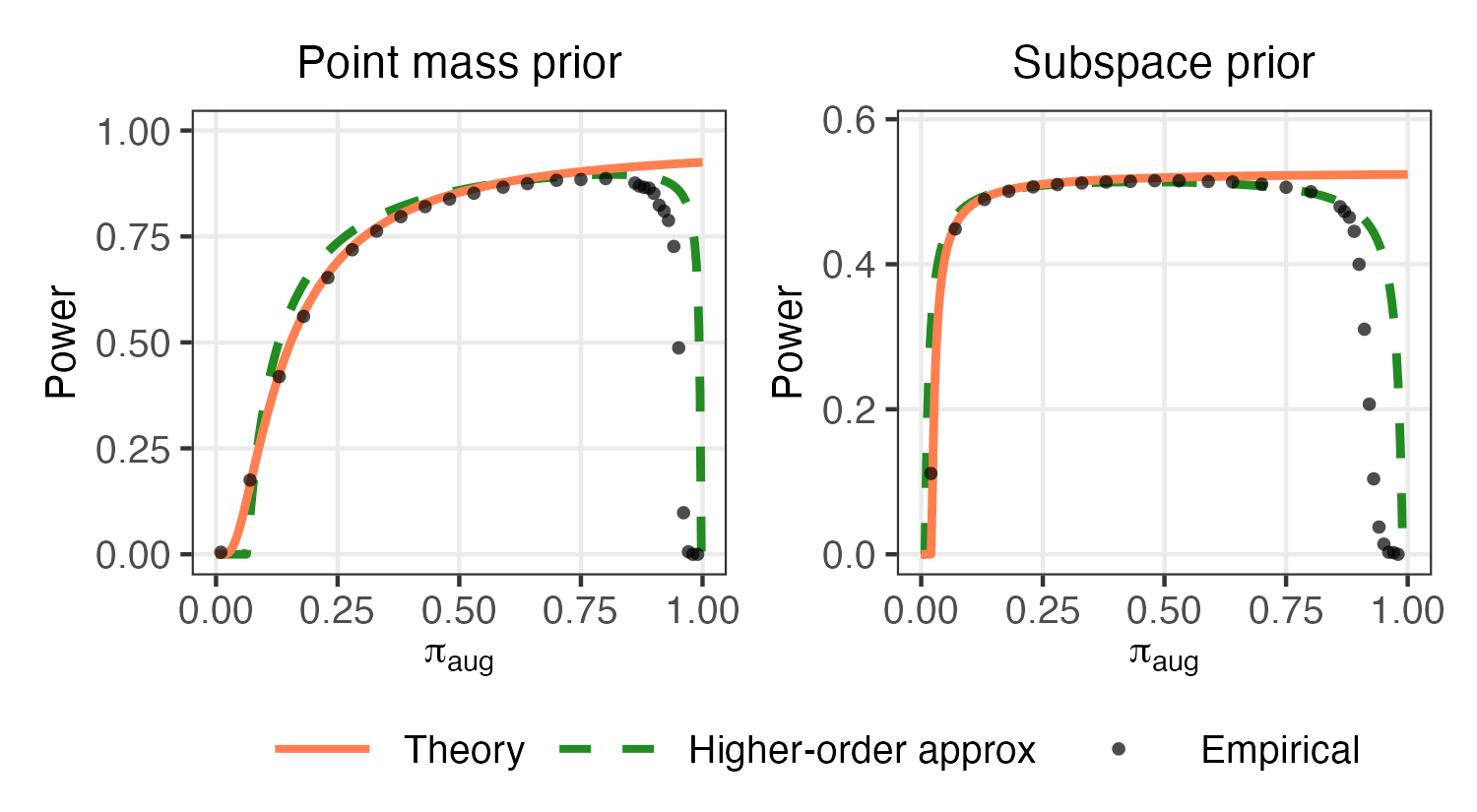}
\vspace{-0.1in}
\caption{Comparing the empirical performance of BONuS with the asymptotic power curves (Theorems~\ref{thm:power-point-mass} and~\ref{thm:power-one-dimensional}) and higher-order variants (Theorems~\ref{thm:point-mass-small-m} and~\ref{thm:subspace-small-m}) to handle small $\tilde{m}$.}
\label{fig:theorem-with-approx}
\end{figure} 

Next, we provide numerical evidence to support Conjecture~\ref{conj:optimal-tilde-m-point-mass} regarding the $\sqrt{m}$ scaling of the optimal $\tilde m$ for the point mass prior (Figure~\ref{fig:point-mass-prior-optimal}, left). For each $m$, we compare the value of $\tilde m$ maximizing empirical power with the value maximizing the approximation in Theorem~\ref{thm:point-mass-small-m}, which grows at a rate of $\sqrt{m}$. We find that the best-fit line through the empirical maximizers is nearly parallel (on a log-log scale) to the theoretical line with slope 1/2; it has a slope of 0.45 (confidence interval [0.41, 0.49]). 
%Note that the optimal values predicted by Theorem~\ref{thm:point-mass-small-m} are smaller than the empirical optimal values; indeed, the predicted optimal values ignore the variance inflation effect of small $\tilde m$, and therefore suggest that $\tilde m$ should be smaller than in reality. This does not contradict Conjecture~\ref{conj:optimal-tilde-m-point-mass}, however, as it addresses the rate but not the constants governing the optimal $\tilde m$. 
The realized powers at the theoretical and empirical optimal values of $\tilde m$ are also quite similar (Figure~\ref{fig:point-mass-prior-optimal}, right).
\begin{figure}[h!]
\centering
\includegraphics[width = \textwidth]{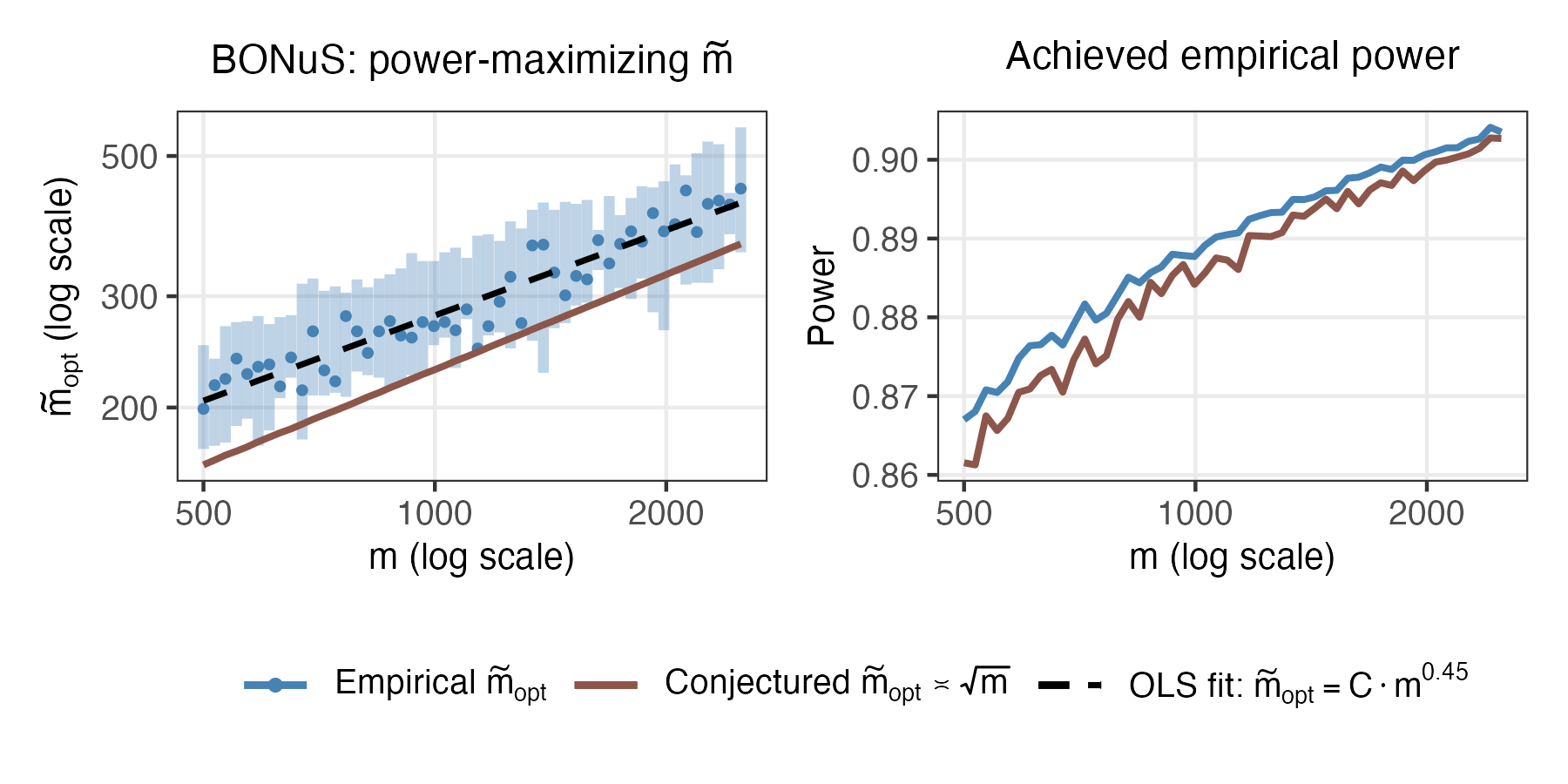}
\vspace{-0.4in}
\caption{Comparing the conjectured power-maximizing $\tilde{m}_{\text{opt}}$ with the empirical maximizer, and the realized powers, supporting the $\sqrt{m}$ scaling of $\tilde{m}_{\text{opt}}$.}
\label{fig:point-mass-prior-optimal}
\end{figure}

\vspace{-0.1in}
 
\section{Discussion} \label{sec:discussion}
 
We have identified masking as a recognizable strategy for avoiding double dipping that is shared across methods and problem settings. The main motivation of this manuscript was to leverage power analysis in a stylized setting to extract insights that may apply more broadly. We hypothesize that the following qualitative insights (one for each of the motivating questions in Section~\ref{sec:questions}) hold for masking methods similar to Split BH and BONuS in settings similar to ours: 1) one-to-many null augmentation methods are more powerful than sample splitting methods based on the same learning procedure, 2) the optimal number of null samples for one-to-many null augmentation methods satisfies $1 \ll \tilde m_{\text{opt}} \ll m$ (recall Remark~\ref{rem:choosing-m-tilde}), and 3) one-to-many null augmentation methods can be nearly as powerful as non-masking methods based on the same learning procedure, while sample splitting methods usually cannot. 

A clear direction for future work is to extend our analyses, conducted in a deliberately simple setting, to other statistical contexts, masking mechanisms, and/or learning procedures. This will help sharpen and verify the conjectures above. A natural starting point is the more general setup outlined in Appendix~\ref{sec:more-general-setup}, which applies to the same three methods analyzed here but accommodates more general data-generating models, learning procedures, and test statistics. Our master theorems (Appendix~\ref{sec:master-theorem-statements}) provide conditions on the induced score distributions under which our power formulas can be extended. Another direction for future work is to verify Conjectures~\ref{conj:optimal-tilde-m-point-mass} and~\ref{conj:optimal-tilde-m-subspace} regarding the optimal scaling of $\tilde m$ for BONuS and extending them to adjacent conformal applications. Furthermore, translating such theoretical insights into practical recommendations and methodology for choosing $\tilde m$ remains a separate open problem.

One of our main findings is that null augmentation methods like BONuS and knockoffs can be more powerful than sample splitting methods like Split BH. Aside from power, domain of applicability is an important consideration in method selection. Both null augmentation methods considered rely on a null exchangeability structure---global exchangeability for BONuS and local exchangeability for knockoffs---that may not apply to all problems. Sample splitting methods, on the other hand, do not require any exchangeability structure, and are more widely applicable. By the same token, non-masking methods may use the data somewhat more efficiently than masking methods, but usually require mathematical derivation \citep{LB16} or Monte Carlo computation \citep{CetL16} of null distributions, which is not always feasible. In sum, the power considerations explored in this work should be used in conjunction with other practical considerations to guide method selection.

\section*{Acknowledgement}
We thank Jiaoyang Huang for help with the proof of Lemma~\ref{lem:joint-scores-in-sample}. We thank Lucas Janson, Zhimei Ren, Jeffrey Zhang, Yaniv Romano, Anna Neufeld, Ronan Perry, Dan Kessler, and Snigdha Panigrahi for feedback on this work. J.L. and E.K. were partially supported by an NSF Graduate Research Fellowship and NSF DMS-2310654, respectively.

\clearpage

\appendix
\counterwithout{equation}{section}
\renewcommand{\theequation}{S\arabic{equation}}

\section{Multi-observation variant of data-generating model} \label{sec:multi-observation-model}

Here we consider an expanded version of the data-generating model~\eqref{eq:data-generating} that is amenable to sample splitting. Suppose that, independently across $j$, we observe $n$ samples for each $j = 1, \dots, m$:
\begin{equation}
\bm Z_{ij} \mid \bm \theta_j \iidsim N(\bm \theta_j/\sqrt{n}, \bm I_d), \quad i = 1, \ldots, n.
\label{eq:data-generating-expanded}
\end{equation}
Then, it is evident that, conditionally on $\bm \theta_j$, the variables
\begin{equation}
\bm X_j \equiv \frac{1}{\sqrt{n}} \sum_{i = 1}^n \bm Z_{ij}
\end{equation}
have the distribution~\eqref{eq:data-generating}. Consider splitting the $n$ samples into two groups $\mathcal I_\learn$ and $\mathcal I_\score$ of sizes $n_\learn \equiv \pi_\spl n$ and $n_\score \equiv (1 - \pi_\spl)n$, respectively (assuming $\pi_\spl n$ is an integer). Then, define
\begin{equation}
\bm X_j^\learn \equiv \frac{1}{\sqrt{n_\learn}} \sum_{i \in \mathcal I_\learn} \bm Z_{ij}, \quad \bm X_j^\score \equiv \frac{1}{\sqrt{n_\score}} \sum_{i \in \mathcal I_\score} \bm Z_{ij}.
\end{equation}
Then, conditionally on $\bm \theta_j$, note that $\{\bm X_j^\learn\}_{j \in [m]}$ and $\{\bm X_j^\score\}_{j \in [m]}$ are each independent across $j$ and independent of each other. Furthermore, we have
\begin{equation}
\begin{split}
\bm X_j^\learn \mid \bm \theta_j &\sim N(\sqrt{n_\learn/n}\cdot \bm \theta_j, \bm I_d)  = N(\sqrt{\pi_{\text{split}}} \cdot \bm \theta_j, \bm I_d); \\ 
\bm X_j^\score \mid \bm \theta_j &\sim N( \sqrt{n_\score/n} \cdot \bm \theta_j, \bm I_d) = N(\sqrt{1 - \pi_{\text{split}}} \cdot \bm \theta_j, \bm I_d),
\end{split}
\end{equation}
which are distributed the same way as $\bm X^\learn$ and $\bm X^\score$ defined in lines 2 and 3 of Algorithm~\ref{alg:split-bh}.

\section{Supporting material for Section~\ref{sec:asymptotic-power-comparison} on asymptotic power comparison}

\subsection{Derivation of the optimal split proportion $\pi_{\spl}^{\opt}$ for point prior} \label{sec:pi-opt}

Throughout, let $\pi \equiv \pi_{\spl}\in(0,1)$ denote the split proportion. Consider the objective
\begin{equation}\label{eq:obj}
f(\pi)
=
\frac{\sqrt{1-\pi}\, h^2\gamma}{\sqrt{h^2\gamma^2 + c/\pi}},
\qquad 0<\pi<1,
\end{equation}
where $h, \gamma, c > 0$. Since $f(\pi)>0$ on $(0,1)$, maximizing $f$ is equivalent to maximizing $f^2$. Dropping the positive constant $(h^2\gamma)^2$, we maximize
\begin{equation}\label{eq:Jpi}
J(\pi)
\;\equiv\;
\frac{1-\pi}{h^2\gamma^2 + c/\pi}
=
\frac{\pi(1-\pi)}{h^2\gamma^2\,\pi + c},
\qquad 0<\pi<1.
\end{equation}
Write $a \equiv h^2\gamma^2$. Up to the positive constant $(h^2\gamma)^2$, maximizing $f(\pi)$ is equivalent to maximizing
\[
J(\pi)=\frac{\pi(1-\pi)}{a\pi+c},\qquad 0<\pi<1.
\]
A direct calculation gives
\[
J'(\pi)
=\frac{(1-2\pi)(a\pi+c)-a(\pi-\pi^2)}{(a\pi+c)^2}
=\frac{c-a\pi^2-2c\pi}{(a\pi+c)^2}.
\]
Thus the stationary point satisfies
\[
a\pi^2+2c\pi-c=0,
\]
with the unique admissible solution
\[
\pi_{\spl}^{\opt}=\frac{-c+\sqrt{c^2+ac}}{a}
=\frac{\sqrt{c(c+h^2\gamma^2)}-c}{h^2\gamma^2}.
\]
Recalling that
\begin{equation}\label{eq:tau-mean}
\tau_{\mean}
\;\equiv\;
\frac{h\gamma}{\sqrt{h^2\gamma^2+c}}
\in(0,1)
\end{equation}
and defining the ratio $r\equiv c/a$, we find
\[
\tau_{\mean}^2=\frac{a}{a+c}=\frac{1}{1+r}
\qquad\Longrightarrow\qquad
r=\frac{1-\tau_{\mean}^2}{\tau_{\mean}^2}.
\]
Therefore, 
\[
\pi_{\spl}^{\opt}
=
\frac{\sqrt{c(c+a)}-c}{a}
=
\sqrt{r(r+1)}-r.
\]
Substituting $r=(1-\tau_{\mean}^2)/\tau_{\mean}^2$ and simplifying yields the compact form
\begin{equation}\label{eq:pi-opt-tau}
\pi_{\spl}^{\opt}
=
\frac{\sqrt{1-\tau_{\mean}^2}}{1+\sqrt{1-\tau_{\mean}^2}}
\end{equation}
stated in equation~\eqref{eq:optimal-pi-split}.

\subsection{Figures supporting Section~\ref{sec:subspace-prior}: power analysis under subspace prior}

Below are Figures~\ref{fig:subspace-prior-alignment} and~\ref{fig:subspace-mass-prior-theorem}, the analogs of Figures~\ref{fig:point-mass-prior-alignment} and~\ref{fig:point-mass-prior-theorem} for the subspace prior. They are qualitatively similar to their counterparts in the main text.
\begin{figure}[h!]
\centering
\includegraphics[width = \textwidth]{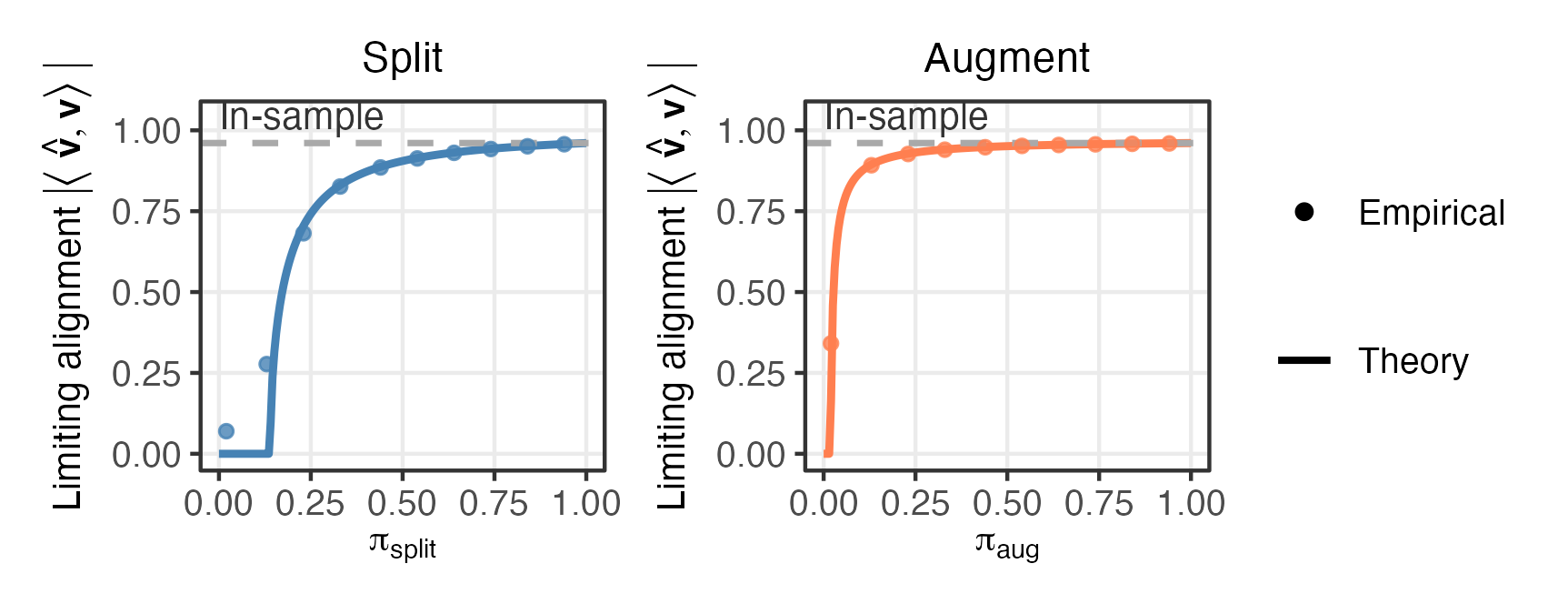}
\caption{Asymptotic alignment of learned direction $\hat{\bm v}$ with the true direction for $(\gamma, c, h) = (0.2, 0.2, 4)$. We overlay points representing the average alignment over 1000 replications of finite-sample ($m = 1000$) simulations for multiple values of $\pi_\aug$ or $\pi_\spl$.}
\label{fig:subspace-prior-alignment}
\end{figure}
\begin{figure}[h!]
\centering
\includegraphics{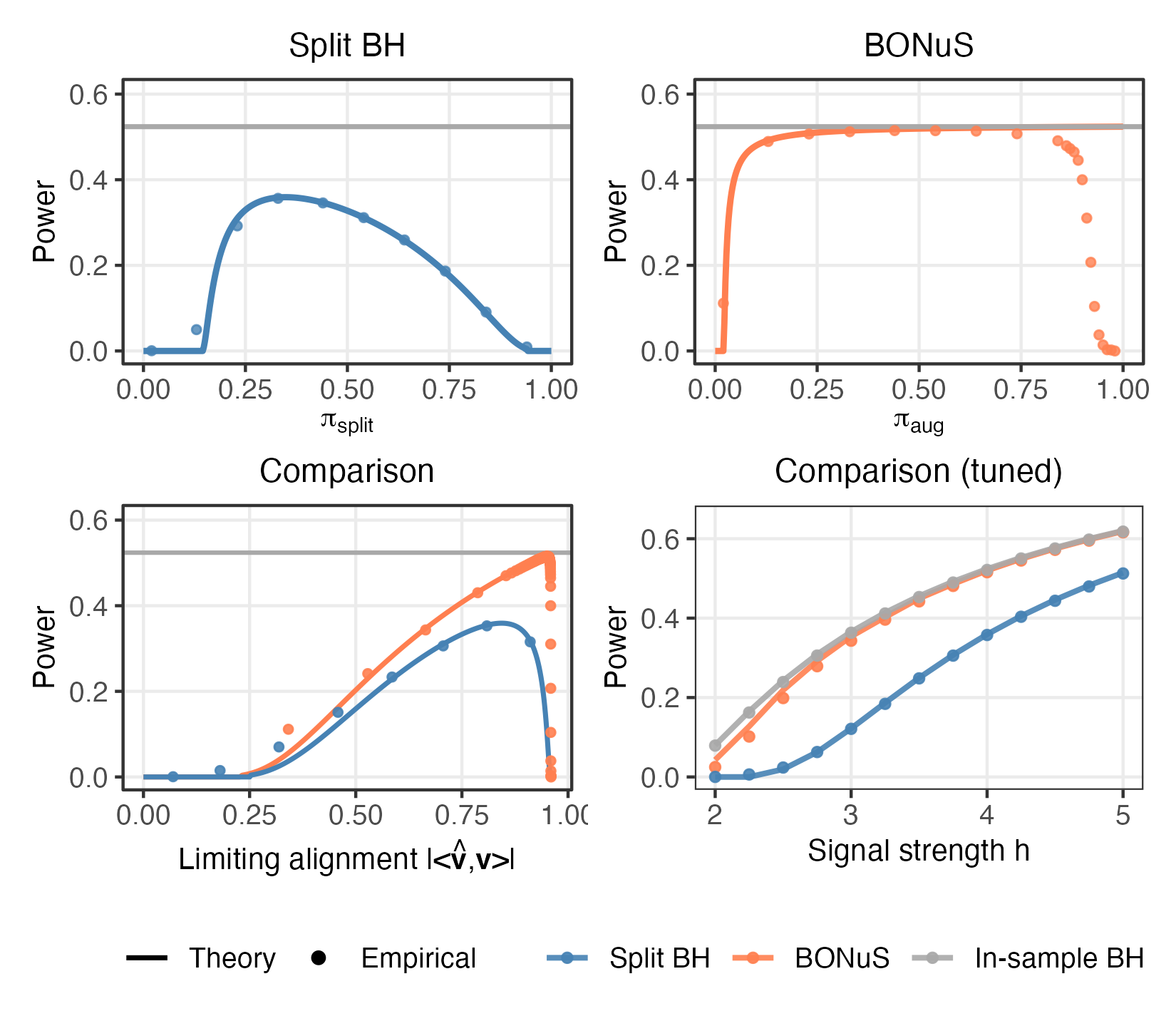}
% \caption{The powers of Split BH (left) and BONuS (middle) versus their tuning parameters, and compared to that of In-sample BH, for the subspace prior with $(\gamma, c, h) = (0.2, 0.2, 4)$. The lines represent the asymptotic powers predicted by Theorem~\ref{thm:power-point-mass} and the points represent numerical simulations with $m = 1000$ and 1000 replications. The powers of the procedures are compared as a function of the limiting alignment of the learned direction with the true direction (right).}
\caption{Comparing the powers of Split BH, BONuS, and In-sample BH under the subspace prior. Other details are as in the caption of Figure~\ref{fig:point-mass-prior-theorem}.} 
\label{fig:subspace-mass-prior-theorem}
\end{figure}

\clearpage

\subsection{Additional experiments showcasing Theorems \ref{thm:power-point-mass} and \ref{thm:power-one-dimensional} in Section~\ref{sec:asymptotic-power-comparison}} \label{sec:additional-experiments}

\paragraph{Higher-dimensional setting}

We complement the experiments in Section \ref{sec:asymptotic-power-comparison} by considering a higher-dimensional setting with $(\gamma, c, h) = (0.2, 1, 5)$ and $m = 1000$ (so we increased the aspect ratio $c$ by a factor of five, from 0.2 to 1). Figures \ref{fig:point-mass-prior-theorem-high-dim} and \ref{fig:subspace-mass-prior-theorem-high-dim} illustrate the asymptotic power comparison between Split BH and BONuS under the point mass prior and one-dimensional subspace prior, respectively. We again overlay points representing the average power over 1000 replications of finite-sample ($m = 1000$) simulations for multiple values of $\pi_\aug$ or $\pi_\spl$. The trends are similar to those observed in the main text for $c = 0.2$: the power of BONuS always exceeds that of Split BH, and the agreement between the simulated power and asymptotic power is still quite good everywhere but the ``finite-sample regime'' of $\pi_\aug \approx 1$. 

\begin{figure}[h]
\centering
\includegraphics[width = \textwidth]{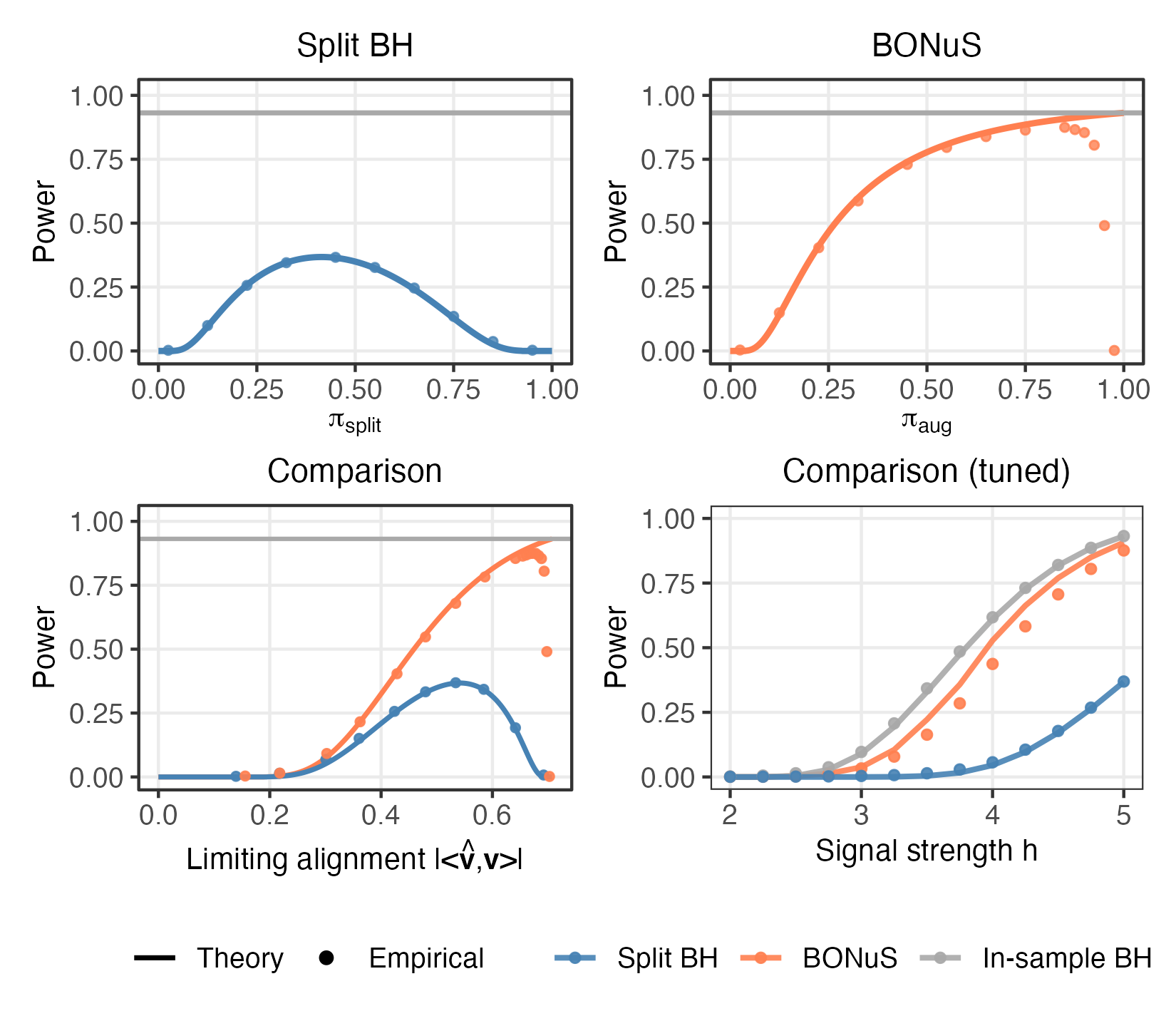}
\caption{Comparing the powers of Split BH, BONuS, and In-sample BH under the point mass prior for $(\gamma, c, h) = (0.2, 1, 5)$. Other details are as in the caption of Figure~\ref{fig:point-mass-prior-theorem}.}
\label{fig:point-mass-prior-theorem-high-dim}
\end{figure}

\clearpage

\begin{figure}[h]
\centering
\includegraphics[width = \textwidth]{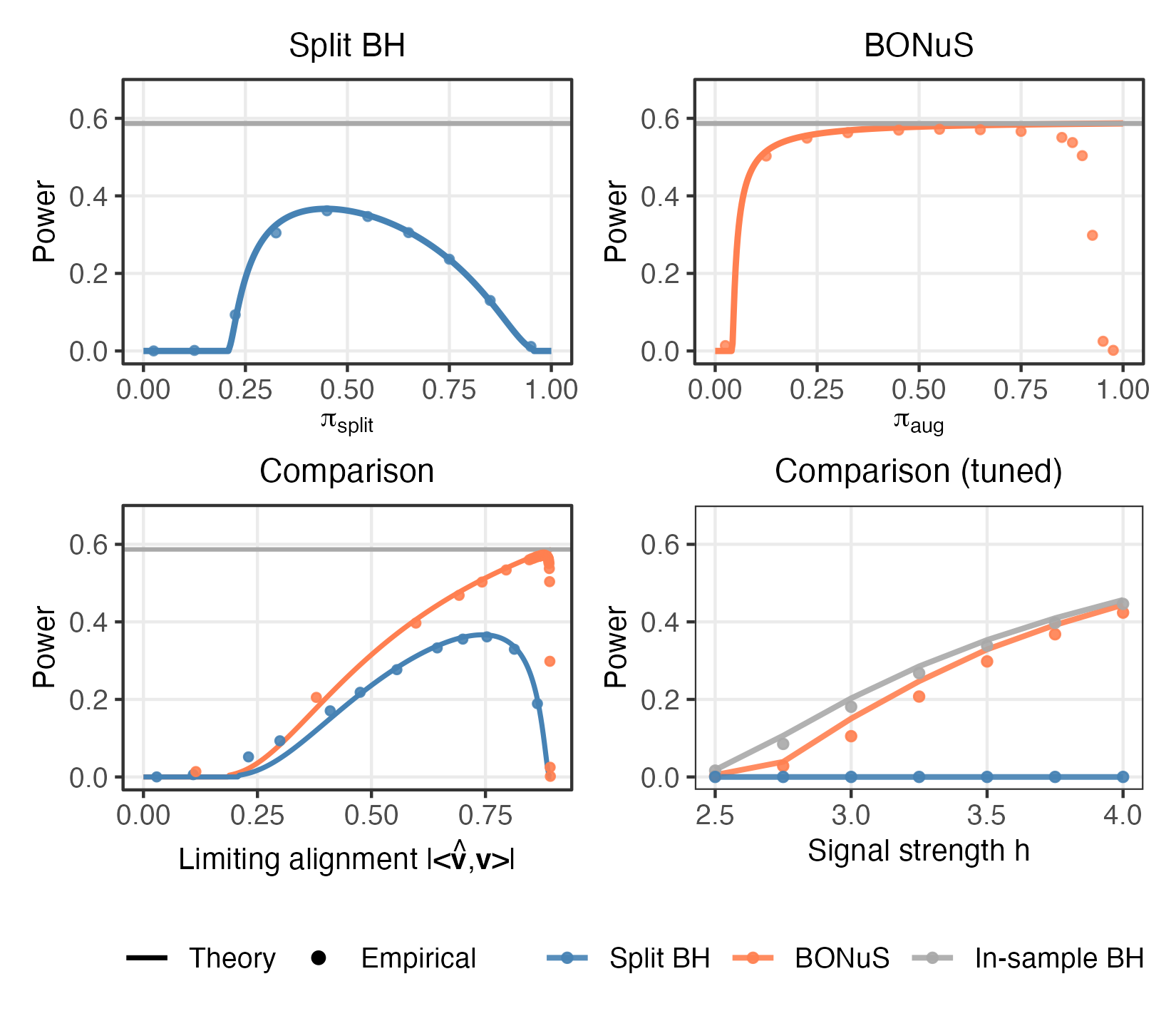}
\caption{Comparing the powers of Split BH, BONuS, and In-sample BH under the subspace prior for $(\gamma, c, h) = (0.2, 1, 5)$. Other details are as in the caption of Figure~\ref{fig:point-mass-prior-theorem}.}
\label{fig:subspace-mass-prior-theorem-high-dim}
\end{figure}

\clearpage

\paragraph{Sparse setting}

We also compare the power curves, both empirical and asymptotic, in a sparse setting where only a small fraction of features are non-null. We set $(\gamma, c, h) = (0.05, 0.2, 8)$ and $m = 1000$ (so we decreased the non-null proportion $\gamma$ by a factor of four, from 0.2 to 0.05). Figures \ref{fig:point-mass-prior-theorem-sparse} and \ref{fig:subspace-mass-prior-theorem-sparse} illustrate the asymptotic and finite-sample ($m=1000$, 1000 replications) power comparison between Split BH and BONuS under the point mass prior and one-dimensional subspace prior, respectively. We again see good agreement between the theory we develop and 
the simulated power, but with a higher penalty associated with the ``finite-sample regime'' of $\pi_\aug \approx 1$.  

\begin{figure}[h]
\centering
\includegraphics[width = \textwidth]{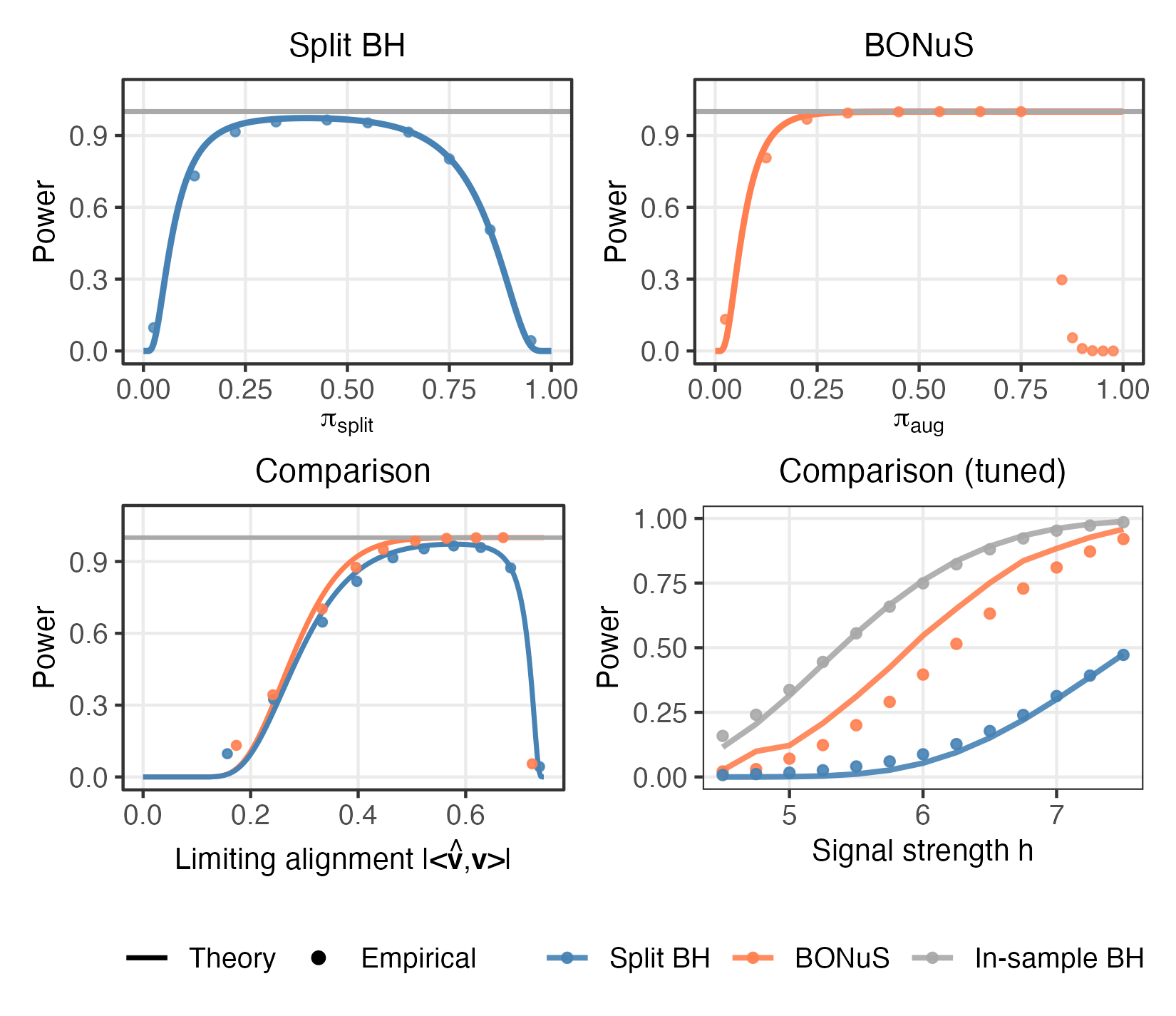}
\caption{Comparing the powers of Split BH, BONuS, and In-sample BH under the point mass prior for $(\gamma, c, h) = (0.05, 0.2, 8)$. Other details are as in the caption of Figure~\ref{fig:point-mass-prior-theorem}.}
\label{fig:point-mass-prior-theorem-sparse}
\end{figure}

\clearpage

\begin{figure}[h]
\centering
\includegraphics[width = \textwidth]{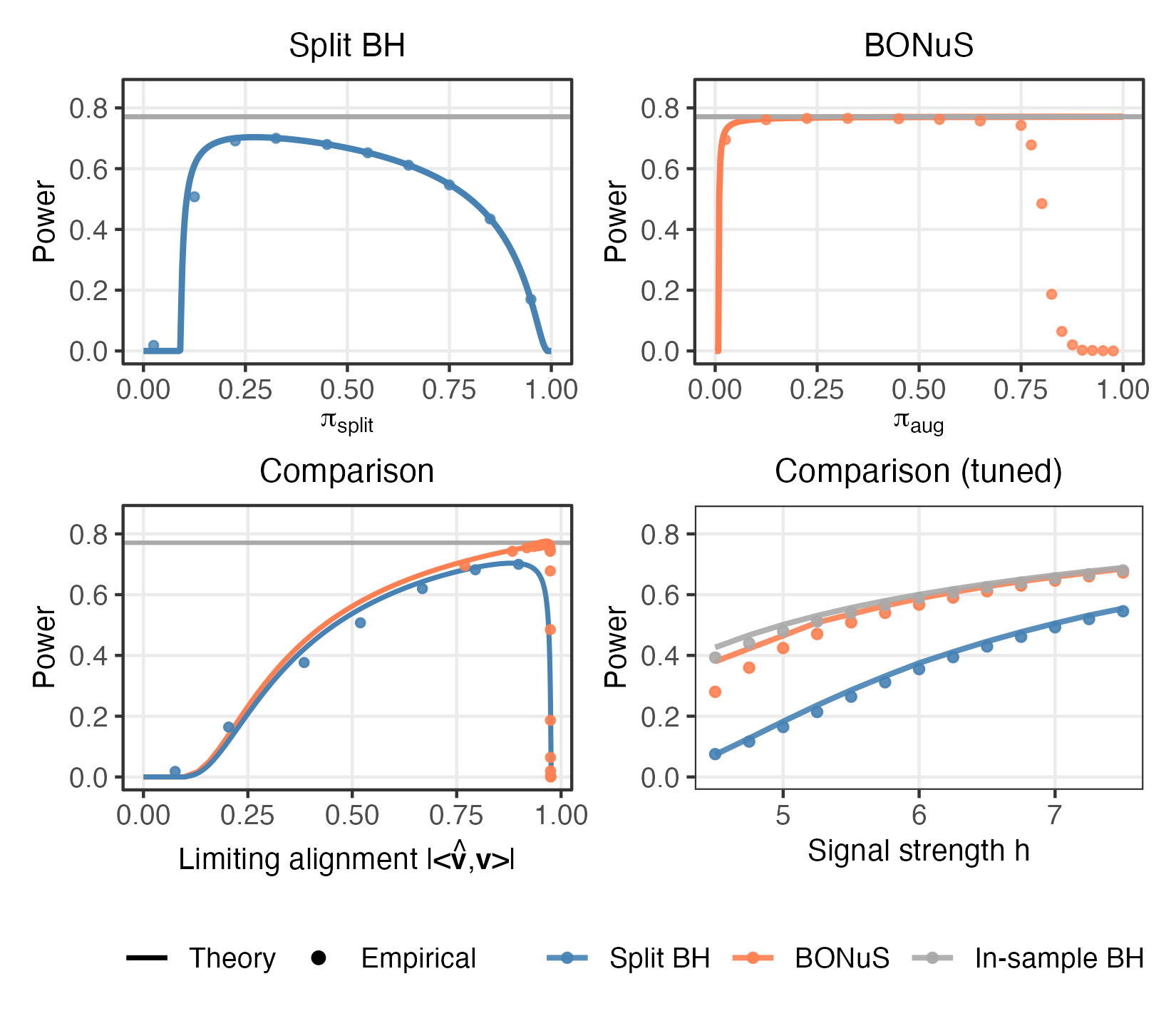}
\caption{Comparing the powers of Split BH, BONuS, and In-sample BH under the subspace prior for $(\gamma, c, h) = (0.05, 0.2, 8)$. Other details are as in the caption of Figure~\ref{fig:point-mass-prior-theorem}.}
\label{fig:subspace-mass-prior-theorem-sparse}
\end{figure}

\clearpage

\section{A tractable approximation to the power of BONuS under point mass prior} \label{sec:tractable-approximation}

Let $S_j = T_{\hat \Lambda}(\bm X_j)$ and $\tilde S_k = T_{\hat \Lambda}(\tilde{\bm X_k})$ be the BONuS scores for hypothesis $j$ and null sample $k$, respectively. Throughout, we drop the ``$\bonus$'' superscript on these scores and other quantities to ease notation. An equivalent expression for the BONuS rejection set is given in Appendix~\ref{sec:algorithmic-variants}:
\begin{equation}
\begin{split}
  &\widehat{\mathcal R} = \{j \in [m]: S_j\geq t_*\}, \quad \text{where} \\
&t_* = \inf\{t \in \R: \widehat{\text{FDP}}(t) \leq q\}; \quad \widehat{\text{FDP}}(t) = \frac{\frac{1}{1+\tilde m}\left(1 + \sum_{k = 1}^{\tilde m}\indicator(\tilde S_k \geq t) \right)}{\frac1m \sum_{j = 1}^m \indicator(S_j \geq t)}.
\end{split}
\end{equation} 
It is challenging to analyze the impact of $\tilde m$ on this procedure within the regime $\tilde m \ll m$. We instead analyze an approximation to the power of BONuS, which we obtain as follows.

% Let us use these expressions to trace the effect of $\tilde m$ on the power of BONuS. If $\tilde m$ is too small, then the numerator of $\widehat{\text{FDP}}(t)$ becomes highly discrete. Compared to the ``ideal'' numerator $\P[\tilde S_k \geq t]$, this discreteness manifests itself through the pseudocount term $1/(1+\tilde m)$ (an upward bias) and the randomness in the summation term (variance). This tends to make the procedure less powerful. On the other hand, if $\tilde m$ is too large, 

First, we approximate the distributions of the scores $S_j$ and $\tilde S_k$. Recall from Proposition~\ref{prop:point-mass-prior-score-distributions} that the limiting mean shift of the BONuS scores $S_j$ for non-null hypotheses is
\begin{equation}
\mu(\gamma, c, h; \pi_\aug) = \frac{h^2 \gamma}{\sqrt{h^2 \gamma^2 + c/\pi_\aug}} = \lim_{\tilde m, m \rightarrow \infty} \frac{h^2 \gamma}{\sqrt{h^2 \gamma^2 + c(1 + \tilde m / m)}} \equiv \lim_{\tilde m, m \rightarrow \infty} \mu'(\tilde m).
\label{eq:approximate-mean-shift}
\end{equation}
The quantity $\mu'(\tilde m)$ captures the mean shift as a function of $\tilde m$, leaving the dependence on $h, \gamma, c, m$ implicit. We can then use Proposition~\ref{prop:point-mass-prior-score-distributions} and the relation~\eqref{eq:approximate-mean-shift} to approximate the distributions of the scores:
\begin{equation}
\begin{split}
S_j, \tilde S_k \overset \cdot \sim N(K, 1) \ \text{for } j \in \mathcal H_0, k \in [\tilde m], \quad 
 S_j \overset \cdot \sim N(K + \mu'(\tilde m), 1) \ \text{for } j \in \mathcal H_1.
\end{split}
\end{equation}

Next, we use these approximate score distributions to obtain an approximation for the power of BONuS. To this end, we approximate 
\begin{equation}
\begin{split}
\widehat{\text{FDP}}(t + K) &= \frac{\frac{1}{1+\tilde m}\left(1 + \sum_{k = 1}^{\tilde m}\indicator(\tilde S_k - K \geq t) \right)}{\frac1m \sum_{j = 1}^m \indicator(S_j - K \geq t)} \\
&\approx \frac{\frac{1}{1 + \tilde m} + \frac{\tilde m}{1 + \tilde m}\bar \Phi(t)}{(1-\gamma)\bar \Phi(t) + \gamma \bar \Phi(t - \mu'(\tilde m))} \equiv \widehat{\text{FDR}}'(t; \tilde m).
\end{split}
\label{eq:approximate-fdr}
\end{equation}
The quantity $\widehat{\text{FDR}}'(t; \tilde m)$ is a deterministic approximation to $\widehat{\text{FDP}}(t + K)$, obtained by taking expectations of the numerator and denominator under the approximate score distributions. Next, we approximate
\begin{equation}
\begin{split}
t_* - K = \inf\{t \in \R: \widehat{\text{FDP}}(t + K) \leq q\} &\approx \inf\{t \in \R: \widehat{\text{FDR}}'(t; \tilde m) \leq q\} \\
&\approx \inf\{t \leq C: \widehat{\text{FDR}}'(t; \tilde m) \leq q\} \equiv t'_*(\tilde m),
\end{split}
\label{eq:approximate-threshold}
\end{equation}
where $C > 0$ is an arbitrary large constant required for technical reasons. Finally,
\begin{equation}
\begin{split}
\TPP(\tilde m) = \frac{1}{|\mathcal H_1|}\sum_{j \in \mathcal H_1}\indicator(S_j \geq t_*) &\approx \frac{1}{|\mathcal H_1|}\sum_{j \in \mathcal H_1}\indicator(S_j - K \geq t'_*(\tilde m)) \\
&\approx \bar \Phi(t'_*(\tilde m) - \mu'(\tilde m)) \equiv \TPR'_\bonus(\tilde m).
\end{split}
\label{eq:approximate-power}
\end{equation}
The function $\TPR'_\bonus(\tilde m)$ defined through equations~\eqref{eq:approximate-mean-shift}, \eqref{eq:approximate-fdr}, \eqref{eq:approximate-threshold}, and \eqref{eq:approximate-power} is a tractable approximation to the power of BONuS for finite $\tilde m$. 

\section{One-to-one null augmentation} \label{sec:knockoffs}

\subsection{Knockoffs as one-to-one null augmentation} 

BONuS is an effective augmentation-based method for scenarios (like ours) where all null data points are exchangeable, so a common pool of additional nulls can be used to test each hypothesis (\textit{one-to-many} null augmentation). In other scenarios, such as variable selection, there is usually no such exchangeable structure to leverage. However, it may still be possible to generate one synthetic null per hypothesis, so that exchangeability holds per hypothesis rather than globally (\textit{one-to-one} null augmentation). This idea underlies \textit{knockoffs} \citep{BC15,CetL16}, a leading approach to variable selection. In this section, we empirically compare a variant of knockoffs adapted to our problem setting to the three methods studied previously.

\begin{wrapfigure}[11]{r}{0.4\textwidth}
\vspace{-0.225in}                    % small pull-up (tune or drop)
\begin{center}
\begin{minipage}[t]{0.4\textwidth}
\begin{algorithm}[H]
\LinesNumbered
\caption{Knockoffs}
\label{alg:knockoffs}
$\tilde{\bm X}_{j} \sim N(\bm 0, \bm I_d)$, $j = 1, \ldots, m$\;
$\bm X^\learn \gets [\bm X; \tilde{\bm X}]$\;
\(\hat \Lambda \gets L(\bm X^\learn)\)\;
$S_j \gets W_{\hat \Lambda}(\bm X_j, \tilde{\bm X}_j)$\;
\(\mathcal R(t)\gets\{j:\,S_j\ge t\}\)\;
$\widehat V(t) \gets 1 + \sum_{j = 1}^{m}\indicator(S_j \leq -t)$\;
\(\hat t \gets \inf\{t > 0:\,\widehat V(t)/|\mathcal R(t)|\le q\}\)\;
\KwRet{\(\widehat{\mathcal R} = \mathcal R(\hat t)\)}
\end{algorithm}
\end{minipage}
\end{center}
\end{wrapfigure}

Knockoffs (Algorithm~\ref{alg:knockoffs}) starts similarly to BONuS (Algorithm~\ref{alg:bonus}), by generating a set of synthetic nulls $\tilde{\bm X}$ and then learning $\hat \Lambda$ based on the augmented data $[\bm X; \tilde{\bm X}]$. Unlike BONuS, exactly $m$ synthetic null samples are generated (one per original sample). Furthermore, knockoffs uses a scoring function $W_{\hat \Lambda}(\bm X_j, \tilde{\bm X}_j)$ that contrasts the signal present in $\bm X_j$ and $\tilde{\bm X}_j$. This scoring function must satisfy the anti-symmetry property that $W_{\hat \Lambda}(\tilde{\bm x}_j, \bm x_j) = -W_{\hat \Lambda}(\bm x_j, \tilde{\bm x}_j)$ for any $\bm x_j, \tilde{\bm x}_j \in \R^d$. The estimation of the number of false discoveries at a threshold $t > 0$ follows a different formula, based on the property that every null score has a symmetric distribution about the origin due to anti-symmetry. 

\subsection{The choice of knockoff statistic} 

Whereas for the three procedures studied previously, the clear choice of test statistic is $T_{\hat \Lambda}(\bm X^\score_j)$~\eqref{eq:likelihood-ratio-statistic}, the contrast form of the knockoffs test statistic makes it nontrivial to choose. A natural choice is to directly contrast the values of the log likelihood ratio statistics for the original and knockoff samples:
\begin{equation}
W_{\hat \Lambda}^{\text{vanilla}}(\bm X_j, \tilde{\bm X}_j) \equiv T_{\hat \Lambda}(\bm X_j) - T_{\hat \Lambda}(\tilde{\bm X}_j).
\label{eq:vanilla-knockoff-stat}
\end{equation}
This construction echoes common knockoff statistic choices in the literature, like the lasso coefficient difference \citep{CetL16}. On the other hand, \citet{Spector2022e} investigated the optimality of knockoff statistics in a Bayesian context that includes the present setup (Section~\ref{sec:problem-setup}). They found that the optimal statistic is the \textit{masked likelihood ratio} (MLR): 
\begin{equation}
\begin{split}
W_j^*(\bm x, \tilde{\bm x}) &\equiv \log \left(\frac{\E_{\bm \theta}[L_{\bm \theta}(\bm X_{j} = \bm x_{j} \mid \{\bm X_{j'}, \tilde{\bm X}_{j'}\}_{j' \in [m]} = \{\bm x_{j'}, \tilde{\bm x}_{j'}\}_{j' \in [m]})]}{\E_{\bm \theta}[L_{\bm \theta}(\bm X_{j} = \tilde{\bm x}_{j} \mid \{\bm X_{j'}, \tilde{\bm X}_{j'}\}_{j' \in [m]} = \{\bm x_{j'}, \tilde{\bm x}_{j'}\}_{j' \in [m]})]}\right), 
\end{split}
\end{equation}
where $ \{\bm X_{j'}, \tilde{\bm X}_{j'}\}_{j' \in [m]}$ is the set of unordered pairs $\{\bm X_{j'}, \tilde{\bm X}_{j'}\}$ for each $j' \in [m]$ and $L_{\bm \theta}$ is the likelihood given $\bm \theta = (\bm \theta_1, \dots, \bm \theta_m)$. In general variable selection problems, this expression requires Monte Carlo to compute but in our case can be computed explicitly. Letting $p$ denote the marginal density of each $\bm X_j$ under our data-generating model, we can rewrite the optimal statistic $W_j^*$ via
\begin{equation}
\begin{split}
W_j^*(\bm x, \tilde{\bm x}) &= \log \left(\frac{\E_{\bm \theta}[L_{\bm \theta}(\bm X_{j} = \bm x_{j} \mid \{\bm X_{j'}, \tilde{\bm X}_{j'}\}_{j' \in [m]} = \{\bm x_{j'}, \tilde{\bm x}_{j'}\}_{j' \in [m]})]}{\E_{\bm \theta}[L_{\bm \theta}(\bm X_{j} = \tilde{\bm x}_{j} \mid \{\bm X_{j'}, \tilde{\bm X}_{j'}\}_{j' \in [m]} = \{\bm x_{j'}, \tilde{\bm x}_{j'}\}_{j' \in [m]})]}\right) \\
&= \log \left(\frac{\E_{\bm \theta}[L_{\bm \theta}(\bm X_j = \bm x_j \mid \{\bm X_j, \tilde{\bm X}_j\} = \{\bm x_j, \tilde{\bm x}_j\})]}{\E_{\bm \theta}[L_{\bm \theta}(\bm X_j = \tilde{\bm x}_j \mid \{\bm X_j, \tilde{\bm X}_j\} = \{\bm x_j, \tilde{\bm x}_j\})]}\right) \\
&= \log \left(\frac{p(\bm X_j = \bm x_j \mid \{\bm X_j, \tilde{\bm X}_j\} = \{\bm x_j, \tilde{\bm x}_j\})}{p(\bm X_j = \tilde{\bm x}_j \mid \{\bm X_j, \tilde{\bm X}_j\} = \{\bm x_j, \tilde{\bm x}_j\})}\right) \\
&= \log \left(\frac{p(\bm X_j = \bm x_j, \tilde{\bm X}_j = \tilde{\bm x}_j)}{p(\bm X_j = \tilde{\bm x}_j, \tilde{\bm X}_j = \bm x_j)}\right) \\
&= \log \left(\frac{p(\bm X_j = \bm x_j)}{p(\tilde{\bm X}_j = \bm x_j)}\right) - \log \left(\frac{p(\bm X_j = \tilde{\bm x}_j)}{p(\tilde{\bm X}_j = \tilde{\bm x}_j)}\right),
\end{split}
\end{equation}
where we have used the mutual independence of $\bm X_j$ and $\tilde{\bm X}_j$ across $j$ in the first and last steps. Given our data-generating model~\eqref{eq:data-generating} and~\eqref{eq:two-groups}, we can derive the log-likelihood ratio for each $j$ as
\begin{equation}
\begin{split}
\log \left(\frac{p(\bm X_j = \bm x_j)}{p(\tilde{\bm X}_j = \bm x_j)}\right) &= \log \left(\frac{(1-\gamma) \phi_d(\bm x_j) + \gamma \int_{\R^d}\phi_d(\bm x_j - \bm \theta)d\Lambda(\bm \theta)}{\phi_d(\bm x_j)}\right) \\
&= \log\left(1-\gamma + \gamma \frac{\int_{\R^d}\phi_d(\bm x_j - \bm \theta)d\Lambda(\bm \theta)}{\phi_d(\bm x_j)}\right) \\
&= \log\left(1-\gamma + \gamma \exp(T_{\Lambda}(\bm x_j))\right).
\end{split}
\end{equation}
Therefore, MLR knockoff statistics reduce to
\begin{equation}
\begin{split}
W_{\hat \Lambda}^{\text{MLR}}(\bm X_j, \tilde{\bm X}_j) &= \log\left(1 - \gamma + \gamma \exp(T_{\hat \Lambda}(\bm X_j))\right) - \log\left(1 - \gamma + \gamma \exp(T_{\hat \Lambda}(\tilde{\bm X}_j))\right) \\
&\equiv f(T_{\hat \Lambda}(\bm X_j)) - f(T_{\hat \Lambda}(\tilde{\bm X}_j)),
\end{split}
\label{eq:mlr-knockoff-stat}
\end{equation}
where $f(t) \equiv \log(1-\gamma + \gamma \exp(t))$. 

\subsection{An empirical comparison of knockoffs to Split BH, BONuS, and In-sample BH}

We empirically compare knockoffs to the three methods studied previously. We include both statistics $W_{\hat \Lambda}^{\text{vanilla}}$ and $W_{\hat \Lambda}^{\text{MLR}}$ to represent options that are less powerful but easy to compute and more powerful but hard to compute (in more general settings), respectively. We implement $W_{\hat \Lambda}^{\text{MLR}}$ with oracle knowledge of $\gamma$ and $h$ for simplicity. Since neither $W_{\hat \Lambda}^{\text{vanilla}}$ nor $W_{\hat \Lambda}^{\text{MLR}}$ is necessarily invariant to monotonic transformations of $T_{\hat \Lambda}$ (unlike the other three methods considered), we keep $T_{\hat \Lambda}$ as the exact log-likelihood ratios defined in equations~\eqref{eq:likelihood-ratio-statistic-point-mass} and \eqref{eq:likelihood-ratio-statistic-subspace}. In order to put all methods compared on equal footing, we ensure that the quality of the learned direction is the same across methods. Since $\tilde m = m$ for knockoffs, the augmentation proportion is fixed at $\pi_\aug = 1/2$. We therefore choose $\pi_\spl$ and $\pi_\aug$ for Split BH and BONuS, respectively, to match the resulting quality of the learned direction $\hat{\bm v}$ to that of knockoffs. We assess the power of all five methods under both choices of prior~\eqref{eq:priors}.

Figure~\ref{fig:theorem-over_h} shows the results of our simulation. The ordering of the powers of Split BH, BONuS, and In-sample BH remains as in Section~\ref{sec:asymptotic-power-comparison}. Compared to these methods, we find that the power of MLR knockoffs essentially matches that of BONuS for both cases of prior. By contrast, vanilla knockoffs has lower power than MLR knockoffs for the point mass prior (essentially matching the power of Split BH), while matching the power of both MLR knockoffs and BONuS for the subspace prior.

\begin{figure}[h!]
\centering
\includegraphics{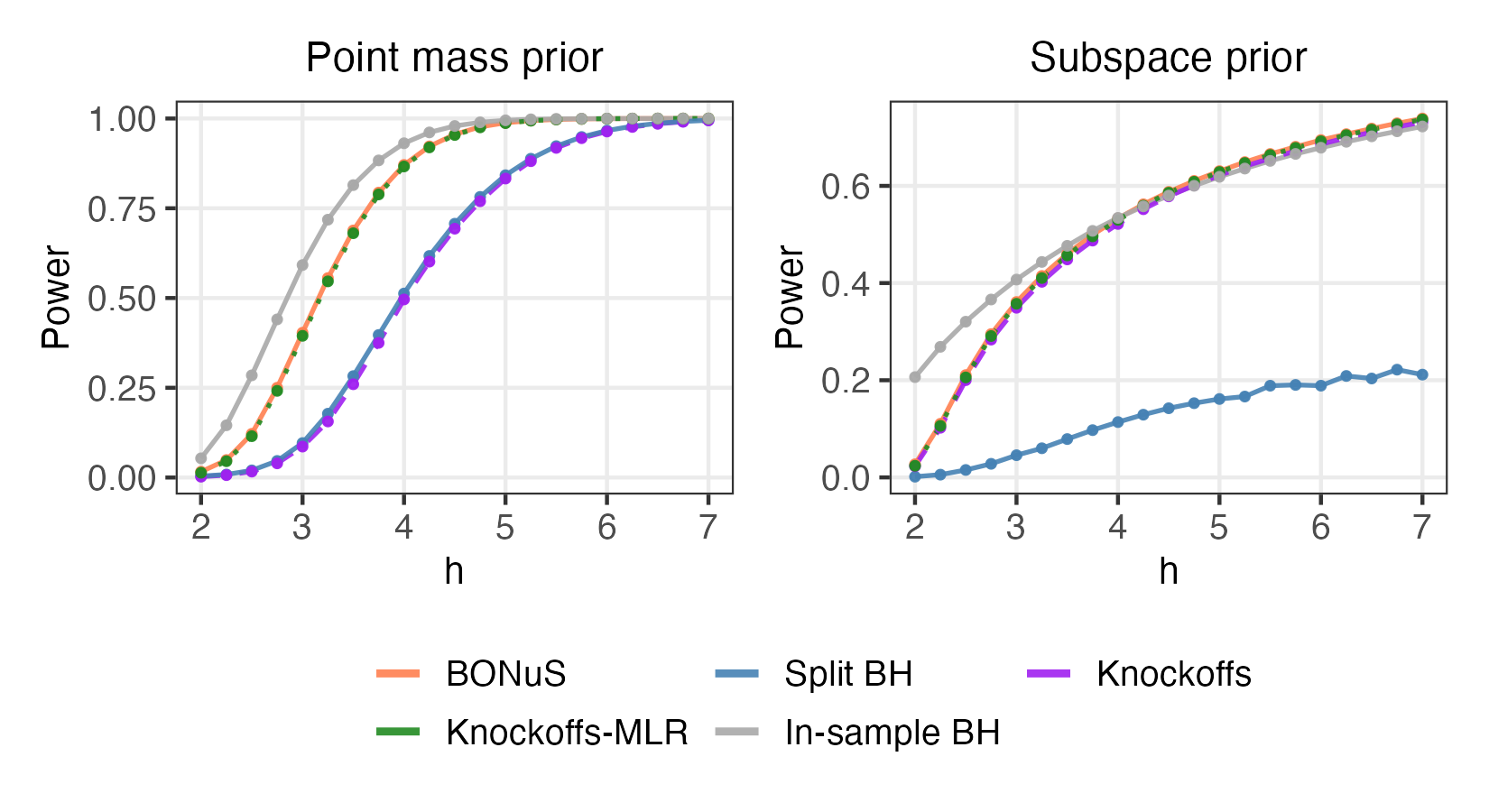}
\caption{Empirical power of Split BH, BONuS, In-sample BH, and two knockoffs variants. To fix learner quality across methods, we set $\tilde m = m = 1000$ for the augmentation methods (i.e., $\pi_\aug = 0.5$) and tune $\pi_\spl$ for Split BH accordingly using Propositions \ref{prop:mean-learner-masked} and \ref{prop:pca-learner-masked}.}
\label{fig:theorem-over_h}
\end{figure}

To better understand these results, we make several observations. We start with the point prior, where 
\begin{equation}
W_{\hat \Lambda}^{\text{vanilla}}(\bm X_j, \tilde{\bm X}_j) = (h\hat{\bm v}^\top \bm X_j - \tfrac12 h^2) - (h\hat{\bm v}^\top \tilde{\bm X}_j - \tfrac12 h^2) = h\hat{\bm v}^\top (\bm X_j - \tilde{\bm X}_j).
\end{equation}
The contrast between $\bm X_j$ and $\tilde{\bm X}_j$ effectively doubles the variance of the test statistic, decreasing the power of knockoffs relative to BONuS, whose score is based solely on $\hat{\bm v}^\top \bm X_j$. In fact, the term $\bm X_j - \tilde{\bm X}_j$ is a constant multiple of $\bm X^\score_j$ in line 3 of Split BH (Algorithm~\ref{alg:split-bh}). This suggests why the power of vanilla knockoffs matches that of Split BH for the point prior. This detrimental impact of the knockoff contrast has been observed previously \citep{Wang2020b,Weinstein2023}. However, knockoffs does not always suffer from this issue, as we see in moving from the vanilla knockoff statistic to the MLR knockoff statistic
\begin{equation}
W_{\hat \Lambda}^{\text{MLR}}(\bm X_j, \tilde{\bm X}_j) = f(h\hat{\bm v}^\top \bm X_j - \tfrac12 h^2) - f(h\hat{\bm v}^\top \tilde{\bm X}_j - \tfrac12 h^2).
\end{equation}
This operation is visualized in the top row of Figure~\ref{fig:mlr-transformation-examples}, which displays the inputs to the scores for the alternative hypotheses. The top-left panel shows $T_{\hat \Lambda}(\bm X_j)$ (red) and $T_{\hat \Lambda}(\tilde{\bm X}_j)$ (blue). The blue histogram is subtracted from the red histogram in the construction of the knockoff score, leading to the variance-doubling effect compared to BONuS, which uses just the red histogram. The MLR transformation (top middle) has the effect of compressing negative inputs to roughly zero, while keeping positive inputs roughly unchanged. The result is that the transformation collapses the blue histogram while keeping the red histogram nearly intact. For this reason, the MLR knockoff statistics are not very different from the BONuS statistics (the red histogram), leading to the two methods having similar power. By contrast, for the subspace prior, the knockoff log-likelihood ratios $T_{\hat \Lambda}(\tilde{\bm X}_j)$ are much smaller than the original log-likelihood ratios $T_{\hat \Lambda}(\bm X_j)$ (bottom left of Figure~\ref{fig:mlr-transformation-examples}). Therefore, even the vanilla knockoff statistics $W_{\hat \Lambda}^{\text{vanilla}}$ are close to the BONuS statistics, leading to similar power between the two methods. The MLR transformation (bottom middle) has little effect in this case, so MLR knockoffs and vanilla knockoffs also have similar power. Other instances where the augmentation of knockoffs does not have a detrimental effect, even in the absence of MLR statistics, have been observed \citep{Weinstein2017,Wang2020b,Weinstein2023, Ke2021a}.

\begin{figure}[h!]
\centering
\includegraphics[width = \textwidth]{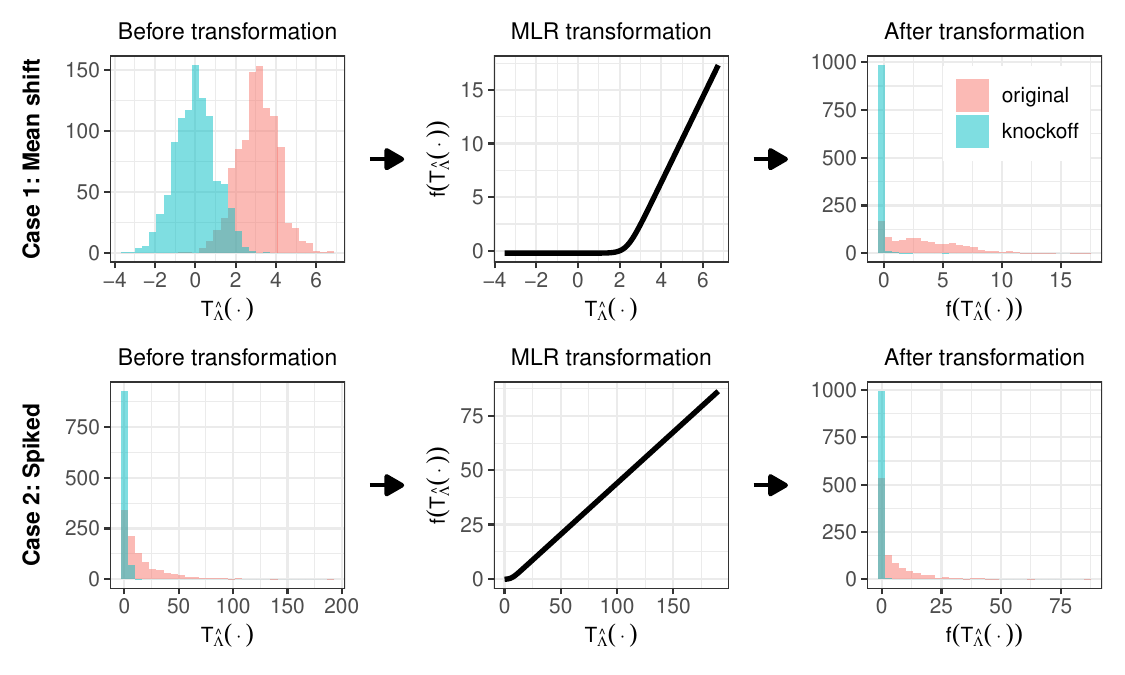}
\caption{The impact of the MLR transformation $f(t) \equiv \log(1-\gamma + \gamma \exp(t))$ on the distributions of original and knockoff scores for alternative hypotheses.}
\label{fig:mlr-transformation-examples}
\end{figure}

In conclusion, knockoffs can be less powerful than BONuS because the contrast form of the knockoff statistic effectively adds noise to the BONuS statistic. However, the MLR transformation can mitigate this issue by collapsing the noise term, leading to knockoff statistics that are close to the BONuS statistics and thus similar power. While MLR knockoffs statistics are computationally challenging in general settings, perhaps this insight can inspire improved knockoff statistics that mimic the MLR effect while being easier to compute.

\section{Small-$\tilde m$ behavior of BONuS for subspace prior} \label{sec:small-tilde-m-subspace-prior}

We follow a similar strategy to that of Section~\ref{sec:small-tilde-m-point-mass-prior}. We require a \emph{detectable signal condition}
\begin{equation}\label{eq:strong-signal-subspace}
	\gamma^2 h^4 > c,
\end{equation}
which ensures that the alternative direction learned on the entire data via the PCA learner has nontrivial alignment with the true direction (recall Lemma~\ref{lem:pca-learner-untransformed}). As in Section~\ref{sec:small-tilde-m-point-mass-prior}, we define a finite-sample variant of the BONuS signal strength~\eqref{eq:subspace-prior-signal-strengths}
% Under this regime, the In-sample BH alternative scaling factor in equation~\eqref{eq:subspace-prior-asymptotic-power} is
% \[
% 1+\sigma^2_{\insamp}
% \;=\;
% 1+h^2\cdot\frac{1-\frac{c}{\gamma^2 h^4}}{1+\frac{c}{\gamma h^2}}
% \;>\;1
% \]
% is well defined and strictly greater than one, confirming the presence of a nontrivial alternative component.
% The effect of $\tilde m$ on power is governed by two competing mechanisms. On one hand, the inclusion of synthetic nulls introduces a \emph{pseudocount bias} through the term $\frac{1}{1+\tilde m}$ in the FDP estimator of \textsc{BONuS}. On the other hand, increasing $\tilde m$ weakens the learned direction by reducing the effective learning fraction $\pi_{\aug} = \frac{m}{m+\tilde m}$ below one. To isolate these effects, we adopt the same proxy analysis strategy used in the point prior case.
\begin{equation}\label{eq:sigma2-finite-tilde-m}
	\mu'(\tilde m)
	\equiv 
	h\cdot
	\sqrt{\frac{\Bigl(1 - \frac{c (1+\tilde m/m)}{\gamma^2 h^4}\Bigr)_+}{1 + \frac{c}{\gamma h^2}}}
\end{equation}
and the deterministic approximation
\begin{equation}
\widehat{\textnormal{FDR}}'(t; \tilde m) \equiv \frac{\frac{1}{1+\tilde m}
		+
		\frac{\tilde m}{1+\tilde m}\,\bar F_0(t)}
	{(1-\gamma)\,\bar F_0(t)
		+
		\gamma\,\bar F_0\!\Big(\frac{t}{1 + (\mu'(\tilde m))^2}\Big)},
\label{eq:approximate-fdr-subspace}
\end{equation}
where $F_0$ denotes the CDF of $\chi^2_1$. Define
\[
t'_*(\tilde m):=\inf\{t\le C:\widehat{\textnormal{FDR}}'(t;\tilde m)\le q\},\qquad
\TPR'_\bonus(\tilde m):=\bar F_0\!\left(\frac{t'_*(\tilde m)}{1+(\mu'(\tilde m))^2}\right).
\]
We then expand \(\TPR'_\bonus(\tilde m)\) around the power of In-sample BH~\eqref{eq:subspace-prior-asymptotic-power}:
\begin{equation}
\TPR_{\insamp} \equiv \TPR_{\textnormal{BH}}(\chi^2_1, (1 + \mu^2_{\insamp}(\gamma, c, h))\chi^2_1, \gamma),
\end{equation}
giving the following result.
\begin{theorem}[Small–$\tilde m$ Expansion under Subspace Prior]
	\label{thm:subspace-small-m}
	Assume the detectable signal condition~\eqref{eq:strong-signal-subspace}. Then, as $\tilde m \to \infty$ with $\tilde m/m \to 0$,
	\begin{equation}\label{eq:expansion-subspace-tilde-m}
		\TPR'_\bonus(\tilde m)
		=
		\TPR_{\insamp}
		-
		\eta_1\,\frac{1}{\tilde m}
		-
		\eta_2\,\frac{\tilde m}{m}
		+
		o\!\left(\frac{1}{\tilde m} + \frac{\tilde m}{m}\right),
	\end{equation}
	where $\eta_1,\eta_2 > 0$ depend only on $(\gamma,c,h,q)$.
\end{theorem}
The proof can be found in Section \ref{sec:small-tilde-m-subspace-prior-proof}.
Except for the constants involved, this result is identical to Theorem~\ref{thm:point-mass-small-m}, and we therefore make a conjecture similar to Conjecture~\ref{conj:optimal-tilde-m-point-mass}:
\begin{conjecture} \label{conj:optimal-tilde-m-subspace} Let $\tilde m_{\textnormal{opt}}$ be the choice of $\tilde m$ that maximizes the power of BONuS in the case of subspace prior for a fixed value of $m$. Under the strong-signal regime~\eqref{eq:strong-signal-subspace}, $\tilde m_{\textnormal{opt}} \asymp \sqrt{m}$.
\end{conjecture}
 
Finally, we note that we did not validate Conjecture~\ref{conj:optimal-tilde-m-subspace} empirically, as the structure of the problem with the subspace prior leads to BONuS's power being fairly insensitive to $\tilde m$ for most parameter settings. This phenomenon is related to the fast saturation of the limiting alignment of the learned direction with the true direction (Figure~\ref{fig:subspace-prior-alignment}, right).
 
\section{Additional details on numerical simulations}
\label{sec:omitted_details_sims}

\subsection{Figure \ref{fig:point-mass-prior-optimal}}
\label{sec:omitted_details_optimal_mtilde}

This figure consists of two panels: the left plots $\tilde{m}_\opt$ as a function of $m$ (on a log-log scale) to capture the theoretical phenomenon conjectured in Conjecture \ref{conj:optimal-tilde-m-point-mass}, while the right panel compares the corresponding actualized TPR at the conjectured optimal $\tilde{m}_\opt = \sqrt{\eta_1/\eta_2} \sqrt m$ with the empirically found maximum TPR, again as a function of $m$.

\paragraph{Left subplot.}

\begin{itemize}
	\item Empirical points: For each $m$ in a grid of 50 values between $500$ and $2500$, chosen such that they are equally spaced on the log scale, the following logic is used to estimate the true TPR and its corresponding $\tilde m$:
	\begin{enumerate}
		\item $\tilde m$ is initialized at 10. 
		\item For the current $\tilde m$, the TPR is estimated by simulating the BONuS procedure using $5{,}000$ replications. The standard error is also recorded. 
		\item $\tilde m$ is increased by 1.
		\item Steps 2-3 are repeated until the TPR is below the running maximum TPR by at least two of its standard errors \textit{for 5 consecutive values of $\tilde m$}, at which point the procedure terminates.
	\end{enumerate} 
	This allows a method which can efficiently estimate the optimal $\tilde m$ without having to search over an excessively large grid. The estimated optimal $\tilde m$ corresponds to the running maximum TPR at the time of termination.

	\item ``Error'' bars: since we are estimating $\tilde m _\opt$, we also strive to visualize a notion of uncertainty associated with these estimates. For each $m$, we return an interval around the estimated $\tilde m _\opt$ with the upper and lower endpoints corresponding to the largest and smallest values of $\tilde m$ which returns a TPR within two standard errors of the estimated maximum TPR.
	
	\item OLS fit: an ordinary least squares regression is performed on the log-log data points to estimate the slope. The fitted line is then plotted. We aim to show that the slope is close to 0.5, corresponding to the conjectured square-root relationship $\tilde m_\opt \asymp \sqrt{m}$.
	
	\item Conjectured line: the line $\tilde m = \sqrt{\eta_1/\eta_2}\sqrt{m}$ is also plotted for reference.
\end{itemize}

\paragraph{Right subplot.}

The blue line references the estimated maximum TPR found in the left subplot for each $m$. The brown line corresponds to the TPR achieved at the conjectured optimal $\tilde m = \sqrt{\eta_1/\eta_2}\sqrt{m}$, which is computed by simulating the BONuS procedure with $5{,}000$ replications at this specific $\tilde m$ value for each $m$ (re-using the simulation results from the left subplot). We show that the TPR at the conjectured $\tilde m_\opt$ is very close to the empirically found maximum TPR.

\subsection{Figures 
\ref{fig:point-mass-prior-theorem},
\ref{fig:subspace-mass-prior-theorem},   
\ref{fig:point-mass-prior-theorem-high-dim},
\ref{fig:subspace-mass-prior-theorem-high-dim} 
\ref{fig:point-mass-prior-theorem-sparse},
\ref{fig:subspace-mass-prior-theorem-sparse} 
}

The fourth subplot (bottom right) for each of these plots are tuned comparisons of multiple methods over varying values of $h$, the signal strength. Since they are tuned to have optimal power, the logic used to generate these plots is exactly the same as that used in the left plot of Figure \ref{fig:point-mass-prior-optimal} (see: Appendix \ref{sec:omitted_details_optimal_mtilde}), with the only difference being that we search over a grid of $\pi_\aug$ or $\pi_\spl$ values instead of the individual $\tilde m$ values. This allows the simulation study to be  more computationally tractable, but trades off some precision in estimating the maximum power for BONuS and Split BH. Each of these plots also compares the two masking methods to 
In-sample BH (Algorithm \ref{alg:in-sample-bh}).
The number of resamples to compute the p-values for In-sample BH is chosen via hyperparameter $B$ which varies across cases. For Case 1, we use $B=50{,}000$ resamples. For Case 2, we use $B=10{,}000$ resamples for the default setting, $B=2{,}500$ for the sparse setting, and $B=1{,}000$ for the high-dimensional setting. Case 2 requires more computational resources due to the cost of computing the direction via SVD, especially for high dimensional data. These choices of $B$ are made to prevent our simulation scripts from taking an excessively long time to run, while still providing a high-quality estimate of the power of In-sample BH. 

 \section{Master theorems} \label{sec:master-theorems}
 We prove Theorems~\ref{thm:power-point-mass} and~\ref{thm:power-one-dimensional} by verifying statements about the limiting distributions of the scores produced by the three methods compared. To facilitate future extensions of Theorems~\ref{thm:power-point-mass} and~\ref{thm:power-one-dimensional}, we provide a more general setup (Section~\ref{sec:more-general-setup}) that includes generic versions of the data-generating model and of the methods analyzed. We then state generic ``master theorems'' for each of the three methods analyzed (Section~\ref{sec:master-theorem-statements}). Our master theorems resemble \citet[Theorems 10 and 11]{Wang2020b} and rely on similar technical machinery, while differing in their details. We prove these master theorems in Appendix~\ref{sec:master-theorem-proofs}.
 
 \subsection{A more general setup} \label{sec:more-general-setup}

\paragraph{Generic data-generating model.} For each hypothesis $j \in [m]$, we observe $\bm X_j \in \mathcal X$, distributed as
\begin{equation}
\bm X_{j} \indsim P_{\bm \theta_j}, \quad \bm \theta_j \in \Theta, \quad j=1,\ldots,m.
\label{eq:data-generating-general}
\end{equation}
Here, $\mathcal X$ is a generic sample space, $\Theta$ is a generic parameter space, and $\{P_{\bm \theta}: \bm \theta \in \Theta\}$ is a family of distributions on $\mathcal X$. For some point $\bm \theta_{0} \in \Theta$, our goal is to test $H_{0j}: \bm \theta_j = \bm \theta_0$ for each $j \in [m]$. We compare procedures returning rejection sets $\widehat{\mathcal R} \subseteq [m]$ that target FDR control at level $q$ for each fixed setting of $(\bm \theta_j)_{j = 1}^m$. We evaluate the power of these procedures with respect to a two-groups model \citep{Efron2008} where $\bm \theta_j$ are drawn from a mixture distribution
\begin{equation}
\bm \theta_j \iidsim (1-\gamma) \cdot \delta_{\bm \theta_0} + \gamma \cdot \Lambda, \quad j = 1, \ldots, m,
\label{eq:two-groups-general}
\end{equation}
where $\gamma \in (0,1]$ is the non-null proportion and $\Lambda$ is the alternative prior. % Denoting by $\mathcal H_0 \equiv \{j \in [m]: \bm \theta_j = \bm \theta_0\}$ and $\mathcal H_1 \equiv \{j \in [m]: \bm \theta_j \neq \bm \theta_0\}$ the sets of null and alternative hypotheses, respectively, we define the true positive proportion as $\TPP(\widehat{\mathcal R}) \equiv |\widehat{\mathcal R} \cap \mathcal H_1|/|\mathcal H_1|$. We quantify the asymptotic power of a method as the in-probability limit of its TPP in a suitable scaling regime with respect to the distributions~\eqref{eq:data-generating-general} and~\eqref{eq:two-groups-general}.

\paragraph{Generic method definitions.} Our master theorems are compatible with generic versions of Split BH (Algorithm~\ref{alg:split-bh-general}), BONuS (Algorithm~\ref{alg:bonus-general}), and In-sample BH (Algorithm~\ref{alg:in-sample-bh-general}). Split BH is generically formulated based on a pair of splitting functions $(\mathcal S^\learn, \mathcal S^\score)$ with the property that $\mathcal S^\learn(\bm X, \pi_\spl)$ and $\mathcal S^\score(\bm X, \pi_\spl)$ are independent conditionally on $(\bm \theta_j)_{j = 1}^m$. BONuS is generically formulated by drawing the null samples from the null distribution $P_{\bm \theta_0}$. In-sample BH is nearly unchanged from Algorithm~\ref{alg:in-sample-bh}, except that the index $j_0$ is chosen so that $\bm \theta_{j_0} = \bm \theta_0$. All three methods are formulated with respect to a generic learner $L$ and a generic score function $T_{\Lambda}$; note that the latter need not take the form~\eqref{eq:likelihood-ratio-statistic}.

\begin{minipage}[t]{0.49\textwidth}
\begin{algorithm}[H]
\LinesNumbered
\setlength{\algomargin}{0em}
\setlength{\algoheightrule}{0pt}
\setlength{\algotitleheightrule}{0pt}
\setlength{\interspacetitleruled}{0pt}
\setlength{\interspacealgoruled}{0pt}
\SetAlgoSkip{}
\SetInd{0em}{0em}
\caption{Split BH (generic)}
\label{alg:split-bh-general}
\KwIn{Splitting proportion $\smash{\pi_\spl} \in (0,1)$}
\AlgoStrut$\bm X^\learn \gets \mathcal S^\learn(\bm X, \pi_\spl)$\;
\AlgoStrut$\bm X^\score \gets \mathcal S^\score(\bm X, \pi_\spl)$\;
\AlgoStrut\(\hat \Lambda \gets L(\bm X^\learn)\)\;
\AlgoStrut$S_j \gets T_{\hat \Lambda}(\bm X_j^\score)$\;
\AlgoStrut$\bar G_{0m}(t) \gets \mathbb P[T_{\hat \Lambda}(\bm X_j^\score) \geq t \mid \hat \Lambda, \bm \theta_j = \bm \theta_0]$\;
\AlgoStrut$p_j \gets \bar G_{0m}(S_j)$\;
\AlgoStrut\KwRet{\(\widehat{\mathcal R} = \textnormal{BH}(\{p_j\}_{j = 1}^m)\)}
\end{algorithm}

\end{minipage}\begin{minipage}[t]{0.51\textwidth}
\begin{algorithm}[H]
\LinesNumbered
\setlength{\algomargin}{0em}
\setlength{\algoheightrule}{0pt}
\setlength{\algotitleheightrule}{0pt}
\setlength{\interspacetitleruled}{0pt}
\setlength{\interspacealgoruled}{0pt}
\SetAlgoSkip{}
\SetInd{0em}{0em}
\caption{BONuS (generic)}
\label{alg:bonus-general}
\KwIn{Number of null samples $\tilde m \in \N$}
\setcounter{AlgoLine}{-1}
\AlgoStrut$\tilde{\bm X}_{k} \sim P_{\bm \theta_0}$, $k = 1, \ldots, \tilde m$\;
\AlgoStrut$\bm X^\learn \gets [\bm X; \tilde{\bm X}]_\Pi$, \ \ $\Pi \sim \text{Sym}([m + \tilde m])$\;
\AlgoStrut$\bm X^\score \gets \bm X$\;
\AlgoStrut\(\hat \Lambda \gets L(\bm X^\learn)\)\;
\AlgoStrut$S_j \gets T_{\hat \Lambda}(\bm X_j^\score)$\;
\AlgoStrut$\bar G_{0m}(t) \gets \tfrac{1}{1 + \tilde m}(1 + \sum_{k = 1}^{\tilde m}\indicator(T_{\hat \Lambda}(\tilde{\bm X}_k) \geq t))$\;
\AlgoStrut$p_j \gets \bar G_{0m}(S_j)$\;
\AlgoStrut\KwRet{\(\widehat{\mathcal R} = \textnormal{BH}(\{p_j\}_{j = 1}^m)\)}
\end{algorithm}
\end{minipage}

\begin{minipage}[t]{0.44\textwidth}
\begin{algorithm}[H]
\LinesNumbered
\caption{In-sample BH (generic)}
\label{alg:in-sample-bh-general}
\KwIn{\{$\bm \theta_j\}_{j = 1}^m$, $j_0$ s.t. $\bm \theta_{j_0} = \bm \theta_0$}
\AlgoStrut$\bm X^\learn \gets \bm X$\;
\AlgoStrut$\bm X^\score \gets \bm X$\;
\AlgoStrut\(\hat \Lambda \gets L(\bm X^\learn)\)\;
\AlgoStrut$S_j \gets T_{\hat \Lambda}(\bm X_j^\score)$\;
\AlgoStrut$\bar G_{0m}(t) \gets \mathbb P[T_{\hat \Lambda}(\bm X_{j_0}) \geq t \mid \{\bm \theta_j\}_{j=1}^m]$\;
\AlgoStrut$p_j \gets \bar G_{0m}(S_j)$\;
\AlgoStrut\KwRet{\(\widehat{\mathcal R} = \textnormal{BH}(\{p_j\}_{j = 1}^m)\)}
\end{algorithm}
\end{minipage}

\paragraph{Generic asymptotic regime.} We consider a sequence of settings indexed by $m$, where the objects $\mathcal X$, $\Theta$, $\bm \theta_0$, $\Lambda$, $L$, $T_\Lambda$, and $\tilde m$ can vary with $m$ (this dependence is suppressed to keep the notation light). As in the main text, we fix $\gamma$.

\subsection{Master theorems} \label{sec:master-theorem-statements}

The section presents three master theorems; although they share much of the same underlying machinery, especially in establishing weak-law-type limits, we state them separately for conceptual clarity and ease of presentation.

 \begin{theorem}[Master theorem for Split BH]\label{thm:master-split-BH}
 	Fix a target FDR level $q\in(0,1)$ and let $\widehat{\mathcal R}$ be the rejection set produced by Algorithm~\ref{alg:split-bh-general}. 
 	Let $G_0$ and $G_1$ be distribution functions on $\R$ with survival functions $\bar G_0=1-G_0$ and $\bar G_1=1-G_1$, and let $\gamma\in(0,1]$ denote the non-null proportion.
 	Assume:
 	
 	\begin{enumerate}[label=(\roman*)]
 		\item \textbf{Pointwise convergence of score tails conditional on $\hat{\Lambda}$ under the null.}  \label{ass:pointwise-convg-score-tails}
 		For every fixed $t \in \R$ ,
 	\[
 	\P\!\big(T_{\hat\Lambda}(X_j^{\score}) \ge t \,\big|\, \hat \Lambda, j \in \cH_0,(\cH_0,\cH_1)\big) \ \convp\ \bar G_0(t).
 	\]
 	
 		\item \textbf{Asymptotic pairwise convergence of scores.} \label{ass:splitBH-indep}  
 	For any distinct indices $j_1 \neq j_2$, any $t_1,t_2 \in \R$, and $k_1,k_2 \in \{0,1\}$,
 	\[
 	\P\!\big(T_{\hat\Lambda}(X_{j_1}^{\score}) \ge t_1,\; T_{\hat\Lambda}(X_{j_2}^{\score}) \ge t_2 \,\big|\, j_1 \in \cH_{k_1}, j_2 \in \cH_{k_2},(\cH_0,\cH_1)\big) 
 	\ \convp\ \bar G_{k_1}(t_1)\,\bar G_{k_2}(t_2).
 	\]
 		\item \textbf{Regularity of limits.} \label{ass:regularity-of-cdfs}
 		$G_0$ and $G_1$ are continuous and, for every $C\in \R$, $K = (-\infty,C] \subset\R$, 
 		\[
 		\sup_{t\in K} G_0(t) < 1
 		\quad\text{and}\quad
 		\sup_{t\in K} G_1(t) < 1.
 		\]
 		
 		\item \textbf{Monotone survival ratio and tail separation.} \label{ass:mlr}
 		The ratio $t\mapsto \bar G_1(t)/\bar G_0(t)$ is strictly increasing on $\R$ (or the support of $G_0 \cap G_1$), and moreover
 		\[
 		\frac{\bar G_1(t)}{\bar G_0(t)} \;\to\; \infty
 		\qquad\text{as } t\to\infty.
 		\]
 	\end{enumerate}
 	
 	Then the true positive proportion of Split BH converges in probability to the BH power functional:
 	\[
 	\TPP(\widehat{\mathcal R})\ \convp\ \TPR_{\mathrm{BH}}(G_0,G_1,\gamma).
 	\]
 \end{theorem}

 \begin{theorem}[Master theorem for BONuS]
 	\label{thm:master-bonus}
 	Fix a target FDR level $q \in (0,1)$ and let $\widehat{\mathcal R}$ be the rejection set produced by Algorithm~\ref{alg:bonus-general}. 
 	Let $G_0$ and $G_1$ be distribution functions on $\R$ with survival functions $\bar G_0=1-G_0$ and $\bar G_1=1-G_1$, and let $\gamma \in (0,1]$ denote the non-null proportion.
 	Assume:
 	
 	\begin{enumerate}[label=(\roman*)]
 		\item \textbf{Pointwise convergence of score tails.} \label{ass:bonus-tails}  
 		For every fixed $t \in \R$ and $k \in \{0,1\}$,
 		\[
 		\P\!\big(T_{\hat\Lambda}(X_j^{\score}) \ge t \,\big|\, j \in \cH_k,(\cH_0,\cH_1)\big) \ \convp\ \bar G_k(t).
 		\]
 		
 		\item \textbf{Asymptotic pairwise independence of scores.} \label{ass:bonus-indep}  
	 		For any distinct indices $j_1 \neq j_2$, any $t_1,t_2 \in \R$, and $k_1,k_2 \in \{0,1\}$,
	 		\[
	 		\P\!\big(T_{\hat\Lambda}(X_{j_1}^{\score}) \ge t_1,\; T_{\hat\Lambda}(X_{j_2}^{\score}) \ge t_2 \,\big|\, j_1 \in \cH_{k_1}, j_2 \in \cH_{k_2},(\cH_0,\cH_1)\big) 
	 		\ \convp\ \bar G_{k_1}(t_1)\,\bar G_{k_2}(t_2).
	 		\]
 		
 		\item \textbf{Regularity of limits.} \label{ass:bonus-regularity}  
 		$G_0$ and $G_1$ are continuous and, for every $C\in \R$, $K = (-\infty,C] \subset\R$, 
 		\[
 		\sup_{t\in K} G_0(t) < 1
 		\quad\text{and}\quad
 		\sup_{t\in K} G_1(t) < 1.
 		\]

 		\item \textbf{Monotone survival ratio and tail separation.} \label{ass:bonus-mlr}  
 		The ratio $t \mapsto \bar G_1(t)/\bar G_0(t)$ is strictly increasing on $\R$ (or the support of $G_0 \cap G_1$), and moreover
 		\[
 		\frac{\bar G_1(t)}{\bar G_0(t)} \;\to\; \infty
 		\qquad\text{as } t \to \infty.
 		\]
 	\end{enumerate}
 	
 	Then the true positive proportion of BONuS converges in probability to the BH power functional:
 	\[
 	\TPP(\widehat{\mathcal R}) \ \convp\ \TPR_{\mathrm{BH}}(G_0,G_1,\gamma).
 	\]
 \end{theorem}

 \begin{theorem}[Master theorem for In-sample BH]
 	\label{thm:master-in-sample}
 	Fix a target FDR level $q \in (0,1)$ and let $\widehat{\mathcal R}$ be the rejection set produced by Algorithm~\ref{alg:in-sample-bh-general}. 
 	Let $G_0$ and $G_1$ be distribution functions on $\R$ with survival functions $\bar G_0=1-G_0$ and $\bar G_1=1-G_1$, and let $\gamma \in (0,1]$ denote the non-null proportion.
 	Assume:
 	
 	\begin{enumerate}[label=(\roman*)]
 		\item \textbf{Conditional convergence of score tails under null.} \label{ass:insample-tails}
 		For every fixed $t \in \R$,
 		\[
 		\P\!\big(T_{\hat\Lambda}(X_j) \ge t \,\big|\, j \in \cH_0,(\theta_k)_{k=1}^m\big) \ \convp\ \bar G_0(t).
 		\]
 		
 		\item \textbf{Asymptotic pairwise independence of scores.} \label{ass:insample-indep}
 		For any distinct indices $j_1 \neq j_2$, any $t_1,t_2 \in \R$, and $k_1,k_2 \in \{0,1\}$,
 		\[
 		\P\!\big(T_{\hat\Lambda}(X_{j_1}) \ge t_1,\; T_{\hat\Lambda}(X_{j_2}) \ge t_2 \,\big|\, j_1 \in \cH_{k_1}, j_2 \in \cH_{k_2},(\cH_0,\cH_1)\big) 
 		\ \convp\ \bar G_{k_1}(t_1)\,\bar G_{k_2}(t_2).
 		\]
 		
 		\item \textbf{Regularity of limits.} \label{ass:insample-regularity}
 		$G_0$ and $G_1$ are continuous and, for every $C\in \R$, $K = (-\infty,C] \subset\R$, 
 		\[
 		\sup_{t\in K} G_0(t) < 1
 		\quad\text{and}\quad
 		\sup_{t\in K} G_1(t) < 1.
 		\]
 		
 		\item \textbf{Monotone survival ratio and tail separation.} \label{ass:insample-mlr}
 		The ratio $t \mapsto \bar G_1(t)/\bar G_0(t)$ is strictly increasing on $\R$ (or the support of $G_0 \cap G_1$), and moreover
 		\[
 		\frac{\bar G_1(t)}{\bar G_0(t)} \;\to\; \infty
 		\qquad\text{as } t \to \infty.
 		\]
 	\end{enumerate}
 	
 	Then the true positive proportion of the in-sample BH procedure converges in probability to the BH power functional:
 	\[
 	\TPP(\widehat{\mathcal R}) \ \convp\ \TPR_{\mathrm{BH}}(G_0,G_1,\gamma).
 	\]
 \end{theorem}

\section{Proofs of master theorems} \label{sec:master-theorem-proofs}

\subsection{Algorithmic variants and equivalence} \label{sec:algorithmic-variants}
Steps 5-7 of Algorithms~\ref{alg:split-bh-general}--\ref{alg:in-sample-bh-general} are presented based on the construction of $p$-values and BH correction. Given scores $S_1,\dots,S_m$ and a nonincreasing survival function $\bar G_{0m}$, we define
\begin{equation*}
p_j \gets \bar G_{0m}(S_j),\quad
\widehat{\mathcal R} \gets \mathrm{BH}(\{p_j\}_{j=1}^m).
\end{equation*}
For convenience in our proofs, we instead use an equivalent representation:
\begin{equation*}
\begin{split}
\mathcal R(t) \gets \{j\in[m]:S_j\ge t\},\quad
\widehat V(t) \gets m\,\bar G_{0m}(t), \quad
\hat t \gets \inf\{t\in\R:\widehat V(t)/|\mathcal R(t)|\le q\}, \quad \widehat{\mathcal R}' \gets \mathcal R(\hat t).
\end{split}
\label{eq:equivalent-calibration}
\end{equation*}
\begin{lemma}[Equivalence of two calibration constructions]
\label{lem:bh-threshold-equivalence}
$\widehat{\mathcal R}=\widehat{\mathcal R}'.$
\end{lemma}
The proof of this standard result is as in \citet{Storey02} (note that here we have continuous scores for which ties don't occur a.s.). In the proofs that follow, we work with this variant of Algorithms~\ref{alg:split-bh-general}--\ref{alg:in-sample-bh-general}.

 \subsection{Proof of Theorem~\ref{thm:master-split-BH} (Split BH)}
 
 \begin{proof}
 	Write $S_j:=T_{\hat\Lambda}(X_j^{\score})$ and, for $t\in\R$,
 	\[
 	\mathcal R(t):=\{j:\,S_j\ge t\},\qquad 
 	R_m(t):=\frac{1}{m}\,|\mathcal R(t)|=\frac{1}{m}\sum_{j=1}^m \mathbf 1\{S_j\ge t\},
 	\]
 	and
 	\[
 	\widehat V(t):=m\cdot \P(S_j\ge t\mid \hat\Lambda,\, j\in\cH_0,(\cH_0,\cH_1)),\qquad 
 	V_m(t):=\frac{1}{m}\widehat V(t).
 	\]
 	The Split BH threshold is $\hat t:=\inf\{t\in\R:\, V_m(t)/R_m(t)\le q\}$ and the limit target is
 	\[
 	\phi(t):=\frac{\bar G_0(t)}{(1-\gamma)\bar G_0(t)+\gamma\,\bar G_1(t)},\qquad 
 	t_*:=\inf\{t\in\R:\,\phi(t)\le q\}.
 	\]
 	
 	\paragraph{Step 1: Pointwise limits of $R_m(t)$ and $V_m(t)$.}
 	Fix $t\in\R$, and write
 	\[
 	S_j:=T_{\hat\Lambda}(X_j^{\score}),\qquad
 	\cF_m:=\sigma(\cH_0,\cH_1).
 	\]
 	
 	\emph{(a) Limit of $R_m(t)$.}
 	For $k\in\{0,1\}$, define
 	\[
 	p_{k,m}(t):=\P(S_j\ge t \mid j\in\cH_k,\cF_m).
 	\]
 	Also let
 	\[
 	\gamma_m:=\frac{|\cH_1|}{m}.
 	\]
 	By the two-groups model, $\gamma_m\convp \gamma$.
 	
 	We first verify the marginal tail limits. For the nulls, Assumption~\ref{ass:pointwise-convg-score-tails} gives
 	\[
 	\P(S_j\ge t\mid \hat\Lambda, j\in\cH_0,(\cH_0,\cH_1))
 	\convp \bar G_0(t).
 	\]
 	Since this random variable is bounded in $[0,1]$, convergence in probability implies convergence in $L^1$, and hence
 	\[
 	p_{0,m}(t)
 	=
 	\E\!\left[
 	\P(S_j\ge t\mid \hat\Lambda, j\in\cH_0,(\cH_0,\cH_1))
 	\,\middle|\,
 	\cF_m
 	\right]
 	\convp \bar G_0(t).
 	\]
 For the alternatives, the marginal tail limit is a consequence of
 Assumption~\ref{ass:splitBH-indep}. Fix $u\in\R$, and for distinct
 $i\neq j$ in $\cH_1$ set
 \[
 J_m(t,u):=\P(S_i\ge t,\;S_j\ge u\mid i,j\in\cH_1,\cF_m).
 \]
 Then
 \[
 J_m(t,u)\le p_{1,m}(t)\le J_m(t,u)+1-p_{1,m}(u).
 \]
 Also,
 \[
 \P(S_i\ge u,S_j\ge u\mid i,j\in\cH_1,\cF_m)\le p_{1,m}(u),
 \]
 so Assumption~\ref{ass:splitBH-indep} gives
 \[
 p_{1,m}(u)\ge \bar G_1(u)^2+o_p(1).
 \]
 Hence
 \[
 J_m(t,u)\le p_{1,m}(t)\le J_m(t,u)+1-\bar G_1(u)^2+o_p(1).
 \]
 Since again by Assumption~\ref{ass:splitBH-indep},
 \[
 J_m(t,u)\convp \bar G_1(t)\bar G_1(u),
 \]
 we obtain
 \[
 \bar G_1(t)\bar G_1(u)+o_p(1)\le p_{1,m}(t)
 \le \bar G_1(t)\bar G_1(u)+1-\bar G_1(u)^2+o_p(1).
 \]
 Letting $u\to-\infty$ and using $\bar G_1(u)\to 1$ yields
 \[
 p_{1,m}(t)\convp \bar G_1(t).
 \]
 	Next we verify the vanishing conditional covariance condition needed for
 	Lemma~\ref{lemma:wlln-bounded-rv}. Let $i\neq j$, and suppose
 	$i\in\cH_{k_1}$ and $j\in\cH_{k_2}$ with $k_1,k_2\in\{0,1\}$. Then
 	\begin{align*}
 		&\Cov\!\big(\mathbf 1\{S_i\ge t\},\mathbf 1\{S_j\ge t\}\mid \cF_m\big) \\
 		&\qquad=
 		\P(S_i\ge t,S_j\ge t\mid \cF_m)
 		-
 		\P(S_i\ge t\mid \cF_m)\,\P(S_j\ge t\mid \cF_m).
 	\end{align*}
 	By Assumption~\ref{ass:splitBH-indep},
 	\[
 	\P(S_i\ge t,S_j\ge t\mid \cF_m)
 	\convp
 	\bar G_{k_1}(t)\bar G_{k_2}(t),
 	\]
 	while the marginal limits just established give
 	\[
 	\P(S_i\ge t\mid \cF_m)\,\P(S_j\ge t\mid \cF_m)
 	=
 	p_{k_1,m}(t)p_{k_2,m}(t)
 	\convp
 	\bar G_{k_1}(t)\bar G_{k_2}(t).
 	\]
 	Therefore
 	\[
 	\Cov\!\big(\mathbf 1\{S_i\ge t\},\mathbf 1\{S_j\ge t\}\mid \cF_m\big)\convp 0.
 	\]
 	
 	Since the array $\{\mathbf 1\{S_j\ge t\}\}_{j=1}^m$ is bounded, Lemma~\ref{lemma:wlln-bounded-rv}
 	yields
 	\[
 	\frac1m\sum_{j=1}^m \mathbf 1\{S_j\ge t\}
 	-
 	\E\!\left[\frac1m\sum_{j=1}^m \mathbf 1\{S_j\ge t\}\,\middle|\,\cF_m\right]
 	\convp 0.
 	\]
 	By exchangeability within each group,
 	\[
 	\E\!\left[\frac1m\sum_{j=1}^m \mathbf 1\{S_j\ge t\}\,\middle|\,\cF_m\right]
 	=
 	(1-\gamma_m)p_{0,m}(t)+\gamma_m p_{1,m}(t).
 	\]
 	Using $\gamma_m\convp\gamma$, $p_{0,m}(t)\convp \bar G_0(t)$, and
 	$p_{1,m}(t)\convp \bar G_1(t)$, we obtain
 	\[
 	(1-\gamma_m)p_{0,m}(t)+\gamma_m p_{1,m}(t)
 	\convp
 	(1-\gamma)\bar G_0(t)+\gamma \bar G_1(t).
 	\]
 	Combining the last two displays,
 	\[
 	R_m(t):=\frac1m\sum_{j=1}^m \mathbf 1\{S_j\ge t\}
 	\convp
 	(1-\gamma)\bar G_0(t)+\gamma \bar G_1(t).
 	\]
 	
 	\medskip
 	\noindent
 	\emph{(b) Limit of $V_m(t)$.}
 	By definition,
 	\[
 	V_m(t)
 	=
 	\P(S_j\ge t \mid \hat\Lambda, j\in\cH_0,(\cH_0,\cH_1)).
 	\]
 	Assumption~\ref{ass:pointwise-convg-score-tails} therefore gives, for each fixed $t$,
 	\[
 	V_m(t)\convp \bar G_0(t).
 	\]
 	
 	\paragraph{Step 2: Uniform convergence .}
 	Both $t\mapsto R_m(t)$ and $t\mapsto V_m(t)$ are (random) monotone decreasing and bounded, and their limits
 	\(
 	t\mapsto (1-\gamma)\bar G_0(t)+\gamma\bar G_1(t)
 	\)
 	and
 	\(
 	t\mapsto \bar G_0(t)
 	\)
 	are continuous and decreasing by Assumption~\ref{ass:regularity-of-cdfs}.
 	Therefore, Lemma~\ref{lemma:unif-convg-monotone} (since $(R_m,(1-\gamma)\bar G_0-\gamma\bar G_1)$ and $(V_m, \bar G_0)$ agree on $\pm \infty$) yields,
 	\[
 	\|R_m-(1-\gamma)\bar G_0-\gamma\bar G_1\|_{\infty}\ \convp\ 0,
 	\qquad
 	\|V_m-\bar G_0\|_{\infty}\ \convp\ 0.
 	\]
 	
 	\paragraph{Step 3: Uniform convergence of the ratio and identification of $t_*$.}
 	Define $\phi_m(t):=V_m(t)/R_m(t)$. By Assumption~\ref{ass:regularity-of-cdfs}, the limit denominator
 	\(
 	(1-\gamma)\bar G_0(t)+\gamma\bar G_1(t)
 	\)
 	is bounded away from $0$ on any  $K =(-\infty,C]$, $C\in\R$; by Step~2, the same holds with high probability for $R_m$ on $K$. Lemma~\ref{lemma:unif-convg-ratios} then applies and gives, for every such $K\subset\R$,
 	\[
 	\|\phi_m-\phi\|_{\infty,K}\ \convp\ 0,
 	\qquad
 	\phi(t)=\frac{\bar G_0(t)}{(1-\gamma)\bar G_0(t)+\gamma\bar G_1(t)}.
 	\]
 	Assumption~\ref{ass:mlr} (strictly increasing $\bar G_1/\bar G_0$) implies $\phi$ is strictly decreasing and continuous on $\R$.
 	Moreover, since  $\bar G_1/\bar G_0\to\infty$ as $t\to+\infty$ (by assumption \ref{ass:mlr}), then $\phi(+\infty)=0$ and its easy to $\phi(-\infty)=1$ (since $G_i$'s are cdf's ), so any $q\in(0,1)$ is attained.
 	
 	\paragraph{Step 4: Consistency of the data-driven threshold.}
 	Since $q$ is attained by $\phi$ and $\phi$ is strictly decreasing and continuous, we can invoke Lemma~\ref{lemma:uniform-convg-of-inverse} (with $\phi_n=\phi_m$ and level $q$) to conclude
 	\[
 	\hat t=\inf\{t:\phi_m(t)\le q\}\ \convp\ t_*=\inf\{t:\phi(t)\le q\}.
 	\]
 	
 	\paragraph{Step 5: Convergence of TPP at the random threshold.}
 	Let
 	\[
 	\TPP_m(t):=\frac{|\{j\in\cH_1:\,S_j\ge t\}|}{|\cH_1|}
 	=\frac{\frac{1}{m}\sum_{j=1}^m \mathbf 1\{j\in\cH_1,S_j\ge t\}}{\frac{1}{m}\sum_{j=1}^m \mathbf 1\{j\in\cH_1\}}.
 	\]
 	As in Step~1, Lemma~\ref{lemma:wlln-bounded-rv} and Assumption~\ref{ass:pointwise-convg-score-tails} yield, for each fixed $t$,
 	\[
 	\frac{1}{m}\sum_{j=1}^m \mathbf 1\{j\in\cH_1,S_j\ge t\}\ \convp\ \gamma\,\bar G_1(t),
 	\qquad
 	\frac{1}{m}\sum_{j=1}^m \mathbf 1\{j\in\cH_1\}\ \convp\ \gamma,
 	\]
 	whence $\TPP_m(t)\convp\bar G_1(t)$ pointwise. By monotonicity and continuity of the limit, Lemma~\ref{lemma:unif-convg-monotone} (since $\TPP_m$ and $\bar G_1$ agree on $\pm \infty$ ) gives uniform convergence on a neighbourhood of $t_*$. Combining with $\hat t\convp t_*$ (Step~4) and the continuous mapping argument,
 	\[
 	\TPP_m(\hat t)\ \convp\ \bar G_1(t_*).
 	\]
 	
 	\paragraph{Step 6: Identify the BH power functional.}
 	By definition,
 	\(
 	\TPR_{\mathrm{BH}}(G_0,G_1,\gamma)=\bar G_1(t_*)
 	\)
 	with $t_*$ determined by $\phi(t_*)=q$. Thus,
 	\[
 	\TPP(\widehat{\mathcal R})=\TPP_m(\hat t)\ \convp\ \TPR_{\mathrm{BH}}(G_0,G_1,\gamma).
 	\]
 \end{proof}

 \subsection{Proof of Theorem~\ref{thm:master-bonus} (BONuS)}
 
 \begin{proof}
 	Let $S_j:=T_{\hat\Lambda}(X_j^{\score})$ for original features and $\tilde S_k:=T_{\hat\Lambda}(\tilde X_k)$ for augmented null features. For $t\in\R$ define
 	\[
 	\mathcal R(t):=\{j\in[m]:\, S_j\ge t\},\qquad 
 	R_m(t):=\frac{1}{m}\,|\mathcal R(t)|=\frac{1}{m}\sum_{j=1}^m \mathbf 1\{S_j\ge t\}.
 	\]
 	For BONuS,
 	\[
 	\widehat V(t)=\frac{m}{1+\tilde m}\Bigl(1+\sum_{k=1}^{\tilde m}\mathbf 1\{\tilde S_k\ge t\}\Bigr),
 	\qquad
 	V_m(t):=\frac{1}{m}\widehat V(t).
 	\]
 	Define the limit BH curve and target threshold
 	\[
 	\phi(t):=\frac{\bar G_0(t)}{(1-\gamma)\bar G_0(t)+\gamma\,\bar G_1(t)},\qquad
 	t_*:=\inf\{t\in\R:\,\phi(t)\le q\}.
 	\]
 	By Assumption~\ref{ass:bonus-mlr}, $\phi$ is strictly decreasing and continuous with $\phi(-\infty)=1$ and $\phi(+\infty)=0$, so $t_*$ is finite and well-defined for any $q\in(0,1)$.
 	
 	\paragraph{Step 1: Pointwise limits of $R_m(t)$ and $V_m(t)$.}
 	\emph{(a) Original features.}  
 	By Assumption~\ref{ass:bonus-tails},
 	\[
 	\P(S_j\ge t\mid j\in\cH_0,(\cH_0,\cH_1))\convp \bar G_0(t),\qquad
 	\P(S_j\ge t\mid j\in\cH_1,(\cH_0,\cH_1))\convp \bar G_1(t),
 	\]
 	for each fixed $t$. By Assumption~\ref{ass:bonus-tails}  and~\ref{ass:bonus-indep}, for any $i\neq j$,
 	\[
 	\P\!\big(S_i\!\ge t,\;S_j\!\ge t \,\big|\, \cF_m\big)
 	-\P\!\big(S_i\!\ge t \,\big|\, \cF_m\big)\,
 	\P\!\big(S_j\!\ge t \,\big|\, \cF_m\big)
 	\;\convp\;0,
 	\quad\text{with }\;\cF_m:=\sigma(\cH_0,\cH_1).
 	\]
 	Hence
 	\[
 	\Cov\!\big(\mathbf 1\{S_i\ge t\},\,\mathbf 1\{S_j\ge t\}\,\big|\,\cF_m\big)
 	\;=\;
 	\P\!\big(S_i\!\ge t,\;S_j\!\ge t \,\big|\, \cF_m\big)
 	-\P\!\big(S_i\!\ge t \,\big|\, \cF_m\big)\,
 	\P\!\big(S_j\!\ge t \,\big|\, \cF_m\big)
 	\;\convp\;0.
 	\]
 	Therefore the bounded array $\{\mathbf 1\{S_j\ge t\}\}_{j=1}^m$ satisfies the vanishing conditional covariance condition of Lemma~\ref{lemma:wlln-bounded-rv} (with $\cF_n=\cF_m$), and so
 	\[
 	\frac{1}{m}\sum_{j=1}^m \mathbf 1\{S_j\ge t\}
 	-\E\!\left[\frac{1}{m}\sum_{j=1}^m \mathbf 1\{S_j\ge t\}\,\middle|\,\cF_m\right]
 	\convp 0.
 	\]
 	
 	Write $\gamma_m:=m^{-1}|\cH_1|$ and
 	\[
 	p_{k,m}(t):=\P(S_j\ge t \mid j\in\cH_k,(\cH_0,\cH_1)),\quad k\in\{0,1\}.
 	\]
 	By the tower property and exchangeability within each group,
 	\[
 	\E\!\left[\frac{1}{m}\sum_{j=1}^m \mathbf 1\{S_j\ge t\}\,\middle|\,\cF_m\right]
 	= (1-\gamma_m)\,p_{0,m}(t) + \gamma_m\,p_{1,m}(t).
 	\]
 	Assumption~\ref{ass:bonus-tails} gives $p_{k,m}(t)\convp \bar G_k(t)$ for each fixed $t$, and by the two-groups model  $\gamma_m\convp \gamma$. Hence,
 	\[
 	(1-\gamma_m)\,p_{0,m}(t) + \gamma_m\,p_{1,m}(t)
 	\;\convp\; (1-\gamma)\,\bar G_0(t) + \gamma\,\bar G_1(t).
 	\]
 	Combining with the previous display (Slutsky),
 	\[
 	R_m(t)
 	=\frac{1}{m}\sum_{j=1}^m \mathbf 1\{S_j\ge t\}
 	\;\convp\; (1-\gamma)\,\bar G_0(t) + \gamma\,\bar G_1(t).
 	\]
 	
 	\emph{(b) Augmented nulls and exchangeability.}  
 	By construction and permutation invariance of learner, synthetic null scores are exchangeable with true null scores:
 	\[
 	\P(\tilde S_k\ge t\mid(\cH_0,\cH_1))\;=\;\P(S_j\ge t\mid j\in\cH_0,(\cH_0,\cH_1))\ \convp\ \bar G_0(t).
 	\]
 	Similarly we also have 
 	\[
 	\P(\tilde S_{k_1}\ge t, \tilde S_{k_2}\ge t\mid(\cH_0,\cH_1))\;=\;\P(S_{j_1}\ge t, S_{j_2}\ge t\mid j_1,j_2\in\cH_0,(\cH_0,\cH_1))\ \convp\ \bar G^2_0(t).
 	\]
 	where we used Assumption~\ref{ass:bonus-indep}. Hence the $\mathbf 1\{\tilde S_k\ge t\}$ are asymptotically pairwise independent (conditional on $(\cH_0,\cH_1)$) using a similar argument as before. Lemma~\ref{lemma:wlln-bounded-rv} then gives
 	\[
 	\frac{1}{\tilde m}\sum_{k=1}^{\tilde m}\mathbf 1\{\tilde S_k\ge t\}\ \convp\ \bar G_0(t).
 	\]
 	Therefore, as $\tilde m\to\infty$,
 	\[
 	V_m(t)\ \convp\ \bar G_0(t)\qquad\text{for each fixed }t.
 	\]
 	
 	\paragraph{Step 2: Uniform convergence.}
 	The processes $t\mapsto R_m(t)$ and $t\mapsto V_m(t)$ are monotone decreasing and bounded. Their limits,
 	\[
 	t\mapsto (1-\gamma)\bar G_0(t)+\gamma\bar G_1(t)
 	\quad\text{and}\quad
 	t\mapsto \bar G_0(t),
 	\]
 	are continuous and decreasing by Assumption~\ref{ass:bonus-regularity}. Lemma~\ref{lemma:unif-convg-monotone} (since $(R_m,(1-\gamma)\bar G_0-\gamma\bar G_1)$ and $(V_m, \bar G_0)$ agree on $\pm \infty$) then gives,
 	\[
 	\|R_m-(1-\gamma)\bar G_0-\gamma\bar G_1\|_{\infty}\ \convp\ 0,
 	\qquad
 	\|V_m-\bar G_0\|_{\infty}\ \convp\ 0.
 	\]
 	
 	\paragraph{Step 3: Uniform convergence of the BH ratio and consistency of the threshold.}
 	Define $\phi_m(t):=V_m(t)/R_m(t)$. By Assumption~\ref{ass:bonus-regularity}, the denominator $(1-\gamma)\bar G_0+\gamma\bar G_1$ is bounded away from $0$ on every set $K =(-\infty,C]$, $C\in \R$, and by Step~2 the same holds for $R_m$ with high probability. Lemma~\ref{lemma:unif-convg-ratios} then gives
 	\[
 	\|\phi_m-\phi\|_{\infty,K}\ \convp\ 0.
 	\]
 	By Assumption~\ref{ass:bonus-mlr}, $\phi$ is strictly decreasing and continuous, and by tail separation every $q\in(0,1)$ is attained. Hence Lemma~\ref{lemma:uniform-convg-of-inverse} applies, giving
 	\[
 	\hat t:=\inf\{t:\phi_m(t)\le q\}\ \convp\ t_*:=\inf\{t:\phi(t)\le q\}.
 	\]
 	
 	\paragraph{Step 4: Convergence of TPP at the random threshold.}
 	Define
 	\[
 	\TPP_m(t):=\frac{|\{j\in\cH_1:\, S_j\ge t\}|}{|\cH_1|}
 	=\frac{\frac{1}{m}\sum_{j=1}^m \mathbf 1\{j\in\cH_1,S_j\ge t\}}{\frac{1}{m}\sum_{j=1}^m \mathbf 1\{j\in\cH_1\}}.
 	\]
 	As in Step~1(a), using Lemma~\ref{lemma:wlln-bounded-rv} and Assumption~\ref{ass:bonus-tails},
 	\[
 	\frac{1}{m}\sum_{j=1}^m \mathbf 1\{j\in\cH_1,S_j\ge t\}\ \convp\ \gamma\,\bar G_1(t),
 	\qquad
 	\frac{1}{m}\sum_{j=1}^m \mathbf 1\{j\in\cH_1\}\ \convp\ \gamma,
 	\]
 	so $\TPP_m(t)\convp \bar G_1(t)$. By monotonicity and continuity of the limit in $t$, Lemma~\ref{lemma:unif-convg-monotone} (since $\TPP_m$ and $\bar G_1$ agree on $\pm \infty$ ) yields uniform convergence near $t_*$. With $\hat t\convp t_*$,
 	\[
 	\TPP_m(\hat t)\ \convp\ \bar G_1(t_*).
 	\]
 	
 	\paragraph{Step 5: Identify the BH power.}
 	By definition of the BH TPR functional,
 	\[
 	\TPR_{\mathrm{BH}}(G_0,G_1,\gamma)=\bar G_1(t_*).
 	\]
 	Therefore,
 	\[
 	\TPP(\widehat{\mathcal R})=\TPP_m(\hat t)\ \convp\ \TPR_{\mathrm{BH}}(G_0,G_1,\gamma).
 	\]
 \end{proof}
 
 \subsection{Proof of Theorem~\ref{thm:master-in-sample} (In-sample BH)}
 
 \begin{proof}
 	Let $S_j:=T_{\hat\Lambda}(X_j)$ and, for $t\in\R$,
 	\[
 	\cR(t):=\{j\in[m]:\,S_j\ge t\},\qquad
 	R_m(t):=\frac{1}{m}|\cR(t)|=\frac{1}{m}\sum_{j=1}^m \mathbf 1\{S_j\ge t\}.
 	\]
 	The (oracle) calibration term in Algorithm~\ref{alg:in-sample-bh} is
 	\[
 	\widehat V(t):= m\cdot \P\!\big(T_{\hat\Lambda}(X_j)\ge t \,\big|\, j\in\cH_0,(\theta_k)_{k=1}^m\big),
 	\qquad
 	V_m(t):=\frac{1}{m}\widehat V(t)= \P\!\big(T_{\hat\Lambda}(X_j)\ge t \,\big|\, j\in\cH_0,(\theta_k)_{k=1}^m\big).
 	\]
 	Define the limit BH curve and target threshold
 	\[
 	\phi(t):=\frac{\bar G_0(t)}{(1-\gamma)\bar G_0(t)+\gamma\,\bar G_1(t)},
 	\qquad
 	t_*:=\inf\{t\in\R:\,\phi(t)\le q\}.
 	\]
 	By Assumption~\ref{ass:insample-mlr}, $\phi$ is strictly decreasing and continuous with
 	$\phi(-\infty)=1$ and $\phi(+\infty)=0$, so $t_*$ is finite and well-defined for any $q\in(0,1)$.
 	
 	\paragraph{Step 1: Pointwise limit of $R_m(t)$.}
 	Fix $t\in\R$ and set $Y_j(t):=\mathbf 1\{S_j\ge t\}$. By Assumption~\ref{ass:insample-indep},
 	for any $i\neq j$,
 	\[
 	\P\!\big(S_i\ge t,S_j\ge t \,\big|\,\cF_m\big)
 	- \P(S_i\ge t \mid \cF_m)\,\P(S_j\ge t \mid \cF_m)\ \convp\ 0,
 	\quad \cF_m:=\sigma(\cH_0,\cH_1),
 	\]
 	hence
 	\[
 	\Cov\!\big(Y_i(t),Y_j(t)\mid \cF_m\big)\ \convp\ 0.
 	\]
 	Since $|Y_j(t)|\le 1$, Lemma~\ref{lemma:wlln-bounded-rv} yields
 	\[
 	R_m(t)- \E\!\Big[\frac{1}{m}\sum_{j=1}^m Y_j(t)\,\Big|\,\cF_m\Big]\ \convp\ 0.
 	\]
 	Write $\gamma_m:=m^{-1}|\cH_1|$ and
 	$p_{k,m}(t):=\P(S_j\ge t\mid j\in\cH_k,(\cH_0,\cH_1))$.
 	By the tower property and exchangeability within groups,
 	\[
 	\E\!\Big[\frac{1}{m}\sum_{j=1}^m Y_j(t)\,\Big|\,\cF_m\Big]
 	= (1-\gamma_m)p_{0,m}(t)+\gamma_m p_{1,m}(t).
 	\]
 	Assumption~\ref{ass:insample-indep} (by setting $t_2 =-\infty$, more formally follow a similar argument to Step 1 part (a) of Proof of Theorem \ref{thm:master-in-sample}) gives $p_{k,m}(t)\convp \bar G_k(t)$, and $\gamma_m\convp \gamma$,
 	so
 	\[
 	R_m(t)\ \convp\ (1-\gamma)\bar G_0(t)+\gamma\,\bar G_1(t).
 	\]
 	
 	\paragraph{Step 2: Pointwise limit of $V_m(t)$.}
 	By Assumption~\ref{ass:insample-tails} 
 	\[
 	V_m(t)=\P\!\big(T_{\hat\Lambda}(X_j)\ge t \,\big|\, j\in\cH_0,(\theta_k)_{k=1}^m\big)\    \to\ \bar G_0(t).
 	\]
 	(Here $V_m(t)$ is an oracle quantity in Algorithm~\ref{alg:in-sample-bh} and does not require empirical averaging.)
 	
 	\paragraph{Step 3: Uniform convergence.}
 	Both $R_m(t)$ and $V_m(t)$ are bounded, monotone decreasing in $t$.
 	Their limits $(1-\gamma)\bar G_0+\gamma\bar G_1$ and $\bar G_0$ are continuous and decreasing by Assumption~\ref{ass:insample-regularity}. Hence Lemma~\ref{lemma:unif-convg-monotone} (since $(R_m,(1-\gamma)\bar G_0-\gamma\bar G_1)$ and $(V_m, \bar G_0)$ agree on $\pm \infty$) implies
 	\[
 	\|R_m-(1-\gamma)\bar G_0-\gamma\bar G_1\|_{\infty}\ \convp\ 0,
 	\qquad
 	\|V_m-\bar G_0\|_{\infty}\ \convp\ 0.
 	\]
 	
 	\paragraph{Step 4: Ratio convergence and threshold consistency.}
 	Let $\phi_m(t):=V_m(t)/R_m(t)$. By Assumption~\ref{ass:insample-regularity}, the limit denominator $(1-\gamma)\bar G_0+\gamma\bar G_1$ is bounded away from $0$ on $(-\infty,C]$ for all $C\in R$; Step~3 implies the same for $R_m$ with high probability. Lemma~\ref{lemma:unif-convg-ratios} yields
 	\[
 	\|\phi_m-\phi\|_{\infty,K}\ \convp\ 0\qquad\text{for each }K =(-\infty,C].
 	\]
 	Assumption~\ref{ass:insample-mlr} (strict monotonicity and tail separation) guarantees that $q$ is attained. Therefore, by Lemma~\ref{lemma:uniform-convg-of-inverse},
 	\[
 	\hat t:=\inf\{t:\phi_m(t)\le q\}\ \convp\ t_*:=\inf\{t:\phi(t)\le q\}.
 	\]
 	
 	\paragraph{Step 5: TPP at the random threshold.}
 	Define
 	\[
 	\TPP_m(t):=\frac{|\{j\in\cH_1:S_j\ge t\}|}{|\cH_1|}
 	=\frac{\frac{1}{m}\sum_{j=1}^m \mathbf 1\{j\in\cH_1,S_j\ge t\}}{\frac{1}{m}\sum_{j=1}^m \mathbf 1\{j\in\cH_1\}}.
 	\]
 	Using the same conditional–covariance argument as in Step~1 and Assumption~\ref{ass:insample-indep},
 	\[
 	\TPP_m(t)\ \convp\ \bar G_1(t)\qquad\text{for each fixed }t.
 	\]
 	By Lemma~\ref{lemma:unif-convg-monotone} (uniform convergence near $t_*$, since $\TPP_m$ and $\bar G_1$ agree on $\pm \infty$ ) and $\hat t\convp t_*$,
 	\[
 	\TPP_m(\hat t)\ \convp\ \bar G_1(t_*).
 	\]
 	
 	\paragraph{Step 6: Identify the BH power.}
 	By definition, $\TPR_{\mathrm{BH}}(G_0,G_1,\gamma)=\bar G_1(t_*)$. Hence
 	\[
 	\TPP(\widehat{\cR})=\TPP_m(\hat t)\ \convp\ \TPR_{\mathrm{BH}}(G_0,G_1,\gamma).
 	\]
 \end{proof}
 
 \section{Results Related to point prior}
 
 \subsection{Quality of learned direction}
 
 \paragraph{Proof of Lemma \ref{lem:mean-learner-untransformed} and Proposition \ref{prop:mean-learner-masked}.} The proof of Lemma \ref{lem:mean-learner-untransformed} follows directly from Lemma \ref{lem:mean-learner-split} with $\pi_{\spl} =1$ and proof of Proposition \ref{prop:mean-learner-masked} follows from Lemma \ref{lem:mean-learner-split}  and Lemma \ref{lem:mean-learner-bonus}.
 \begin{lemma}[Quality of mean learner on split data]
 	\label{lem:mean-learner-split}
 	Assume the model \eqref{eq:data-generating}–\eqref{eq:two-groups} with the point-mass prior
 	\(\Lambda=\delta_{h\bm v}\), and the asymptotic regime \eqref{eq:asymptotic-regime}.
 	Let \(\bm X^\learn\) be obtained by data fission with splitting proportion \(\pi_{\spl}\in(0,1]\),
 	i.e.
 	\[
 	\bm X_j^\learn \;=\; \sqrt{\pi_{\spl}}\;\bm X_j \;+\; \sqrt{1-\pi_{\spl}}\;\widetilde{\bm X}_j,
 	\qquad \widetilde{\bm X}_j \iidsim N(\bm 0,\bm I_d) .
 	\]
 	Define the mean learner (cf. \eqref{eq:mean-learner})
 	\[
 	\hat{\bm v}_{\spl} \;=\; \frac{\bar{\bm X}^{\learn}}{\|\bar{\bm X}^{\learn}\|},
 	\qquad
 	\bar{\bm X}^{\learn} \;:=\; \frac{1}{m}\sum_{j=1}^m \bm X_j^\learn .
 	\]
 	Then
 	\[
 	\hat{\bm v}_{\spl}^{\top}\bm v \;\convp\;
 	\tau_{\textnormal{mean}}(\gamma,c,\sqrt{\pi_{\spl}}\,h)
 	\;=\;
 	\frac{h\gamma}{\sqrt{h^2\gamma^2 + c/\pi_{\spl}}}\, .
 	\]
 \end{lemma}
 
 \begin{proof}
 	Under the point-mass prior, \(\bm\theta_j\in\{\bm 0,h\bm v\}\) with
 	\(\P(\bm\theta_j=h\bm v)=\gamma\) i.i.d.\ across \(j\).
 	By construction of \(\bm X_j^\learn\),
 	\[
 	\bm X_j^\learn \;\overset{d}{=}\; \sqrt{\pi_{\spl}}\;\bm\theta_j \;+\; \bm Z_j,
 	\qquad \bm Z_j \iidsim N(\bm 0,\bm I_d),
 	\]
 	so
 	\[
 	\bar{\bm X}^{\learn}
 	\;=\;
 	\sqrt{\pi_{\spl}}\;\bar{\bm\theta} \;+\; \bar{\bm Z},
 	\qquad
 	\bar{\bm\theta}:=\frac{1}{m}\sum_{j=1}^m \bm\theta_j,
 	\quad
 	\bar{\bm Z}:=\frac{1}{m}\sum_{j=1}^m \bm Z_j \sim N\!\Big(\bm 0, \frac{1}{m}\bm I_d\Big).
 	\]
 	Let \(\gamma_m:=m^{-1}|\cH_1|\). Then \(\bar{\bm\theta}= \gamma_m h\,\bm v\) and \(\gamma_m\convp\gamma\) by the law of large numbers across features.
 	
 	Without loss of generality rotate coordinates so that \(\bm v=\bm e_1\).
 	Write the decomposition of \(\bar{\bm Z}\) into the component along \(\bm e_1\) and its orthogonal complement:
 	\[
 	B_m := \bm e_1^{\top}\bar{\bm Z} \sim N\!\Big(0,\frac{1}{m}\Big), 
 	\qquad
 	T_m := \big\|\bar{\bm Z} - B_m \bm e_1\big\|^2 \;\overset{d}{=}\; \frac{1}{m}\chi^2_{d-1}.
 	\]
 	Then
 	\[
 	\bar{\bm X}^{\learn}
 	= \big(\sqrt{\pi_{\spl}}\,\gamma_m h + B_m\big)\bm e_1
 	\;+\; \big(\bar{\bm Z} - B_m \bm e_1\big),
 	\]
 	and consequently
 	\[
 	\hat{\bm v}_{\spl}^{\top}\bm v
 	= \frac{\sqrt{\pi_{\spl}}\,\gamma_m h + B_m}
 	{\sqrt{\big(\sqrt{\pi_{\spl}}\,\gamma_m h + B_m\big)^2 + T_m}}.
 	\]
 	
 	Now \(\gamma_m\convp\gamma\), \(B_m\to 0\) in probability, and
 	\(T_m \convp \lim (d-1)/m = c\) by \eqref{eq:asymptotic-regime}.
 	By the continuous mapping theorem,
 	\[
 	\hat{\bm v}_{\spl}^{\top}\bm v
 	\;\convp\;
 	\frac{\sqrt{\pi_{\spl}}\,\gamma h}{\sqrt{\big(\sqrt{\pi_{\spl}}\,\gamma h\big)^2 + c}}
 	\;=\;
 	\frac{h\gamma}{\sqrt{h^2\gamma^2 + c/\pi_{\spl}}}\,,
 	\]
 	which is \(\tau_{\textnormal{mean}}(\gamma,c,\sqrt{\pi_{\spl}}\,h)\) as claimed.
 \end{proof}

 \begin{lemma}[Quality of mean learner on augmented data (BONuS)]
 	\label{lem:mean-learner-bonus}
 	Assume the model \eqref{eq:data-generating}–\eqref{eq:two-groups} with the point-mass prior
 	\(\Lambda=\delta_{h\bm v}\), and the asymptotic regime \eqref{eq:asymptotic-regime}.
 	Let the BONuS learning data be
 	\[
 	\bm X^{\learn}=[\bm X,\tilde{\bm X}],
 	\qquad
 	\tilde{\bm X}_k \iidsim N(\bm 0,\bm I_d),\ k=1,\dots,\tilde m,
 	\]
 	and define the mean learner (cf.\ \eqref{eq:mean-learner})
 	\[
 	\hat{\bm v}_{\bonus}\;=\;\frac{\bar{\bm X}^{\learn}}{\|\bar{\bm X}^{\learn}\|},
 	\qquad
 	\bar{\bm X}^{\learn}\;:=\;\frac{1}{m+\tilde m}
 	\left(\sum_{j=1}^m \bm X_j+\sum_{k=1}^{\tilde m}\tilde{\bm X}_k\right).
 	\]
 	Then
 	\[
 	\hat{\bm v}_{\bonus}^{\top}\bm v
 	\;\convp\;
 	\tau_{\textnormal{mean}}(\pi_{\aug}\gamma,\pi_{\aug}c, h)
 	\;=\;
 	\frac{h\gamma}{\sqrt{h^2\gamma^2 + c/\pi_{\aug}}}\,,
 	\]
 	where \(\pi_{\aug}=\lim m/(m+\tilde m)\in(0,1]\).
 \end{lemma}
 
 \begin{proof}
 	Under the point-mass prior, \(\bm\theta_j\in\{\bm 0,h\bm v\}\) with
 	\(\P(\bm\theta_j=h\bm v)=\gamma\) i.i.d.\ across \(j\), and
 	\(\bm X_j\sim N(\bm\theta_j,\bm I_d)\) while each \(\tilde{\bm X}_k\sim N(\bm 0,\bm I_d)\).
 	Hence
 	\[
 	\bar{\bm X}^{\learn}
 	=\frac{1}{m+\tilde m}\Big(\sum_{j=1}^m \bm\theta_j\Big)
 	\;+\;\bar{\bm Z},
 	\qquad
 	\bar{\bm Z}
 	:=\frac{1}{m+\tilde m}\left(\sum_{j=1}^m(\bm X_j-\bm\theta_j)
 	+ \sum_{k=1}^{\tilde m}\tilde{\bm X}_k\right)
 	\sim N\!\Big(\bm 0,\tfrac{1}{m+\tilde m}\bm I_d\Big).
 	\]
 	Let \(\gamma_m:=m^{-1}|\cH_1|\). Then \(\frac{1}{m}\sum_{j=1}^m \bm\theta_j=\gamma_m h\,\bm v\), so
 	\[
 	\bar{\bm X}^{\learn}
 	= \frac{m}{m+\tilde m}\,\gamma_m h\,\bm v \;+\; \bar{\bm Z}
 	= \pi_{\aug,m}\,\gamma_m h\,\bm v \;+\; \bar{\bm Z},
 	\qquad
 	\pi_{\aug,m}:=\frac{m}{m+\tilde m}\ \to\ \pi_{\aug}.
 	\]
 	
 	Without loss of generality rotate coordinates so that \(\bm v=\bm e_1\).
 	Write the decomposition of \(\bar{\bm Z}\):
 	\[
 	B_m := \bm e_1^{\top}\bar{\bm Z} \sim N\!\Big(0,\frac{1}{m+\tilde m}\Big), 
 	\qquad
 	T_m := \big\|\bar{\bm Z} - B_m \bm e_1\big\|^2 
 	\;\overset{d}{=}\; \frac{1}{m+\tilde m}\chi^2_{d-1}.
 	\]
 	Then
 	\[
 	\hat{\bm v}_{\bonus}^{\top}\bm v
 	= \frac{\pi_{\aug,m}\gamma_m h + B_m}
 	{\sqrt{\big(\pi_{\aug,m}\gamma_m h + B_m\big)^2 + T_m}}.
 	\]
 	
 	By the law of large numbers across features \(\gamma_m\convp\gamma\), by assumption \eqref{eq:asymptotic-regime} we have \(d/m\to c\) and hence
 	\[
 	T_m \;\convp\; c\,\pi_{\aug},
 	\qquad
 	B_m \;\convp\; 0.
 	\]
 	since $\frac{d}{m+\tilde m} \;\to\; \frac{d}{m}\cdot\frac{m}{m+\tilde m} \;=\; c\,\pi_{\aug}$.
 	Therefore, by the continuous mapping theorem,
 	\[
 	\hat{\bm v}_{\bonus}^{\top}\bm v
 	\;\convp\;
 	\frac{\pi_{\aug}\gamma h}{\sqrt{(\pi_{\aug}\gamma h)^2 + c\,\pi_{\aug}}}
 	\;=\;
 	\frac{h\gamma}{\sqrt{h^2\gamma^2 + c/\pi_{\aug}}}
 	\;=\;
 	\tau_{\textnormal{mean}}(\pi_{\aug}\gamma,\pi_{\aug}c,h),
 	\]
 	as claimed.
 \end{proof}

 \subsection{Results on power}
 
 \paragraph{Proof of Theorem \ref{thm:power-point-mass}}
 
 The proof follows by using Lemma \ref{lem:power-point-mass-split}, \ref{lem:power-point-mass-bonus} and \ref{lem:power-point-mass-insample}.
 
 \begin{lemma}[Asymptotic power for Split BH under point-mass prior]
 	\label{lem:power-point-mass-split}
 	Assume the model \eqref{eq:data-generating}–\eqref{eq:two-groups} with $\Lambda=\delta_{h\bm v}$ and the asymptotic regime \eqref{eq:asymptotic-regime}. 
 	Let $\widehat{\cR}$ be the rejection set of \emph{Split BH} with splitting proportion $\pi_{\spl}\in(0,1)$ and the mean learner \eqref{eq:mean-learner} fit on $\bm X^{\learn}$. 
 	Then
 	\[
 	\TPP(\widehat{\cR})\ \convp\ \TPR_{\textnormal{BH}}\!\big(N(0,1),\,N(\mu_{\spl},1),\,\gamma\big),
 	\]
 	where
 	\[
 	\mu_{\spl}
 	\;:=\;
 	\sqrt{1-\pi_{\spl}}\;h\;\tau_{\textnormal{mean}}(\gamma,c,\sqrt{\pi_{\spl}}\,h)
 	\;=\;
 	\sqrt{1-\pi_{\spl}}\;h\;\frac{h\gamma}{\sqrt{h^2\gamma^2+c/\pi_{\spl}}}.
 	\]
 \end{lemma}
 
 \begin{proof}
 	Let
 	\[
 	S_j:=T_{\hat\Lambda}(X_j^{\score})=\hat v^{\top}X_j^{\score},
 	\qquad
 	\mu_{\spl}:=\sqrt{1-\pi_{\spl}}\,h\,\tau_{\mathrm{mean}}(\gamma,c,\sqrt{\pi_{\spl}}\,h).
 	\]
 	By data fission, \(X^{\learn}\independent X^{\score}\) given the parameter, and by Lemma~\ref{lem:mean-learner-split},
 	\[
 	\hat v^{\top}v \convp \tau_{\mathrm{mean}}(\gamma,c,\sqrt{\pi_{\spl}}\,h).
 	\]
 	Set
 	\[
 	\bar G_0(t):=\bar\Phi(t),
 	\qquad
 	\bar G_1(t):=\bar\Phi(t-\mu_{\spl}).
 	\]
 	
 	\emph{(i) Pointwise null-tail convergence.}
 	If \(j\in\cH_0\), then conditional on \(\hat v\) and \((\cH_0,\cH_1)\),
 	\[
 	S_j=\hat v^{\top}X_j^{\score}\sim N(0,1),
 	\]
 	so for every fixed \(t\in\R\),
 	\begin{equation}\label{eq:score-split-bh-null}
 	\P(S_j\ge t\mid \hat v,\,j\in\cH_0,(\cH_0,\cH_1))
 	=\bar\Phi(t)=\bar G_0(t).
 	\end{equation}
 	Hence Assumption~\ref{ass:pointwise-convg-score-tails} holds exactly.
 	
 	For later reference, if \(j\in\cH_1\), then
 	\[
 	X_j^{\score}\mid (j\in\cH_1,(\cH_0,\cH_1))
 	\sim N(\sqrt{1-\pi_{\spl}}\,h\,v,I_d),
 	\]
 	and therefore, conditional on \(\hat v\),
 	\[
 	S_j\mid (\hat v,\,j\in\cH_1,(\cH_0,\cH_1))
 	\sim N(\sqrt{1-\pi_{\spl}}\,h\,\hat v^{\top}v,\,1).
 	\]
 	Since \(\hat v^{\top}v\convp \tau_{\mathrm{mean}}(\gamma,c,\sqrt{\pi_{\spl}}\,h)\),
 	\begin{equation}\label{eq:score-split-bh-alt}
 	\P(S_j\ge t\mid \hat v,\,j\in\cH_1,(\cH_0,\cH_1))
 	=\bar\Phi\!\bigl(t-\sqrt{1-\pi_{\spl}}\,h\,\hat v^{\top}v\bigr)
 	\convp \bar G_1(t).
 	\end{equation}
 	
 	\medskip
 	\noindent
 \emph{(ii) Asymptotic pairwise independence.}
 Fix distinct \(i\neq j\), \(k_1,k_2\in\{0,1\}\), and \(t_1,t_2\in\R\).
 Under the point prior, once we condition on \((\cH_0,\cH_1)\), the means of the score-split observations are fixed. Moreover, by data splitting, conditional on \(\hat v\) and \((\cH_0,\cH_1)\), the variables \(X_i^{\score}\) and \(X_j^{\score}\) are independent. Hence \(S_i=\hat v^\top X_i^{\score}\) and \(S_j=\hat v^\top X_j^{\score}\) are conditionally independent given \(\hat v\) and \((\cH_0,\cH_1)\), so
 \[
 \P(S_i\ge t_1,\ S_j\ge t_2 \mid \hat v,\, i\in\cH_{k_1},\,j\in\cH_{k_2},\,(\cH_0,\cH_1))
 =
 \P(S_i\ge t_1 \mid \hat v,\, i\in\cH_{k_1},\,(\cH_0,\cH_1))
 \]
 \[
 \hspace{5cm}\times\,
 \P(S_j\ge t_2 \mid \hat v,\, j\in\cH_{k_2},\,(\cH_0,\cH_1)).
 \]
 From part (i), for each fixed \(t\),
 \[
 \P(S_\ell\ge t \mid \hat v,\, \ell\in\cH_k,\,(\cH_0,\cH_1))
 \convp \bar G_k(t),\qquad k\in\{0,1\}.
 \]
 Since these conditional tail probabilities are bounded in \([0,1]\), taking expectations and using bounded convergence yields
 \[
 \P\!\big(S_i\ge t_1,\ S_j\ge t_2 \mid i\in\cH_{k_1},\,j\in\cH_{k_2},\,(\cH_0,\cH_1)\big)
 \to \bar G_{k_1}(t_1)\bar G_{k_2}(t_2).
 \]
 Thus Assumption~\ref{ass:splitBH-indep} holds.
 	\medskip
 	
 	\noindent
 	\emph{(iii) Regularity of the limit cdfs.}
 	Here
 	\[
 	G_0(t)=\Phi(t),
 	\qquad
 	G_1(t)=\Phi(t-\mu_{\spl}).
 	\]
 	Both are continuous cdfs on \(\R\), and for every \(C\in\R\),
 	\[
 	\sup_{t\le C} G_0(t)<1,
 	\qquad
 	\sup_{t\le C} G_1(t)<1.
 	\]
 	So Assumption~\ref{ass:regularity-of-cdfs} holds.
 	
 	\medskip
 	\noindent
 	\emph{(iv) Monotone survival ratio and tail separation.}
 	We have
 	\[
 	\frac{\bar G_1(t)}{\bar G_0(t)}
 	=
 	\frac{\bar\Phi(t-\mu_{\spl})}{\bar\Phi(t)}.
 	\]
 	By Lemma~\ref{lemma:mono_survival_gaussian}, this ratio is strictly increasing in \(t\)
 	whenever \(\mu_{\spl}>0\), and moreover
 	\[
 	\frac{\bar G_1(t)}{\bar G_0(t)}\to\infty
 	\qquad\text{as }t\to\infty.
 	\]
 	Hence Assumption~\ref{ass:mlr} holds. (If \(\mu_{\spl}=0\), then \(G_0=G_1\) and the conclusion is immediate.)
 	
 	\medskip
 	All assumptions of Theorem~\ref{thm:master-split-BH} are satisfied. Therefore,
 	\[
 	\TPP(\widehat{\cR})\convp \TPR_{\mathrm{BH}}(G_0,G_1,\gamma)
 	=
 	\TPR_{\mathrm{BH}}\bigl(N(0,1),N(\mu_{\spl},1),\gamma\bigr).
 	\]
 \end{proof}
 
 \begin{lemma}[Asymptotic power for BONuS under point-mass prior]
 	\label{lem:power-point-mass-bonus}
 	Assume the model \eqref{eq:data-generating}–\eqref{eq:two-groups} with $\Lambda=\delta_{h\bm v}$ and the asymptotic regime \eqref{eq:asymptotic-regime}.
 	Run \emph{BONuS} with augmentation size $\tilde m$ and a permutation-invariant mean learner \eqref{eq:mean-learner} trained on $\bm X^{\learn}=[\bm X,\tilde{\bm X}]$.
 	Let $\widehat{\cR}$ be the resulting rejection set.
 	Define
 	\[
 	\mu_{\bonus}
 	\;:=\;
 	h\;\tau_{\textnormal{mean}}(\pi_{\aug}\gamma,\pi_{\aug}c,h)
 	\;=\;
 	h\cdot\frac{h\gamma}{\sqrt{h^2\gamma^2 + c/\pi_{\aug}}},
 	\qquad
 	\pi_{\aug}=\lim\frac{m}{m+\tilde m}\in(0,1].
 	\]
 	Then
 	\[
 	\TPP(\widehat{\cR})\ \convp\ \TPR_{\textnormal{BH}}\!\big(N(0,1),\,N(\mu_{\bonus},1),\,\gamma\big).
 	\]
 \end{lemma}
 
 \begin{proof}
 	Let $S_j:=T_{\hat\Lambda}(X_j)=\hat v_{\bonus}^{\top}X_j$ be the BONuS score.
 	
 	\emph{(i) Pointwise convergence of score tails.}  
 	By Lemma~\ref{lemma:bonus-limit-law}, under the asymptotic regime,
 	\[
 	S_j \mid j\in\cH_0,(\cH_0,\cH_1) \;\convd\; N(c/D,1),
 	\quad
 	S_j \mid j\in\cH_1,(\cH_0,\cH_1) \;\convd\; N\big((c+h^2\gamma)/D,1\big),
 	\quad \text{almost surely}\]
 	with $D=\sqrt{h^2\gamma^2+c/\pi_{\aug}}$.  
 	Hence for each fixed $t$,
 	\begin{equation}\label{eq:score-bonus-null}
 	\P(S_j \ge t \mid j\in\cH_0,(\cH_0,\cH_1))\ \convp\ \bar\Phi(t-c/D)=:\bar G_0(t),
 	\end{equation}
 	and
 	\begin{equation}\label{eq:score-bonus-alt}
 	\P(S_j \ge t \mid j\in\cH_1,(\cH_0,\cH_1))\ \convp\ \bar\Phi\!\big(t-(c+h^2\gamma)/D\big)=:\bar G_1(t).
 	\end{equation}
 	
 	\emph{(ii) Asymptotic pairwise independence.}  
 	By Lemma~\ref{lemma:bonus-pairwise-independence} (pairwise asymptotic independence of BONuS scores),
 	for any distinct $j_1\neq j_2$, any fixed $t_1,t_2\in\R$, and $k_1,k_2\in\{0,1\}$,
 	\[
 	\P\!\big(S_{j_1}\!\ge t_1,\ S_{j_2}\!\ge t_2 \,\big|\, \cH_{0},\cH_{1}\big)
 	\;-\;
 	\P\!\big(S_{j_1}\!\ge t_1 \,\big|\, \cH_{0},\cH_{1}\big)\,
 	\P\!\big(S_{j_2}\!\ge t_2 \,\big|\, \cH_{0},\cH_{1}\big)
 	\ \convp\ 0,
 	\]
 	so the joint exceedance probability factorizes in the limit, verifying the pairwise-independence requirement of the BONuS master theorem.
 	
 	\emph{(iii) Regularity of limits.}  
 	$G_0(t)=\Phi(t-c/D)$ and $G_1(t)=\Phi(t-(c+h^2\gamma)/D)$ are continuous distribution functions, and for set $K = (-\infty,C]$, $\sup_{t\in K}G_k(t)<1$.
 	
 	\emph{(iv) Monotone survival ratio and tail separation.}  
 	We have
 	\[
 	\frac{\bar G_1(t)}{\bar G_0(t)}=\frac{\bar\Phi\!\big(t-(c+h^2\gamma)/D\big)}{\bar\Phi(t-c/D)}.
 	\]
 	By Lemma~\ref{lemma:mono_survival_gaussian}, this ratio is strictly increasing in $t$ and diverges to $+\infty$ as $t\to\infty$.
 	
 	Thus all assumptions of the BONuS master theorem hold in the point-mass case, and applying it gives
 	\[
 	\TPP(\widehat{\cR})\ \convp\ \TPR_{\textnormal{BH}}\!\big(N(c/D,1),\,N((c+h^2\gamma)/D,1),\,\gamma\big).
 	\]
 	
 	Finally, observe that $\big(N(c/D,1),N(c/D+\mu_{\bonus},1)\big)$
 	is a common shift by $a=c/D$ of $\big(N(0,1),N(\mu_{\bonus},1)\big)$.
 	By Lemma~\ref{lem:bh-translation-invariance},
 	\[
 	\TPR_{\mathrm{BH}}\!\big(N(c/D,1),\,N(c/D+\mu_{\bonus},1),\,\gamma\big)
 	=
 	\TPR_{\mathrm{BH}}\!\big(N(0,1),\,N(\mu_{\bonus},1),\,\gamma\big).
 	\]
 	
 \end{proof}
 
 \begin{lemma}[Asymptotic power for in-sample BH under point-mass prior]
 	\label{lem:power-point-mass-insample}
 	Assume the model \eqref{eq:data-generating}–\eqref{eq:two-groups} with $\Lambda=\delta_{h\bm v}$ and the asymptotic regime \eqref{eq:asymptotic-regime}.
 	Run \emph{in-sample BH} with a permutation-invariant mean learner \eqref{eq:mean-learner} trained on $\bm X^{\learn}=\bm X$.
 	Let $\widehat{\cR}$ be the resulting rejection set.
 	Define
 	\[
 	\mu_{\textnormal{in-sample}}
 	\;:=\;
 	h\;\tau_{\textnormal{mean}}(\gamma,c,h)
 	\;=\;
 	h\cdot\frac{h\gamma}{\sqrt{h^2\gamma^2 + c}}.
 	\]
 	Then
 	\[
 	\TPP(\widehat{\cR})\ \convp\ \TPR_{\textnormal{BH}}\!\big(N(0,1),\,N(\mu_{\textnormal{in-sample}},1),\,\gamma\big).
 	\]
 \end{lemma}
 
 \begin{proof}
 	Let $S_j:=T_{\hat\Lambda}(X_j)=\hat{\bm v}^{\!\top}X_j$, where the learner is the mean learner \eqref{eq:mean-learner} trained on $\bm X^{\learn}=\bm X$ (in-sample).
 	Under the point prior $\Lambda=\delta_{h\bm v}$ the likelihood-ratio score is linear, and the quality of the learned direction satisfies
 	\[
 	\hat{\bm v}^{\!\top}\bm v \ \convp\ \tau_{\textnormal{mean}}(\gamma,c,h)
 	\ =\ \frac{h\gamma}{\sqrt{h^2\gamma^2+c}},
 	\]
 	by the BONuS alignment lemmas specialized to the unmasked case (set $\pi_{\aug}=1$ in Proposition~\ref{prop:mean-learner-masked}).
 	
 	We verify the assumptions of Theorem~\ref{thm:master-in-sample} with the following choices:
 	\[
 	G_0(t)=\Phi\!\big(t-\mu_0\big),\qquad
 	G_1(t)=\Phi\!\big(t-\mu_1\big),\qquad
 	\mu_0:=\frac{c}{\sqrt{h^2\gamma^2+c}},\quad
 	\mu_1:=\mu_0+\mu_{\textnormal{in-sample}},
 	\]
 	where $\mu_{\textnormal{in-sample}}:=h\,\tau_{\textnormal{mean}}(\gamma,c,h)=\dfrac{h^2\gamma}{\sqrt{h^2\gamma^2+c}}$.
 	
 	\paragraph{(i) Pointwise convergence of score tails (Assumption~\ref{ass:insample-tails}).}
 	The “limit law of the test statistic” argument used for BONuS (Lemma~\ref{lemma:bonus-limit-law}) carries over verbatim with $\tilde m=0$ (equivalently, $\pi_{\aug}=1$): use the leave-one-out decomposition with
 	\(
 	A_j:=\sum_{i\neq j}X_i
 	\)
 	(in place of $A_j$ that also included synthetics). Exactly the same order calculations yield that on set of probability $1$
 	\[
 	S_j \mid j\in\cH_0,(\cH_0,\cH_1) \ \convd\ N(\mu_0,1),\qquad
 	S_j \mid j\in\cH_1 ,(\cH_0,\cH_1)\ \convd\ N(\mu_1,1),
 	\]
 	with $\mu_0=c/\sqrt{h^2\gamma^2+c}$ and $\mu_1=\mu_0+h^2\gamma/\sqrt{h^2\gamma^2+c}$, hence for each fixed $t$,
 	\[
 	\P(S_j\ge t\mid j\in\cH_k,(\cH_0,\cH_1))\ \convp\ \bar G_k(t),\qquad k\in\{0,1\}.
 	\]
 	Since $(\theta_k)_{k=1}^m$ is a deterministic function of $(\cH_0,\cH_1)$ we have that
 	\[
 	\P(S_j\ge t\mid j\in\cH_0,(\theta_k)_{k=1}^m)\ \convp\ \bar G_0(t),\qquad.
 	\]
 	\paragraph{(ii) Asymptotic pairwise independence (Assumption~\ref{ass:insample-indep}).}
 	Use the BONuS pairwise-independence lemma (Lemma \ref{lemma:bonus-pairwise-independence}) for original scores with the same leave-two-out device but with no synthetics (set $\tilde m=0$). Define
 	\(
 	A_{-12}:=\sum_{i\notin\{j_1,j_2\}}X_i
 	\)
 	and replicate the decomposition from the BONuS proof: for $\nu\in\{1,2\}$,
 	\[
 	S_{j_\nu}
 	=\frac{A_{-12}^{\top}X_{j_\nu}}{\|A_{-12}\|}\;+\;\frac{c}{\sqrt{h^2\gamma^2+c}}\;+\;o_p(1).
 	\]
 	Conditioning on $A_{-12}$, the two leading Gaussian terms
 	\(
 	A_{-12}^{\top}X_{j_\nu}/\|A_{-12}\|
 	\)
 	are independent (because $X_{j_1}\independent X_{j_2}$ and both are independent of $A_{-12}$). Hence the unconditional joint law factors in the limit, yielding the desired pairwise tail factorization:
 	\begin{align*}
 	\P&(S_{j_1}\ge t_1,\ S_{j_2}\ge t_2\mid j_\nu\in\cH_{k_\nu},(\cH_0,\cH_1))\\
 	\;&=\;
 	\P(S_{j_1}\ge t_1\mid j_1\in\cH_{k_1},(\cH_0,\cH_1))
 	\;\P(S_{j_2}\ge t_2\mid j_2\in\cH_{k_2},(\cH_0,\cH_1))
 	\;+\;o_p(1).
 	\end{align*}
 	
 	\paragraph{(iii) Regularity (Assumption~\ref{ass:insample-regularity}).}
 	$G_0$ and $G_1$ are Gaussian cdfs with finite means, hence continuous and bounded away from $1$ on any interval of the form $(-\infty,C]$ for $C\in R$.
 	
 	\paragraph{(iv) Monotone survival ratio and tail separation (Assumption~\ref{ass:insample-mlr}).}
 	With $\bar G_0(t)=\bar\Phi(t-\mu_0)$ and $\bar G_1(t)=\bar\Phi(t-\mu_1)$ where $\mu_1>\mu_0$, Lemma~\ref{lemma:mono_survival_gaussian} gives that the survival ratio
 	\(
 	t\mapsto \bar G_1(t)/\bar G_0(t)=\bar\Phi(t-\mu_1)/\bar\Phi(t-\mu_0)
 	\)
 	is strictly increasing in $t$ and diverges to $\infty$ as $t\to\infty$.
 	
 	\medskip
 	With (i)–(iv) verified, Theorem~\ref{thm:master-in-sample} applies and yields
 	\[
 	\TPP(\widehat{\cR})\ \convp\ \TPR_{\mathrm{BH}}\!\big(N(\mu_0,1),\,N(\mu_1,1),\,\gamma\big).
 	\]
 	Finally, by the translation invariance of the BH TPR functional (Lemma~\ref{lem:bh-translation-invariance}),
 	\[
 	\TPR_{\mathrm{BH}}\!\big(N(\mu_0,1),\,N(\mu_1,1),\,\gamma\big)
 	=\TPR_{\mathrm{BH}}\!\big(N(0,1),\,N(\mu_1-\mu_0,1),\,\gamma\big)
 	=\TPR_{\mathrm{BH}}\!\big(N(0,1),\,N(\mu_{\textnormal{in-sample}},1),\,\gamma\big),
 	\]
 	which proves the claim.
 \end{proof}
  \subsection{Results on scores}
 \paragraph{Proof of Proposition \ref{prop:point-mass-prior-score-distributions}}
 \begin{enumerate}
 	\item Split BH: We start at~\eqref{eq:score-split-bh-null} and use Lemma \ref{lemma:conditional-to-unconditional-convg} (to integrate out $\{\hat{\bm v}_{\spl},\cH_0,\cH_1\}$ ) with $X_n =S^\spl_j  $ and $\cF_n = \{\hat{\bm v}_{\spl},\cH_0,\cH_1\}$ to obtain that $S^\spl_j \mid j \in \mathcal H_0 \convd N(0,1)$, similarly we can use~\eqref{eq:score-split-bh-alt} along with Lemma \ref{lemma:conditional-to-unconditional-convg} to obtain $S^\spl_j \mid j \in \mathcal H_1 \convd N(\mu_\spl(\gamma, c, h; \pi_\spl), 1)$.
 	
 	\item BONuS: We start at~\eqref{eq:score-bonus-null} and use Lemma \ref{lemma:conditional-to-unconditional-convg} with$X_n =S^\bonus_j  $ and $\cF_n = \{\cH_0,\cH_1\}$ to obtain that $S^\bonus_j \mid j \in \mathcal H_0 \convd N(K_\bonus,1)$, similarly we can use~\eqref{eq:score-bonus-alt} along with Lemma \ref{lemma:conditional-to-unconditional-convg} to obtain $S^\bonus_j \mid j \in \mathcal H_1 \convd N(K_\bonus + \mu_\bonus(\gamma, c, h; \pi_\aug), 1)$.
 	
 	\item In-Sample BH: The proof directly follows from BONuS score result with $\pi_{\aug} = 1$. 
 \end{enumerate}

 \subsection{Auxiliary Lemmas}
 \begin{lemma}[Limit law of the BONuS test statistic under the point prior]
 	\label{lemma:bonus-limit-law}
 	Let $\pi_{\aug}:=\lim m/(m+\tilde m)\in(0,1]$ and $c:=\lim d/m\in[0,\infty)$ under the asymptotic regime~\eqref{eq:asymptotic-regime}.  
 	Assume the point prior $\Lambda=\delta_{h\bm v}$ and generate features
 	\[
 	\bm X_j \;=\; b_j h \bm v + \bm \xi_j,
 	\qquad
 	b_j \stackrel{\text{i.i.d.}}{\sim} \mathrm{Ber}(\gamma),\quad
 	\bm \xi_j \iidsim \mathcal N(0,I_d).
 	\]
 	Run \textnormal{BONuS}: augment with $\tilde m$ synthetic nulls, fit the permutation-invariant mean learner
 	\[
 	\hat{\bm v}_{\bonus}
 	\;:=\;
 	\frac{\tfrac{1}{m+\tilde m}\Big(\sum_{j=1}^m \bm X_j + \sum_{k=1}^{\tilde m} \tilde{\bm X}_k\Big)}
 	{\left\|\tfrac{1}{m+\tilde m}\Big(\sum_{j=1}^m \bm X_j + \sum_{k=1}^{\tilde m} \tilde{\bm X}_k\Big)\right\|_2},
 	\]
 	and compute scores
 	\[
 	S_j^{\bonus} \;:=\; \hat{\bm v}_{\bonus}^{\!\top}\bm X_j,\qquad j=1,\dots,m.
 	\]
 	Define
 	\[
 	D \;:=\; \sqrt{h^2\gamma^2 + c/\pi_{\aug}}.
 	\]
 	Then as $m,d,\tilde m\to\infty$ with $d/m\to c$ and $m/(m+\tilde m)\to\pi_{\aug}$,
 	\[
 	S_j^{\bonus}\,\Big|\,j\in\cH_0,(\cH_0,\cH_1) \;\;\convd\;\; \mathcal N\!\Big(\tfrac{c}{D},1\Big),
 	\quad
 	S_j^{\bonus}\,\Big|\,j\in\cH_1,(\cH_0,\cH_1) \;\;\convd\;\; \mathcal N\!\Big(\tfrac{c+h^2\gamma}{D},1\Big) \quad \text{almost surely}.
 	\]
 \end{lemma}
 
 \begin{proof}
 	W.l.o.g.\ take $\bm v=\bm e_1$, so $\bm X_j=b_j h \bm e_1+\bm\xi_j$ with $\bm\xi_j\sim\mathcal N(0,I_d)$.  
 	Let $p=m+\tilde m$ and define the empirical mean
 	\[
 	\hat{\bm\mu}:=\frac{1}{p}\Big(\sum_{i=1}^m \bm X_i + \sum_{k=1}^{\tilde m}\tilde{\bm X}_k\Big),
 	\qquad
 	\hat{\bm v}_{\bonus}=\frac{\hat{\bm\mu}}{\|\hat{\bm\mu}\|_2}.
 	\]
 	For a fixed $j$, write the leave-one-out decomposition
 	\[
 	A_j:=\sum_{i\ne j}\bm X_i+\sum_{k=1}^{\tilde m}\tilde{\bm X}_k,
 	\qquad
 	\hat{\bm\mu}=\tfrac{1}{p}(A_j+\bm X_j),\qquad
 	\hat{\bm v}_{\bonus}=\frac{A_j+\bm X_j}{\|A_j+\bm X_j\|_2}.
 	\]
 	Note $A_j\independent \bm X_j$.

 	\emph{Step 1 (Scale of $\|A_j+\bm X_j\|$).}
 	Recall
 	\[
 	A_j \;=\; \sum_{i\ne j}\bm X_i \;+\; \sum_{k=1}^{\tilde m}\tilde{\bm X}_k
 	\;=\;
 	\underbrace{\sum_{i\ne j} b_i h\,\bm e_1}_{\text{signal}}
 	\;+\;
 	\underbrace{\sum_{i\ne j}\bm\xi_i \;+\; \sum_{k=1}^{\tilde m}\tilde{\bm\xi}_k}_{\text{noise}},
 	\]
 	where $\bm\xi_i,\tilde{\bm\xi}_k\overset{\mathrm{i.i.d.}}{\sim}\mathcal N(\bm0,I_d)$ and
 	$b_i\overset{\mathrm{i.i.d.}}{\sim}\mathrm{Ber}(\gamma)$, independent of the noises.
 	
 	\paragraph{Signal size.}
 	By the strong law of large numbers there exists a set $\mathcal{S}$ such that $\P(\mathcal{S}) =1$ and on $\mathcal{S}$,
 	\[
 	\frac{1}{m}\sum_{i\ne j} b_i \;\to\; \gamma
 	\quad\Longrightarrow\quad
 	\Big\|\sum_{i\ne j} b_i h\,\bm e_1\Big\|
 	\;=\; h\sum_{i\ne j} b_i
 	\;=\; m\gamma h \,\big(1+o(1)\big).
 	\]
 	
 	\paragraph{Noise size.}
 	Let
 	\[
 	\bm Z \;:=\; \sum_{i\ne j}\bm\xi_i \;+\; \sum_{k=1}^{\tilde m}\tilde{\bm\xi}_k.
 	\]
 	Then $\bm Z\sim\mathcal N\!\big(\bm0,(p-1)\,I_d\big)$ with $p=m+\tilde m$.
 	Hence
 	\[
 	\|\bm Z\|^2 \;\stackrel{d}{=}\; (p-1)\,\chi^2_d
 	\quad\convd\quad
 	\frac{\|\bm Z\|^2}{(p-1)\,d} \;\convp\; 1,
 	\qquad
 	\frac{\|\bm Z\|}{\sqrt{(p-1)\,d}} \;\convp\; 1.
 	\]
 	Therefore
 	\[
 	\|\bm Z\| \;=\; \sqrt{d\,p}\,\big(1+o_p(1)\big).
 	\]
 	
 	\paragraph{Cross term is negligible.}
 	Write $A_j = S_j + \bm Z$ with $S_j:=\sum_{i\ne j} b_i h\,\bm e_1$.
 	Then
 	\[
 	\|A_j\|^2 \;=\; \|S_j\|^2 + \|\bm Z\|^2 + 2\,S_j^{\!\top}\bm Z.
 	\]
 	Since $S_j\parallel \bm e_1$ and $\bm Z\sim\mathcal N(\bm0,(p-1)I_d)$,
 	$S_j^{\!\top}\bm Z \,\big|\, S_j \sim \mathcal N\!\big(0,\,(p-1)\,\|S_j\|^2\big)$.
 	Hence
 	\[
 	\frac{S_j^{\!\top}\bm Z}{m^2}
 	\;=\;
 	\frac{\|S_j\|}{m}\cdot \frac{1}{m}\,\mathcal N\!\big(0,p-1\big)
 	\;=\;
 	O_p\!\Big(\frac{\|S_j\|}{m}\cdot \sqrt{\frac{p}{m^2}}\Big)
 	\;=\;
 	O_p\!\Big(\gamma h\cdot \sqrt{\frac{p}{m^2}}\Big)
 	\;=\; o_p(1),
 	\]
 	because $p/m\to 1/\pi_{\aug}$ is constant, so $\sqrt{p}/m\to 0$.
 	
 	\paragraph{Asymptotic scale of $\|A_j\|$.}
 	Combining the three pieces,
 	\[
 	\|A_j\|^2
 	\;=\;
 	(m\gamma h)^2\,(1+o(1)) \;+\; d\,p\,(1+o_p(1)) \;+\; o_p(m^2).
 	\]
 	Divide by $m^2$ and use $d/m\to c$ and $p/m\to 1/\pi_{\aug}$:
 	\[
 	\frac{\|A_j\|^2}{m^2}
 	\;\convp\;
 	\gamma^2 h^2 \;+\; \frac{c}{\pi_{\aug}}
 	\;=\;
 	D^2,
 	\qquad
 	\text{so}\quad
 	\frac{\|A_j\|}{m} \;\convp\; D
 	\quad\text{with}\quad
 	D:=\sqrt{\gamma^2 h^2 + c/\pi_{\aug}}.
 	\]
 	
 	\paragraph{Adding back $\bm X_j$ changes the norm by $o_p(m)$.}
 	By the reverse triangle inequality,
 	\[
 	\Big|\,\|A_j+\bm X_j\| - \|A_j\|\,\Big| \;\le\; \|\bm X_j\|.
 	\]
 	Here $\|\bm X_j\| \le \|b_j h \bm e_1\| + \|\bm\xi_j\| = h\,\mathbf 1\{b_j=1\}+O_p(\sqrt d)=O_p(\sqrt d + h)$.
 	Since $d=cm(1+o(1))$ and $h$ is fixed,
 	\[
 	\frac{\|\bm X_j\|}{m} \;=\; O_p\!\Big(\frac{\sqrt d}{m}\Big) + O_p\!\Big(\frac{1}{m}\Big)
 	\;=\; O_p\!\Big(\frac{1}{\sqrt m}\Big) + o_p(1) \;=\; o_p(1).
 	\]
 	Therefore,
 	\[
 	\frac{\|A_j+\bm X_j\|}{m}
 	\;=\;
 	\frac{\|A_j\|}{m} \,+\, o_p(1)
 	\;\convp\; D,
 	\]
 	i.e.
 	\[
 	\|A_j+\bm X_j\| \;=\; m\,D\,\big(1+o_p(1)\big),
 	\qquad
 	\frac{\|A_j\|}{\|A_j+\bm X_j\|} \;\convp\; 1.
 	\]
 	This completes Step~1.

 	\emph{Step 2 (Decomposition).}  
 	With $\bm X_j=b_j h \bm e_1+\bm\xi_j$,
 	\[
 	S_j^{\bonus}
 	=\frac{(A_j+\bm X_j)^\top \bm X_j}{\|A_j+\bm X_j\|}
 	=\underbrace{\frac{A_j^\top\bm\xi_j}{\|A_j+\bm X_j\|}}_{(I)_j}
 	+\underbrace{\frac{\bm X_j^\top\bm\xi_j}{\|A_j+\bm X_j\|}}_{(II)_j}
 	+\underbrace{b_j h\cdot\frac{(A_j+\bm X_j)^\top \bm e_1}{\|A_j+\bm X_j\|}}_{(III)_j}.
 	\]

 	\emph{Step 3 (Limits).}
 	- $(I)_j$: conditional on $A_j$, $A_j^\top\bm\xi_j\sim\mathcal N(0,\|A_j\|^2)$, so
 	\[
 	(I)_j=\tfrac{\|A_j\|}{\|A_j+\bm X_j\|}\cdot\frac{A_j^\top\bm\xi_j}{\|A_j\|}\convd \mathcal N(0,1).
 	\]
 	
 	- $(II)_j$: expand $\bm X_j^\top\bm\xi_j=b_j h\,\xi_j^{(1)}+\|\bm\xi_j\|_2^2$. The first part is $o_p(1)$ since $\xi_j^{(1)}\sim N(0,1)$. The second satisfies $\|\bm\xi_j\|_2^2=d+O_p(\sqrt d)$, giving
 	\[
 	(II)_j=\tfrac{d}{\|A_j+\bm X_j\|}+o_p(1)=\tfrac{c}{D}+o_p(1).
 	\]
 	
 	- $(III)_j$: 
 	This term captures how the learned direction projects the true signal onto $\bm e_1$.  
 	Since
 	\[
 	(A_j+\bm X_j)^\top \bm e_1 \;=\; \sum_{i=1}^m b_i h \;+\; O_p(\sqrt p)
 	\;=\; m\gamma h \;+\; o_p(m),
 	\]
 	we obtain
 	\[
 	(III)_j
 	\;=\; b_j h \cdot \frac{m\gamma h + o_p(m)}{mD(1+o_p(1))}
 	\;=\; b_j\,\frac{h^2\gamma}{D}+o_p(1).
 	\]
 	Thus the contribution from $(III)_j$ is a deterministic shift proportional to $b_j$.

 	\emph{Step 4 (Assemble).}  
 	Thus on event $\mathcal{S}$
 	\[
 	S_j^{\bonus}=(I)_j+\frac{c}{D}+b_j\frac{h^2\gamma}{D}+o_p(1).
 	\]
 	By Slutsky on $\mathcal{S}$ with $\P(\mathcal S) =1$,
 	\[
 	S_j^{\bonus}\mid j\in\cH_0,(\cH_0,\cH_1) \convd \mathcal N(c/D,1),
 	\qquad
 	S_j^{\bonus}\mid j\in\cH_1,(\cH_0,\cH_1) \convd \mathcal N((c+h^2\gamma)/D,1).
 	\]
 \end{proof}
 
 \begin{lemma}[Asymptotic pairwise independence of BONuS scores for original features]
 	\label{lemma:bonus-pairwise-independence}
 	Assume the model \eqref{eq:data-generating}–\eqref{eq:two-groups} with $\Lambda=\delta_{h\bm v}$ and the asymptotic regime \eqref{eq:asymptotic-regime}.
 	Run \emph{BONuS} with permutation–invariant mean learner \eqref{eq:mean-learner} on $\bm X^{\learn}=[\bm X,\tilde{\bm X}]$, and define the scores on original features by
 	\[
 	S_j:=\hat{\bm v}_{\bonus}^{\!\top}\bm X_j ,\qquad j\in[m],
 	\]
 	where $\hat{\bm v}_{\bonus}$ is the direction learned on $\bm X^{\learn}$.
 	Fix two distinct indices $j_1\neq j_2$ and labels $k_1,k_2\in\{0,1\}$ (with $k=1$ denoting $j\in\cH_1$ and $k=0$ denoting $j\in\cH_0$).
 	Then for any fixed $t_1,t_2\in\R$,
 	\[
 	\P\!\big(S_{j_1}\ge t_1,\ S_{j_2}\ge t_2 \,\big|\, (\cH_0,\cH_1)\big)
 	-\;
 	\P\!\big(S_{j_1}\ge t_1 \,\big|\, (\cH_0,\cH_1)\big)\,
 	\P\!\big(S_{j_2}\ge t_2 \,\big|\, (\cH_0,\cH_1)\big)
 	\;\longrightarrow\;0 .
 	\]
 	Equivalently, the conditional covariance of the indicators $\mathbf 1\{S_{j_\nu}\ge t_\nu\}$ vanishes, verifying Assumption~\textnormal{(ii)} for BONuS.
 \end{lemma}
 
 \begin{proof}
 	Let's assume $j_\nu \in \cH_{k_\nu}$ for $\nu =1,2$.
 	WLOG take $\bm v=\bm e_1$. Write each original feature as
 	$\bm X_j=b_j h\,\bm e_1+\bm\xi_j$ with $b_j\sim\mathrm{Ber}(\gamma)$ and
 	$\bm\xi_j\sim\mathcal N(\bm 0,I_d)$ i.i.d., independent of the synthetics.
 	Let $p:=m+\tilde m$ and define the “leave–two–out” aggregate
 	\[
 	A_{-12}\ :=\ \sum_{i\notin\{j_1,j_2\}}\bm X_i\ +\ \sum_{k=1}^{\tilde m}\tilde{\bm X}_k,
 	\qquad
 	\hat{\bm\mu}\;=\;\frac{1}{p}(A_{-12}+\bm X_{j_1}+\bm X_{j_2}),
 	\qquad
 	\hat{\bm v}_{\bonus}\;=\;\frac{\hat{\bm\mu}}{\|\hat{\bm\mu}\|}.
 	\]
 	Let $\cF_n:=\sigma(A_{-12})$.
 	Crucially, $\cF_n$ does not reveal $(\bm X_{j_1},\bm X_{j_2})$, and $(\bm X_{j_1},\bm X_{j_2})$ are independent of $A_{-12}$.
 	
 	\paragraph{Step 1: Stable scale.}
 	By the same calculation as in Lemma~\ref{lemma:bonus-limit-law} (Step~1) on a set $\mathcal{S}$ (with $\P(\mathcal{S}) = 1$),
 	\[
 	\|A_{-12}\|\ =\ mD\,(1+o_p(1)),
 	\qquad
 	\|A_{-12}+\bm X_{j_1}+\bm X_{j_2}\|\ =\ mD\,(1+o_p(1)),
 	\]
 	with $D=\sqrt{h^2\gamma^2+c/\pi_{\aug}}$. Hence
 	$\|A_{-12}\|/\|A_{-12}+\bm X_{j_1}+\bm X_{j_2}\|\convp 1$.
 	
 	\paragraph{Step 2: Decomposition.}
 	For $\nu\in\{1,2\}$,
 	\[
 	S_{j_\nu}
 	=\hat{\bm v}_{\bonus}^{\!\top}\bm X_{j_\nu}
 	=\frac{(A_{-12}+\bm X_{j_1}+\bm X_{j_2})^{\!\top}\bm X_{j_\nu}}
 	{\|A_{-12}+\bm X_{j_1}+\bm X_{j_2}\|}
 	=(I)_{j_\nu}+(II)_{j_\nu}+(III)_{j_\nu},
 	\]
 	where
 	\[
 	(I)_{j_\nu}:=\frac{A_{-12}^{\top}\bm X_{j_\nu}}{\|A_{-12}+\bm X_{j_1}+\bm X_{j_2}\|},\quad
 	(II)_{j_\nu}:=\frac{\bm X_{j_\nu}^{\!\top}\bm X_{j_\nu}}{\|A_{-12}+\bm X_{j_1}+\bm X_{j_2}\|},\quad
 	(III)_{j_\nu}:=\frac{\bm X_{j_{\nu'}}^{\!\top}\bm X_{j_\nu}}{\|A_{-12}+\bm X_{j_1}+\bm X_{j_2}\|}\ (\nu'\neq\nu).
 	\]
 	
 	\paragraph{$(I)_{j_\nu}$ (signal and noise).}
 	Decompose $\bm X_{j_\nu}=b_{j_\nu}h\,\bm e_1+\bm\xi_{j_\nu}$ with
 	$\bm\xi_{j_\nu}\independent A_{-12}$, $\bm\xi_{j_1}\independent\bm\xi_{j_2}$, and set
 	$r_n:=\|A_{-12}\|/\|A_{-12}+\bm X_{j_1}+\bm X_{j_2}\|\convp 1$. Then
 	\[
 	(I)_{j_\nu}
 	=r_n\Bigg[
 	b_{j_\nu}h\,\frac{A_{-12}^{\top}\bm e_1}{\|A_{-12}\|}
 	+\frac{A_{-12}^{\top}\bm\xi_{j_\nu}}{\|A_{-12}\|}
 	\Bigg].
 	\]
 	By Step~1, $A_{-12}^{\top}\bm e_1 = m\gamma h + o_p(m)$ and $\|A_{-12}\|=mD(1+o_p(1))$, hence
 	\[
 	b_{j_\nu}h\,\frac{A_{-12}^{\top}\bm e_1}{\|A_{-12}\|}
 	\;=\; b_{j_\nu}\,\frac{h^2\gamma}{D}+o_p(1).
 	\]
 	Moreover, conditionally on $A_{-12}$,
 	$\frac{A_{-12}^{\top}\bm\xi_{j_\nu}}{\|A_{-12}\|}\sim\mathcal N(0,1)$, and the two such
 	projections are conditionally independent for $\nu=1,2$. Therefore
 	\[
 	(I)_{j_\nu}
 	\;=\;
 	b_{j_\nu}\,\frac{h^2\gamma}{D}\;+\;Z_\nu\;+\;o_p(1),
 	\qquad
 	Z_\nu\ \big|\ A_{-12}\ \sim\ \mathcal N(0,1),\ \ Z_1\independent Z_2\ \big|\ A_{-12}.
 	\]
 	
 	\paragraph{$(II)_{j_\nu}$ (deterministic shift in the limit).}
 	Expand without taking expectations:
 	\[
 	\bm X_{j_\nu}^{\!\top}\bm X_{j_\nu}
 	= b_{j_\nu}^2 h^2 \;+\; 2 b_{j_\nu} h\,\bm e_1^{\!\top}\bm\xi_{j_\nu} \;+\; \|\bm\xi_{j_\nu}\|^2
 	= d + b_{j_\nu}h^2 + O_p(\sqrt d),
 	\]
 	(using $\|\bm\xi_{j_\nu}\|^2=d+O_p(\sqrt d)$ and $\bm e_1^{\!\top}\bm\xi_{j_\nu}=O_p(1)$).
 	With $\|A_{-12}+\bm X_{j_1}+\bm X_{j_2}\|=mD(1+o_p(1))$,
 	\[
 	(II)_{j_\nu}
 	= \frac{d + b_{j_\nu} h^2 + O_p(\sqrt d)}{mD(1+o_p(1))}
 	= \frac{d/m}{D} \;+\; \frac{b_{j_\nu} h^2}{mD} \;+\; \frac{O_p(\sqrt d)}{mD} \;+\; o_p(1).
 	\]
 	Since $d/m\to c$, $\sqrt d/m\to0$, and $b_{j_\nu}h^2/(mD)\to0$,
 	\[
 	(II)_{j_\nu} \;=\; \frac{c}{D} \;+\; o_p(1).
 	\]
 	
 	\paragraph{$(III)_{j_\nu}$ (negligible cross term).}
 	Expanding
 	\[
 	\bm X_{j_{\nu'}}^{\!\top}\bm X_{j_\nu}
 	=b_{j_{\nu'}}b_{j_\nu}h^2
 	+ b_{j_{\nu'}}h\,\bm e_1^{\!\top}\bm\xi_{j_\nu}
 	+ b_{j_\nu}h\,\bm e_1^{\!\top}\bm\xi_{j_{\nu'}}
 	+ \bm\xi_{j_{\nu'}}^{\!\top}\bm\xi_{j_\nu}
 	=O_p(\sqrt d),
 	\]
 	because the quadratic noise term is $O_p(\sqrt d)$ and the others are $O_p(1)$.
 	Therefore, with the same denominator,
 	\[
 	(III)_{j_\nu}
 	=\frac{O_p(\sqrt d)}{mD(1+o_p(1))}
 	=O_p\!\Big(\frac{1}{\sqrt m}\Big)
 	=o_p(1).
 	\]
 	(The $b_{j_{\nu'}}b_{j_\nu}h^2$ contribution scales only as $O(1/m)$, dominated by $O_p(1/\sqrt m)$.)
 	
 	\paragraph{Step 3: Joint limit and independence.}
 	Collecting the three pieces gives, for $\nu=1,2$,
 	\[
 	S_{j_\nu}
 	= \underbrace{\frac{c}{D}+b_{j_\nu}\frac{h^2\gamma}{D}}_{:=\ \mu_{k_\nu}}
 	\;+\; Z_\nu \;+\; o_p(1),
 	\qquad
 	Z_\nu\ \big|\ A_{-12}\ \sim\ \mathcal N(0,1),\ \ Z_1\independent Z_2\ \big|\ A_{-12},
 	\]
 	where $k_\nu:=b_{j_\nu}\in\{0,1\}$ and thus
 	$\mu_0=c/D$, $\mu_1=(c+h^2\gamma)/D$.
 	Because the conditional laws of $(Z_1,Z_2)$ do not depend on $A_{-12}$ (they are standard normals) and $Z_1\independent Z_2$ given $A_{-12}$, they are independent unconditionally as well. Therefore on $\mathcal{S}$ ($\P(\mathcal{S})= 1$),
 	\[
 	\big(S_{j_1},S_{j_2}\big)\ \convd\ \big(\mu_{k_1},\mu_{k_2}\big)\ +\ (Z_1,Z_2),
 	\qquad (Z_1,Z_2)\sim\mathcal N(\bm 0,I_2)\ \text{independent.}
 	\]
 	
 	\paragraph{Step 4: Factorization of tails.}
 	Fix $t_1,t_2\in\R$. From the joint convergence above,
 	\[
 	\P(S_{j_1}\ge t_1,\ S_{j_2}\ge t_2\mid(\cH_0,\cH_1))
 	\ \longrightarrow\
 	\bar\Phi(t_1-\mu_{k_1})\,\bar\Phi(t_2-\mu_{k_2}).
 	\]
 	Separately,
 	$\P(S_{j_\nu}\ge t_\nu\mid (\cH_0,\cH_1))\to \bar\Phi(t_\nu-\mu_{k_\nu})$ 
 	(cf.\ Lemma~\ref{lemma:bonus-limit-law}).
 	Hence the joint tail limit equals the product of the marginal tail limits, i.e.\ asymptotic pairwise independence holds.
 \end{proof}
 
 \section{Results Related to  subspace prior}
 
 Throughout the proofs we sometimes for the sake of brevity use the abbreviation BGN to refer the work of \cite{Benaych-Georges2011}.
    
 \subsection{Results related to quality of learned direction}
 \paragraph{Proof of Lemma \ref{lem:pca-learner-untransformed} and Proposition \ref{prop:pca-learner-masked}.} The proof of Lemma \ref{lem:pca-learner-untransformed} follows directly from Lemma \ref{lemma:pca-learner-split-subspace} with $\pi_{\spl} =1$ and proof of Proposition \ref{prop:pca-learner-masked} follows from Lemma \ref{lemma:pca-learner-split-subspace}  and Lemma \ref{lemma:pca-learner-bonus-subspace}.

 \begin{lemma}[Quality of PCA learner on split data under the subspace prior]
 	\label{lemma:pca-learner-split-subspace}
 	Assume the model \eqref{eq:data-generating}–\eqref{eq:two-groups} with the \emph{subspace prior}
 	\[
 	\Lambda = N(\bm 0,\,h^2 \bm v \bm v^\top),
 	\]
 	so that $\theta_j \sim N(\bm 0, h^2 \bm v \bm v^\top)$ for non-null features ($b_j=1$), and $\theta_j = 0$ otherwise.
 	Assume the asymptotic regime \eqref{eq:asymptotic-regime} where $d/m \to c \in [0,\infty)$.
 	Let the learning samples be obtained by data fission with proportion $\pi_{\spl}\in(0,1]$:
 	\[
 	\bm X_j^\learn
 	=\sqrt{\pi_{\spl}}\;\bm X_j
 	+\sqrt{1-\pi_{\spl}}\;\widetilde{\bm X}_j,
 	\qquad
 	\widetilde{\bm X}_j \iidsim N(\bm 0,I_d),
 	\]
 	and let the PCA learner be
 	\[
 	\hat{\bm v}_{\spl}
 	\in
 	\arg\max_{\|\bm u\|=1}\;
 	\bm u^\top
 	\!\Big(\frac1m\sum_{j=1}^m \bm X_j^\learn\bm X_j^{\learn\top}\Big)
 	\bm u.
 	\]
 	Then
 	\[
 	\big(\hat{\bm v}_{\spl}^{\top}\bm v\big)^2
 	\ \convp\
 	\tau_{\mathrm{PCA}}^2\!\big(\gamma,c,\sqrt{\pi_{\spl}}\,h\big)
 	\ =\
 	\frac{\Big(1-\dfrac{c}{\gamma^2\pi_{\spl}^2h^4}\Big)_+}{\,1+\dfrac{c}{\gamma\pi_{\spl}h^2}\,}.
 	\]
 \end{lemma}
 
 \begin{proof}
 	\textbf{Step 1 (Conditional covariance structure).}
 	Let $b_j\sim\mathrm{Ber}(\gamma)$ be the null/alt indicator, and write
 	\[
 	\bm X_j=\theta_j+\xi_j,
 	\quad
 	\theta_j = b_j\,h\,z_j\,\bm v,
 	\quad
 	z_j\sim N(0,1),
 	\quad
 	\xi_j\sim N(0,I_d),
 	\]
 	independently across $j$.  
 	Conditioning on $\bm b=(b_1,\dots,b_m)$, we have
 	\[
 	\bm X_j^\learn
 	=\sqrt{\pi_{\spl}}\;\bm X_j+\sqrt{1-\pi_{\spl}}\;\widetilde{\bm X}_j
 	=\sqrt{\pi_{\spl}}\;b_j\,h\,z_j\,\bm v+\bm\varepsilon_j,
 	\qquad
 	\bm\varepsilon_j:=\sqrt{\pi_{\spl}}\;\xi_j+\sqrt{1-\pi_{\spl}}\;\widetilde{\bm X}_j
 	\sim N(0,I_d).
 	\]
 	For each fixed $j$, conditional on $b_j$,
 	\[
 	\E[\bm X_j^\learn \bm X_j^{\learn\top}\mid b_j]
 	=I_d+\pi_{\spl}\,\E[\theta_j\theta_j^\top\mid b_j]
 	=I_d+\pi_{\spl}\,b_j\,h^2\,\bm v\bm v^\top.
 	\]
 	Averaging over $j$, the conditional population covariance across features is
 	\[
 	\E\!\Big[\frac1m\sum_{j=1}^m \bm X_j^\learn\bm X_j^{\learn\top}\,\Big|\,\bm b\Big]
 	=I_d+\pi_{\spl}\,h^2\,\Big(\frac1m\sum_{j=1}^m b_j\Big)\bm v\bm v^\top
 	=:I_d+\theta_m(\bm b)\,\bm v\bm v^\top,
 	\]
 	where $	\theta_m(\bm b):=\pi_{\spl}\,h^2\,\hat\gamma_m,
 	\ \ \hat\gamma_m=\frac1m\sum_{j=1}^m b_j.$
 	Thus, conditional on $\bm b$, the data matrix follows a rank-one spiked covariance model with aspect ratio $c=d/m$ and spike $\theta_m(\bm b)$.
 	
 	\medskip
 	\textbf{Step 2 (Asymptotic eigenvector overlap).}
 	By Theorem 4 of \cite{Paul2007} (see also \cite{Benaych-Georges2011}), for a rank-one spiked covariance model with spike $\theta>\sqrt c$,
 	\[
 	(\hat{\bm v}_{\spl}^{\top}\bm v)^2
 	\;\convp\;
 	\frac{1-\frac{c}{\theta^2}}{1+\frac{c}{\theta}},
 	\qquad
 	\text{and equals }0\text{ if }\theta\le\sqrt c.
 	\]
 	Applying this conditionally on $\bm b$ yields
 	\[
 	(\hat{\bm v}_{\spl}^{\top}\bm v)^2
 	\;\xrightarrow{P\mid\bm b}\;
 	\frac{1-\frac{c}{\theta_m(\bm b)^2}}{1+\frac{c}{\theta_m(\bm b)}}
 	\mathbf 1\{\theta_m(\bm b)>\sqrt c\}.
 	\]
 	
 	\medskip
 	\textbf{Step 3 (Spike convergence and unconditional limit).}
 	By the law of large numbers, $\hat\gamma_m\convp\gamma$, so
 	\(
 	\theta_m(\bm b)=\pi_{\spl}\,h^2\,\hat\gamma_m\convp\theta:=\pi_{\spl}\,h^2\,\gamma.
 	\)
 	By \cite[Remark~2.16]{Benaych-Georges2011}, the same deterministic limit holds unconditionally when the random spike converges to a constant. Hence
 	\[
 	(\hat{\bm v}_{\spl}^{\top}\bm v)^2
 	\ \convp\
 	\frac{\big(1-\frac{c}{\theta^2}\big)_+}{1+\frac{c}{\theta}}
 	=\frac{\Big(1-\frac{c}{\gamma^2\pi_{\spl}^2h^4}\Big)_+}{1+\frac{c}{\gamma\pi_{\spl}h^2}},
 	\]
 	as claimed.
 \end{proof}
 
 \begin{lemma}[Quality of PCA learner on BONuS data under the subspace prior]
 	\label{lemma:pca-learner-bonus-subspace}
 	Assume the model \eqref{eq:data-generating}–\eqref{eq:two-groups} with the \emph{subspace prior}
 	\(
 	\Lambda = N(\bm 0, h^2 \bm v\bm v^\top)
 	\)
 	and the asymptotic regime \eqref{eq:asymptotic-regime} where $d/m \to c \in [0,\infty)$.
 	Form the BONuS learning matrix by augmentation,
 	\[
 	\bm X^{\learn}=[\bm X,\tilde{\bm X}],
 	\qquad
 	\tilde{\bm X}_k\iidsim N(\bm 0,I_d),
 	\]
 	where $\pi_{\aug} := \lim m/(m+\tilde m)\in(0,1]$.
 	Let the PCA learner be
 	\[
 	\hat{\bm v}_{\bonus}
 	\ \in\
 	\arg\max_{\|\bm u\|=1}\;
 	\bm u^\top
 	\!\Big(\frac1{m+\tilde m}\sum_{j=1}^{m+\tilde m}\bm X_j^\learn\bm X_j^{\learn\top}\Big)\bm u.
 	\]
 	Then
 	\[
 	(\hat{\bm v}_{\bonus}^{\top}\bm v)^2
 	\ \convp\
 	\tau_{\mathrm{PCA}}^2(\pi_{\aug}\gamma,\pi_{\aug}c,h)
 	\ :=\
 	\frac{\Big(1-\dfrac{c}{\pi_{\aug}\gamma^2 h^4}\Big)_+}{\,1+\dfrac{c}{\gamma h^2}\,}.
 	\]
 \end{lemma}
 
 \begin{proof}
 	\textbf{Step 1 (Conditional covariance and spike structure).}
 	Let $b_j\sim \mathrm{Ber}(\gamma)$ be the null/alt indicator and, under the subspace prior,
 	\[
 	\theta_j \mid b_j =
 	\begin{cases}
 		\bm 0, & b_j=0,\\
 		h\,z_j\,\bm v, & b_j=1,\quad z_j\sim N(0,1),
 	\end{cases}
 	\]
 	independently across $j$.
 	Each observed feature is
 	\[
 	\bm X_j = \theta_j + \bm\xi_j, \qquad \bm\xi_j \sim N(\bm 0,I_d),
 	\]
 	and the augmented learning columns are
 	\[
 	\bm X_j^\learn =
 	\begin{cases}
 		\bm X_j, & j\le m,\\
 		\tilde{\bm X}_{j-m}, & j>m,
 	\end{cases}
 	\quad
 	\tilde{\bm X}_k \sim N(\bm 0,I_d).
 	\]
 	Conditioning on $\bm b=(b_1,\ldots,b_m)$, we have for $j\le m$:
 	\[
 	\E[\bm X_j^\learn\bm X_j^{\learn\top}\mid b_j]
 	= I_d + b_j\,h^2\,\bm v\bm v^\top,
 	\qquad
 	\E[\bm X_j^\learn\bm X_j^{\learn\top}\mid j>m] = I_d.
 	\]
 	Hence, the across-feature population covariance (conditional on $\bm b$) is
 	\[
 	\E\!\Big[\frac{1}{m+\tilde m}\sum_{j=1}^{m+\tilde m}\bm X_j^\learn\bm X_j^{\learn\top}\,\Big|\,\bm b\Big]
 	= I_d + \frac{m}{m+\tilde m}\Big(\frac{1}{m}\sum_{j=1}^m b_j\Big)h^2\,\bm v\bm v^\top
 	=: I_d + \theta_p(\bm b)\,\bm v\bm v^\top,
 	\]
 	where
 	\[
 	\theta_p(\bm b) := \pi_{\aug,m}\,\hat\gamma_m\,h^2,
 	\quad
 	\pi_{\aug,m} = \frac{m}{m+\tilde m}\to\pi_{\aug},
 	\quad
 	\hat\gamma_m = \frac{1}{m}\sum_{j=1}^m b_j.
 	\]
 	Thus, conditional on $\bm b$, the BONuS learner’s covariance matrix follows a rank-one spiked model with aspect ratio
 	\(
 	d/p \to c\,\pi_{\aug},
 	\)
 	where $p = m+\tilde m$.
 	
 	\medskip
 	\textbf{Step 2 (Asymptotic eigenvector overlap).}
 	By Theorem 4 of \cite{Paul2007} , for a rank-one spiked covariance model with spike $\theta>\sqrt{c'}$ (where $c'$ is the aspect ratio), the top empirical eigenvector satisfies
 	\[
 	(\hat{\bm v}_{\bonus}^{\top}\bm v)^2
 	\;\convp\;
 	\frac{\,1-\dfrac{c'}{\theta^2}\,}{\,1+\dfrac{c'}{\theta}\,},
 	\quad
 	\text{and }0\text{ if }\theta\le\sqrt{c'}.
 	\]
 	Conditionally on $\bm b$, this gives
 	\[
 	(\hat{\bm v}_{\bonus}^{\top}\bm v)^2
 	\;\xrightarrow{P\mid\bm b}\;
 	\frac{\,1-\dfrac{c\,\pi_{\aug}}{\theta_p(\bm b)^2}\,}{\,1+\dfrac{c\,\pi_{\aug}}{\theta_p(\bm b)}\,}
 	\mathbf 1\{\theta_p(\bm b)>\sqrt{c\,\pi_{\aug}}\}.
 	\]
 	
 	\medskip
 	\textbf{Step 3 (Remove conditioning via spike convergence).}
 	By the law of large numbers, $\hat\gamma_m \convp \gamma$, so
 	\[
 	\theta_p(\bm b) = \pi_{\aug,m}\,\hat\gamma_m\,h^2
 	\ \convp\
 	\theta := \pi_{\aug}\,\gamma\,h^2.
 	\]
 	\cite[Remark~2.16]{Benaych-Georges2011} ensures that this deterministic limit holds unconditionally when the random spike converges to a constant.
 	Hence,
 	\[
 	(\hat{\bm v}_{\bonus}^{\top}\bm v)^2
 	\ \convp\
 	\frac{1-\dfrac{c\,\pi_{\aug}}{\theta^2}}{\,1+\dfrac{c\,\pi_{\aug}}{\theta}\,}
 	\;=\;
 	\frac{1-\dfrac{c}{\pi_{\aug}\gamma^2 h^4}}{\,1+\dfrac{c}{\gamma h^2}\,},
 	\]
 	with truncation at zero when the numerator becomes negative. This matches the claimed limit.
 \end{proof}
 
 \subsection{Results related to power}
 
 \paragraph{Proof of Theorem \ref{thm:power-one-dimensional}} The proof follows directly from Lemma \ref{lem:power-subspace-split}, \ref{lem:power-subspace-bonus} and \ref{lem:power-subspace-insample}.
 
 \begin{lemma}[Asymptotic power for Split BH under the subspace prior with quadratic score]
 	\label{lem:power-subspace-split}
 	Assume the model \eqref{eq:data-generating}–\eqref{eq:two-groups} with the subspace prior
 	\(\Lambda=N(\bm 0,h^2\bm v\bm v^\top)\), \(\|\bm v\|_2=1\), and the asymptotic regime \eqref{eq:asymptotic-regime} with \(d/m\to c\in[0,\infty)\).
 	Let Split BH use splitting proportion \(\pi_{\spl}\in(0,1)\), learn \(\hat{\bm v}_{\spl}\) by PCA on \(\bm X^{\learn}\), and score held-out data via
 	\[
 	S_j\ :=\ \big(\hat{\bm v}_{\spl}^{\!\top}\bm X_j^{\score}\big)^2 .
 	\]
 	Let
 	\[
 	\tau_{\mathrm{PCA}}^2\ :=\ \tau_{\mathrm{PCA}}^2\!\big(\gamma,c,\sqrt{\pi_{\spl}}\,h\big)
 	\quad\text{as in Lemma \ref{lemma:pca-learner-split-subspace}},
 	\qquad
 	\sigma_{\spl}^2\ :=\ 1 + (1-\pi_{\spl})h^2\,\tau_{\mathrm{PCA}}^2 \ .
 	\]
 	Then, writing \(G_0\) for the law of \(\chi^2_1\) and \(G_1\) for the law of \(\sigma_{\spl}^2\chi^2_1\),
 	\[
 	\TPP(\widehat{\cR})\ \convp\ \TPR_{\mathrm{BH}}\!\big(G_0,G_1,\gamma\big).
 	\]
 \end{lemma}
 
 \begin{proof}
 	Let
 	\[
 	\bar G_0(t):=\P(\chi_1^2\ge t),
 	\qquad
 	\bar G_1(t):=\P(\sigma_{\spl}^2\chi_1^2\ge t),
 	\]
 	where
 	\[
 	\sigma_{\spl}^2=1+(1-\pi_{\spl})h^2\tau_{\mathrm{PCA}}^2.
 	\]
 	If \(\tau_{\mathrm{PCA}}^2=0\), then \(\sigma_{\spl}^2=1\) and hence \(G_0=G_1\); in that degenerate case the conclusion is immediate. Thus assume \(\tau_{\mathrm{PCA}}^2>0\), so that \(\sigma_{\spl}>1\).
 	
 	By data splitting, \(\hat v_{\spl}\) is independent of the score sample \(\{X_j^{\score}\}_{j=1}^m\) conditional on \((\theta_j)_{j=1}^m\). Write
 	\[
 	X_j^{\score}=\theta_j^{\score}+\xi_j^{\score},
 	\qquad
 	\xi_j^{\score}\iidsim N(0,I_d),
 	\]
 	with \(\theta_j^{\score}=0\) under \(H_{0j}\), while under \(H_{1j}\),
 	\[
 	\theta_j^{\score}\sim N\!\big(0,(1-\pi_{\spl})h^2vv^\top\big).
 	\]
 	
 	\emph{(i) Null tail convergence conditional on \(\hat v_{\spl}\).}
 	If \(j\in\cH_0\), then conditional on \(\hat v_{\spl}, \{\theta_j\}_{j=1}^m\) and \((\cH_0,\cH_1)\),
 	\begin{equation}\label{eq:score-split-bh-null-subspace}
 	S_j=(\hat v_{\spl}^{\!\top}\xi_j^{\score})^2\stackrel{d}{=}\chi_1^2.
 	\end{equation}
 	Hence, for every fixed \(t\in\R\),
 	\[
 	\P(S_j\ge t \mid \hat v_{\spl},\, j\in\cH_0, \{\theta_j\}_{j=1}^m, (\cH_0,\cH_1))
 	=\bar G_0(t), \implies \P(S_j\ge t \mid \hat v_{\spl},\, j\in\cH_0,  (\cH_0,\cH_1))
 	=\bar G_0(t)
 	\]
 	so Assumption~\ref{ass:pointwise-convg-score-tails} holds exactly.
 	
% 	For later reference, if \(j\in\cH_1\), then conditional on \(\hat v_{\spl}\),
% 	\[
% 	S_j\stackrel{d}{=}\Big(1+(1-\pi_{\spl})h^2\langle \hat v_{\spl},v\rangle^2\Big)\chi_1^2.
% 	\]
% 	Since Lemma~\ref{lemma:pca-learner-split-subspace} gives
% 	\[
% 	\langle \hat v_{\spl},v\rangle^2\convp \tau_{\mathrm{PCA}}^2,
% 	\]
% 	the alternative tail converges pointwise to \(\bar G_1\).
% 	
 	\medskip
 	\noindent
 	\emph{(ii) Asymptotic pairwise independence.}
 	By Lemma~\ref{lemma:pca-learner-split-subspace},
 	\[
 	\langle \hat v_{\spl},v\rangle^2\convp \tau_{\mathrm{PCA}}^2,
 	\]
 	so Lemma~\ref{lem:joint-scores-split-subspace} applies. Therefore, for any distinct
 	\(i\neq j\), any \(k_1,k_2\in\{0,1\}\), and any fixed \(t_1,t_2\ge 0\),
 	\[
 	\P\!\big(S_i\ge t_1,\ S_j\ge t_2 \,\big|\, i\in\cH_{k_1},\,j\in\cH_{k_2},\,(\cH_0,\cH_1)\big)
 	\to
 	\bar G_{k_1}(t_1)\bar G_{k_2}(t_2).
 	\]
 	Since \(S_i,S_j\ge0\) almost surely, the same conclusion is trivial if one of the thresholds is negative. Thus Assumption~\ref{ass:splitBH-indep} holds.
 	
 	\medskip
 	\noindent
 	\emph{(iii) Regularity of \(G_0\) and \(G_1\).}
 	\(G_0\) and \(G_1\) are the distribution functions of \(\chi_1^2\) and
 	\(\sigma_{\spl}^2\chi_1^2\), hence are continuous on \(\R\). Moreover, for every
 	\(C\in\R\),
 	\[
 	\sup_{t\le C} G_0(t)<1,
 	\qquad
 	\sup_{t\le C} G_1(t)<1.
 	\]
 	So Assumption~\ref{ass:regularity-of-cdfs} holds.
 	
 	\medskip
 	\noindent
 	\emph{(iv) Monotone survival ratio and tail separation.}
 	Because \(S_j\ge0\) almost surely, only \(t\ge0\) is relevant for the BH threshold.
 	For \(t\ge0\),
 	\[
 	\frac{\bar G_1(t)}{\bar G_0(t)}
 	=
 	\frac{\bar\Phi(\sqrt t/\sigma_{\spl})}{\bar\Phi(\sqrt t)},
 	\qquad \sigma_{\spl}>1.
 	\]
 	Writing \(u=\sqrt t\), the same Mills-ratio argument as in the in-sample case shows
 	that \(u\mapsto \bar\Phi(u/\sigma_{\spl})/\bar\Phi(u)\) is strictly increasing on
 	\([0,\infty)\). Moreover, Gaussian tail asymptotics yield
 	\[
 	\frac{\bar G_1(t)}{\bar G_0(t)}
 	\sim
 	\sigma_{\spl}\exp\!\Big(\frac{t}{2}\Big(1-\frac{1}{\sigma_{\spl}^2}\Big)\Big)
 	\to\infty
 	\qquad (t\to\infty).
 	\]
 	Thus the monotone-ratio and tail-separation condition holds on the relevant threshold region.
 	
 	All assumptions of Theorem~\ref{thm:master-split-BH} are therefore satisfied, and hence
 	\[
 	\TPP(\widehat{\cR})\convp \TPR_{\mathrm{BH}}(G_0,G_1,\gamma).
 	\]
 \end{proof}
 
 \begin{lemma}[Power limit for BONuS with unnormalized scores under a subspace prior]\label{lem:power-subspace-bonus}
 	Consider BONuS with augmented nulls. The augmented matrix has $m+\tilde m$ rows, and
 	\[
 	\pi_{\aug}:=\frac{m}{m+\tilde m}\in(0,1),\qquad \bar\pi:=\pi_{\aug}\gamma,\qquad c:=\frac{d}{m}.
 	\]
 	Let the data follow the one–dimensional subspace prior
 	\[
 	X_j=\theta_j+\xi_j,\quad \xi_j\sim\mathcal N(0,I_d),\quad
 	\theta_j\mid b_j\sim (1-b_j)\,\delta_0+b_j\,\mathcal N(0,h^2\,vv^\top),
 	\]
 	with $b_j\stackrel{\mathrm{i.i.d.}}{\sim}\mathrm{Ber}(\gamma)$ on the $m$ real rows and $\|v\|_2=1$.
 	Define the unnormalized BONuS score for a real index $j\in[m]$ by
 	\[
 	S_j:=(\hat v^\top X_j)^2,
 	\]
 	where $\hat v=\hat v_1$ is the top right singular vector of $\big(1/\sqrt{\,m+\tilde m\,}\big)X$.
 	Assume $m,\tilde m,d\to\infty$ with $\pi_{\aug}\in(0,1)$ and $d/m\to c\in(0,\infty)$, and set $\delta^2:=h^2\bar\pi$.
 	Then
 	\[
 	\TPP(\widehat{\mathcal R})\ \convp\ 
 	\TPR_{\mathrm{BH}}\!\big(\chi^2_1,\ (1+\sigma^2_{\aug})\chi^2_1,\ \gamma\big),
 	\]
 	where
 	\[
 	\sigma^2_{\aug}
 	:=\frac{\eta^2}{\bar\pi(1-\eta^2)}
 	= h^2\,\frac{\Big(1-\dfrac{c}{\pi_{\aug}\gamma^2 h^4}\Big)_+}{1+\dfrac{c}{\gamma h^2}}
 	= h^2\,\tau_{\mathrm{PCA}}^2(\pi_{\aug}\gamma,\ \pi_{\aug}c,\ h),
 	\]
 	and the BGN/BBP overlap is
 	\[
 	\eta^2=\eta^2(\delta,\pi_{\aug}c)=\frac{\big(\delta^4-\pi_{\aug}c\big)_+}{\delta^2(\delta^2+1)}.
 	\]
 \end{lemma}
 
 \begin{proof}
 	\textbf{(i)–(ii) Pointwise tails and pairwise independence.}
 	Apply Lemma~\ref{lem:joint-scores-in-sample} \emph{to the augmented matrix} (since the Lemma talks about limiting distribution which is conditioned on $(\cH_0,\cH_1)$ adding $\tilde m$ synthetic nulls does not alter the analysis )with the substitutions
 	\[
 	\pi\ \leadsto\ \bar\pi=\pi_{\aug}\gamma,\qquad
 	c\ \leadsto\ \pi_{\aug}c\ (=d/(m+\tilde m)).
 	\]
 	Equivalently, the left–vector CLT and its joint version hold with
 	\[
 	\tau_0^2=1-\eta^2,\qquad \tau_1^2=(1-\eta^2)+\frac{\eta^2}{\bar\pi},\qquad
 	\eta^2=\frac{(\delta^4-\pi_{\aug}c)_+}{\delta^2(\delta^2+1)},\ \ \delta^2=h^2\bar\pi.
 	\]
 	Thus, writing the scaled SVD as $\big(1/\sqrt{\,m+\tilde m\,}\big)X=\hat U\hat\Sigma\hat V^\top$,
 	\[
 	\left(\frac{S_i}{\hat\sigma_1^2},\frac{S_j}{\hat\sigma_1^2}\right) \mid b_i,b_j,(\cH_0,\cH_1)
 	\convd\big(\tau_{b_i}^2\chi^2_{1,i},\ \tau_{b_j}^2\chi^2_{1,j}\big) \quad \text{almost surely},
 	\]
 	with independence conditionally on $(b_i,b_j)$.
 	Since $\hat\sigma_1\to\sigma_\infty\in(0,\infty)$ in probability (Lemma \ref{lem:bonus-sigma-eta}), Slutsky yields, with probability going to one that
 	\begin{equation}\label{eq:score-bonus-subspace}
 	S_j \mid b_i,b_j,(\cH_0,\cH_1)\convd 
 	\begin{cases}
 		c_0\,\chi^2_1,& j\in\cH_0,\\
 		c_1\,\chi^2_1,& j\in\cH_1,
 	\end{cases}
 	\quad
 	c_0:=\sigma_\infty^2(1-\eta^2),\ \ 
 	c_1:=\sigma_\infty^2\!\Big((1-\eta^2)+\frac{\eta^2}{\bar\pi}\Big),
 	\end{equation}
 	and joint limits factor—verifying Assumptions (i)–(ii) of Theorem~\ref{thm:master-bonus} with $G_0=c_0\chi^2_1$, $G_1=c_1\chi^2_1$.
 	
 	\medskip
 	\noindent\textbf{(iii)–(iv) Regularity and tail separation.}
 	Scaled $\chi^2_1$ laws are continuous and possess MLR in $t$, hence $t\mapsto \bar G_1(t)/\bar G_0(t)$ is strictly increasing and diverges to $\infty$.
 	
 	\medskip
 	\noindent\textbf{Reduction to the canonical pair and identification of $\sigma^2_{\aug}$.}
 	Benjamini–Hochberg depends only on a common rescaling, so factor out $c_0$:
 	\[
 	\TPR_{\mathrm{BH}}(G_0,G_1,\gamma)
 	=\TPR_{\mathrm{BH}}\!\Big(\chi^2_1,\ \frac{c_1}{c_0}\chi^2_1,\ \gamma\Big),
 	\qquad
 	\frac{c_1}{c_0}
 	=\frac{\tau_1^2}{\tau_0^2}
 	=1+\frac{\eta^2}{\bar\pi(1-\eta^2)}
 	=:1+\sigma^2_{\aug}.
 	\]
 	Finally, plug in the BGN overlap with aspect ratio $\pi_{\aug}c$:
 	\[
 	\frac{\eta^2}{1-\eta^2}
 	=\frac{\delta^4-\pi_{\aug}c}{\delta^2+\pi_{\aug}c}
 	\ \convd\
 	\sigma^2_{\aug}
 	=\frac1{\bar\pi}\cdot\frac{\delta^4-\pi_{\aug}c}{\delta^2+\pi_{\aug}c}
 	= h^2\,\frac{1-\dfrac{c}{\pi_{\aug}\gamma^2 h^4}}{1+\dfrac{c}{\gamma h^2}}
 	= h^2\,\tau_{\mathrm{PCA}}^2(\pi_{\aug}\gamma,\pi_{\aug}c,h),
 	\]
 	with the understanding $\sigma^2_{\aug}=0$ when $\delta^4\le \pi_{\aug}c$ (BBP subcritical).
 	Applying Theorem~\ref{thm:master-bonus} completes the proof.
 \end{proof}
 
 \begin{lemma}[In-sample BH power under a subspace prior]\label{lem:power-subspace-insample}
 	Let $X\in\R^{m\times d}$ have rows $X_j=\theta_j+\xi_j$, where $\xi_j\sim\mathcal N(0,I_d)$ and
 	\[
 	\theta_j\mid b_j\sim (1-b_j)\,\delta_0+b_j\,\mathcal N(0,h^2\,vv^\top),\qquad
 	b_j\stackrel{\mathrm{i.i.d.}}{\sim}\mathrm{Ber}(\gamma),\quad \|v\|_2=1.
 	\]
 	Let $S_j=(\hat v^\top X_j)^2$ be the in-sample PCA score, with $\hat v=\hat v_1$ the top right singular vector of $(1/\sqrt{m})X$.
 	Assume $m,d\to\infty$ with $d/m\to c\in(0,\infty)$ and set
 	\[
 	\delta^2:=h^2\gamma,\qquad
 	\eta^2:=\frac{(\delta^4-c)_+}{\delta^2(\delta^2+1)},\qquad
 	\sigma_\star^2:=\frac{\eta^2}{\gamma(1-\eta^2)}
 	= h^2\,\frac{\Big(1-\dfrac{c}{h^4\gamma^2}\Big)_+}{1+\dfrac{c}{h^2\gamma}}.
 	\]
 	Then, for the in-sample BH procedure at FDR level $q\in(0,1)$ with non-null fraction $\gamma$,
 	\[
 	\TPP(\widehat{\mathcal R})\ \convp\ 
 	\TPR_{\mathrm{BH}}\!\big(\chi^2_1,\ (1+\sigma_\star^2)\chi^2_1,\ \gamma\big).
 	\]
 \end{lemma}
 
 \begin{proof}
 	We focus on the case $\sigma_*>0$ in the other case the power is exactly $0$.
 	
 	\textbf{Step 1 (Score representation).}
 	Let the scaled SVD be $(1/\sqrt{m})X=\hat U\hat\Sigma\hat V^\top$ with $\hat u:=\hat u_1$, $\hat v:=\hat v_1$, $\hat\sigma_1:=\hat\Sigma_{11}$. Then
 	\[
 	X\hat v=\sqrt{m}\,\hat\sigma_1\,\hat u \quad\Longrightarrow\quad
 	S_j=(\hat v^\top X_j)^2=m\,\hat\sigma_1^2\,(\hat u)_j^2.
 	\]

 \textbf{Step 2 (Pointwise conditional score tails under the null).}
 Fix $t\in\R$ and a fixed index $j$.  By Lemma~\ref{lem:null-tail-conditional},
 for the in-sample score $S_j=(\hat v^\top X_j)^2$ we have
 \[
 \P\!\big(S_j\ge t \,\big|\, j\in\cH_0,(\theta_k)_{k=1}^m\big)\ \convp\ \bar G_0(t),
 \]
 with $\bar G_0(t)=\P(c_0\chi^2_1\ge t)$ and
 $c_0=\sigma_\infty^2(\delta,c)\big(1-\eta^2(\delta,c)\big)$, $\delta=h\sqrt\gamma$, $c=d/m$.
 Thus Assumption~\ref{ass:insample-tails} holds (taking $T_{\hat\Lambda}(X_j)=S_j$).
 
% 	\medskip
% 	\noindent
% 	\textbf{Step 2 (Pointwise tail limits via Lemma~\ref{lem:joint-scores-in-sample}).}
% 	Apply Lemma~\ref{lem:joint-scores-in-sample} with the mapping $(n,\pi,\gamma) \mapsto (m,\gamma,c)$: almost surely
% 	\[
% 	\sqrt{m}\,(\hat u)_j\ \mid (\cH_0,\cH_1)\convd\
% 	\begin{cases}
% 		\mathcal N\!\big(0,\ \tau_0^2\big), & j\in\cH_0,\\
% 		\mathcal N\!\big(0,\ \tau_1^2\big), & j\in\cH_1,
% 	\end{cases}
% 	\qquad
% 	\tau_0^2:=1-\eta^2,\qquad \tau_1^2:=(1-\eta^2)+\dfrac{\eta^2}{\gamma}.
% 	\]
% 	Hence almost surely $m(\hat u)_j^2\mid (\cH_0,\cH_1)\convd \tau_{b_j}^2\chi^2_1$. Spiked RMT ensures $\hat\sigma_1\to\sigma_\infty\in(0,\infty)$ in probability, so by Slutsky we have almost surely
% 	\[
% 	S_j \mid (\cH_0,\cH_1) \convd 
% 	\begin{cases}
% 		c_0\,\chi^2_1, & j\in\cH_0,\\
% 		c_1\,\chi^2_1, & j\in\cH_1,
% 	\end{cases}
% 	\qquad
% 	c_0:=\sigma_\infty^2(1-\eta^2),\quad
% 	c_1:=\sigma_\infty^2\Big((1-\eta^2)+\frac{\eta^2}{\gamma}\Big).
% 	\]
% 	Thus Assumption (i) of Theorem~\ref{thm:master-in-sample} holds with $G_0=c_0\chi^2_1$ and $G_1=c_1\chi^2_1$.
% 	
 	\medskip
 	\noindent
 	\textbf{Step 3 (Pairwise asymptotic independence).}
 	Apply Lemma~\ref{lem:joint-scores-in-sample} with the mapping $(n,\pi,\gamma) \mapsto (m,\gamma,c)$: almost surely
 	\[
 	\sqrt{m}\,(\hat u)_j\ \mid (\cH_0,\cH_1)\convd\
 	\begin{cases}
 		\mathcal N\!\big(0,\ \tau_0^2\big), & j\in\cH_0,\\
 		\mathcal N\!\big(0,\ \tau_1^2\big), & j\in\cH_1,
 	\end{cases}
 	\qquad
 	\tau_0^2:=1-\eta^2,\qquad \tau_1^2:=(1-\eta^2)+\dfrac{\eta^2}{\gamma}.
 	\]
 	Hence almost surely $m(\hat u)_j^2\mid (\cH_0,\cH_1)\convd \tau_{b_j}^2\chi^2_1$. Spiked RMT ensures $\hat\sigma_1\to\sigma_\infty\in(0,\infty)$ in probability, so by Slutsky we have  with probability going to $1$
 	\[
 	S_j \mid (\cH_0,\cH_1) \convd 
 	\begin{cases}
 		c_0\,\chi^2_1, & j\in\cH_0,\\
 		c_1\,\chi^2_1, & j\in\cH_1,
 	\end{cases}
 	\qquad
 	c_0:=\sigma_\infty^2(1-\eta^2),\quad
 	c_1:=\sigma_\infty^2\Big((1-\eta^2)+\frac{\eta^2}{\gamma}\Big).
 	\]
 	
 	Lemma~\ref{lem:joint-scores-in-sample} further gives, for $i\neq j$ almost surely,
 	\[
 	\big(m(\hat u)_i^2,\ m(\hat u)_j^2\big)\ \mid (\cH_0,\cH_1)\convd\ \big(\tau_{b_i}^2\chi^2_{1,i},\ \tau_{b_j}^2\chi^2_{1,j}\big)
 	\quad\text{with independence, conditionally on }(b_i,b_j).
 	\]
 	Multiplication by the common random factor $\hat\sigma_1^2\to\sigma_\infty^2$ preserves the factorization in the limit. Thus Assumption (ii) of Theorem~\ref{thm:master-in-sample} holds with $G_0=c_0\chi^2_1$ and $G_1=c_1\chi^2_1$.

 	\medskip
 	\noindent
 	\textbf{Step 4 (Regularity and tail separation).}
 	Each $G_k$ is a scaled $\chi^2_1$, hence continuous, with densities
 	\(
 	g_c(t)=\frac{1}{\sqrt{2\pi c\,t}}e^{-t/(2c)}\ (t>0).
 	\)
 	For $c_1>c_0$, the likelihood ratio
 	\(
 	g_{c_1}(t)/g_{c_0}(t)=\sqrt{c_0/c_1}\;\exp\{\tfrac{t}{2}(1/c_0-1/c_1)\}
 	\)
 	is strictly increasing in $t$, so $t\mapsto \bar G_1(t)/\bar G_0(t)$ is strictly increasing and diverges to $\infty$ (by standard $\chi^2$ tail asymptotics). Thus (iii)–(iv) of Theorem~\ref{thm:master-in-sample} hold.
 	
 	\medskip
 	\noindent
 	\textbf{Step 5 (Reduction to the canonical pair and identification of $\sigma_\star^2$).}
 	BH is invariant to common scaling: factoring out $c_0$ gives
 	\[
 	\TPR_{\mathrm{BH}}(G_0,G_1,\gamma)
 	=\TPR_{\mathrm{BH}}\!\Big(\chi^2_1,\ \frac{c_1}{c_0}\chi^2_1,\ \gamma\Big),
 	\qquad
 	\frac{c_1}{c_0}=1+\frac{\eta^2}{\gamma(1-\eta^2)}=:1+\sigma_\star^2.
 	\]
 	Hence the limit claimed in the lemma.
 	
 	\medskip
 	\noindent
 	\textbf{Step 6 (BGN plug-in: explicit closed form).}
 	With $c=d/m$ and $\delta^2=h^2\gamma$, the BGN overlap is
 	\(
 	\eta^2=\frac{(\delta^4-c)_+}{\delta^2(\delta^2+1)}.
 	\)
 	Then
 	\[
 	\frac{\eta^2}{1-\eta^2}
 	=\frac{(\delta^4-c)_+}{\delta^2+\;c}
 	\quad\Longrightarrow\quad
 	\sigma_\star^2
 	=\frac{1}{\gamma}\cdot\frac{(\delta^4-c)_+}{\delta^2+c}
 	=\frac{(h^4\gamma^2-c)_+}{\gamma(h^2\gamma+c)}
 	= h^2\,\frac{\Big(1-\dfrac{c}{h^4\gamma^2}\Big)_+}{1+\dfrac{c}{h^2\gamma}}.
 	\]
 	Below the BBP threshold ($\delta^4\le c$), $\eta^2=0$ hence $\sigma_\star^2=0$; above threshold the displayed expression holds.
 	
 	\medskip
 	All assumptions of Theorem~\ref{thm:master-in-sample} are verified; the conclusion follows.
 \end{proof}
 \subsection{Auxiliary results}
\begin{lemma}[Joint asymptotic law of two split-sample scores under the Gaussian subspace prior]
	\label{lem:joint-scores-split-subspace}
	Assume the Gaussian subspace prior \(\Lambda=N(0,h^2vv^\top)\) with \(\|v\|_2=1\), and let \(b_j:=\mathbf 1\{j\in\cH_1\}\in\{0,1\}\). Conditional on \((\cH_0,\cH_1)\), write
	\[
	\theta_j=h\,b_jU_jv,\qquad U_j\iidsim N(0,1),
	\]
	with \((U_j)_{j=1}^m\) independent. Let
	\[
	X_j^{\score}=\sqrt{1-\pi_{\spl}}\,\theta_j+\xi_j^{\score},\qquad \xi_j^{\score}\iidsim N(0,I_d),
	\]
	where \((\xi_j^{\score})_{j=1}^m\) are independent of \((U_j)_{j=1}^m\). Let \(\hat v_{\spl}\) be the unit vector learned from the learning split, oriented so that \(\langle \hat v_{\spl},v\rangle\ge 0\), and define
	\[
	S_j:=(\hat v_{\spl}^{\!\top}X_j^{\score})^2,\qquad j=1,\dots,m.
	\]
	
	Fix distinct \(i\neq j\), and assume that, conditional on \((\cH_0,\cH_1)\),
	\[
	\langle \hat v_{\spl},v\rangle^2 \convp \tau_{\mathrm{PCA}}^2
	\qquad\text{for some }\tau_{\mathrm{PCA}}^2\in[0,1].
	\]
	Then, conditional on \(b_i,b_j,(\cH_0,\cH_1)\),
	\[
	(S_i,S_j)\convd \big(\sigma_{b_i}^2\chi^2_{1,i},\,\sigma_{b_j}^2\chi^2_{1,j}\big),
	\]
	where \(\chi^2_{1,i},\chi^2_{1,j}\) are independent \(\chi^2_1\) random variables and
	\[
	\sigma_0^2=1,\qquad \sigma_1^2=1+(1-\pi_{\spl})h^2\tau_{\mathrm{PCA}}^2.
	\]
	Equivalently, for every fixed \(t_1,t_2\ge 0\),
	\[
	\P\!\big(S_i\ge t_1,\ S_j\ge t_2 \mid b_i,b_j,(\cH_0,\cH_1)\big)
	\to \bar G_{b_i}(t_1)\bar G_{b_j}(t_2),
	\]
	where \(\bar G_b(t):=2\bar\Phi(\sqrt t/\sigma_b)\) for \(b\in\{0,1\}\). In particular, the limit factorizes, so the limiting pair is independent.
\end{lemma}
 
 \begin{proof}
 	Write
 	\[
 	a_{\spl}:=\sqrt{1-\pi_{\spl}},
 	\qquad
 	\rho_n:=\langle \hat v_{\spl},v\rangle \ge 0.
 	\]
 	Since \(\rho_n^2\convp\tau_{\mathrm{PCA}}^2\) and \(\rho_n\ge0\), it follows that
 	\[
 	\rho_n \convp \tau_{\mathrm{PCA}}
 	\qquad\text{conditional on }(\cH_0,\cH_1),
 	\]
 	where \(\tau_{\mathrm{PCA}}:=\sqrt{\tau_{\mathrm{PCA}}^2}\in[0,1]\).
 	
 	For \(\ell\in\{i,j\}\),
 	\[
 	\hat v_{\spl}^{\!\top}X_\ell^{\score}
 	=
 	a_{\spl}\,\hat v_{\spl}^{\!\top}\theta_\ell
 	+
 	\hat v_{\spl}^{\!\top}\xi_\ell^{\score}
 	=
 	a_{\spl}\,h\,b_\ell U_\ell\,\rho_n + Z_{\ell,n},
 	\]
 	where
 	\[
 	Z_{\ell,n}:=\hat v_{\spl}^{\!\top}\xi_\ell^{\score}.
 	\]
 	Conditional on \((\hat v_{\spl},U_i,U_j,\cH_0,\cH_1)\), the pair
 	\[
 	(Z_{i,n},Z_{j,n})
 	\]
 	consists of independent \(N(0,1)\) random variables, because
 	\((\xi_i^{\score},\xi_j^{\score})\) are independent standard Gaussian vectors and are independent of the learning split.
 	
 	Define the projected pair
 	\[
 	Y_n
 	:=
 	\big(\hat v_{\spl}^{\!\top}X_i^{\score},\ \hat v_{\spl}^{\!\top}X_j^{\score}\big)
 	=
 	a_{\spl}h\rho_n\,(b_iU_i,\ b_jU_j) + (Z_{i,n},Z_{j,n}).
 	\]
 	We claim that, conditional on \(b_i,b_j,(\cH_0,\cH_1)\),
 	\[
 	a_{\spl}h\rho_n\,(b_iU_i,b_jU_j)
 	-
 	a_{\spl}h\tau_{\mathrm{PCA}}\,(b_iU_i,b_jU_j)
 	\ \convp\ 0.
 	\]
 	Indeed,
 	\[
 	\left\|
 	a_{\spl}h(\rho_n-\tau_{\mathrm{PCA}})(b_iU_i,b_jU_j)
 	\right\|_2
 	\le
 	a_{\spl}h\,|\rho_n-\tau_{\mathrm{PCA}}|\,
 	\sqrt{U_i^2+U_j^2},
 	\]
 	and conditional on \((\cH_0,\cH_1)\), the Gaussian pair \((U_i,U_j)\) is \(O_{\P}(1)\), while
 	\(\rho_n-\tau_{\mathrm{PCA}}\convp0\). Hence the claim follows.
 	
 	Therefore, by Slutsky's theorem,
 	\[
 	Y_n
 	\ \convd\
 	\big(a_{\spl}h\tau_{\mathrm{PCA}}b_iU_i + Z_i,\ 
 	a_{\spl}h\tau_{\mathrm{PCA}}b_jU_j + Z_j\big)
 	\qquad
 	\text{conditional on }b_i,b_j,(\cH_0,\cH_1),
 	\]
 	where \(Z_i,Z_j\) are independent \(N(0,1)\), independent of \(U_i,U_j\).
 	
 	Now \(U_i,U_j,Z_i,Z_j\) are mutually independent conditional on \((\cH_0,\cH_1)\), so the two limiting coordinates are independent centered Gaussians with variances
 	\[
 	\Var(a_{\spl}h\tau_{\mathrm{PCA}}b_\ell U_\ell + Z_\ell \mid b_\ell)
 	=
 	1 + a_{\spl}^2 h^2\tau_{\mathrm{PCA}}^2 b_\ell
 	=
 	1 + (1-\pi_{\spl})h^2\tau_{\mathrm{PCA}}^2 b_\ell
 	=
 	\sigma_{b_\ell}^2.
 	\]
 	Hence
 	\[
 	Y_n
 	\ \convd\
 	\big(\sigma_{b_i}W_i,\ \sigma_{b_j}W_j\big)
 	\qquad
 	\text{conditional on }b_i,b_j,(\cH_0,\cH_1),
 	\]
 	where \(W_i,W_j\) are independent \(N(0,1)\).
 	
 	Finally, since
 	\[
 	(S_i,S_j)=(Y_{n,1}^2,Y_{n,2}^2),
 	\]
 	the continuous mapping theorem yields
 	\[
 	(S_i,S_j)
 	\ \convd\
 	\big(\sigma_{b_i}^2W_i^2,\ \sigma_{b_j}^2W_j^2\big)
 	=
 	\big(\sigma_{b_i}^2\chi^2_{1,i},\ \sigma_{b_j}^2\chi^2_{1,j}\big),
 	\]
 	conditional on \(b_i,b_j,(\cH_0,\cH_1)\), as claimed.
 	
 	The tail-factorization statement follows immediately from the product form of the limiting law.
 \end{proof}

 \begin{lemma}[Joint asymptotic law of two in-sample scores]\label{lem:joint-scores-in-sample}
 	Let $X\in\R^{n\times d}$ have rows
 	\[
 	X_i \;=\; h\,b_i Z_i\,v \;+\; \varepsilon_i,\qquad
 	\varepsilon_i\sim\mathcal N(0,I_d),\ \ Z_i\sim\mathcal N(0,1),\ \ b_i\sim\mathrm{Ber}(\pi),
 	\]
 	with all variables independent and $\|v\|_2=1$. Let
 	\[
 	\frac{1}{\sqrt{n}}X=\hat U\hat\Sigma\hat V^\top,\qquad
 	\hat u:=\hat u_1,\ \hat v:=\hat v_1,\ \hat\sigma_1:=\hat\Sigma_{11},
 	\]
 	where the top vectors are oriented so that  $\langle\hat u,u\rangle\ge 0$ with
 	\[
 	u:=\frac{b\odot Z}{\|b\odot Z\|}\in\R^n,\qquad (b\odot Z)_k=b_k Z_k.
 	\]
 	For two distinct indices $i\neq j$, define the in-sample scores
 	\[
 	S_\ell:=(\hat v^\top X_\ell)^2 \;=\; n\,\hat\sigma_1^2\,(\hat u)_\ell^2,\qquad \ell\in\{i,j\}.
 	\]
 	Assume $n,d\to\infty$ with $d/n\to\gamma\in(0,\infty)$ and that
 	\[
 	\hat\eta_n^2:=\langle \hat u,u\rangle^2\ \convp\ \eta^2\in[0,1].
 	\]
 	Then almost surely
 	\[
 	\Big(\,n(\hat u)_i^2,\ n(\hat u)_j^2\,\Big)\ \mid b_i,b_j,(\cH_0,\cH_1) \convd\ \big(\tau_{b_i}^2\,\chi^2_{1,\,i},\ \tau_{b_j}^2\,\chi^2_{1,\,j}\big),
 	\qquad
 	\tau_0^2:=1-\eta^2,\ \ \tau_1^2:=(1-\eta^2)+\eta^2/\pi,
 	\]
 	where $\chi^2_{1,\,i},\chi^2_{1,\,j}$ are independent $\chi^2_1$ and independent of $(b_i,b_j)$.
 	Consequently, almost surely the normalized scores factorize in the limit:
 	\[
 	\left(\frac{S_i}{\hat\sigma_1^2},\ \frac{S_j}{\hat\sigma_1^2}\right)\ \mid b_i,b_j,(\cH_0,\cH_1) \convd\
 	\big(\tau_{b_i}^2\,\chi^2_{1,\,i},\ \tau_{b_j}^2\,\chi^2_{1,\,j}\big).
 	\]
 	If, in addition, $\hat\sigma_1\to\sigma_\infty\in(0,\infty)$ in probability, then with  probability going to $1$
 	\[
 	(S_i,S_j) \mid b_i,b_j,(\cH_0,\cH_1) \convd\ \sigma_\infty^2\,
 	\big(\tau_{b_i}^2\,\chi^2_{1,\,i},\ \tau_{b_j}^2\,\chi^2_{1,\,j}\big).
 	\]
 \end{lemma}
 
 \begin{proof}
 	
 	Note that by strong law of large number on a set $\mathcal{S}$ with $\P(\mathcal{S}) =1$ we have that  $\frac{1}{n}\sum_{i=1}^n b_i \to \pi$. We would work on the event $\mathcal{S}$ throughout the rest of the proof.
 	
 \textbf{Step 1 (Decomposition and uniformity on $u^\perp$).}
 Define $\hat\eta_n:=\langle \hat u,u\rangle\in[0,1]$ and
 \[
 w:=\frac{\hat u-\hat\eta_n u}{\sqrt{1-\hat\eta_n^2}}\in u^\perp,\qquad \|w\|_2=1.
 \]
 Write the signal part of $X$ as
 \[
 X^{\mathrm{sig}} := h\,(b\odot Z)\,v^\top
 = h\,\|b\odot Z\|\,u\,v^\top,
 \]
 so conditional on $(u,\|b\odot Z\|)$ the left singular direction of the signal is exactly $u$.
 Let $E$ denote the noise matrix with i.i.d.\ $\mathcal N(0,1)$ entries, so $X=X^{\mathrm{sig}}+E$.
 For any orthogonal $R\in O_n$ satisfying $Ru=u$, we have
 \[
 RX = R X^{\mathrm{sig}} + R E
 = h\,\|b\odot Z\|\,u\,v^\top + R E
 \stackrel{d}{=} h\,\|b\odot Z\|\,u\,v^\top + E = X
 \]
 conditionally on $(u,\|b\odot Z\|)$, since $R E\stackrel{d}{=}E$ by rotational invariance of the Gaussian noise.
 By equivariance of the SVD, this implies that $(\hat u,\hat\eta_n)$ has the same conditional law as $(R\hat u,\hat\eta_n)$ for all such $R$.
 Hence the conditional law of $\hat u$ given $(u,\hat\eta_n)$ is invariant under all rotations that fix $u$, so the component
 \[
 w = \frac{\hat u-\hat\eta_n u}{\sqrt{1-\hat\eta_n^2}}\in u^\perp
 \]
 must be rotationally invariant in $u^\perp$ and have unit norm.
 Therefore, conditional on $(u,\hat\eta_n)$,
 \begin{equation}\label{eq:haar}
 	\mathcal L\big(w\,\big|\,u,\hat\eta_n\big)=\mathrm{Unif}\!\big(S^{n-2}\subset u^\perp\big).
 \end{equation}

 	\medskip
 \noindent\textbf{Step 2 (LLN for $\|b\odot Z\|$ and entry size of $u$).}
 Recall that
 \[
 \mathcal S \;=\;\Bigl\{\frac{1}{n}\sum_{k=1}^n b_k \to \pi\Bigr\},
 \]
 which depends only on $(b_k)_{k\ge1}$ and has $\P(\mathcal S)=1$.
 On $\mathcal S$, let
 \[
 m_n := |\cH_1| = \sum_{k=1}^n b_k \sim \pi n \;\to\; \infty.
 \]
 
 Conditionally on $(\cH_0,\cH_1)$ (equivalently on $(b_k)_{k\ge1}$), the variables $\{Z_k : k\in\cH_1\}$ are i.i.d.\ $\mathcal N(0,1)$.
 By the strong law of large numbers applied to this subsequence,
 \[
 \frac{1}{m_n}\sum_{k\in\cH_1} Z_k^2 \;\to\; 1
 \quad\text{a.s.\ (w.r.t.\ the $Z_k$),}
 \]
 and hence on $\mathcal S$,
 \[
 \frac{1}{n}\|b\odot Z\|^2
 = \frac{1}{n}\sum_{k=1}^n b_k Z_k^2
 = \frac{m_n}{n}\cdot \frac{1}{m_n}\sum_{k\in\cH_1} Z_k^2
 \;\to\; \pi.
 \]
 Thus $\|b\odot Z\| = \sqrt{n\pi}\,(1+o(1))$, and for any fixed index $\ell$,
 \begin{equation}\label{eq:u-entry}
 	u_\ell
 	= \frac{b_\ell Z_\ell}{\|b\odot Z\|}
 	= \frac{b_\ell Z_\ell}{\sqrt{n\pi}}\,(1+o(1))
 	= \frac{b_\ell Z_\ell}{\sqrt{n\,\pi}} + o_P(n^{-1/2}),
 	\qquad \ell\in\{i,j\},
 \end{equation}
 where the $o_P(n^{-1/2})$ is with respect to the randomness of $(Z_k)$, conditional on $(\cH_0,\cH_1)$.
 
 	\medskip
 	\noindent\textbf{Step 3 (Moments for a Haar coordinate in $u^\perp$).}
 	Fix $u$ and choose an orthogonal $R$ with $R^\top u=e_n$. Let
 	$A:=R\begin{bmatrix}I_{n-1}\\ 0\end{bmatrix}$ map $\R^{n-1}$ isometrically onto $u^\perp$.
 	If $W'\sim\mathrm{Unif}(S^{n-2})$ in $\R^{n-1}$ and $\widetilde w:=A W'$, then \eqref{eq:haar} implies $\widetilde w\stackrel{d}{=}w$ given $(u,\hat\eta_n)$. For any coordinate,
 	\[
 	\E(\widetilde w_\ell\mid u)=0,\qquad
 	\Var(\widetilde w_\ell\mid u)=\frac{\|A^\top e_\ell\|_2^2}{n-1}=\frac{1-u_\ell^2}{n-1}.
 	\]
 	
 	\medskip
 	\noindent\textbf{Step 4 (Stable coordinate CLT on the sphere and joint version).}
 	Using the Gaussian representation $W'=G'/\|G'\|$ with $G'\sim\mathcal N(0,I_{n-1})$,
 	\[
 	\sqrt{n-1}\,\widetilde w_\ell \ \convd\ \mathcal N(0,1-u_\ell^2)\quad\text{conditionally on }u.
 	\]
 	Since $u_\ell^2=o_P(1)$ by \eqref{eq:u-entry}, a conditional characteristic–function argument yields the \emph{stable} limit
 	\begin{equation}\label{eq:stable-w1}
 		\sqrt{n}\,w_\ell \ \convd\ \mathcal N(0,1)\quad\text{w.r.t.\ }\sigma(u),\qquad \ell\in\{i,j\}.
 	\end{equation}
 	Jointly,
 	\[
 	\sqrt{n-1}\,(w_i,w_j)\ \convd\ \mathcal N\!\Big(0,\ \begin{bmatrix}
 		1-u_i^2 & -u_i u_j\\
 		-u_i u_j & 1-u_j^2
 	\end{bmatrix}\Big)\ \ \text{cond.\ on }u,
 	\]
 	and since $u_i,u_j=O_P(n^{-1/2})$, Slutsky gives
 	\begin{equation}\label{eq:joint-w}
 		\sqrt{n}\,(w_i,w_j)\ \convd\ (G_i,G_j),\quad G_i,G_j\ \text{i.i.d. }\mathcal N(0,1),
 	\end{equation}
 	independent of $\sigma(u)$.
 	
 	\medskip
 	\noindent\textbf{Step 5 (Signal-aligned piece).}
 	By \eqref{eq:u-entry} and $\hat\eta_n\to\eta$,
 	\[
 	\sqrt{n}\,\hat\eta_n\,u_\ell
 	= \hat\eta_n \cdot \frac{b_\ell Z_\ell}{\|b\odot Z\|/\sqrt{n}}
 	\ \convd\ \eta\cdot \frac{b_\ell Z_\ell}{\sqrt{\pi}},\qquad \ell\in\{i,j\},
 	\]
 	so conditionally on $b_\ell$ this limit is $0$ if $b_\ell=0$ and $\mathcal N(0,\eta^2/\pi)$ if $b_\ell=1$.
 	
 	\medskip
 	\noindent\textbf{Step 6 (Combine the two pieces for the coordinates).}
 	Write
 	\[
 	\sqrt{n}\,(\hat u)_\ell
 	= \underbrace{\sqrt{n}\,\hat\eta_n\,u_\ell}_{\text{signal-aligned}}
 	+ \underbrace{\sqrt{n}\,\sqrt{1-\hat\eta_n^2}\,w_\ell}_{\text{sphere term}},\qquad \ell\in\{i,j\}.
 	\]
 	Since $\hat\eta_n\to\eta$, we have $\sqrt{1-\hat\eta_n^2}\to\sqrt{1-\eta^2}$. By \eqref{eq:joint-w}, the sphere terms converge jointly to independent $\mathcal N(0,1-\eta^2)$ and are asymptotically independent of the signal-aligned limits. Therefore, conditional on $(b_i,b_j)$,
 	\[
 	\big(\sqrt{n}\,(\hat u)_i,\ \sqrt{n}\,(\hat u)_j\big)\ \convd\ (Y_i,Y_j),
 	\]
 	with $Y_i,Y_j$ independent mean-zero Gaussians and
 	\[
 	\Var(Y_\ell\mid b_\ell)=
 	\begin{cases}
 		1-\eta^2, & b_\ell=0,\\[4pt]
 		(1-\eta^2)+\eta^2/\pi, & b_\ell=1.
 	\end{cases}
 	\]
 	
 	\medskip
 	\noindent\textbf{Step 7 (Square to obtain the normalized scores).}
 	By the continuous mapping theorem and independence, we have on $\mathcal{S}$,
 	\[
 	\Big(n(\hat u)_i^2,\ n(\hat u)_j^2\Big)\ \mid b_i,b_j,(\cH_0,\cH_1) \convd\
 	\big(\tau_{b_i}^2\,\chi^2_{1,\,i},\ \tau_{b_j}^2\,\chi^2_{1,\,j}\big),
 	\]
 	with $\chi^2_{1,\,i},\chi^2_{1,\,j}$ independent.
 	
 	\medskip
 	\noindent\textbf{Step 8 (scaling by the top singular value).}
 	Since $X\hat v=\sqrt{n}\,\hat\sigma_1\,\hat u$, we have $S_\ell=n\hat\sigma_1^2(\hat u)_\ell^2$. Thus on $\mathcal{S}$ 
 	\[
 	\left(\frac{S_i}{\hat\sigma_1^2},\ \frac{S_j}{\hat\sigma_1^2}\right)\ \mid b_i,b_j,(\cH_0,\cH_1)\convd\
 	\big(\tau_{b_i}^2\,\chi^2_{1,\,i},\ \tau_{b_j}^2\,\chi^2_{1,\,j}\big).
 	\]
 	If additionally $\hat\sigma_1\to \sigma_\infty$, Slutsky yields  (on the event $\mathcal{S}$)
 	\(
 	(S_i,S_j)\mid b_i,b_j,(\cH_0,\cH_1)\convd \sigma_\infty^2\big(\tau_{b_i}^2\,\chi^2_{1,\,i},\ \tau_{b_j}^2\,\chi^2_{1,\,j}\big).
 	\) 
 	with probability going to $1$.
 \end{proof}
 
 \begin{lemma}[Limits of $\hat\sigma_1$ and $\hat\eta_n$ under BONuS]\label{lem:bonus-sigma-eta}
 	Let $X\in\R^{n\times d}$ with $n=m+\tilde m$ be the BONuS-augmented panel,
 	\[
 	X_i=\begin{cases}
 		h\,b_i Z_i\, v + \varepsilon_i, & 1\le i\le m,\\
 		\varepsilon_i, & m<i\le n,
 	\end{cases}
 	\qquad
 	\varepsilon_i\stackrel{\mathrm{i.i.d.}}{\sim}\mathcal N(0,I_d),\ \ 
 	Z_i\stackrel{\mathrm{i.i.d.}}{\sim}\mathcal N(0,1),
 	\]
 	where $b_i\stackrel{\mathrm{i.i.d.}}{\sim}\mathrm{Ber}(\pi)$ for $i\le m$ and $b_i\equiv 0$ for $i>m$, and $\|v\|_2=1$.
 	Let the scaled SVD be
 	\[
 	\frac{1}{\sqrt{n}}\,X=\hat U\hat\Sigma\hat V^\top,\qquad
 	\hat\sigma_1:=\hat\Sigma_{11},\ \hat u:=\hat u_1,\ \hat v:=\hat v_1,
 	\]
 	and define
 	\[
 	u:=\frac{b\odot Z}{\|b\odot Z\|}\in\R^n,\qquad
 	\hat\eta_n:=\langle \hat u,u\rangle\in[0,1].
 	\]
 	Assume $d,m,\tilde m\to\infty$ with
 	\[
 	\frac{d}{m}\to c\in(0,\infty),\qquad \frac{d}{n}\to\gamma'=\pi_{\aug}c\in(0,\infty),\qquad \pi_{\aug}:=\frac{m}{n}\to\pi_{\aug}\in(0,1).
 	\]
 	Let the effective signal fraction be $\bar\pi:=\pi_{\aug}\pi$ and set the (population) spike size
 	\[
 	\delta_n\ :=\ \frac{h}{\sqrt{n}}\|b\odot Z\|\ \convp\ \delta\ :=\ h\sqrt{\bar\pi}.
 	\]
 	Then the following limits hold (almost surely):
 	
 	\smallskip
 	\noindent\textbf{(a) Overlap (BBP transition and closed form).}
 	\[
 	\hat\eta_n^2\ \mid (\cH_0,\cH_1)\convp\ \eta^2(\delta,\gamma')\ :=\
 	\begin{cases}
 		0, & \delta^4\le \gamma',\\[6pt]
 		\dfrac{\delta^4-\gamma'}{\delta^2(\delta^2+1)}, & \delta^4>\gamma'.
 	\end{cases}
 	\]
 	
 	\smallskip
 	\noindent\textbf{(b) Top singular value.}
 	\[
 	\hat\sigma_1\ \mid (\cH_0,\cH_1)\convp\ \sigma_\infty(\delta,\gamma')\ :=\
 	\begin{cases}
 		1+\sqrt{\gamma'}, & \delta^4\le \gamma',\\[8pt]
 		\sqrt{(1+\delta^2)\Big(1+\dfrac{\gamma'}{\delta^2}\Big)}, & \delta^4>\gamma'.
 	\end{cases}
 	\]
 \end{lemma}
 
 \begin{proof}
 	Note that by strong law of large number on a set $\mathcal{S}$ with $\P(\mathcal{S}) =1$ we have that  $\frac{1}{m}\sum_{i=1}^m b_i \to \pi$. We would work on the event $\mathcal{S}$ throughout the rest of the proof.

 	Write the scaled matrix $Y:=(1/\sqrt{n})X^\top\in\R^{d\times n}$. Then
 	\[
 	Y\ =\ S_n\ +\ W,\qquad
 	S_n\ :=\ \frac{h}{\sqrt{n}}\,v\,(b\odot Z)^\top,\qquad
 	W\ :=\ \frac{1}{\sqrt{n}}\,\big[\varepsilon_1,\ldots,\varepsilon_n\big]^\top,
 	\]
 	where $W$ has i.i.d.\ $\mathcal N(0,1/n)$ entries. Thus $Y$ is the rank–one \emph{additive} spiked rectangular model considered by \cite{Benaych-Georges2011} (BGN): the signal part has singular value
 	\[
 	\|S_n\|\ =\ \frac{h}{\sqrt{n}}\|v\|\,\|b\odot Z\|\ =\ \delta_n\ \convp\ \delta,
 	\]
 	and the left/right spike directions are $v$ and $u=(b\odot Z)/\|b\odot Z\|$.
 	
 	\medskip
 	\noindent\textit{Convergence of the spike size.}
 	Because only the first $m$ rows can carry signal, a similar argument as Lemma \ref{lem:joint-scores-in-sample} with law of large numbers gives
 	\[
 	\frac{1}{n}\|b\odot Z\|^2=\frac{m}{n}\cdot\frac{1}{m}\sum_{i=1}^m b_i Z_i^2\ \xrightarrow{a.s.}\ \pi_{\aug}\pi\ =\ \bar\pi,
 	\]
 	almost surely conditional on $b_i's$,
 	hence $\delta_n\to\delta=h\sqrt{\bar\pi}$ in probability on $\mathcal{S}$.
 	
 	\medskip
 	\noindent\textit{Apply BGN (deterministic spike case) and continuity in the spike.}
 	For a fixed spike \(\delta>0\), BGN’s formulas give:
 	(i) a BBP threshold at \(\delta^4=\gamma'\);
 	(ii) the outlier singular value limit \(\sqrt{(1+\delta^2)(1+\gamma'/\delta^2)}\) above threshold and the noise-edge limit \(1+\sqrt{\gamma'}\) below;
 	(iii) the squared cosine (overlap) between the empirical and population singular vectors equals
 	\(
 	\frac{\delta^4-\gamma'}{\delta^2(\delta^2+1)}
 	\)
 	when \(\delta^4>\gamma'\) and vanishes otherwise.
 	Moreover, BGN’s continuity remark for random spikes implies that the same limits hold when the spike strength \(\delta_n\) is random but \(\delta_n\to\delta\) in probability (which happens in our case on the event $\mathcal{S}$).
 	
 	\medskip
 	\noindent\textit{Map to the present notation.}
 	The right singular vector of \(Y\) equals the left singular vector \(\hat u\) of \(X\) (and its population counterpart is \(u\)); the top singular value of \(Y\) is \(\hat\sigma_1\). Substituting \(\delta=h\sqrt{\bar\pi}\) and \(\gamma'=d/n\) into the BGN formulas yields the claims in (a)–(b).
 \end{proof}
 
 \begin{lemma}[Conditional asymptotic law of a single in-sample score (quenched in $b,Z$)]
 	\label{lem:single-score-conditional-Zb}
 	Let $X\in\R^{n\times d}$ have rows
 	\[
 	X_k \;=\; h\,b_k Z_k\,v \;+\; \varepsilon_k,\qquad
 	\varepsilon_k\stackrel{\mathrm{i.i.d.}}{\sim}\mathcal N(0,I_d),
 	\]
 	where $\|v\|_2=1$. For each $n$, condition on the realized sequences $(b_1,\dots,b_n)$ and $(Z_1,\dots,Z_n)$.
 	Let
 	\[
 	\widehat\Sigma=\frac{1}{n}X^\top X,\qquad 
 	\frac{1}{\sqrt{n}}X=\hat U\hat\Sigma\hat V^\top,\qquad
 	\hat u:=\hat u_1,\ \hat v:=\hat v_1,\ \hat\sigma_1:=\hat\Sigma_{11},
 	\]
 	with the usual orientation $\langle \hat v,v\rangle\ge 0$ and $\langle \hat u,u\rangle\ge 0$, where
 	\[
 	u:=\frac{b\odot Z}{\|b\odot Z\|}\in\R^n.
 	\]
 	Fix an index $i$ and define the in-sample score
 	\[
 	S_i:=(\hat v^\top X_i)^2 \;=\; n\,\hat\sigma_1^2\,(\hat u)_i^2.
 	\]
 	
 	Assume $n,d\to\infty$ with $d/n\to\gamma\in(0,\infty)$ and that the following \emph{sequence conditions}
 	hold for the realized $(b,Z)$:
 	\begin{align}
 		\label{eq:seq-B}
 		\frac{1}{n}\sum_{k=1}^n b_k \to \pi\in(0,1),
 		\qquad m_n:=\sum_{k=1}^n b_k \to\infty;\\
 		\label{eq:seq-Z2}
 		\frac{1}{n}\sum_{k=1}^n b_k Z_k^2 \to \pi
 		\qquad\Big(\text{equivalently } \frac{1}{n}\|b\odot Z\|^2\to\pi\Big);\\
 		\label{eq:seq-deloc}
 		\max_{1\le k\le n}\frac{|b_k Z_k|}{\|b\odot Z\|}\to 0.
 	\end{align}
 	Assume also the (conditional) overlap convergence
 	\[
 	\hat\eta_n^2:=\langle \hat u,u\rangle^2\ \xrightarrow{P(\cdot\,|\,b,Z)}\ \eta^2\in[0,1].
 	\]
 	Then, conditional on $(b,Z)$,
 	\[
 	\sqrt n\,(\hat u)_i\ \xRightarrow{\ \ \ \ \ }\ Y_i,
 	\qquad
 	Y_i\ \sim\ \mathcal N\!\Big(\mu_i,\ 1-\eta^2\Big),
 	\qquad
 	\mu_i\ :=\frac{b_i Z_i}{\sqrt \pi}.
 	\]
 	Consequently, conditional on $(b,Z)$,
 	\[
 	n(\hat u)_i^2\ \xRightarrow{\ \ \ \ \ }\ (1-\eta^2)\,\chi^2_{1}(\lambda_i),
 	\qquad
 	\lambda_i\ :=\ \frac{\mu_i^2}{1-\eta^2}
 	\ =\ \frac{\eta^2}{1-\eta^2}\cdot \frac{b_i Z_i^2}{\pi},
 	\]
 	where $\chi^2_{1}(\lambda_i)$ denotes a noncentral chi-square with one degree of freedom and noncentrality
 	$\lambda_i$ (and we used \eqref{eq:seq-Z2} in the last equality).
 	Moreover,
 	\[
 	\frac{S_i}{\hat\sigma_1^2}\ \Big|\ (b,Z)
 	\ \xRightarrow{\ \ \ \ \ }\ (1-\eta^2)\,\chi^2_{1}(\lambda_i).
 	\]
 	If, in addition, $\hat\sigma_1\to\sigma_\infty\in(0,\infty)$ in $P(\cdot\,|\,b,Z)$-probability, then
 	\[
 	S_i\ \Big|\ (b,Z)\ \xRightarrow{\ \ \ \ \ }\
 	\sigma_\infty^2\,(1-\eta^2)\,\chi^2_{1}(\lambda_i).
 	\]
 \end{lemma}
 \begin{proof}
 	Throughout the proof we work \emph{conditionally on the realized sequences} $(b_1,\dots,b_n)$ and $(Z_1,\dots,Z_n)$, and assume the sequence conditions
 	\eqref{eq:seq-B}--\eqref{eq:seq-deloc} hold. Under this conditioning, the only randomness is from the Gaussian noise
 	$E\in\R^{n\times d}$ with i.i.d.\ $\mathcal N(0,1)$ entries, where
 	\[
 	X \;=\; X^{\mathrm{sig}}+E,
 	\qquad
 	X^{\mathrm{sig}} \;:=\; h\,(b\odot Z)\,v^\top
 	\;=\; h\,\|b\odot Z\|\,u\,v^\top,
 	\qquad
 	u=\frac{b\odot Z}{\|b\odot Z\|}.
 	\]
 	
 	\medskip
 	\noindent\textbf{Step 1 (Decomposition and Haar-uniformity on $u^\perp$).}
 	Define $\hat\eta_n:=\langle \hat u,u\rangle\in[0,1]$ and
 	\[
 	w\;:=\;\frac{\hat u-\hat\eta_n u}{\sqrt{1-\hat\eta_n^2}}
 	\in u^\perp,
 	\qquad
 	\|w\|_2=1,
 	\]
 	so that
 	\begin{equation}\label{eq:decomp-hatu-single}
 		\hat u \;=\; \hat\eta_n u + \sqrt{1-\hat\eta_n^2}\,w.
 	\end{equation}
 	Let $R\in O_n$ be any orthogonal matrix satisfying $Ru=u$. Then, conditionally on $(b,Z)$,
 	\[
 	RX
 	=R X^{\mathrm{sig}} + R E
 	= h\,\|b\odot Z\|\,u\,v^\top + R E
 	\stackrel{d}{=} h\,\|b\odot Z\|\,u\,v^\top + E
 	= X,
 	\]
 	since $RE\stackrel{d}{=}E$ by rotational invariance of i.i.d.\ Gaussian noise.
 	By SVD equivariance, $(\hat u,\hat\eta_n)$ has the same conditional law as $(R\hat u,\hat\eta_n)$ for all such $R$.
 	Hence, conditional on $(b,Z)$ and on $(u,\hat\eta_n)$, the law of $\hat u$ is invariant under all rotations fixing $u$.
 	Equivalently, conditional on $(u,\hat\eta_n)$,
 	\begin{equation}\label{eq:haar-single}
 		\mathcal L\big(w\,\big|\,u,\hat\eta_n\big)
 		=\mathrm{Unif}\!\big(S^{n-2}\subset u^\perp\big).
 	\end{equation}
 	
 	\medskip
 	\noindent\textbf{Step 2 (Consequences of the sequence conditions for $\|b\odot Z\|$ and $u_i$).}
 	By \eqref{eq:seq-Z2},
 	\[
 	\frac{1}{n}\|b\odot Z\|^2 \to \pi
 	\qquad\Longrightarrow\qquad
 	\frac{1}{\sqrt n}\|b\odot Z\| \to \sqrt{\pi}.
 	\]
 	Therefore, for the fixed index $i$,
 	\begin{equation}\label{eq:ui-asymp-single}
 		\sqrt n\,u_i
 		=\frac{b_i Z_i}{\|b\odot Z\|/\sqrt n}
 		\longrightarrow
 		\frac{b_i Z_i}{\sqrt\pi}.
 	\end{equation}
 	Moreover, \eqref{eq:seq-deloc} implies $\|u\|_\infty\to 0$, hence in particular
 	\begin{equation}\label{eq:ui-small-single}
 		u_i^2 \to 0.
 	\end{equation}
 	
 	\medskip
 	\noindent\textbf{Step 3 (Conditional CLT for one Haar coordinate in $u^\perp$).}
 	Fix $u$ and let $w\sim \mathrm{Unif}(S^{n-2}\subset u^\perp)$. A standard representation yields
 	\[
 	w \;=\; \frac{\Pi_{u^\perp}G}{\|\Pi_{u^\perp}G\|},
 	\qquad G\sim\mathcal N(0,I_n),
 	\]
 	where $\Pi_{u^\perp}=I-uu^\top$.
 	Then, conditionally on $u$, the scalar $\langle e_i,w\rangle=w_i$ has mean $0$ and variance
 	\[
 	\Var(w_i\mid u)=\frac{1-u_i^2}{n-1}.
 	\]
 	Moreover, the usual spherical coordinate CLT (equivalently, $w$ as a normalized Gaussian in $u^\perp$) gives, conditionally on $u$,
 	\[
 	\sqrt{n-1}\,w_i \ \convd\ \mathcal N(0,1-u_i^2).
 	\]
 	Using \eqref{eq:ui-small-single} and Slutsky,
 	\begin{equation}\label{eq:clt-wi-single}
 		\sqrt n\,w_i \ \convd\ \mathcal N(0,1),
 		\qquad\text{under }P(\cdot\,|\,b,Z).
 	\end{equation}
 	
 	\medskip
 	\noindent\textbf{Step 4 (Limit of $\sqrt n\,\hat u_i$).}
 	From the decomposition \eqref{eq:decomp-hatu-single},
 	\[
 	\sqrt n\,\hat u_i
 	=\underbrace{\sqrt n\,\hat\eta_n u_i}_{A_n}
 	+\underbrace{\sqrt n\,\sqrt{1-\hat\eta_n^2}\,w_i}_{B_n}.
 	\]
 	For the aligned term, by \eqref{eq:ui-asymp-single} and the assumption
 	$\hat\eta_n^2\to\eta^2$ in $P(\cdot\,|\,b,Z)$-probability,
 	\[
 	A_n
 	=\hat\eta_n\cdot \sqrt n\,u_i
 	\ \xrightarrow{P(\cdot\,|\,b,Z)}\
 	\eta\cdot \frac{b_i Z_i}{\sqrt\pi}
 	=:\mu_i.
 	\]
 	For the spherical term, since $w_i$ is (conditionally on $(u,\hat\eta_n)$) Haar as in \eqref{eq:haar-single},
 	its conditional distribution does not depend on $\hat\eta_n$, and \eqref{eq:clt-wi-single} yields
 	\[
 	\sqrt n\,w_i \ \convd\ G,\qquad G\sim\mathcal N(0,1),
 	\]
 	under $P(\cdot\,|\,b,Z)$. Combining this with $\sqrt{1-\hat\eta_n^2}\to\sqrt{1-\eta^2}$
 	in $P(\cdot\,|\,b,Z)$-probability and Slutsky gives
 	\[
 	B_n \ \convd\ \sqrt{1-\eta^2}\,G \ \sim\ \mathcal N(0,1-\eta^2),
 	\qquad\text{under }P(\cdot\,|\,b,Z).
 	\]
 	Since $A_n\to\mu_i$ in probability and $B_n\Rightarrow \mathcal N(0,1-\eta^2)$, we conclude
 	\[
 	\sqrt n\,\hat u_i \ \convd\ \mathcal N(\mu_i,1-\eta^2)
 	\qquad\text{under }P(\cdot\,|\,b,Z),
 	\]
 	which is the claimed limit for $\sqrt n(\hat u)_i$.
 	
 	\medskip
 	\noindent\textbf{Step 5 (Square to obtain the noncentral chi-square limit).}
 	Let $Y_i\sim \mathcal N(\mu_i,1-\eta^2)$. Then
 	\[
 	\frac{Y_i^2}{1-\eta^2}\ \sim\ \chi^2_1\!\left(\lambda_i\right),
 	\qquad
 	\lambda_i:=\frac{\mu_i^2}{1-\eta^2}.
 	\]
 	By the continuous mapping theorem applied to the convergence of $\sqrt n\,\hat u_i$,
 	\[
 	n(\hat u)_i^2 \ \convd\ (1-\eta^2)\,\chi^2_1(\lambda_i),
 	\qquad\text{under }P(\cdot\,|\,b,Z).
 	\]
 	Finally, using $\mu_i=\eta\,\frac{b_i Z_i}{\sqrt\pi}$ from Step 4 and \eqref{eq:seq-Z2},
 	\[
 	\lambda_i
 	=\frac{\eta^2}{1-\eta^2}\cdot \frac{b_i Z_i^2}{\pi},
 	\]
 	as stated.
 	
 	\medskip
 	\noindent\textbf{Step 6 (Scaling by the top singular value).}
 	Since $X\hat v=\sqrt n\,\hat\sigma_1\,\hat u$, we have
 	$S_i=n\hat\sigma_1^2(\hat u)_i^2$, hence $S_i/\hat\sigma_1^2=n(\hat u)_i^2$.
 	This proves
 	\[
 	\frac{S_i}{\hat\sigma_1^2}\ \Big|\ (b,Z)
 	\ \convd\ (1-\eta^2)\,\chi^2_1(\lambda_i).
 	\]
 	If additionally $\hat\sigma_1\to\sigma_\infty$ in $P(\cdot\,|\,b,Z)$-probability, then Slutsky yields
 	\[
 	S_i\ \Big|\ (b,Z)\ \convd\ \sigma_\infty^2\,(1-\eta^2)\,\chi^2_1(\lambda_i).
 	\]
 \end{proof}
 
 \begin{lemma}[Conditional limits of $\hat\sigma_1$ and $\hat\eta_n$ (in-sample; quenched in $b,Z$)]
 	\label{lem:insample-sigma-eta-conditional-Zb}
 	Let $X\in\R^{n\times d}$ have rows
 	\[
 	X_k \;=\; h\,b_k Z_k\,v \;+\; \varepsilon_k,\qquad
 	\varepsilon_k\stackrel{\mathrm{i.i.d.}}{\sim}\mathcal N(0,I_d),
 	\qquad \|v\|_2=1,
 	\]
 	and for each $n$ condition on the realized sequences $(b_1,\dots,b_n)$ and $(Z_1,\dots,Z_n)$.
 	Let the scaled SVD be
 	\[
 	\frac{1}{\sqrt{n}}\,X=\hat U\hat\Sigma\hat V^\top,\qquad
 	\hat\sigma_1:=\hat\Sigma_{11},\ \hat u:=\hat u_1,\ \hat v:=\hat v_1,
 	\]
 	with the usual orientation $\langle \hat v,v\rangle\ge 0$ and $\langle \hat u,u\rangle\ge 0$, and define
 	\[
 	u:=\frac{b\odot Z}{\|b\odot Z\|}\in\R^n,\qquad
 	\hat\eta_n:=\langle \hat u,u\rangle\in[0,1].
 	\]
 	Assume $n,d\to\infty$ with $d/n\to\kappa\in(0,\infty)$ and assume the sequence conditions
 	\eqref{eq:seq-B}--\eqref{eq:seq-deloc} from Lemma~\ref{lem:single-score-conditional-Zb} hold for the realized $(b,Z)$.
 	Define the (deterministic, given $b,Z$) spike size
 	\[
 	\delta_n\ :=\ \frac{h}{\sqrt{n}}\|b\odot Z\|,
 	\qquad
 	\delta\ :=\ h\sqrt{\pi}.
 	\]
 	Then, conditional on $(b,Z)$,
 	\[
 	\hat\eta_n^2\ \xrightarrow{P(\cdot\,|\,b,Z)}\ \eta^2(\delta,\kappa),
 	\qquad
 	\hat\sigma_1\ \xrightarrow{P(\cdot\,|\,b,Z)}\ \sigma_\infty(\delta,\kappa),
 	\]
 	where (BBP transition and closed forms)
 	\[
 	\eta^2(\delta,\kappa)=
 	\begin{cases}
 		0, & \delta^4\le \kappa,\\[6pt]
 		\dfrac{\delta^4-\kappa}{\delta^2(\delta^2+1)}, & \delta^4>\kappa,
 	\end{cases}
 	\qquad
 	\sigma_\infty(\delta,\kappa)=
 	\begin{cases}
 		1+\sqrt{\kappa}, & \delta^4\le \kappa,\\[8pt]
 		\sqrt{(1+\delta^2)\Big(1+\dfrac{\kappa}{\delta^2}\Big)}, & \delta^4>\kappa.
 	\end{cases}
 	\]
 \end{lemma}
 
 \begin{proof}
 	Fix $n$ and work throughout \emph{conditionally on} the realized $(b_1,\dots,b_n)$ and $(Z_1,\dots,Z_n)$.
 	Under this conditioning, the only randomness is from the i.i.d.\ Gaussian noise vectors $(\varepsilon_k)_{k\le n}$.
 	
 	\medskip
 	\noindent\textbf{Step 1 (Map to the additive spiked rectangular model).}
 	Let
 	\[
 	Y\;:=\;\frac{1}{\sqrt n}X^\top\in\R^{d\times n}.
 	\]
 	Then $Y=S_n+W$ where
 	\[
 	S_n\ :=\ \frac{h}{\sqrt n}\,v\,(b\odot Z)^\top
 	\ =\ \delta_n\, v\,u^\top,
 	\qquad
 	W\ :=\ \frac{1}{\sqrt n}\,[\varepsilon_1,\dots,\varepsilon_n].
 	\]
 	Conditionally on $(b,Z)$, the matrix $W$ has i.i.d.\ $\mathcal N(0,1/n)$ entries. Thus, conditional on $(b,Z)$,
 	$Y$ is a rank-one additive perturbation of an i.i.d.\ Gaussian matrix with left spike $v$, right spike $u$, and
 	spike strength $\delta_n$.
 	
 	\medskip
 	\noindent\textbf{Step 2 (Deterministic convergence of the spike size).}
 	By the sequence condition \eqref{eq:seq-Z2},
 	\[
 	\frac{1}{n}\|b\odot Z\|^2 \to \pi,
 	\qquad\text{hence}\qquad
 	\delta_n=\frac{h}{\sqrt n}\|b\odot Z\| \to h\sqrt{\pi}=\delta,
 	\]
 	deterministically given the realized $(b,Z)$ satisfying \eqref{eq:seq-Z2}.
 	
 	\medskip
 	\noindent\textbf{Step 3 (Invoke Benaych-Georges--Nadakuditi and plug in $\delta$).}
 	Since $d/n\to\kappa$ and $Y=\delta_n vu^\top+W$ with i.i.d.\ Gaussian noise, we may apply the
 	Benaych-Georges--Nadakuditi (2011) asymptotics for the rank-one spiked rectangular model.
 	The delocalization condition \eqref{eq:seq-deloc} gives $\|u\|_\infty\to 0$, placing us in the standard regime for the spike vector.
 	Therefore, conditional on $(b,Z)$,
 	the top singular value of $Y$ (equivalently of $(1/\sqrt n)X$) and the squared overlap between the empirical and population
 	right singular vectors converge in $P(\cdot\,|\,b,Z)$-probability to the deterministic limits given by the BGN formulas
 	at spike strength $\delta$ and aspect ratio $\kappa$:
 	\[
 	\hat\sigma_1 \ \xrightarrow{P(\cdot\,|\,b,Z)}\ \sigma_\infty(\delta,\kappa),
 	\qquad
 	\hat\eta_n^2=\langle \hat u,u\rangle^2 \ \xrightarrow{P(\cdot\,|\,b,Z)}\ \eta^2(\delta,\kappa).
 	\]
 	The piecewise definitions above encode the BBP transition at $\delta^4=\kappa$.
 	
 	\medskip
 	\noindent\textbf{Step 4 (Translate back to the notation for $X$).}
 	The nonzero singular values of $(1/\sqrt n)X$ and $Y=(1/\sqrt n)X^\top$ coincide, so the top singular value is $\hat\sigma_1$.
 	Moreover, the right singular vector of $Y$ equals the left singular vector $\hat u$ of $X$, and the population right spike direction
 	is $u=(b\odot Z)/\|b\odot Z\|$. Hence $\hat\eta_n=\langle \hat u,u\rangle$, completing the proof.
 \end{proof}
 
 \begin{lemma}[A.s.\ sequence regularity and conditional null tail limits]\label{lem:null-tail-conditional}
 	Let $b_k\stackrel{\mathrm{i.i.d.}}{\sim}\mathrm{Ber}(\pi)$ and $Z_k\stackrel{\mathrm{i.i.d.}}{\sim}N(0,1)$ be independent, with
 	$\pi\in(0,1)$.
 	Then \eqref{eq:seq-B}--\eqref{eq:seq-deloc} sequence conditions hold almost surely. Let $\mathcal E$ denote the probability-one event on which \eqref{eq:seq-B}--\eqref{eq:seq-deloc} all hold.
 	Now fix $h>0$, $\|v\|_2=1$, and consider the in-sample Gaussian model
 	\[
 	\theta_k:=h\,b_k Z_k\,v,\qquad X_k=\theta_k+\varepsilon_k,\qquad
 	\varepsilon_k\stackrel{\mathrm{i.i.d.}}{\sim}N(0,I_d),\qquad k=1,\dots,m,
 	\]
 	with $d/m\to c\in(0,\infty)$. Let $(1/\sqrt m)X=\hat U\hat\Sigma\hat V^\top$ and define
 	\[
 	S_j:=(\hat v^\top X_j)^2=m\,\hat\sigma_1^2\,(\hat u)_j^2.
 	\]
 	Then on $\mathcal E$, for every fixed $t\in\R$ and every fixed index $j$,
 	\begin{equation}\label{eq:cond-tail-nulltail}
 		\P\!\big(S_j\ge t \,\big|\, j\in\cH_0,(\theta_k)_{k=1}^m\big)\ \longrightarrow\ \bar G_0(t),
 	\end{equation}
 	where
 	\[
 	\bar G_0(t):=\P\!\big(c_0\,\chi^2_1\ge t\big),\qquad
 	c_0:=\sigma_\infty^2(\delta,c)\big(1-\eta^2(\delta,c)\big),\qquad
 	\delta:=h\sqrt{\pi},
 	\]
 	and $\eta^2(\delta,c)$ and $\sigma_\infty(\delta,c)$ are the BGN closed forms from
 	Lemma~\ref{lem:insample-sigma-eta-conditional-Zb}. In particular
 	\[
 	\P\!\big(S_j\ge t \,\big|\, j\in\cH_0,(\theta_k)_{k=1}^m\big)\ \convp\ \bar G_0(t).
 	\]
 \end{lemma}
 
\begin{proof}
	\textbf{Part I: almost sure validity of the sequence conditions.}
	Let $b_k\stackrel{\mathrm{i.i.d.}}{\sim}\mathrm{Ber}(\pi)$ and $Z_k\stackrel{\mathrm{i.i.d.}}{\sim}N(0,1)$ be independent.
	
	\smallskip
	\noindent\emph{Condition \eqref{eq:seq-B}.}
	By the strong law of large numbers (SLLN) applied to $(b_k)$,
	\[
	\frac{1}{m}\sum_{k=1}^m b_k \to \E[b_1]=\pi
	\qquad\text{a.s.}
	\]
	In particular $\sum_{k=1}^m b_k\to\infty$ almost surely since $\pi>0$.
	
	\smallskip
	\noindent\emph{Condition \eqref{eq:seq-Z2}.}
	Apply SLLN to the i.i.d.\ sequence $(b_k Z_k^2)_{k\ge1}$. Since
	\[
	\E[b_1 Z_1^2]=\E[b_1]\E[Z_1^2]=\pi,
	\qquad
	\E[(b_1 Z_1^2)^2]=\E[b_1 Z_1^4]=\pi\,\E[Z_1^4]<\infty,
	\]
	SLLN yields
	\[
	\frac{1}{m}\sum_{k=1}^m b_k Z_k^2 \to \pi
	\qquad\text{a.s.}
	\]
	Equivalently, $\|b\odot Z\|^2/m\to\pi$ almost surely.
	
	\smallskip
	\noindent\emph{Condition \eqref{eq:seq-deloc}.}
	By a union bound and the Gaussian tail inequality,
	\[
	\P\!\left(\max_{1\le k\le m}|Z_k|\ge 3\sqrt{\log m}\right)
	\le 2m\exp\!\left(-\frac{(3\sqrt{\log m})^2}{2}\right)
	=2m^{-7/2}.
	\]
	Hence
	\[
	\sum_{m\ge1}\P\!\left(\max_{1\le k\le m}|Z_k|\ge 3\sqrt{\log m}\right)<\infty,
	\]
	and by Borel-Cantelli,
	\begin{equation}\label{eq:maxZ-growth}
		\max_{1\le k\le m}|Z_k|=O(\sqrt{\log m})
		\qquad\text{a.s.}
	\end{equation}
	On the other hand, by \eqref{eq:seq-Z2},
	\begin{equation}\label{eq:normbZ-growth}
		\|b\odot Z\|^2
		=\sum_{k=1}^m b_k Z_k^2
		= m\cdot\Big(\frac{1}{m}\sum_{k=1}^m b_k Z_k^2\Big)
		= \pi m\,(1+o(1))
		\qquad\text{a.s.,}
	\end{equation}
	and therefore $\|b\odot Z\|=\sqrt{\pi m}\,(1+o(1))$ almost surely. Combining
	\eqref{eq:maxZ-growth}--\eqref{eq:normbZ-growth},
	\[
	\max_{1\le k\le m}\frac{|b_k Z_k|}{\|b\odot Z\|}
	\le \frac{\max_{k\le m}|Z_k|}{\|b\odot Z\|}
	=O\!\left(\frac{\sqrt{\log m}}{\sqrt m}\right)\to 0
	\qquad\text{a.s.}
	\]
	This proves \eqref{eq:seq-deloc}. Let $\mathcal E$ be the (probability-one) event on which
	\eqref{eq:seq-B}--\eqref{eq:seq-deloc} all hold.
	
	\medskip
	\textbf{Part II: conditional null tail limit on $\mathcal E$.}
	Fix $t\in\R$ and a fixed index $j$. Consider
	\[
	\theta_k=h\,b_k Z_k v,\qquad X_k=\theta_k+\varepsilon_k,
	\qquad \varepsilon_k\stackrel{\mathrm{i.i.d.}}{\sim}N(0,I_d),
	\]
	with $d/m\to c\in(0,\infty)$, and define
	$S_j=(\hat v^\top X_j)^2=m\hat\sigma_1^2(\hat u)_j^2$ from the scaled SVD
	$(1/\sqrt m)X=\hat U\hat\Sigma\hat V^\top$.
	
	\smallskip
	\noindent\emph{Conditioning.}
	Conditioning on the realized panel $(\theta_k)_{k=1}^m$ is stronger than conditioning on the pair $(b,Z)$, since
	$\theta_k$ determines $b_k=\mathbf 1\{\theta_k\neq 0\}$ and the product $b_kZ_k=(v^\top\theta_k)/h$, but does not determine $Z_k$ itself on the event $\{b_k=0\}$.
	However, this distinction is immaterial for the argument below, because the conditional law of the data depends on $(b,Z)$ only through the products $(b_kZ_k)_{k=1}^m$:
	\[
	X_k = h\,b_k Z_k\,v + \varepsilon_k .
	\]
	Hence the conditional law of the data, and therefore of all score quantities constructed from the data, is the same whether we condition on the realized panel $(\theta_k)_{k=1}^m$ or on the realized values of $(b_kZ_k)_{k=1}^m$ (equivalently, on $(b,Z)$ modulo the irrelevant values of $Z_k$ when $b_k=0$).
	Accordingly, below, conditional statements given $\{\theta_k\}$ may be read as conditional statements given the effective signal coordinates $(b_kZ_k)_{k=1}^m$.
	
	\smallskip
	\noindent\emph{Deterministic spike size.}
	Work on $\mathcal E$. By \eqref{eq:seq-Z2} on $\mathcal E$,
	\[
	\delta_m:=\frac{h}{\sqrt m}\|b\odot Z\|
	\ \longrightarrow\ h\sqrt{\pi}=:\delta.
	\]
	
	\smallskip
	\noindent\emph{Limits of $\hat\sigma_1$ and $\hat\eta_m$.}
	By Lemma~\ref{lem:insample-sigma-eta-conditional-Zb}, conditional on $(\theta_k)_{k=1}^m$ (equivalently on $(b,Z)$),
	\[
	\hat\sigma_1 \xrightarrow{P(\cdot\,|\,\{\theta_k\})} \sigma_\infty(\delta,c),
	\qquad
	\hat\eta_m^2:=\langle \hat u,u\rangle^2 \xrightarrow{P(\cdot\,|\,\{\theta_k\})} \eta^2(\delta,c),
	\]
	where $u=(b\odot Z)/\|b\odot Z\|$.
	
	\smallskip
	\noindent\emph{Null score limit.}
	Now suppose $j\in\cH_0$, i.e.\ $\theta_j=0$ (equivalently $b_j=0$).
	By Lemma~\ref{lem:single-score-conditional-Zb} (with $n=m$),
	conditional on $(\theta_k)_{k=1}^m$,
	\[
	\frac{S_j}{\hat\sigma_1^2}=m(\hat u)_j^2 \Rightarrow (1-\eta^2(\delta,c))\,\chi^2_1,
	\]
	since under $b_j=0$ the noncentrality parameter is $0$. Slutsky then yields
	\[
	S_j \Rightarrow c_0\,\chi^2_1,
	\qquad
	c_0:=\sigma_\infty^2(\delta,c)\big(1-\eta^2(\delta,c)\big),
	\]
	conditionally on $(j\in\cH_0,(\theta_k)_{k=1}^m)$.
	
	\smallskip
	\noindent\emph{Tail convergence.}
	Because $c_0\chi^2_1$ has a continuous distribution, the portmanteau theorem gives, on $\mathcal E$,
	\[
	\P\!\big(S_j\ge t \,\big|\, j\in\cH_0,(\theta_k)_{k=1}^m\big)
	\ \longrightarrow\ \P(c_0\chi^2_1\ge t)=:\bar G_0(t),
	\]
	which is \eqref{eq:cond-tail-nulltail}.
	
	\smallskip
	\noindent\emph{Upgrade to $\convp$.}
	Since $\P(\mathcal E)=1$ and the limit $\bar G_0(t)$ is deterministic, convergence on $\mathcal E$
	implies
	\[
	\P\!\big(S_j\ge t \,\big|\, j\in\cH_0,(\theta_k)_{k=1}^m\big)\ \convp\ \bar G_0(t).
	\]
\end{proof}
 
  \subsection{Results on scores}
 \paragraph{Proof of Proposition \ref{prop:subspace-prior-score-distributions}}
 \begin{enumerate}
 	\item Split BH: We start at~\eqref{eq:score-split-bh-null-subspace} and use Lemma \ref{lemma:conditional-to-unconditional-convg} with $X_n =S^\spl_j  $ and $\cF_n = \{\hat{\bm v}_{\spl},\cH_0,\cH_1\}$ to obtain that $S^\spl_j \mid j \in \mathcal H_0 \convd \chi^2_1$, similarly we can use Lemma \ref{lem:joint-scores-split-subspace} along with Lemma \ref{lemma:conditional-to-unconditional-convg} to obtain $S^\spl_j \mid j \in \mathcal H_1 \convd (1 + \mu^2_\spl)\chi^2_1$.
 	
 	\item BONuS: We start at~\eqref{eq:score-bonus-subspace} and use Lemma \ref{lemma:conditional-to-unconditional-convg} with $X_n =S^\bonus_j  $ and $\cF_n = \{\cH_0,\cH_1\}$ to obtain that $S^\bonus_j \mid j \in \mathcal H_0 \convd  K_\bonus \chi^2_1$, similarly we can use~\eqref{eq:score-bonus-subspace} along with Lemma \ref{lemma:conditional-to-unconditional-convg} to obtain $S^\bonus_j \mid j \in \mathcal H_1 \convd K_\bonus (1 + \mu^2_\bonus)\chi^2_1$.
 	
 	\item In-Sample BH: The proof directly follows from BONuS score result with $\pi_{\aug} = 1$. 
 \end{enumerate}

 \section{Miscellaneous Lemmas}
 
 \subsection{Probability Lemmas}
 
 \begin{lemma}\label{lemma:convg-cond-mean-zero}
 	If $X_n$ is a random variable such that $\E(X_n \mid \cF_n)= 0$, then 
 	\[
 	\E \V(X_n \mid \cF_n) \to 0  \implies X_n \overset{p}{\to} 0.
 	\]
 \end{lemma}
 
 \begin{proof}
 	$\E X_n^2 = \E \E X_n^2 \mid \cF_n$ by tower property. Since $\E X_n \mid \cF_n =0$ we have $\E \E X_n^2 \mid \cF_n = \E \V(X_n \mid \cF_n)$.
 \end{proof}
 \begin{lemma}[Weak Law of Large Numbers for Bounded Arrays]
 	\label{lemma:wlln-bounded-rv}
 	Let $\{X_{ni}\}_{i=1}^n$ be an array of random variables such that 
 	$|X_{ni}| \le B$ almost surely for all $i,n$, for some constant $B > 0$. 
 	Suppose that $\{X_{ni}\}_{i=1}^n$ are conditionally exchangeable given $\mathcal F_n$ and for every $i \neq j$, the conditional covariance given a filtration 
 	$\cF_n$ vanishes asymptotically:
 	\[
 	\Cov(X_{ni}, X_{nj} \mid \cF_n) \;\convp\; 0.
 	\]
 	Then,
 	\[
 	\frac{1}{n}\sum_{i=1}^n X_{ni} \;-\; 
 	\E\!\left[\frac{1}{n}\sum_{i=1}^n X_{ni} \,\middle|\, \cF_n\right]
 	\;\convp\; 0.
 	\]
 \end{lemma}

 \begin{proof}
 	We apply Lemma~\ref{lemma:convg-cond-mean-zero}, which states that convergence in probability to the conditional expectation holds if 
 	\[
 	\E\left[ \V\left( \frac{1}{n} \sum_{i=1}^n X_{ni} \,\middle|\, \mathcal{F}_n \right) \right] \to 0.
 	\]	
 	Let $S_n = \sum_{i=1}^n X_{ni}$. By the formula for conditional variance, we have:
 	\begin{align*}
 		\V\left( \frac{1}{n} S_n \mid \mathcal{F}_n \right)
 		&= \frac{1}{n^2} \sum_{i=1}^n \V(X_{in} \mid \mathcal{F}_n)
 		+ \frac{1}{n^2} \sum_{i \ne j} \mathrm{Cov}(X_{ni}, X_{nj} \mid \mathcal{F}_n).
 	\end{align*}	
 	To conclude, we verify that both terms vanish in expectation:	
 	Since $|X_{in}| \leq B$ almost surely, $\V(X_{ni} \mid \mathcal{F}_n) \leq B^2$ uniformly, so
 	\[
 	\frac{1}{n^2} \sum_{i=1}^n \E[\V(X_{ni} \mid \mathcal{F}_n)] \leq \frac{B^2}{n} \to 0.
 	\]	
 	By assumption, $\mathrm{Cov}(X_{ni}, X_{nj} \mid \mathcal{F}_n) \overset{p}{\to} 0$ for $i \ne j$, and since $|\mathrm{Cov}(X_{ni}, X_{nj} \mid \mathcal{F}_n)| \leq B^2$, we can apply the bounded convergence theorem:
 	\[
 	\E[\mathrm{Cov}(X_{ni}, X_{nj} \mid \mathcal{F}_n)] \to 0.
 	\]
 	this implies that 
 	\[
 	\E\left[\frac{1}{n^2} \sum_{i\neq j}\mathrm{Cov}(X_{ni}, X_{nj} \mid \mathcal{F}_n)\right] = \frac{n(n-1)}{n^2}	\E[\mathrm{Cov}(X_{n1}, X_{n2} \mid \mathcal{F}_n)] \to 0
 	\]
 	where we used the fact that $X_{ni}$'s for $i=1,\ldots,n$ are identically distributed conditional on $\cF_n$.
 	Combining both, we obtain
 	\[
 	\E\left[ \V\left( \frac{1}{n} S_n \,\middle|\, \mathcal{F}_n \right) \right] \to 0,
 	\]
 	which completes the proof.
 \end{proof}
 
\begin{lemma}\label{lemma:conditional-to-unconditional-convg}
	Let $(X_n)$ be a sequence of real--valued random variables and $(\mathcal F_n)$ a sequence of $\sigma$--fields. 
	Assume that for every continuity point $t$ of a distribution function $F$,
	\[
	\mathbb{P}(X_n \ge t \mid \mathcal F_n) \xrightarrow{p} \bar F(t),
	\qquad \text{where } \bar F(t) := 1 - F(t).
	\]
	Then 
	\[
	X_n \xrightarrow{d} F .
	\]
\end{lemma}

\begin{proof}
	Fix a continuity point $t$ of $F$. 
	By assumption,
	\[
	\mathbb{P}(X_n \ge t \mid \mathcal F_n) \xrightarrow{p} \bar F(t).
	\]
	Taking expectations on both sides, and using bounded convergence theorem (since $\mathbb{P}(X_n \ge t \mid \mathcal F_n)$ is uniformly bounded by $1$)
	\[
	\mathbb{P}(X_n \ge t)
	= \mathbb{E}\!\left[\mathbb{P}(X_n \ge t \mid \mathcal F_n)\right]
	\;\longrightarrow\; \bar F(t).
	\]
	Hence
	\[
	\mathbb{P}(X_n \le t)
	= 1 - \mathbb{P}(X_n \ge t)
	\;\longrightarrow\; F(t).
	\]
	
	Thus the distribution functions $F_n(t) := \mathbb{P}(X_n \le t)$ converge pointwise to $F(t)$ at all continuity points of $F$. 
	By the classical characterization of weak convergence (Helly--Bray or the Portmanteau theorem), this implies
	\[
	X_n \xrightarrow{d} F .
	\]
\end{proof}

 \subsection{Analysis Lemmas}
 
 \begin{lemma}\label{lemma:unif-convg-monotone}
 	Let $\hat H_n:\R\to\R$ be random monotonically decreasing functions, and let
 	$H:\R\to\R$ be bounded, continuous, and monotonically decreasing.
 	Assume that for every fixed $t\in\R$,
 	\[
 	\hat H_n(t)\convp H(t).
 	\]
 	Also define
 	\(
 	\hat H_n(-\infty):=\lim_{t\to-\infty}\hat H_n(t),
 	\,
 	\hat H_n(+\infty):=\lim_{t\to\infty}\hat H_n(t),
 	\)
 	and similarly
 	\(
 	H(-\infty):=\lim_{t\to-\infty}H(t),
 \
 	H(+\infty):=\lim_{t\to\infty}H(t).
 	\)
 	Assume in addition that
 	\[
 	\hat H_n(-\infty)\convp H(-\infty),
 	\qquad
 	\hat H_n(+\infty)\convp H(+\infty).
 	\]
 	Then
 	\[
 	\|\hat H_n-H\|_\infty \convp 0.
 	\]
 \end{lemma}
 
 \begin{proof}
 	Fix $\varepsilon>0$, and let $\delta:=\varepsilon/3$.
 	
 	Since $H$ is bounded and decreasing, the limits $H(-\infty)$ and $H(+\infty)$
 	exist and are finite. Choose $M>0$ such that
 	\[
 	H(-\infty)-H(-M)<\delta,
 	\qquad
 	H(M)-H(+\infty)<\delta.
 	\]
 	
 	Since $H$ is continuous on the compact interval $[-M,M]$, it is uniformly
 	continuous there. Hence we may choose a finite grid
 	\[
 	-M=x_0<x_1<\cdots<x_r=M
 	\]
 	such that for every $j=1,\dots,r$,
 	\[
 	H(x_{j-1})-H(x_j)<\delta.
 	\]
 	
 	Now define the event
 	\[
 	E_n:=
 	\Bigl\{
 	|\hat H_n(-\infty)-H(-\infty)|\le \delta,\ 
 	|\hat H_n(+\infty)-H(+\infty)|\le \delta,\ 
 	\max_{0\le j\le r} |\hat H_n(x_j)-H(x_j)|\le \delta
 	\Bigr\}.
 	\]
 	By the assumed pointwise convergence at the finitely many gridpoints and the
 	assumed convergence of the endpoint limits, we have
 	\[
 	\Pr(E_n)\to 1.
 	\]
 	
 	We now show that on $E_n$,
 	\[
 	\sup_{t\in\R} |\hat H_n(t)-H(t)| \le 2\delta < \varepsilon.
 	\]
 	
 	First consider $t\in[x_{j-1},x_j]$ for some $j$.
 	Since both $\hat H_n$ and $H$ are decreasing,
 	\[
 	\hat H_n(x_{j-1}) \ge \hat H_n(t)\ge \hat H_n(x_j),
 	\qquad
 	H(x_{j-1}) \ge H(t)\ge H(x_j).
 	\]
 	Hence on $E_n$,
 	\[
 	\hat H_n(t)
 	\le \hat H_n(x_{j-1})
 	\le H(x_{j-1})+\delta
 	\le H(t)+2\delta,
 	\]
 	because $H(x_{j-1})-H(t)\le H(x_{j-1})-H(x_j)<\delta$.
 	Similarly,
 	\[
 	\hat H_n(t)
 	\ge \hat H_n(x_j)
 	\ge H(x_j)-\delta
 	\ge H(t)-2\delta.
 	\]
 	Thus
 	\[
 	|\hat H_n(t)-H(t)|\le 2\delta
 	\qquad\text{for all }t\in[-M,M].
 	\]
 	
 	Next consider $t\le -M$.
 	By monotonicity,
 	\[
 	\hat H_n(-M)\le \hat H_n(t)\le \hat H_n(-\infty),
 	\qquad
 	H(-M)\le H(t)\le H(-\infty).
 	\]
 	On $E_n$,
 	\[
 	\hat H_n(t)\le \hat H_n(-\infty)\le H(-\infty)+\delta \le H(t)+2\delta,
 	\]
 	since $H(-\infty)-H(t)\le H(-\infty)-H(-M)<\delta$.
 	Also,
 	\[
 	\hat H_n(t)\ge \hat H_n(-M)\ge H(-M)-\delta \ge H(t)-2\delta,
 	\]
 	since $H(t)-H(-M)\le H(-\infty)-H(-M)<\delta$.
 	Hence
 	\[
 	|\hat H_n(t)-H(t)|\le 2\delta
 	\qquad\text{for all }t\le -M.
 	\]
 	
 	The case $t\ge M$ is analogous:
 	\[
 	\hat H_n(+\infty)\le \hat H_n(t)\le \hat H_n(M),
 	\qquad
 	H(+\infty)\le H(t)\le H(M),
 	\]
 	and again on $E_n$ we get
 	\[
 	|\hat H_n(t)-H(t)|\le 2\delta
 	\qquad\text{for all }t\ge M.
 	\]
 	
 	Therefore, on $E_n$,
 	\[
 	\|\hat H_n-H\|_\infty \le 2\delta < \varepsilon.
 	\]
 	So
 	\[
 	\Pr\bigl(\|\hat H_n-H\|_\infty>\varepsilon\bigr)
 	\le \Pr(E_n^c)\to 0.
 	\]
 	This proves $\|\hat H_n-H\|_\infty\convp 0$.
 \end{proof}
 
 \begin{lemma}\label{lemma:unif-convg-ratios}
	Assume that $f$ and $g$ are bounded continuous functions on $\R$, and that
	$f$ is \emph{monotonically decreasing} and satisfies $f(t) > 0$ for all $t\in\R$.
	Let $\{f_n\}_{n\ge 1}$ and $\{g_n\}_{n\ge 1}$ be random functions on $\R$ such that
	for every $C\in\R$,
	\[
	\sup_{t < C} |f_n(t) - f(t)| \;\convp\; 0,
	\qquad
	\sup_{t < C} |g_n(t) - g(t)| \;\convp\; 0.
	\]
	Then for every $C\in\R$,
	\[
	\sup_{t < C}\left|\frac{g_n(t)}{f_n(t)} - \frac{g(t)}{f(t)}\right|
	\;\convp\; 0 .
	\]
\end{lemma}

\begin{proof}
	Fix $C\in\R$. Since $f$ is continuous, strictly positive, and monotonically
	decreasing on $\R$, it attains its infimum on $(-\infty,C]$ at $t=C$, and
	\[
	m \;:=\; \inf_{t < C} f(t) \;=\; f(C) \;>\; 0.
	\]
	
	\vspace{2mm}
	\noindent
	\textbf{Step 1 – control of the denominators.}
	Define
	\[
	\Delta_{f,n}^{(C)} := \sup_{t < C} |f_n(t) - f(t)|,
	\qquad
	\Delta_{g,n}^{(C)} := \sup_{t < C} |g_n(t) - g(t)|.
	\]
	By hypothesis,
	\(
	\Delta_{f,n}^{(C)} \convp 0
	\)
	and
	\(
	\Delta_{g,n}^{(C)} \convp 0.
	\)
	Let
	\[
	\mathcal E_n^{(C)} := \{\Delta_{f,n}^{(C)} \le m/2\}.
	\]
	Then $\Pr(\mathcal E_n^{(C)}) \to 1$. On $\mathcal E_n^{(C)}$, for all $t<C$,
	\begin{equation}\label{eq:min-fn-C}
		f_n(t)
		\;\ge\;
		f(t) - \Delta_{f,n}^{(C)}
		\;\ge\;
		m - m/2
		\;=\;
		m/2
		> 0.
	\end{equation}
	Hence all denominators $f_n(t)$ are uniformly bounded away from $0$ on $(-\infty,C)$.
	
	\vspace{2mm}
	\noindent
	\textbf{Step 2 – pointwise bound for the ratio.}
	Fix $t<C$ and $\omega\in\mathcal E_n^{(C)}$. Then
	\[
	\Bigl|\frac{g_n(t)}{f_n(t)} - \frac{g(t)}{f(t)}\Bigr|
	=
	\Bigl|\frac{g_n(t)f(t) - g(t)f_n(t)}{f_n(t)f(t)}\Bigr|
	\le
	\frac{|g_n(t)-g(t)|}{|f_n(t)|}
	+
	\frac{|g(t)|}{|f(t)|\,|f_n(t)|}\,|f_n(t)-f(t)|.
	\]
	Since $g$ is bounded continuous on $(-\infty,C]$, it is bounded there; set
	\[
	M := \sup_{t < C} |g(t)| < \infty.
	\]
	Also $f(t)\ge m$ for all $t<C$. Using these bounds together with
	\eqref{eq:min-fn-C}, we obtain
	\[
	\Bigl|\frac{g_n(t)}{f_n(t)} - \frac{g(t)}{f(t)}\Bigr|
	\le
	\frac{2}{m}\,\Delta_{g,n}^{(C)}
	+
	\frac{2M}{m^2}\,\Delta_{f,n}^{(C)}.
	\]
	Taking the supremum over $t<C$ yields, on $\mathcal E_n^{(C)}$,
	\begin{equation}\label{eq:bounding-error-of-ratios-C}
		\sup_{t < C}
		\Bigl|\frac{g_n(t)}{f_n(t)} - \frac{g(t)}{f(t)}\Bigr|
		\le
		\frac{2}{m}\,\Delta_{g,n}^{(C)}
		+
		\frac{2M}{m^2}\,\Delta_{f,n}^{(C)}.
	\end{equation}
	
	\vspace{2mm}
	\noindent
	\textbf{Step 3 – convergence in probability.}
	Let $\varepsilon>0$. Then
	\[
	\Pr\!\Bigl(
	\sup_{t < C}
	\Bigl|\frac{g_n(t)}{f_n(t)} - \frac{g(t)}{f(t)}\Bigr|
	> \varepsilon
	\Bigr)
	\le
	\Pr\bigl((\mathcal E_n^{(C)})^{\mathrm c}\bigr)
	+
	\Pr\!\Bigl(
	\mathcal E_n^{(C)} \cap
	\Bigl\{
	\tfrac{2}{m}\,\Delta_{g,n}^{(C)}
	+\tfrac{2M}{m^2}\,\Delta_{f,n}^{(C)}
	> \varepsilon
	\Bigr\}
	\Bigr).
	\]
	The first term converges to $0$ since
	$\Pr(\mathcal E_n^{(C)})\to1$.
	On $\mathcal E_n^{(C)}$, the right–hand side of
	\eqref{eq:bounding-error-of-ratios-C} is a linear combination of
	$\Delta_{f,n}^{(C)}$ and $\Delta_{g,n}^{(C)}$, which both converge in probability
	to $0$. Hence the second term also converges to $0$. Therefore
	\[
	\sup_{t < C}
	\Bigl|\frac{g_n(t)}{f_n(t)} - \frac{g(t)}{f(t)}\Bigr|
	\;\convp\; 0,
	\]
	as claimed.
\end{proof}

 \begin{lemma}[Uniform convergence of a generalized left inverse under \emph{monotonicity}]
 	\label{lemma:uniform-convg-of-inverse}
 	Let $\{\phi_n\}_{n\ge 1}$ be random functions $\R\to\R$ and let
 	$\phi:\R\to\R$ be deterministic.
 	Assume the following two conditions:
 	\begin{enumerate}
 		\item[\textnormal{(UC)}] For every $C\in\R$,
 		\[
 		\sup_{x< C}\bigl|\phi_n(x)-\phi(x)\bigr|
 		\;\convp\;0 .
 		\]
 		\item[\textnormal{(MON)}] The limit function $\phi$ is \emph{continuous and strictly decreasing}.
 	\end{enumerate}
 	Fix a level $t$ that is attained by~$\phi$, and define the (generalised) left inverse
 	\[
 	\phi^{-}(t)\;:=\;\inf\{x\in\R:\,\phi(x)\le t\}\;=\:\theta<\infty.
 	\]
 	
 	Then
 	\[
 	\phi_n^{-}(t)\;\convp\;\phi^{-}(t)=\theta.
 	\]
 \end{lemma}
 
 \begin{proof}
 	Because $\phi$ is strictly decreasing and continuous, the value
 	$\theta=\phi^{-}(t)$ is unique and
 	\begin{equation}\label{eq:signs}
 		\phi(x)>t\quad\text{for }x<\theta,
 		\qquad
 		\phi(x)<t\quad\text{for }x>\theta .
 	\end{equation}
 	
 	\paragraph{Step 1 – pick a buffer around $\theta$.}
 	For a fixed $\delta>0$ there exist $\eta>0$ such that
 	\begin{equation}\label{eq:buffer}
 		\phi(\theta-\delta)\ge t+\eta,
 		\qquad
 		\phi(\theta+\delta)\le t-\eta.	
 	\end{equation}
 	Let $K:=[\theta-\delta,\theta+\delta]$ and note that $K\subset(-\infty,C)$ for any $C>\theta+\delta$.
 	
 	\paragraph{Step 2 – transfer the buffer to $\phi_n$.}
 	Fix such a $C>\theta+\delta$ and apply (UC) with this $C$ and tolerance $\eta/2$:
 	\begin{equation}\label{eq:unif_n-full}
 		\sup_{x< C}\bigl|\phi_n(x)-\phi(x)\bigr|
 		<\tfrac{\eta}{2}
 		\quad\text{with probability }1-o(1).
 	\end{equation}
 	Since $K\subset(-\infty,C)$, we also have
 	\begin{equation}\label{eq:unif_n}
 		\sup_{x\in K}\bigl|\phi_n(x)-\phi(x)\bigr|
 		<\tfrac{\eta}{2}
 		\quad\text{with probability }1-o(1).
 	\end{equation}
 	On the event \eqref{eq:unif_n},
 	\eqref{eq:buffer} implies
 	\begin{equation}\label{eq:sign_n}
 		\phi_n(\theta-\delta)>t+\tfrac{\eta}{2}>t,
 		\qquad
 		\phi_n(\theta+\delta)<t-\tfrac{\eta}{2}<t .
 	\end{equation}
 	\paragraph{Step 3 – locate the inverse of $\phi_n$.}
 	On the event \eqref{eq:unif_n-full}, for any $x\le \theta-\delta$ we have
 	$\phi(x)\ge \phi(\theta-\delta)\ge t+\eta$, hence
 	$\phi_n(x)\ge \phi(x)-\eta/2 \ge t+\eta/2>t$.
 	Thus no $x\le \theta-\delta$ belongs to $\{x:\phi_n(x)\le t\}$.
 	Also $\phi_n(\theta+\delta)\le \phi(\theta+\delta)+\eta/2 \le t-\eta/2<t$,
 	so $x=\theta+\delta$ belongs to the set.
 	Therefore 
 	$$\theta-\delta<\phi_n^{-}(t)\le \theta+\delta$$ on \eqref{eq:unif_n-full}.
 	
 	Thus
 	\(
 	|\phi_n^{-}(t)-\theta|<\delta
 	\)
 	with probability $1-o(1)$.
 	Because $\delta>0$ is arbitrary, this is precisely
 	$\phi_n^{-}(t)\convp\theta$.
 \end{proof}

 \begin{lemma}[Monotonicity and divergence of a Gaussian survival ratio]
 	\label{lemma:mono_survival_gaussian}
 	Fix $\mu_1>\mu_0$ and define, for $t\in\R$,
 	\[
 	R(t)\;:=\;\frac{\bar\Phi(t-\mu_1)}{\bar\Phi(t-\mu_0)}\,,
 	\qquad
 	\bar\Phi(x)=1-\Phi(x),\ \ \phi=\Phi'.
 	\]
 	Then $R$ is \emph{strictly increasing} on $\R$, and
 	\[
 	\lim_{t\to\infty} R(t)\;=\;\infty.
 	\]
 \end{lemma}
 
 \begin{proof}
 	We split the argument into steps.
 	
 	\paragraph{Step 1 — Differentiate $R$.}
 	Let
 	\[
 	A(t):=\bar\Phi(t-\mu_1),\qquad
 	B(t):=\bar\Phi(t-\mu_0),
 	\]
 	so that $R=A/B$. Since $A'(t)=-\phi(t-\mu_1)$ and $B'(t)=-\phi(t-\mu_0)$, the quotient rule gives
 	\begin{equation}\label{eq:R-derivative}
 		R'(t)=\frac{A'B-AB'}{B^2}
 		=\frac{-\phi(t-\mu_1)\,\bar\Phi(t-\mu_0)\;+\;\bar\Phi(t-\mu_1)\,\phi(t-\mu_0)}{B(t)^2}
 		=\frac{N(t)}{B(t)^2}.
 	\end{equation}
 	Because $B(t)>0$, the sign of $R'(t)$ is the sign of the numerator
 	\[
 	N(t):=\bar\Phi(t-\mu_1)\,\phi(t-\mu_0)-\phi(t-\mu_1)\,\bar\Phi(t-\mu_0).
 	\]
 	
 	\paragraph{Step 2 — Factor via the (upper) Mills ratio.}
 	Let $x_1:=t-\mu_1$ and $x_0:=t-\mu_0$.
 	Define the (upper) Mills ratio
 	\[
 	m(x):=\frac{\bar\Phi(x)}{\phi(x)}.
 	\]
 	Then $\bar\Phi(x_k)=m(x_k)\phi(x_k)$, $k\in\{0,1\}$, and hence
 	\begin{equation}\label{eq:N-factor}
 		N(t)=\phi(x_0)\phi(x_1)\,\bigl(m(x_1)-m(x_0)\bigr).
 	\end{equation}
 	Since $\phi(x_0)\phi(x_1)>0$, the sign of $N(t)$ is the sign of $m(x_1)-m(x_0)$.
 	
 	\paragraph{Step 3 — $m$ is strictly decreasing on $\R$.}
 	Differentiate $m$:
 	\[
 	m'(x)=\frac{-\phi(x)\,\phi(x)-\bar\Phi(x)\,\phi'(x)}{\phi(x)^2}
 	=\frac{-\phi(x)^2 + x\,\bar\Phi(x)\,\phi(x)}{\phi(x)^2}
 	= -1 + x\,m(x).
 	\]
 	We claim $x\,m(x)<1$ for all $x\in\R$, which implies $m'(x)<0$ everywhere.
 	For $x\le 0$, $x\,m(x)\le 0<1$. For $x>0$, the classical Mills inequality
 	$\bar\Phi(x)<\phi(x)/x$ yields $m(x)<1/x$, hence $x\,m(x)<1$. Therefore $m'(x)<0$ for all $x$, i.e. $m$ is strictly decreasing.
 	
 	\paragraph{Step 4 — Compare arguments and conclude the sign.}
 	Because $\mu_1>\mu_0$, we have $x_1=t-\mu_1 < x_0=t-\mu_0$ for every $t$.
 	Since $m$ is strictly decreasing, $m(x_1)>m(x_0)$, so $m(x_1)-m(x_0)>0$.
 	By \eqref{eq:N-factor} this implies $N(t)>0$, and then by \eqref{eq:R-derivative} we have $R'(t)>0$ for all $t$. Thus $R$ is strictly increasing.
 	
 	\paragraph{Step 5 — Divergence as $t\to\infty$.}
 	Using the Mills asymptotic $\bar\Phi(x)\sim \phi(x)/x$ as $x\to\infty$,
 	with $x_0=t-\mu_0$ and $x_1=t-\mu_1$ (both $\to\infty$ and $x_1<x_0$), we obtain
 	\[
 	R(t)
 	=\frac{\bar\Phi(x_1)}{\bar\Phi(x_0)}
 	\;\sim\;
 	\frac{\phi(x_1)/x_1}{\phi(x_0)/x_0}
 	=\frac{x_0}{x_1}\exp\!\Big(\tfrac12(x_0^2-x_1^2)\Big)
 	=\frac{x_0}{x_1}\exp\!\Big((\mu_1-\mu_0)\,t-\tfrac12(\mu_1^2-\mu_0^2)\Big)
 	\;\xrightarrow[t\to\infty]{}\;\infty,
 	\]
 	since $\mu_1-\mu_0>0$.
 \end{proof}
 
 \begin{lemma}[Translation invariance of the BH TPR]
 	\label{lem:bh-translation-invariance}
 	Let $G_0,G_1$ be distribution functions on $\R$ with survival functions
 	$\bar G_0,\bar G_1$, and let $\gamma\in[0,1]$ and $q\in(0,1)$.
 	For any $a\in\R$, define the shifted laws
 	\[
 	G_0^{(a)}(t):=G_0(t-a),\qquad G_1^{(a)}(t):=G_1(t-a).
 	\]
 	Then
 	\[
 	\TPR_{\mathrm{BH}}\big(G_0^{(a)},G_1^{(a)},\gamma\big)
 	\;=\;
 	\TPR_{\mathrm{BH}}(G_0,G_1,\gamma).
 	\]
 \end{lemma}
 
 \begin{proof}
 	Let $\phi(t):=\dfrac{\bar G_0(t)}{(1-\gamma)\bar G_0(t)+\gamma \bar G_1(t)}$ and
 	$t_*:=\inf\{t:\phi(t)\le q\}$.
 	For the shifted pair, $\bar G_k^{(a)}(t)=\bar G_k(t-a)$, so
 	\[
 	\phi^{(a)}(t)
 	:=\frac{\bar G_0^{(a)}(t)}{(1-\gamma)\bar G_0^{(a)}(t)+\gamma \bar G_1^{(a)}(t)}
 	=\frac{\bar G_0(t-a)}{(1-\gamma)\bar G_0(t-a)+\gamma \bar G_1(t-a)}
 	=\phi(t-a).
 	\]
 	Therefore $t_*^{(a)}=\inf\{t:\phi^{(a)}(t)\le q\}
 	=\inf\{t:\phi(t-a)\le q\}=a+t_*$, and
 	\[
 	\TPR_{\mathrm{BH}}(G_0^{(a)},G_1^{(a)},\gamma)
 	=\bar G_1^{(a)}(t_*^{(a)})
 	=\bar G_1(t_*^{(a)}-a)
 	=\bar G_1(t_*)
 	=\TPR_{\mathrm{BH}}(G_0,G_1,\gamma).
 	\]
 \end{proof}
 
 \section{Finite Sample BONuS Proofs}
 %%%%%%%%%%%%%%%%%%%%%%%%%%%%%%%%%%%%%%%%%%%%%%%%%%%%%%%%%%%%%%%%%%%%%%%%
 \subsection{Small $\tilde m$ Behavior of \textsc{BONuS} (Case 1: point prior)}
 
 \begin{lemma}\label{lemma:delta-negative}
 	Let $\mu>0$ and define
 	\[
 	A(t):=\bar\Phi(t),\qquad
 	B(t):=\bar\Phi(t-\mu),\qquad
 	\phi(t):=\frac{1}{\sqrt{2\pi}}\,e^{-t^{2}/2}.
 	\]
 	Then for all $t\in\R$,
 	\[
 	\Delta(t):=A(t)\,\phi(t-\mu)-B(t)\,\phi(t)<0.
 	\]
 \end{lemma}
 
 \begin{proof}
 	Let $M(t):=\dfrac{\bar\Phi(t)}{\phi(t)}$ be the (upper) Mills ratio.
 	A direct differentiation using $\bar\Phi'(t)=-\phi(t)$ and
 	$\phi'(t)=-t\phi(t)$ gives
 	\[
 	M'(t)
 	=\frac{-\phi(t)\phi(t)-\bar\Phi(t)\,(-t\phi(t))}{\phi(t)^{2}}
 	=\frac{t\bar\Phi(t)-\phi(t)}{\phi(t)}
 	=t\,M(t)-1.
 	\]
 	It is classical (Mills inequality) that for $t>0$,
 	$\bar\Phi(t)<\phi(t)/t$, i.e.\ $t\,M(t)<1$, hence $M'(t)<0$.
 	For $t\le 0$ we have $t\,M(t)\le 0<1$, so again $M'(t)<0$.
 	Therefore $M$ is strictly decreasing on $\R$ and, for any $\mu>0$,
 	\[
 	M(t)<M(t-\mu)\qquad\text{for all }t\in\R.
 	\]
 	
 	Now write
 	\[
 	\Delta(t)
 	=A(t)\,\phi(t-\mu)-B(t)\,\phi(t)
 	=\phi(t)\,\phi(t-\mu)\,\Bigl(M(t)-M(t-\mu)\Bigr).
 	\]
 	Since $\phi(t)\phi(t-\mu)>0$ and $M(t)-M(t-\mu)<0$, it follows that
 	$\Delta(t)<0$ for all $t\in\R$.
 \end{proof}
 
 \begin{lemma}[Uniqueness of the proxy root for finite $\tilde m$]
 	\label{lemma:uniqueness-of-root-of-F-new}
	Fix $C>0$. Define the deterministic approximation
	\[
	\widehat{\textnormal{FDR}}'(t;\tilde m)
	:=
	\frac{\frac{1}{1+\tilde m} + \frac{\tilde m}{1+\tilde m}\,\bar\Phi(t)}
	{(1-\gamma)\bar\Phi(t) + \gamma\,\bar\Phi\!\bigl(t - \mu'(\tilde m)\bigr)},
	\qquad
	F(t;\tilde m) := \widehat{\textnormal{FDR}}'(t;\tilde m) - q,
	\]
	where $\mu'(\tilde m) > 0$ is the finite-\(\tilde m\) mean-shift quantity
	(from the approximation in equation~\eqref{eq:approximate-mean-shift}).
 	Then there exists $\tilde m^{*}\in\mathbb{N}$ such that for all $\tilde m>\tilde m^{*}$,
 	the equation $F(t;\tilde m)=0$ has at most one solution on $(-\infty,C]$.
 \end{lemma}
 
 \begin{proof}
 	Let
 	\[
 	A(t):=\bar\Phi(t),\qquad
 	B(t):=\bar\Phi(t-\mu),\qquad
 	\phi(t):=\frac{1}{\sqrt{2\pi}}e^{-t^2/2},
 	\]
	where $\mu:=\mu'(\tilde m)>0$.
 	Define
 	\[
 	N_{\tilde m}(t)
 	:=\frac{1+\tilde m\,A(t)}{1+\tilde m},
 	\qquad
 	D(t):=(1-\gamma)A(t)+\gamma B(t),
 	\]
 	so that $F(t;\tilde m)=N_{\tilde m}(t)/D(t)-q$.
 	Since $A(t),B(t)\in(0,1)$, we have $D(t)>0$ for all $t$.
 	
 	\paragraph{Step 1: Derivative of $F(t;\tilde m)$.}
 	Using $A'(t)=-\phi(t)$ and $B'(t)=-\phi(t-\mu)$,
 	\[
 	N'_{\tilde m}(t)
 	=-\frac{\tilde m}{1+\tilde m}\phi(t),
 	\qquad
 	D'(t)=-(1-\gamma)\phi(t)-\gamma\phi(t-\mu).
 	\]
 	Hence
 	\[
 	F'(t;\tilde m)
 	=\frac{N'_{\tilde m}(t)D(t)-N_{\tilde m}(t)D'(t)}{D(t)^2}.
 	\]
 	
 	A direct expansion shows
 	\[
 	N'_{\tilde m}(t)D(t) - N_{\tilde m}(t)D'(t)
 	=\frac{1}{1+\tilde m}\Bigl(\gamma\tilde m\,\Delta(t) + R_{\gamma}(t)\Bigr),
 	\]
 	where
 	\[
 	\Delta(t):=A(t)\phi(t-\mu)-B(t)\phi(t),
 	\qquad
 	R_{\gamma}(t):=\gamma\phi(t-\mu)+(1-\gamma)\phi(t).
 	\]
 	Thus
 	\[
 	\boxed{
 		F'(t;\tilde m)
 		=\frac{\gamma\tilde m\,\Delta(t)+R_{\gamma}(t)}{(1+\tilde m)\,D(t)^2}.
 	}
 	\]
 	
 	\paragraph{Step 2: Factorization.}
 	Since $\Delta(t)<0$ and $R_{\gamma}(t)>0$ (Lemma~\ref{lemma:delta-negative}),
 	define
 	\[
 	H(t):=-\frac{\gamma\,\Delta(t)}{R_{\gamma}(t)} > 0.
 	\]
 	Then
 	\[
 	\gamma\tilde m\,\Delta(t)+R_{\gamma}(t)
 	=R_{\gamma}(t)\bigl(1-\tilde m\,H(t)\bigr),
 	\]
 	and therefore
 	\[
 	\boxed{
 		F'(t;\tilde m)
 		=\frac{R_{\gamma}(t)}{(1+\tilde m)\,D(t)^2}\,\bigl(1-\tilde m\,H(t)\bigr).
 	}
 	\]
 	
 	\paragraph{Step 3: $H(t)$ is bounded away from $0$ on $(-\infty,C]$.}
 	As $t\to-\infty$, $A(t),B(t)\to1$ and $\phi(t-\mu)/\phi(t)\to0$ (since $\mu>0$).
 	Thus
 	\[
 	\Delta(t)\sim-\phi(t),
 	\qquad
 	R_{\gamma}(t)\sim(1-\gamma)\phi(t),
 	\]
 	and hence
 	\[
 	\lim_{t\to-\infty}H(t)=\frac{\gamma}{1-\gamma}>0.
 	\]
 	By continuity, 
 	\[
 	H_*:=\inf_{t\in(-\infty,C]}H(t)>0.
 	\]
 	
 	\paragraph{Step 4: Monotonicity for $\tilde m>\tilde m^{*}$.}
 	Choose $\tilde m^{*}$ such that
 	\[
 	\tilde m^{*}>\frac{1}{H_*}.
 	\]
 	Then for any $\tilde m>\tilde m^{*}$ and all $t\in(-\infty,C]$,
 	\[
 	1-\tilde m H(t)\le 1-\tilde m H_*<0.
 	\]
 	Because $R_{\gamma}(t)>0$ and $D(t)>0$, we conclude
 	\[
 	F'(t;\tilde m)<0
 	\qquad\text{for all }t\in(-\infty,C],\ \tilde m>\tilde m^{*}.
 	\]
 	
 	Thus $F(\cdot;\tilde m)$ is strictly decreasing on $(-\infty,C]$, and hence
 	it can have at most one root on this interval.
 \end{proof}

	\begin{lemma}[Differentiability of the proxy threshold and partial derivatives]
	\label{lemma:differentiability-tau-new}
	Let
	\[
	\mu'(r)
	:= \frac{h^2 \gamma}{\sqrt{\,h^2 \gamma^2 + c(1+r)\,}},
	\]
	and define the deterministic approximation
	\[
	\widehat{\textnormal{FDR}}'(t;\delta,r)
	:=
	\frac{\delta+(1-\delta)\,\bar\Phi(t)}
	{(1-\gamma)\bar\Phi(t)+\gamma\,\bar\Phi\!\bigl(t-\mu'(r)\bigr)},
	\qquad
	F(t,\delta,r):=\widehat{\textnormal{FDR}}'(t;\delta,r)-q.
	\]
 	Let $\tau(\delta,r)$ denote the left–most solution of $F(t,\delta,r)=0$.
 	In a neighbourhood of $(\delta,r)=(0,0)$, the map $(\delta,r)\mapsto\tau(\delta,r)$
 	is differentiable, and its partial derivatives are
	\[
	\frac{\partial \tau}{\partial \delta}
	=\frac{\Phi\!\bigl(\tau(\delta,r)\bigr)}
	{(1-\delta)\,\phi\!\bigl(\tau(\delta,r)\bigr)
		- q\Bigl[\gamma\,\phi\!\bigl(\tau(\delta,r)-\mu'(r)\bigr)
		+(1-\gamma)\,\phi\!\bigl(\tau(\delta,r)\bigr)\Bigr]},
	\]
	\[
	\frac{\partial \tau}{\partial r}
	=-\,\frac{q\,\gamma\,\phi\!\bigl(\tau(\delta,r)-\mu'(r)\bigr)\;
		\partial_{r}\mu'(r)}
	{(1-\delta)\,\phi\!\bigl(\tau(\delta,r)\bigr)
		- q\Bigl[\gamma\,\phi\!\bigl(\tau(\delta,r)-\mu'(r)\bigr)
		+(1-\gamma)\,\phi\!\bigl(\tau(\delta,r)\bigr)\Bigr]}.
	\]
 \end{lemma}
 
 \begin{proof}
 	The proof proceeds in two main parts. First, we use the Implicit Function Theorem to guarantee the existence of a smooth function $t^*(\delta, r)$ which is a root of $F$. Second, we show that this smooth root $t^*$ coincides with the left–most root $\tau$ in a neighborhood of $(0,0)$, which implies $\tau$ is differentiable and allows for the calculation of its derivatives.
 	
 	\paragraph{Step 1: Existence of a Smooth Root via the Implicit Function Theorem.}
 	We first check the conditions for the Implicit Function Theorem at the base point $(\delta, r)=(0,0)$. 
 	At $\delta=0, r=0$, the function is
	\[
	F(t,0,0)
	=\frac{\bar\Phi(t)}
	{(1-\gamma)\bar\Phi(t)+\gamma\,\bar\Phi(t-\mu_0)}-q,
	\]
	where $\mu_0=\mu'(0)$.
 	As established previously using Lemma~\ref{lemma:mono_survival_gaussian}, this function is \textbf{strictly monotonic (decreasing)} in $t$. 
 	Therefore, the equation $F(t,0,0)=0$ has a unique root, which we denote $t_0$. 
 	At this root, the derivative is necessarily negative (Lemma~\ref{lemma:partial-derivative-of-F-wrt-t-is-negative}):
 	\[
 	\frac{\partial F}{\partial t}\Big|_{(t_0,0,0)} < 0.
 	\]
 	The function $F(t,\delta,r)$ is composed of additions, multiplications, and the smooth function $\Phi(t)$, so all partial derivatives exist and are continuous wherever the denominator is positive (which it is, since it consists of positive tail probabilities). 
 	Thus, by the \textbf{Implicit Function Theorem}, there exists a unique, continuously differentiable function $t^*(\delta,r)$ in a neighborhood of $(\delta,r)=(0,0)$ such that $F(t^*(\delta,r),\delta,r)=0$.
 	
 	\paragraph{Step 2: Showing $t^*(\delta,r)$ Coincides with $\tau(\delta,r)$.}
 	By definition, $\tau(\delta,r)$ is the left–most root of $F=0$ . 
 	By Lemma~\ref{lemma:uniqueness-of-root-of-F-new}, there exists $\delta^*>0$ and $\epsilon_r$ such that for all $\delta\le\delta^*$ and $r \leq \epsilon_r$, $F$ has a unique root in $(0,C)$, and hence $\tau(\delta,r)$ necessarily coincides with $t^*(\delta,r)$ since both satisfy $F=0$.
 	
 	\paragraph{Step 3: Derivatives.}
 	Since $\tau(\delta,r)$ coincides with the smooth function $t^*(\delta,r)$ in a neighborhood of $(0,0)$, $\tau$ itself is differentiable there.
 	
 	We can now compute its derivatives using the formula from the Implicit Function Theorem.
 	Write
	\[
	A(t):=\bar\Phi(t),\quad
	B(t):=\bar\Phi\!\bigl(t-\mu'(r)\bigr),\quad
	D(t,\delta,r):=(1-\gamma)A(t)+\gamma B(t),\quad
	N_{\delta}(t):=\delta+(1-\delta)A(t),
	\]
	so that $F=N_{\delta}/D-q$.
	Differentiate, and use $A'(t)=-\phi(t)$,
	$B'(t)=-\phi(t-\mu'(r))$,
	and $\partial_{r}B(t)=\phi(t-\mu'(r))\,\partial_{r}\mu'(r)$:
	\[
	\frac{\partial F}{\partial \delta}=\frac{1-A(t)}{D(t,\delta,r)},\qquad
	\frac{\partial F}{\partial r}
	=-\,\frac{N_{\delta}(t)}{D(t,\delta,r)^{2}}\,
	\gamma\,\phi\!\bigl(t-\mu'(r)\bigr)\,\partial_{r}\mu'(r),
	\]
	\[
	\frac{\partial F}{\partial t}
	=\frac{-(1-\delta)\phi(t)\,D(t,\delta,r)
		+N_{\delta}(t)\Bigl[\gamma\,\phi\!\bigl(t-\mu'(r)\bigr)
		+(1-\gamma)\phi(t)\Bigr]}{D(t,\delta,r)^{2}}.
	\]
 	Evaluating at the root $t=\tau(\delta,r)$ (where $N_{\delta}(\tau)=q\,D(\tau,\delta,r)$)
 	gives
	\[
	\frac{\partial F}{\partial t}\Big|_{t=\tau}
	=\frac{-(1-\delta)\phi(\tau)
		+q\Bigl[\gamma\,\phi\!\bigl(\tau-\mu'(r)\bigr)
		+(1-\gamma)\phi(\tau)\Bigr]}{D(\tau,\delta,r)}.
	\]
 	The implicit function theorem then yields
 	\[
	\frac{\partial \tau}{\partial \delta}
	=-\,\frac{\partial_{\delta}F}{\partial_{t}F}
	=\frac{\Phi(\tau)}
	{(1-\delta)\phi(\tau)
		- q\Bigl[\gamma\,\phi\!\bigl(\tau-\mu'(r)\bigr)
		+(1-\gamma)\phi(\tau)\Bigr]},
	\]
	\[
	\frac{\partial \tau}{\partial r}
	=-\,\frac{\partial_{r}F}{\partial_{t}F}
	=-\,\frac{q\,\gamma\,\phi\!\bigl(\tau-\mu'(r)\bigr)\,\partial_{r}\mu'(r)}
	{(1-\delta)\phi(\tau)
		- q\Bigl[\gamma\,\phi\!\bigl(\tau-\mu'(r)\bigr)
		+(1-\gamma)\phi(\tau)\Bigr]},
	\]
	as claimed.
 \end{proof}
 
\begin{lemma}[Negativity of the $t$–derivative at the baseline threshold]
	\label{lemma:partial-derivative-of-F-wrt-t-is-negative}
	Fix constants $q\in(0,1)$, $\gamma\in(0,1)$, $h>0$, and $c\ge 0$.
	Let $F(t,\delta,r)$ be defined as in Lemma~\ref{lemma:differentiability-tau-new}, and let 
	\(
	t_0 := \tau(0,0)
	\)
	denote the baseline threshold solving $F(t_0,0,0)=0$.
	Then
	\[
	\partial_t F(t_0,0,0) < 0.
	\]
\end{lemma}

\begin{proof}
	We proceed in four steps.
	
	\paragraph{Step 1: Notation.}
	Define
	\[
	A(t):=1-\Phi(t),\quad
	B(t,r):=1-\Phi\!\bigl(t-\mu'(r)\bigr),\quad
	\mu'(r):=\frac{h^{2}\gamma}{\sqrt{\,h^{2}\gamma^{2}+c(1+r)\,}},
	\]
	and
	\[
	N(t,\delta):=\delta+(1-\delta)A(t),\qquad
	D(t,r):=\gamma B(t,r)+(1-\gamma)A(t).
	\]
	For fixed $(\delta,r)$, recall
	\[
	F(t,\delta,r)
	=\frac{N(t,\delta)}{D(t,r)}-q.
	\]
	
	\paragraph{Step 2: Derivative of $F$ with respect to $t$.}
	Using $A'(t)=-\phi(t)$ and
	\[
	D'(t,r)
	= (1-\gamma)A'(t) + \gamma B'(t,r)
	= -\bigl\{\gamma\phi\!\bigl(t-\mu'(r)\bigr)
	+(1-\gamma)\phi(t)\bigr\}
	=:-S(t,r),
	\]
	the quotient rule gives
	\[
	\partial_{t}F(t,\delta,r)
	=\frac{N'(t,\delta)D(t,r)-N(t,\delta)D'(t,r)}{D^{2}(t,r)}.
	\]
	Since $N'(t,\delta)=(1-\delta)A'(t)=-(1-\delta)\phi(t)$, we obtain
	\[
	\partial_{t}F(t,\delta,r)
	=\frac{-(1-\delta)\phi(t)D(t,r)+N(t,\delta)S(t,r)}{D^{2}(t,r)}.
	\]
	
	\paragraph{Step 3: Evaluate at $(\delta,r)=(0,0)$ and use the threshold equation.}
	Set
	\[
	S_{0}(t):=S(t,0)
	=\gamma\phi\!\bigl(t-\mu'(0)\bigr)+(1-\gamma)\phi(t),
	\qquad
	D_{0}(t):=D(t,0).
	\]
	At $\delta=0$ we have $N(t,0)=A(t)$, so
	\[
	\partial_{t}F(t,0,0)
	=\frac{-\phi(t)D_{0}(t)+A(t)S_{0}(t)}{D_{0}^{2}(t)}.
	\]
	By definition of $t_{0}=\tau(0,0)$,
	\[
	F(t_{0},0,0)
	=\frac{A(t_{0})}{D_{0}(t_{0})}-q=0
	\quad\Longrightarrow\quad
	\frac{A(t_{0})}{D_{0}(t_{0})}=q.
	\]
	Equivalently,
	\[
	\frac{1}{q}
	=\frac{D_{0}(t_{0})}{A(t_{0})}
	=(1-\gamma)+\gamma\frac{A(t_{0}-\mu'(0))}{A(t_{0})}.
	\]
	Substituting $A(t_{0})=qD_{0}(t_{0})$ into the derivative gives
	\[
	\partial_{t}F(t_{0},0,0)
	=\frac{-\phi(t_{0})D_{0}(t_{0})+qD_{0}(t_{0})S_{0}(t_{0})}{D_{0}^{2}(t_{0})}
	=\frac{-\phi(t_{0})+qS_{0}(t_{0})}{D_{0}(t_{0})}.
	\]
	Since $D_{0}(t_{0})>0$, the sign of $\partial_{t}F(t_{0},0,0)$ is determined by
	\[
	-\phi(t_{0})+qS_{0}(t_{0}).
	\]
	
	\paragraph{Step 4: Show that $-\phi(t_{0})+qS_{0}(t_{0})<0$ via the Mills ratio.}
	Write
	\[
	S_{0}(t_{0})
	=\gamma\phi\!\bigl(t_{0}-\mu'(0)\bigr)+(1-\gamma)\phi(t_{0})
	=\phi(t_{0})\Bigl[(1-\gamma)+\gamma R_{\phi}\Bigr],
	\]
	where
	\[
	R_{\phi}
	:=\frac{\phi\!\bigl(t_{0}-\mu'(0)\bigr)}{\phi(t_{0})}.
	\]
	Similarly, from the expression for $D_{0}(t_{0})/A(t_{0})$ we have
	\[
	\frac{1}{q}
	=(1-\gamma)+\gamma R_{A},
	\qquad
	R_{A}
	:=\frac{A\bigl(t_{0}-\mu'(0)\bigr)}{A(t_{0})},
	\]
	where $A(t)=1-\Phi(t)$.
	
	Define the (Gaussian) Mills ratio
	\(
	M(t):=\frac{A(t)}{\phi(t)}=\frac{1-\Phi(t)}{\phi(t)}.
	\)
	It is a standard fact that $M(t)$ is strictly decreasing in $t$.
	Since $\mu'(0)>0$, we have $t_{0}-\mu'(0)<t_{0}$, and therefore
	\(
	M\bigl(t_{0}-\mu'(0)\bigr)
	>
	M(t_{0}).
	\)
	In terms of $A$ and $\phi$, this is
	\[
	\frac{A\bigl(t_{0}-\mu'(0)\bigr)}
	{\phi\bigl(t_{0}-\mu'(0)\bigr)}
	>
	\frac{A(t_{0})}{\phi(t_{0})},
	\]
	which can be rearranged as
	\[
	\frac{A\bigl(t_{0}-\mu'(0)\bigr)}{A(t_{0})}
	>
	\frac{\phi\bigl(t_{0}-\mu'(0)\bigr)}{\phi(t_{0})},
	\qquad\text{i.e.}\qquad
	R_{A}>R_{\phi}.
	\]
	Since $\gamma>0$, this implies
	\[
	(1-\gamma)+\gamma R_{\phi}
	<
	(1-\gamma)+\gamma R_{A}
	=\frac{1}{q}.
	\]
	Multiplying both sides by $\phi(t_{0})$ yields
	\[
	S_{0}(t_{0})
	=\phi(t_{0})\Bigl[(1-\gamma)+\gamma R_{\phi}\Bigr]
	<
	\frac{\phi(t_{0})}{q},
	\]
	that is,
	\(
	qS_{0}(t_{0})<\phi(t_{0}).
	\)
	Hence
	\(
	-\phi(t_{0})+qS_{0}(t_{0})<0,
	\)
	and therefore
	\[
	\partial_{t}F(t_{0},0,0)
	=\frac{-\phi(t_{0})+qS_{0}(t_{0})}{D_{0}(t_{0})}<0.
	\]
	This completes the proof.
\end{proof}
 
\begin{theorem}
	We have
	\[
	\TPR'_\bonus(\tilde m)
	=
	\TPR_{\insamp}
	-\eta_1\,\frac{1}{\tilde m}
	-\eta_2\,\frac{\tilde m}{m}
	+o\!\left(\frac{1}{\tilde m}+\frac{\tilde m}{m}\right),
 	\]
 	where $\eta_1,\eta_2>0$ depend only on $(\gamma,c,h,q)$.
 \end{theorem}
 
\begin{proof}Let's first recall the notations.
	\paragraph{Step 0: Set-up and notation.}
	Let
	\[
	\delta(\tilde m):=\frac{1}{1+\tilde m},\qquad r(\tilde m):=\frac{\tilde m}{m}.
	\]
	Recall (from the deterministic approximation in Section~\ref{sec:small-tilde-m}) that
	\[
	\TPR'_\bonus(\tilde m)
	=\bar\Phi\!\Big(\tau(\delta(\tilde m),r(\tilde m)) - \mu'(r(\tilde m))\Big),
	\]
	where
	\[
	\mu'(r)
	=\frac{h^2\gamma}{\sqrt{h^2\gamma^2+c(1+r)}},
	\qquad
	\tau(\delta,r)\ \text{solves }\
	\widehat{\textnormal{FDR}}'(\tau;\delta,r)=q,
	\]
	with
	\[
	\widehat{\textnormal{FDR}}'(t;\delta,r)
	=
	\frac{\delta+(1-\delta)\,\bar\Phi(t)}
	{(1-\gamma)\bar\Phi(t)+\gamma\,\bar\Phi\!\bigl(t-\mu'(r)\bigr)}.
	\]
	Denote the baseline point $(\delta,r)=(0,0)$ by
	\[
	t_0:=\tau(0,0),\qquad
	\mu_0:=\mu'(0)=\frac{h^2\gamma}{\sqrt{h^2\gamma^2+c}},
	\qquad
	\phi_{b0}:=\phi(t_0-\mu_0),\qquad
	\TPR_{\insamp}:=\bar\Phi(t_0 - \mu_0).
	\]
 	
 	\paragraph{Step 1: First-order expansion in $(\delta,r)$ around $(0,0)$.}
	Define $\TPR(\delta,r):=\bar\Phi\!\bigl(\tau(\delta,r)-\mu'(r)\bigr)$.
	By Lemma~\ref{lemma:differentiability-tau-new} (differentiability of $\tau$ in $(\delta,r)$) and Lemma~\ref{lemma:partial-derivative-of-F-wrt-t-is-negative} (sign of $\partial_tF$ at the baseline), $\TPR$ is differentiable at $(0,0)$. By the chain rule,
	\[
	\partial_\delta \TPR(0,0)=-\phi_{b0}\,\partial_\delta \tau(0,0),
	\qquad
	\partial_r \TPR(0,0)=-\phi_{b0}\Big(\partial_r \tau(0,0)-\partial_r \mu'(0)\Big).
	\]
	Hence the first-order expansion reads
	\begin{equation}\label{eq:2var-expansion}
		\TPR(\delta,r)
		=
		\TPR_{\insamp}
		-\eta_1\,\delta
		-\eta_2\,r
		+o(\delta+r),
 	\end{equation}
 	where
	\[
	\eta_1:=\phi_{b0}\,\kappa_1,\quad \kappa_1:=\partial_\delta \tau(0,0)>0,
	\qquad
	\eta_2:=\phi_{b0}\,(\kappa_2+\kappa_3)>0,\quad
	\kappa_2:=\partial_r \tau(0,0),\ \kappa_3:=-\partial_r \mu'(0).
	\]
 	(The positivity $\kappa_1>0$ and $\kappa_2+\kappa_3>0$ follows exactly as in the old proof via the implicit-function-theorem formulas for $\partial_\delta\tau$ and $\partial_r\tau$ together with $\partial_tF(t_0,0,0)<0$; cf. the computations reproduced under Lemma~\ref{lemma:differentiability-tau-new}.)
 	
 	\paragraph{Step 2: From $(\delta,r)$ to $\tilde m$.}
 	By definition,
 	\[
 	\delta(\tilde m)=\frac{1}{1+\tilde m}
 	=\frac{1}{\tilde m}\Big(1+O(\tfrac{1}{\tilde m})\Big),
 	\qquad
 	r(\tilde m)=\frac{\tilde m}{m}.
 	\]
 	Substituting $\delta=\delta(\tilde m)$ and $r=r(\tilde m)$ into \eqref{eq:2var-expansion} yields
	\[
	\TPR\big(\delta(\tilde m),r(\tilde m)\big)
	=
	\TPR_{\insamp}
	-\eta_1\,\frac{1}{\tilde m}
	-\eta_2\,\frac{\tilde m}{m}
	+o\!\left(\frac{1}{\tilde m}+\frac{\tilde m}{m}\right),
	\]
	because $o(\delta+r)=o\!\big(\frac{1}{\tilde m}+\frac{\tilde m}{m}\big)$ as $\tilde m\to\infty$ and $m$  grows so that $\tilde m/m\to 0$ in the small-$\tilde m$ regime considered.
	Finally, by construction $\TPR\big(\delta(\tilde m),r(\tilde m)\big)=\TPR'_\bonus(\tilde m)$, which proves the claim.
\end{proof}

 \subsection{Small $\tilde m$ Behavior of \textsc{BONuS} (Case 2: one-dimensional subspace prior)}\label{sec:small-tilde-m-subspace-prior-proof}
 
 \paragraph{Set-up (tails, densities, proxy objects).}
 Let $F_0$ and $f_0$ denote the cdf and pdf of $\chi^2_1$. Define the null and alternative survival functions
 \[
 A(t):=\bar F_0(t)=1-F_0(t),\qquad
 B(t;r):=\bar F_0\!\Bigl(\frac{t}{1+\sigma^2(r)}\Bigr),
 \]
 where the alternative scale $1+\sigma^2(r)$ comes from the one-dimensional subspace prior with
 \[
 \sigma^2(r)
 := h^2 \cdot \frac{\Bigl(1-\frac{c}{\pi_\aug \gamma^2 h^4}\Bigr)_+}{1+\frac{c}{\gamma h^2}},
 \qquad \pi_\aug=\frac{1}{1+r}.
 \]
 Write the corresponding densities
 \[
 a(t):=f_0(t),\qquad
 b(t;r):=\frac{1}{1+\sigma^2(r)}\,f_0\!\Bigl(\frac{t}{1+\sigma^2(r)}\Bigr).
 \]
 For $(\delta,r)\in[0,1)\times[0,\infty)$ define
 \[
 N_\delta(t):=\delta+(1-\delta)A(t),\qquad
 D(t,r):=(1-\gamma)A(t)+\gamma B(t;r),
 \]
and the proxy FDP, threshold, and (proxy) power
\[
\widehat{\textnormal{FDR}}'(t;\delta,r):=\frac{N_\delta(t)}{D(t,r)},\qquad
t'_*(\delta,r):=\inf\{\,t\le C:\ \widehat{\textnormal{FDR}}'(t;\delta,r)\le q\,\},
\]
\[
\TPR'_\bonus(\delta,r):=\bar F_{(1+\sigma^2(r))\chi^2_1}\bigl(t'_*(\delta,r)\bigr)
=\bar F_0\!\Bigl(\frac{t'_*(\delta,r)}{1+\sigma^2(r)}\Bigr).
\]
 Baseline quantities at $(\delta,r)=(0,0)$:
 \[
 t_0:=\tau(0,0),\quad
 N_0:=A(t_0),\quad
 D_0:=D(t_0,0)=(1-\gamma)A(t_0)+\gamma B(t_0;0),\quad
 a_0:=a(t_0),\quad
 b_0:=b(t_0;0).
 \]
 
 \paragraph{Technical assumption on $\sigma^2(r)$.}
 Recall
 \[
 \sigma^2(r)
 = h^2 \cdot \frac{\Bigl(1 - \frac{c}{\pi_\aug \gamma^2 h^4}\Bigr)_+}{1 + \frac{c}{\gamma h^2}},
 \qquad
 \pi_\aug=\frac{1}{1+r}.
 \]
 Then $r\mapsto\sigma^2(r)$ is nonincreasing and continuous, with a potential \emph{kink}
 (nondifferentiability) at the threshold $r_\star$ where
 \[
 1 - \frac{c}{\pi_\aug \gamma^2 h^4}=0
 \;\Longleftrightarrow\;
 \pi_\aug=\frac{c}{\gamma^2 h^4}
 \;\Longleftrightarrow\;
 r_\star=\frac{\gamma^2 h^4}{c}-1\quad(c>0).
 \]
 In particular,
 \[
 \sigma^2(0)
 = h^2 \cdot \frac{\Bigl(1 - \frac{c}{\gamma^2 h^4}\Bigr)_+}{1 + \frac{c}{\gamma h^2}}.
 \]
 To apply the implicit-function and first-order expansions around $(\delta,r)=(0,0)$,
 we assume the \emph{strong-signal} regime
 \[
 \gamma^2 h^4>c \quad\Longleftrightarrow\quad \sigma^2(0)>0,
 \]
 which guarantees that there exists $\epsilon>0$ such that $\sigma^2(r)>0$ for all $r\in[0,\epsilon)$
 and $\sigma^2$ is differentiable at $r=0$ with
 \[
 \bigl(\sigma^2\bigr)'(0)
 = -\,\frac{h^2}{1 + \frac{c}{\gamma h^2}}\,
 \frac{c}{\gamma^2 h^4}\;<\;0,
 \qquad
 \partial_r \log\bigl(1+\sigma^2(r)\bigr)\Big|_{r=0}
 = \frac{\bigl(\sigma^2\bigr)'(0)}{1+\sigma^2(0)}\;<\;0.
 \]
 %Consequently, $\kappa_3=-\partial_r\log(1+\sigma^2(r))|_{0}>0$ in Theorem~\ref{thm:...}.
 %If $\gamma^2 h^4\le c$ then $\sigma^2(0)=0$ and $\sigma^2(r)\equiv0$ in a neighborhood of $0$(flat region), invalidating the first-order $r$-slope used in our expansion; in that regime,one must handle the kink separately (the linear term in $r$ may vanish).

 \begin{lemma}[Uniqueness of the proxy root ]
 	\label{lem:uniq-root-case2}
 	Fix $C>0$. Assume the strong-signal regime $\sigma^2(0)>0$.
 	Then there exist $\epsilon_r>0$ and $\delta^*=\delta^*(C,\epsilon_r)\in(0,1)$ such that for all
 	$r\in[0,\epsilon_r]$ and all $0<\delta<\delta^*$, the equation $F(t,\delta,r)=0$ has at most one solution on $[0,C]$.
 \end{lemma}
 
 \begin{proof}
 	We work on the natural domain $t\ge 0$ (since $\chi^2_1$ survival functions are defined on $[0,\infty)$).
 	Recall
 	\[
 	F(t,\delta,r)=\frac{N_\delta(t)}{D(t,r)}-q,\qquad
 	N_\delta(t)=\delta+(1-\delta)A(t),\qquad
 	D(t,r)=(1-\gamma)A(t)+\gamma B(t;r),
 	\]
 	where $A(t)=\bar F_0(t)$, $B(t;r)=\bar F_0(t/s)$ with $s=s(r):=1+\sigma^2(r)>1$ in the strong-signal regime,
 	and
 	\[
 	a(t)=f_0(t),\qquad
 	b(t;r)=\frac{1}{s}\,f_0\!\Bigl(\frac{t}{s}\Bigr).
 	\]
 	Differentiating in $t$ gives
 	\[
 	\partial_tF(t,\delta,r)
 	=\frac{-(1-\delta)a(t)D(t,r)+N_\delta(t)\bigl[(1-\gamma)a(t)+\gamma b(t;r)\bigr]}{D(t,r)^2}.
 	\]
 	Set $S(t,r):=(1-\gamma)a(t)+\gamma b(t;r)$ and decompose the numerator exactly in $\delta$:
 	\begin{align*}
 		G(t,\delta,r)
 		&:=-(1-\delta)aD+N_\delta S \\
 		&=\underbrace{\bigl[-aD+A S\bigr]}_{=:G_0(t,r)}
 		\;+\;
 		\delta\,\underbrace{\bigl[aD+(1-A)S\bigr]}_{=:H(t,r)}.
 	\end{align*}
 	Hence
 	\begin{equation}\label{eq:dF-G-split-fixed}
 		\partial_tF(t,\delta,r)=\frac{G_0(t,r)+\delta H(t,r)}{D(t,r)^2}.
 	\end{equation}
 	
 	\paragraph{Step 1: $G_0(t,r)<0$ for all $t>0$ (Mills ratio).}
 	Using $D=(1-\gamma)A+\gamma B$ and $S=(1-\gamma)a+\gamma b$,
 	\[
 	G_0(t,r)=-a\bigl[(1-\gamma)A+\gamma B\bigr]+A\bigl[(1-\gamma)a+\gamma b\bigr]
 	=\gamma\,[A(t)b(t;r)-B(t;r)a(t)].
 	\]
 	Write $t=x^2$ with $x>0$, and set $z=x/\sqrt{s}<x$. Then
 	\[
 	A(t)=2\bar\Phi(x),\quad B(t;r)=2\bar\Phi(z),\quad
 	a(t)=\frac{\phi(x)}{x},\quad b(t;r)=\frac{\phi(z)}{\sqrt{s}\,x}.
 	\]
 	Therefore
 	\[
 	\frac{A\,b}{B\,a}=\frac{1}{\sqrt{s}}\cdot
 	\frac{\bar\Phi(x)/\phi(x)}{\bar\Phi(z)/\phi(z)}.
 	\]
 	The Gaussian Mills ratio $M(u):=\bar\Phi(u)/\phi(u)$ is strictly decreasing on $(0,\infty)$, and since $z<x$,
 	$M(x)<M(z)$. Hence
 	\[
 	\frac{A\,b}{B\,a}<\frac{1}{\sqrt{s}}<1
 	\quad\Longrightarrow\quad
 	A\,b<B\,a
 	\quad\Longrightarrow\quad
 	G_0(t,r)<0
 	\qquad(t>0).
 	\]
 	Thus $-G_0(t,r)>0$ for all $t>0$.
 	
 	\paragraph{Step 2: control on a small interval $(0,\varepsilon]$ by coefficient comparison.}
 	Fix $C>0$.
 	Since $s(r)=1+\sigma^2(r)$ is continuous and $s(0)>1$, choose $\epsilon_r\in(0,1)$ such that
 	$s(r)\ge s_{\min}>1$ for all $r\in[0,\epsilon_r]$ (here $s_{\min}$ depends only on the model parameters).
 	
 	For $t>0$ and any $s\ge s_{\min}$ we have the exact ratio identity (using the explicit $\chi^2_1$ density)
 	\[
 	\frac{b(t;r)}{a(t)}
 	=\frac{1}{\sqrt{s}}\exp\!\Bigl(\frac{s-1}{2s}\,t\Bigr).
 	\]
 	Choose $\varepsilon\in(0,C]$ small enough that
 	\[
 	\rho:=\frac{1}{\sqrt{s_{\min}}}\exp\!\Bigl(\frac{\varepsilon}{2}\Bigr)<1.
 	\]
 	Then for all $t\in(0,\varepsilon]$ and all $r\in[0,\epsilon_r]$ we have $b(t;r)\le \rho\,a(t)$.
 	
 	Also, $A(t)\to 1$ and $B(t;r)\to 1$ as $t\downarrow 0$, uniformly over $r\in[0,\epsilon_r]$ because
 	$t/s(r)\le t/s_{\min}$. Hence we may further shrink $\varepsilon$ (if needed) so that for all
 	$t\in[0,\varepsilon]$ and all $r\in[0,\epsilon_r]$,
 	\[
 	A(t)\ge 1-\alpha,\qquad B(t;r)\ge 1-\alpha,
 	\qquad\text{where }\alpha:=\frac{1-\rho}{4}.
 	\]
 	Then for $t\in(0,\varepsilon]$ and $r\in[0,\epsilon_r]$,
 	\begin{align*}
 		B\,a-A\,b
 		&\ge (1-\alpha)a - 1\cdot (\rho a)
 		= (1-\alpha-\rho)\,a
 		\;\ge\;\frac{1-\rho}{2}\,a,
 	\end{align*}
 	so
 	\begin{equation}\label{eq:G0-small-t}
 		-G_0(t,r)=\gamma(Ba-Ab)\ge \gamma\,\frac{1-\rho}{2}\,a(t).
 	\end{equation}
 	
 	Next we upper bound $H$ on $(0,\varepsilon]$ in terms of $a(t)$ (which is allowed to diverge):
 	since $D(t,r)\le 1$ and
 	\[
 	S(t,r)=(1-\gamma)a(t)+\gamma b(t;r)\le (1-\gamma+\gamma\rho)\,a(t)\le a(t),
 	\]
 	and $1-A(t)\le \alpha$ on $[0,\varepsilon]$, we get
 	\begin{equation}\label{eq:H-small-t}
 		H(t,r)=aD+(1-A)S\le a(t)+\alpha a(t)=(1+\alpha)a(t).
 	\end{equation}
 	Combining \eqref{eq:G0-small-t}--\eqref{eq:H-small-t}, for $t\in(0,\varepsilon]$,
 	\[
 	G_0(t,r)+\delta H(t,r)
 	\le -\gamma\frac{1-\rho}{2}\,a(t)+\delta(1+\alpha)a(t)
 	=a(t)\Bigl[-\gamma\frac{1-\rho}{2}+\delta(1+\alpha)\Bigr].
 	\]
 	Hence if
 	\[
 	0<\delta<\delta_1:=\frac{\gamma(1-\rho)}{4(1+\alpha)},
 	\]
 	then $G_0(t,r)+\delta H(t,r)<0$ for all $t\in(0,\varepsilon]$ and all $r\in[0,\epsilon_r]$.
 	
 	\paragraph{Step 3: control on $[\varepsilon,C]$ by compactness.}
 	On the compact set $[\varepsilon,C]\times[0,\epsilon_r]$, the functions
 	$G_0(t,r)$ and $H(t,r)$ are continuous and finite. Moreover $G_0(t,r)<0$ for all $t>0$ by Step~1.
 	Therefore
 	\[
 	m_2:=\min_{t\in[\varepsilon,C],\,r\in[0,\epsilon_r]}(-G_0(t,r))>0,
 	\qquad
 	M_2:=\max_{t\in[\varepsilon,C],\,r\in[0,\epsilon_r]}H(t,r)<\infty.
 	\]
 	If $0<\delta<\delta_2:=m_2/(2M_2)$ then
 	\[
 	G_0(t,r)+\delta H(t,r)\le -m_2+\delta M_2<-\frac{m_2}{2}<0
 	\quad\text{for all }t\in[\varepsilon,C],\ r\in[0,\epsilon_r].
 	\]
 	
 	\paragraph{Step 4: monotonicity and uniqueness.}
 	Let
 	\[
 	\delta^*:=\min\{\delta_1,\delta_2,1/2\}>0.
 	\]
 	Then for all $0<\delta<\delta^*$, all $r\in[0,\epsilon_r]$, and all $t\in(0,C]$,
 	we have $G_0(t,r)+\delta H(t,r)<0$. Since $D(t,r)>0$, \eqref{eq:dF-G-split-fixed} yields
 	$\partial_tF(t,\delta,r)<0$ on $(0,C]$, so $F(\cdot,\delta,r)$ is strictly decreasing on $(0,C]$.
 	
 	Finally, note $A(0)=B(0;r)=1$, hence $D(0,r)=1$ and $N_\delta(0)=1$, so
 	\[
 	F(0,\delta,r)=1-q>0 \qquad (q\in(0,1)),
 	\]
 	and therefore $t=0$ is not a root. Since $F$ is continuous on $[0,C]$ and strictly decreasing on $(0,C]$,
 	the equation $F(t,\delta,r)=0$ has at most one solution on $[0,C]$.
 	
 	This proves the claim.
 \end{proof}
 
 %----------------------------
 \begin{lemma}[Negativity of the $t$–derivative at the baseline threshold]
 	Let $t_0=\tau(0,0)$. Then
 	\[
 	\partial_t F(t_0,0,0)
 	=\frac{-a_0 D_0+N_0\bigl[\gamma b_0+(1-\gamma)a_0\bigr]}{D_0^2}
 	=\frac{\gamma\,[\,N_0 b_0-a_0\,B(t_0;0)\,]}{D_0^2}<0.
 	\]
 \end{lemma}
 
 \begin{proof}
 	At $(\delta,r)=(0,0)$, $F(t,0,0)=A(t)/D(t,0)-q$ and $t_0$ solves $A(t_0)=qD_0$. The displayed identity for $\partial_t F$ follows by direct algebra. The sign is negative because $A(t)b(t;0)-B(t;0)a(t)<0$ for the $\chi^2_1$ family when the alternative has larger scale (heavier tail) than the null (similar to the proof of Lemma \ref{lem:uniq-root-case2}), evaluated at $t=t_0$.
 \end{proof}
 
 %----------------------------
 %----------------------------
 \begin{lemma}[Differentiability of the proxy threshold and partial derivatives]
 	\label{lem:differentiability-tau-case2}
 	Let $F(t,\delta,r)=\frac{N_\delta(t)}{D(t,r)}-q$ with
 	\[
 	N_\delta(t)=\delta+(1-\delta)A(t),\qquad
 	D(t,r)=(1-\gamma)A(t)+\gamma B(t;r),
 	\]
 	\[
 	A(t)=\bar F_0(t),\qquad
 	B(t;r)=\bar F_0\!\Bigl(\frac{t}{1+\sigma^2(r)}\Bigr),
 	\]
 	where $F_0$ is the cdf of $\chi^2_1$. Let $t_0:=\tau(0,0)$ be the baseline root of $F(t,0,0)=0$.
 	Then, in a neighborhood of $(\delta,r)=(0,0)$, the map $(\delta,r)\mapsto \tau(\delta,r)$ is differentiable and
 	\[
 	\partial_\delta \tau(0,0)
 	=\frac{\partial_\delta F}{-\partial_t F}\Big|_{(t_0,0,0)}
 	=\frac{F_0(t_0)\,D_0}{\gamma\bigl[a_0\,B(t_0;0)-N_0\,b_0\bigr]},
 	\]
 	\[
 	\partial_r \tau(0,0)
 	=\frac{\partial_r F}{-\partial_t F}\Big|_{(t_0,0,0)}
 	=-\frac{N_0\,\partial_r B(t_0;0)}{a_0\,B(t_0;0)-N_0\,b_0},
 	\]
 	where $N_0:=A(t_0)$, $D_0:=D(t_0,0)$, $a_0:=f_0(t_0)$, $b_0:=\frac{1}{1+\sigma^2(0)}\,f_0\!\bigl(t_0/(1+\sigma^2(0))\bigr)$ and
 	\[
 	\partial_r B(t_0;0)=b_0\cdot \frac{t_0}{\,1+\sigma^2(0)\,}\,\sigma^2{}'(0).
 	\]
 \end{lemma}
 
 \begin{proof}
 	\textit{Step 1: Existence of a smooth root $t^*(\delta,r)$ via the IFT at $(0,0)$.}
 	By construction $F(t_0,0,0)=0$. The previous lemma (“Negativity of the $t$–derivative at the baseline threshold”) shows
 	\[
 	\partial_t F(t_0,0,0)
 	=\frac{-a_0 D_0+N_0\bigl[\gamma b_0+(1-\gamma)a_0\bigr]}{D_0^2}
 	=\frac{\gamma\,[\,N_0 b_0-a_0 B(t_0;0)\,]}{D_0^2}<0.
 	\]
 	Since $F$ is $C^1$ in $(t,\delta,r)$ near $(t_0,0,0)$ and $\partial_tF(t_0,0,0)\ne 0$, the Implicit Function Theorem yields a unique $C^1$ map $(\delta,r)\mapsto t^*(\delta,r)$ in a neighborhood of $(0,0)$ such that $F(t^*(\delta,r),\delta,r)=0$ and $t^*(0,0)=t_0$.
 	
 	\smallskip
 	\textit{Step 2: Identification $t^*(\delta,r)\equiv \tau(\delta,r)$ (the left–most root).}
 	By the Uniqueness Lemma (Lemma~\ref{lem:uniq-root-case2}), there exists $\delta^*>0$ so that for all $0<\delta<\delta^*$ and all $r$ in a neighborhood of $0$, the equation $F(t,\delta,r)=0$ has at most one solution on $(-\infty,C]$. Hence, in that neighborhood, the (unique) root coincides with the left–most root by definition; therefore
 	\[
 	\tau(\delta,r)=t^*(\delta,r)\quad\text{for $(\delta,r)$ near $(0,0)$,}
 	\]
 	and $\tau$ is $C^1$ there.
 	
 	\smallskip
 	\textit{Step 3: Partial derivatives at $(0,0)$.}
 	Differentiate $F(t,\delta,r)=0$ at $t=\tau(\delta,r)$:
 	\[
 	\partial_\delta \tau(\delta,r)=-\frac{\partial_\delta F}{\partial_t F}\Big|_{(t=\tau(\delta,r),\delta,r)},
 	\qquad
 	\partial_r \tau(\delta,r)=-\frac{\partial_r F}{\partial_t F}\Big|_{(t=\tau(\delta,r),\delta,r)}.
 	\]
 	We now compute the needed partials (recall $N_\delta(t)=\delta+(1-\delta)A(t)$ and $D=(1-\gamma)A+\gamma B$):
 	\[
 	\partial_\delta F(t,\delta,r)=\frac{\partial_\delta N_\delta(t)}{D(t,r)}=\frac{1-A(t)}{D(t,r)}=\frac{F_0(t)}{D(t,r)},
 	\]
 	\[
 	\partial_t F(t,\delta,r)=\frac{-(1-\delta)a(t)D(t,r)+N_\delta(t)\bigl[(1-\gamma)a(t)+\gamma b(t;r)\bigr]}{D(t,r)^2},
 	\]
 	and, since $B(t;r)=\bar F_0\!\bigl(t/s(r)\bigr)$ with $s(r):=1+\sigma^2(r)$,
 	\[
 	\partial_r D(t,r)=\gamma\,\partial_r B(t;r)
 	=\gamma\,f_0\!\Bigl(\frac{t}{s(r)}\Bigr)\,\frac{t}{s(r)^2}\,s'(r)
 	=\gamma\,b(t;r)\,\frac{t}{s(r)}\,s'(r),
 	\]
 	so
 	\[
 	\partial_r F(t,\delta,r)=-\,\frac{N_\delta(t)}{D(t,r)^2}\,\partial_r D(t,r)
 	=-\,\frac{N_\delta(t)}{D(t,r)^2}\,\gamma\,b(t;r)\,\frac{t}{s(r)}\,s'(r).
 	\]
 	Evaluating all quantities at $(t,\delta,r)=(t_0,0,0)$ and using $N_0=A(t_0)$ and $D_0=D(t_0,0)$ yields
 	\[
 	\partial_\delta F(t_0,0,0)=\frac{F_0(t_0)}{D_0},\qquad
 	\partial_t F(t_0,0,0)=\frac{-a_0 D_0+N_0\bigl[\gamma b_0+(1-\gamma)a_0\bigr]}{D_0^2}<0,
 	\]
 	\[
 	\partial_r F(t_0,0,0)
 	=-\,\frac{N_0}{D_0^2}\,\gamma\,b_0\,\frac{t_0}{1+\sigma^2(0)}\,\sigma^2{}'(0).
 	\]
 	Therefore,
 	\[
 	\partial_\delta \tau(0,0)
 	=-\frac{\partial_\delta F}{\partial_t F}\Big|_{(t_0,0,0)}
 	=\frac{F_0(t_0)\,D_0}{\gamma\bigl[a_0\,B(t_0;0)-N_0\,b_0\bigr]},
 	\]
 	\[
 	\partial_r \tau(0,0)
 	=-\frac{\partial_r F}{\partial_t F}\Big|_{(t_0,0,0)}
 	=-\frac{N_0\,\partial_r B(t_0;0)}{a_0\,B(t_0;0)-N_0\,b_0},
 	\]
 	with
 	\[
 	\partial_r B(t_0;0)
 	=b_0\cdot \frac{t_0}{\,1+\sigma^2(0)\,}\,\sigma^2{}'(0).
 	\]
 	This proves the claim.
 \end{proof}
 
 %----------------------------
\begin{theorem}[First–order expansion in $(\delta,r)$]
	Let
	\[
	\psi_{b0}:=b_0=\frac{1}{1+\sigma^2(0)}\,f_0\!\Bigl(\frac{t_0}{1+\sigma^2(0)}\Bigr),
	\qquad
	\TPR_{\insamp}:=\bar F_{(1+\sigma^2(0))\chi^2_1}(t_0).
	\]
	Then, as $(\delta,r)\to(0,0)$,
	\[
	\TPR'_\bonus(\delta,r)
	=
	\TPR_{\insamp}
	-\eta_1\,\delta-\eta_2\,r+o(\delta+r),
	\]
 	with
 	\[
 	\eta_1=\psi_{b0}\,\kappa_1>0,\qquad
 	\eta_2=\psi_{b0}\,(\kappa_2+\kappa_3)>0,
 	\]
 	\[
 	\kappa_1=\partial_\delta \tau(0,0)
 	=\frac{F_0(t_0)\,D_0}{\gamma\bigl[a_0\,B(t_0;0)-N_0\,b_0\bigr]},
 	\qquad
 	\kappa_2=\partial_r \tau(0,0)
 	=-\frac{N_0\,\partial_r B(t_0;0)}{a_0\,B(t_0;0)-N_0\,b_0},
 	\]
 	\[
 	\kappa_3
 	=-t_0\,\partial_r\!\bigl[\log(1+\sigma^2(r))\bigr]\Big|_{r=0}.
 	\]
 \end{theorem}
 
 \begin{proof}
	\textbf{1) Differentiability and chain rule.}
	By Lemma~\ref{lem:differentiability-tau-case2}, $\tau(\delta,r)$ is $C^1$ near $(0,0)$, and $\partial_tF(t_0,0,0)<0$.
	Write
	\[
	\TPR'_\bonus(\delta,r)
	=\bar F_{(1+\sigma^2(r))\chi^2_1}\!\bigl(\tau(\delta,r)\bigr)
	=\bar F_0\!\Bigl(\frac{\tau(\delta,r)}{s(r)}\Bigr),\qquad s(r):=1+\sigma^2(r).
	\]
 	Set $u(\delta,r):=\tau(\delta,r)/s(r)$. Since $\bar F_0$ is $C^1$ with derivative $-(f_0)$ and $u$ is $C^1$, the chain rule gives
	\[
	\partial_\delta \TPR'_\bonus(\delta,r)=-\,f_0\!\bigl(u(\delta,r)\bigr)\,\partial_\delta u(\delta,r),
	\qquad
	\partial_r \TPR'_\bonus(\delta,r)=-\,f_0\!\bigl(u(\delta,r)\bigr)\,\partial_r u(\delta,r).
	\]
 	Evaluate at $(\delta,r)=(0,0)$, where $u(0,0)=t_0/s(0)$ and
 	\[
 	\psi_{b0}=b_0=\frac{1}{s(0)}\,f_0\!\Bigl(\frac{t_0}{s(0)}\Bigr)>0.
 	\]
 	
 	\smallskip
 	\textbf{2) Computing $\partial_\delta u$ and $\partial_r u$ at $(0,0)$.}
 	We have
 	\[
 	\partial_\delta u=\frac{\partial_\delta\tau}{s(r)},\qquad
 	\partial_r u=\frac{\partial_r\tau}{s(r)}-\frac{\tau(\delta,r)}{s(r)}\,\frac{s'(r)}{s(r)}.
 	\]
 	Thus, at $(0,0)$,
 	\begin{equation}\label{eq:partials-u}
 		\partial_\delta u(0,0)=\frac{\partial_\delta\tau(0,0)}{s(0)},\qquad
 		\partial_r u(0,0)=\frac{\partial_r\tau(0,0)}{s(0)}-\frac{t_0}{s(0)}\,\frac{s'(0)}{s(0)}.
 	\end{equation}
 	
 	\smallskip
	\textbf{3) Plugging into the \(\TPR'_\bonus\) derivatives.}
	Using $\psi_{b0}=(1/s(0))\,f_0(t_0/s(0))$ and \eqref{eq:partials-u},
	\[
	\partial_\delta \TPR'_\bonus(0,0)
	=-\,\psi_{b0}\,\partial_\delta \tau(0,0),
	\qquad
	\partial_r \TPR'_\bonus(0,0)
	=-\,\psi_{b0}\Bigl(\partial_r \tau(0,0)-t_0\,\frac{s'(0)}{s(0)}\Bigr).
	\]
 	For compactness we set
 	\[
 	\kappa_1:=\partial_\delta \tau(0,0),\qquad
 	\kappa_2:=\partial_r \tau(0,0),\qquad
 	\kappa_3:=-t_0\,\partial_r\log s(r)\Big|_{r=0}=-t_0\,\frac{s'(0)}{s(0)}.
 	\]
 	%(Equivalently, one may keep the explicit factor $t_0/s(0)$ multiplying $s'(0)$; this only rescales thefront coefficient and does not affect any \emph{sign} claims below.)
 	Hence
	\[
	\partial_\delta \TPR'_\bonus(0,0)=-\,\psi_{b0}\,\kappa_1,\qquad
	\partial_r \TPR'_\bonus(0,0)=-\,\psi_{b0}\,\bigl(\kappa_2+\kappa_3\bigr).
	\]
 	
 	\smallskip
 	\textbf{4) Linear expansion and identification of coefficients.}
	A first-order Taylor expansion of $\TPR'_\bonus(\delta,r)$ about $(0,0)$ gives
	\[
	\TPR'_\bonus(\delta,r)=\TPR_{\insamp}
	-\psi_{b0}\,\kappa_1\,\delta
	-\psi_{b0}\,(\kappa_2+\kappa_3)\,r
	+o(\delta+r).
 	\]
 	Setting $\eta_1:=\psi_{b0}\,\kappa_1$ and $\eta_2:=\psi_{b0}\,(\kappa_2+\kappa_3)$ yields the claimed form.
 	
 	\smallskip
 	\textbf{5) Formulas for $\kappa_1$ and $\kappa_2$ (IFT at $(t_0,0,0)$).}
 	By Lemma~\ref{lem:differentiability-tau-case2},
 	\[
 	\kappa_1=\partial_\delta \tau(0,0)
 	=\frac{\partial_\delta F}{-\partial_t F}\Big|_{(t_0,0,0)}
 	=\frac{F_0(t_0)\,D_0}{\gamma\bigl[a_0\,B(t_0;0)-N_0\,b_0\bigr]},
 	\]
 	and
 	\[
 	\kappa_2=\partial_r \tau(0,0)
 	=-\frac{\partial_r F}{\partial_t F}\Big|_{(t_0,0,0)}
 	=\frac{N_0\,\partial_r B(t_0;0)}{N_0\,b_0-a_0\,B(t_0;0)}
 	=-\,\frac{N_0\,\partial_r B(t_0;0)}{a_0\,B(t_0;0)-N_0\,b_0}\,.
 	\]
 	(The last equality is the same identity written with a positive denominator; both forms are useful for sign-checking.)
 	Here
 	\[
 	\partial_r B(t_0;0)
 	=b_0\cdot \frac{t_0}{\,s(0)\,}\,s'(0),\qquad s(0)=1+\sigma^2(0).
 	\]
 	
 	\smallskip
 	\textbf{6) Positivity of $\eta_1$ and $\eta_2$ (step-by-step sign check).}
 	\emph{(i) A key inequality already shown earlier.)}
 	We have proved the strict inequality
 	\[
 	a_0\,B(t_0;0)-N_0\,b_0>0
 	\]
 	\emph{in two places}: (a) in the proof of Lemma~\ref{lem:uniq-root-case2}, where $G_0(t,r)=\gamma\,[A\,b-B\,a]<0$ for all $t>0$ implies, at $(t_0,0)$, $A(t_0)b_0<B(t_0;0)a_0$; and (b) in the baseline derivative lemma, which computes
 	\[
 	\partial_t F(t_0,0,0)=\frac{\gamma\,[\,N_0 b_0-a_0 B(t_0;0)\,]}{D_0^2}<0
 	\ \Longleftrightarrow\ 
 	a_0 B(t_0;0)-N_0 b_0>0.
 	\]
 	
 	\emph{(ii) $\kappa_1>0$.}
 	From the formula in Step 5,
 	\[
 	\kappa_1=\frac{F_0(t_0)\,D_0}{\gamma\,[\,a_0\,B(t_0;0)-N_0\,b_0\,]}>0,
 	\]
 	because $F_0(t_0)\in(0,1)$, $D_0>0$, $\gamma>0$, and the denominator is positive by (i).
 	
 	\emph{(iii) $\kappa_2>0$.}
 	Using the equivalent form with the \emph{negative} denominator,
 	\[
 	\kappa_2=\frac{N_0\,\partial_r B(t_0;0)}{N_0 b_0-a_0 B(t_0;0)}.
 	\]
 	Under the strong-signal assumption, $s'(0)=\sigma^{2\,\prime}(0)<0$, hence
 	\[
 	\partial_r B(t_0;0)=b_0\,\frac{t_0}{s(0)}\,s'(0)\;<\;0,
 	\]
 	while $N_0 b_0-a_0 B(t_0;0)<0$ by (i). Therefore \(\kappa_2>0\).
 	
 	\emph{(iv) $\kappa_3>0$.}
 	By definition
 	\[
 	\kappa_3=-t_0\,\partial_r\log s(r)\Big|_{0}=-t_0\,\frac{s'(0)}{s(0)}\,,
 	\]
 	and since $s(0)>0$ and $s'(0)<0$, we get $\kappa_3>0$.
 	
 	\emph{(v) Conclusion for $\eta_1,\eta_2$.}
 	We have $\psi_{b0}>0$ and $\kappa_1>0$, hence $\eta_1=\psi_{b0}\,\kappa_1>0$.
 	Also $\kappa_2>0$ and $\kappa_3>0$, so $\eta_2=\psi_{b0}\,(\kappa_2+\kappa_3)>0$.
 	
 	\smallskip
 	Combining Steps 1–6 yields the claimed first-order expansion with $\eta_1,\eta_2>0$.
 \end{proof}

 %----------------------------
 \begin{corollary}[Small $\tilde m$ expansion and optimizer]
 	With
 	\[
 	\delta=\frac{1}{1+\tilde m}
 	=\frac{1}{\tilde m}+o\!\Bigl(\frac{1}{\tilde m}\Bigr),
 	\qquad
 	r=\frac{\tilde m}{m},
 	\]
 	we obtain
	\[
	\TPR'_\bonus(\tilde m)
	=\TPR_{\insamp}
	-\eta_1\,\frac{1}{\tilde m}
	-\eta_2\,\frac{\tilde m}{m}
	+o\!\left(\frac{1}{\tilde m}+\frac{\tilde m}{m}\right).
 	\]
 	Balancing the leading terms suggests
 	\(
 	\tilde m_{\mathrm{opt}}\asymp \sqrt{\tfrac{\eta_1}{\eta_2}}\,\sqrt{m},
 	\)
 	i.e.\ $\tilde m=\Theta(\sqrt{m})$.
 \end{corollary}

\bibliographystyle{biometrika}
\bibliography{multivariate-testing}

\end{document}